\documentclass[11pt,twoside]{amsart}
\usepackage[utf8]{inputenc}
\usepackage{hyperref, amsmath, amsthm, amssymb, amsfonts, pifont, bbm,  scrextend, verbatim, bbm, accents, enumitem, subcaption}
\usepackage[dvipsnames]{xcolor}

\usepackage[capitalize,noabbrev]{cleveref}
\hypersetup{
	colorlinks=true,
	linkcolor=blue,
	citecolor=black,
	filecolor=black,      
	urlcolor=black,
	linktoc=page
}

\usepackage{microtype}
\usepackage{graphicx}
\usepackage{stmaryrd}

\newtheorem{innercustomthm}{Theorem}
\newenvironment{customthm}[1]
{\renewcommand\theinnercustomthm{#1}\innercustomthm}
  {\endinnercustomthm}

\theoremstyle{plain}
\newtheorem{thm}{Theorem}[section]
\newtheorem{cor}[thm]{Corollary}
\newtheorem{conj}[thm]{Conjecture} 
\newtheorem{prop}[thm]{Proposition}
\newtheorem{lem}[thm]{Lemma}

\numberwithin{quest}{section}
\newtheorem*{theorem*}{Theorem}

\theoremstyle{remark}
\newtheorem{rem}{Remark}
\numberwithin{rem}{section}

\theoremstyle{definition}
\newtheorem{defn}[rem]{Definition}

\AddToHook{env/lem/begin}{\crefalias{thm}{lem}}
\AddToHook{env/cor/begin}{\crefalias{thm}{cor}}
\AddToHook{env/prop/begin}{\crefalias{thm}{prop}}
\AddToHook{env/defn/begin}{\crefalias{rem}{defn}}

\crefname{thm}{Theorem}{Theorems}
\crefname{lem}{Lemma}{Lemmas}
\crefname{cor}{Corollary}{Corollaries}
\crefname{prop}{Proposition}{Propositions}
\crefname{defn}{Definition}{Definitions}

\newcommand{\eps}{\varepsilon}
\newcommand{\Z}{\mathbb{Z}}

\newcommand{\Q}{\mathbb{Q}}
\newcommand{\N}{\mathbb{N}}
\newcommand{\R}{\mathbb{R}}
\renewcommand{\P}{\mathbb{P}}
\newcommand{\E}{\mathbb{E}}

\newcommand{\cA}{\mathcal{A}}
\newcommand{\cC}{\mathcal{C}}
\newcommand{\cS}{\mathcal{S}}
\newcommand{\cN}{\mathcal{N}}

\newcommand{\cF}{\mathcal{F}}
\newcommand{\cJ}{\mathcal{J}}
\newcommand{\cH}{\mathcal{H}}

\newcommand{\eqd}{\stackrel{d}{=}}

\newcommand{\cvgd}{\stackrel{d}{\to}}
\newcommand{\II}[1]{\llbracket #1 \rrbracket}

\newcommand{\cvgp}{\stackrel{\mathbb P}{\to}}

\renewcommand{\tt}[1]{{\mathtt{#1}}}
\newif\ifShowComments
\ShowCommentstrue

\newcounter{CommentCounter}

\title{The Airy line ensemble at the rough-smooth boundary}
\author{Sunil Chhita}
\address{Sunil Chhita, Department of Mathematical Sciences, 
	Durham University, Durham, UK}
\email{sunil.chhita@durham.ac.uk}

\author{Duncan Dauvergne}
\address{Duncan Dauvergne, Department of Mathematics, 
	University of Toronto, Toronto, Canada}
\email{duncan.dauvergne@utoronto.ca}

\author{Thomas Finn}

\date{}

\begin{document}

	\begin{abstract}
We study the rough-smooth boundary in the two-periodic Aztec diamond, a random domino tiling model exhibiting three types of macroscopic regions.   
We show that the height function at this boundary converges to an independent sum of an Airy surface and an i.i.d.\ noise field with fluctuations governed by the full-plane smooth phase. Going further, we prove convergence of a family of Temperleyan backbone paths to the Airy line ensemble. This gives the first convergence result for a family of undirected paths converging to the Airy line ensemble, as well as Airy convergence at a noisy boundary.
	\end{abstract}

\maketitle
    \begin{figure}
    \centering
    \includegraphics[width=0.75\linewidth]{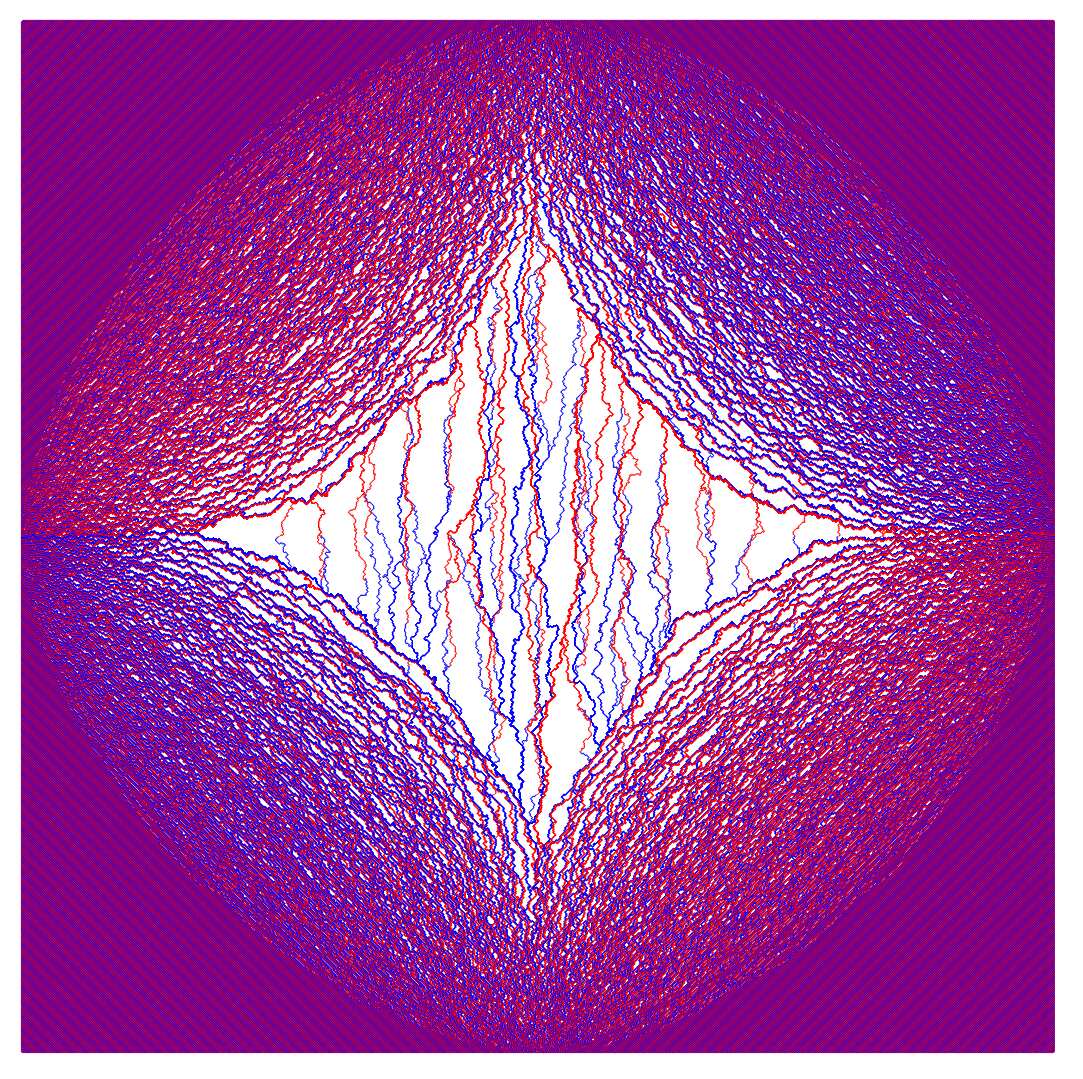}
\caption{Temperleyan paths for the two-periodic Aztec diamond ($n = 1200, a = 0.5$). }\label{fig:placeholder}
\end{figure}

	\newpage

	\setcounter{tocdepth}{1}
	\tableofcontents
	
	\section{Introduction}

	The two-periodic Aztec diamond with bias parameter $a \in (0,1)$ is a domino tiling model of a certain diamond shaped region of the square grid, known as an Aztec diamond. Each domino ($1\times 2$ or $2 \times 1$ rectangle) has weight $1$ or $a$ which depends on its location in a doubly periodic fashion. Large two-periodic Aztec diamonds exhibit three types of macroscopic regions, as characterized for random tiling models in \cite{KOS03}. These regions are known as \textit{frozen}, where tiles form a deterministic brickwork pattern; \textit{rough}, where there is polynomial decay of correlations between tiles; and \textit{smooth}, where there is exponential decay of correlations between tiles. In the simulation in \cref{fig:intro-tiling}, we can identify four frozen regions covering the corners of the Aztec diamond, a smooth region sitting in the center of the Aztec diamond (the region with no long structures), and a doubly-connected rough region which occupies the rest of the Aztec diamond. There are other random tiling models which exhibit all three types of macroscopic regions, such as the three-periodic lozenge tilings \cite{Kui24}, as well as other statistical mechanical models, such as the six vertex model with domain-wall boundary conditions and a certain choice of parameters \cite{allison2005numerical}. The two-periodic Aztec diamond is one of the simplest such models.

 \begin{figure}
		\centering
		\includegraphics[height=9cm]{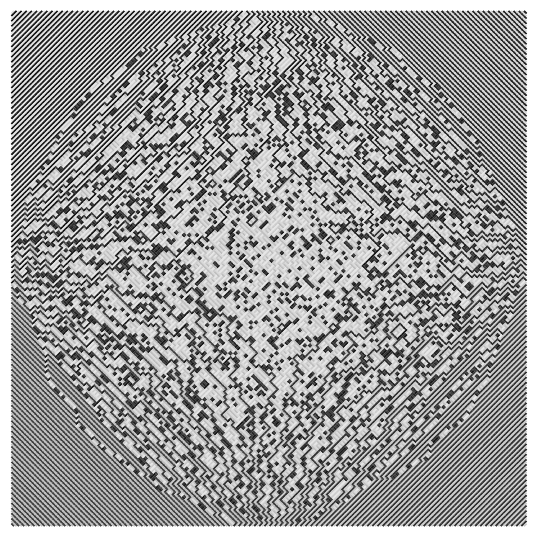}
        \caption{A simulation of a two-periodic Aztec diamond with $n=200$ and $a=0.7$ with the $a$-dominoes shaded darker than the $b$-dominoes. Note the underlying square grid is rotated by $\pi/4$, as is our coordinate convention in this paper.}
        \label{fig:intro-tiling}
\end{figure}

	One of the original motivations for studying the two-periodic Aztec diamond as articulated in \cite{CY14, CJ16} was to understand the probabilistic behavior at the rough-smooth boundary and to compare to the known probabilistic behavior at the frozen-rough boundary. Indeed, universality at the frozen-rough boundary has been recently achieved for a large class of lozenge tilings  \cite{AH25}, see also \cite{okounkov2003correlation, ferrari2003step, Jo03, baik2007discrete, petrov2014asymptotics, duse2018universal} for earlier tiling results and \cite{Gor21} for more references and background.
    In tiling models, there is usually a set of lattice paths that combinatorially defines the boundary between the frozen and rough regions. In the microscopic limit at the interface, these lattice paths converge to the \textit{Airy line ensemble}, a random sequence of continuous functions $\mathcal A = \{\mathcal A_i :\R \to\R, i \in \N\}$ with the ordering $\mathcal A_1 > \mathcal A_2 > \cdots$ believed to govern edge fluctuations of many statistical mechanical models. The Airy line ensemble also appears as a scaling limit in random matrix theory \cite{macedo1994universal, forrester1999correlations, adler2005pdes} and for growth models in the Kardar-Parisi-Zhang universality class, e.g.\ see \cite{prahofer2002scale, imamura2007dynamics, CH11, dauvergne2023uniform, wu2023convergence, aggarwal2024scaling}.

The Airy line ensemble is also expected to govern the fluctuations at the rough-smooth boundary in random tilings. However, here the picture is much murkier than at the rough-smooth boundary. One of the main difficulties is that unlike at the frozen-rough boundary, the rough-smooth boundary features mixing between two non-deterministic regions. This makes the boundary noisy, and means that there is no local combinatorial description of a boundary or a canonical ensemble of paths. Moreover, when one does try to define natural families of paths they are undirected, i.e.\ they cannot be written as graphs of one-dimensional functions and may feature significant backtracking or winding.

Our goal in this paper is to overcome these challenges and understand precisely what is happening at the rough-smooth boundary in the two-periodic Aztec diamond. Our first theorem shows convergence of the height function at the rough-smooth boundary.  For this theorem, we define the \textit{Airy surface} $\mathcal A:\R^2 \to \{0, 1, \dots\}$ by 
$$
\mathcal A(t, x) = \# \{i \in \N : \mathcal A_i(t) \ge x\},
$$
which should be viewed as the height function for the Airy line ensemble.

\begin{customthm}{1}[Informal version of \cref{T:main-1}]
\label{T:custom1}
Fix $a \in (0, 1)$. Let $\mathcal{H}_n$ be the height function of the two-periodic Aztec diamond of size $n$, and let $H_n$ be the average of $\mathcal H_n$ in the smooth region. Let $[\cdot]_n:\R^2 \to \R^2$ be a scaling map which takes limiting coordinates to an $O(n^{2/3}) \times O(n^{1/3})$-sized box where the rough-smooth boundary meets the antidiagonal in the two-periodic Aztec diamond of size $n$. Then, as $n\to \infty$, for any finite set $S \subset \R^2$,
\begin{equation*}
(H_n-\mathcal{H}_n[u]_n, u \in S) \implies (4\mathcal{A}(u) + X_u, u \in S)
\end{equation*}
where $\mathcal{A}$ is the Airy surface and $X = (X_u, u \in F)$ is an i.i.d.\ vector, independent of $\mathcal A$, whose individual entries are given by the  height of a face in the full-plane smooth phase.
\end{customthm}

\cref{T:custom1} makes precise the idea that the rough-smooth boundary is given by an Airy line ensemble $\mathcal A$ sitting on top of a smooth background, represented by the $X$. The central height $H_n$ itself converges in the limit to a discrete Gaussian random variable, see the definition \eqref{E:central-height} and surrounding discussion.

While \cref{T:main-1} confirms the idea that the height profile at the rough-smooth boundary looks like an Airy surface sitting in a smooth background, it says nothing about the existence of a family of paths converging to the Airy line ensemble. 
Most of the paper is devoted to understanding such a family. In particular, we are able to identify a `top path' which converges to the Airy process. The paths we consider are \textit{south and north backbone paths}, introduced briefly in \cite{CJ16}, which arise from a version of Temperley's bijection. See \cref{S:temperley-intro} for details and \cref{fig:intro-paths} for an example. Let $\{\mathcal S_i\}_{i=1}^{n/2}$ be the collection of south backbone paths starting on the bottom boundary and labelled from left to right. We let $1\leq I \leq n$ be the (random) index such that $\mathcal S_1,\dots, \mathcal S_I$ terminate on the left boundary and $\mathcal S_{I+1},\dots, \mathcal S_n$ terminate on the right. We can now state our main theorem. 

\begin{figure}
		\centering
		\includegraphics[height=9cm]{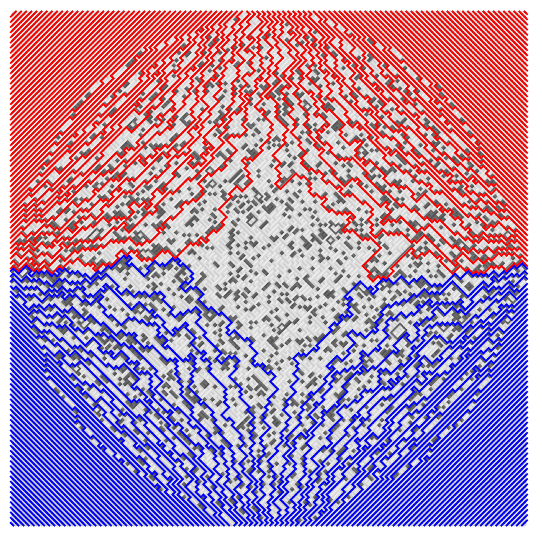}
		\caption{The same simulation as in \cref{fig:intro-tiling}, with an overlay of south ($\mathcal{S}$) and north ($\mathcal{N}$) backbone paths from the bottom and top of the Aztec diamond respectively.}
		\label{fig:intro-paths}
	\end{figure}

\begin{customthm}{2}[Informal version of \cref{T:main-2}]
\label{T:custom2}
Fix $a \in (0, 1)$. Around the main anti-diagonal of the Aztec diamond, as $n \to \infty$,  the collection $\{\mathcal S_{I+1-i}\}_{i \in \mathbb{N}}$ converges after rescaling to the Airy line ensemble, in the sense of finite dimensional distributions (FDD). This convergence can be taken jointly with the convergence in \cref{T:custom2}.
\end{customthm}

Note that since the curves $\mathcal S_i$ are general plane curves whereas the Airy line ensemble consists of one-dimensional functions, certain projections need to be taken to make sense of the FDD convergence above. Note that the splitting point $I$ retains randomness even in the limit, and ends up corresponding with high probability to the central height $H_n$ from \cref{T:custom1}.

To our knowledge, our results are the first to show Airy process or Airy line ensemble convergence for a model in a noisy background, or a model whose boundary is given by genuinely two-dimensional lattice paths. Indeed, without these difficulties the kind of FDD convergence in \cref{T:custom1} and \cref{T:custom2} would be immediate from the extended Airy kernel convergence established in \cite{BCJ18}, as is the case for other models. In our setting we need to worry about even the most basic properties of the model, such as establishing the presence of a last path at the boundary, controlling global or mesoscopic backtracks and path winding, and teasing out the differences between the local and global contributions to the height function.   What saves us is that we are working with a highly integrable model which is amenable to refined computations, and we have a Temperley-type correspondence with certain biased spanning forest measures, which allows us to import probabilistic tools coming from Wilson's algorithm and biased loop-erased random walks. We discuss our analysis in much greater detail below after the precise statement of results.

\subsection{Previous work on the model}

	Recent progress on the two-periodic Aztec diamond and its generalizations has hinged on an inherent algebraic structure.  This algebraic structure yields exact expressions for the determinantal point process correlations associated to the two-periodic Aztec diamond in three different ways \cite{CY14,DK21,BD19} which rely on the same structure \cite{CD23, KP25}.  Moreover, this algebraic structure has also been exploited to introduce bias on the orientation of dominoes \cite{BD22}, consider higher periodicities where the model has multiple smooth regions \cite{Ber21, BB23}, split weights where the model has two different adjacent smooth phases \cite{She25},  
    as well as completely general weights \cite{BdT24}. One feature of all of the aforementioned papers is that the correlation kernel associated to the determinantal process is amenable to precise asymptotic analysis, which gives access to limit shape curves and certain local statistics.  Indeed, this has been exploited for the two-periodic Aztec diamond to find local limit laws \cite{CJ16, DK21}, Airy kernel asymptotics at the rough-smooth boundary \cite{CJ16, BCJ18}, sine-Gordon asymptotics in the $a\to 1$ regime \cite{mason2022two}, precise dimer-dimer correlations at this boundary \cite{JM22}, and Gaussian free field fluctuations \cite{BN25}. This final result also holds for more general periodicities.  More recently, this integrable structure has been to shown to hold for certain randomly weighted domino tilings of the Aztec diamond \cite{DVP25}. 
	
	An alternate way to try to define the rough-smooth boundary is to introduce contractions to certain faces, a process called squishing \cite{BCJ22}, which results in a set of lattice paths known as squished paths.  From simulations, these squished paths seem to determine the rough-smooth boundary and it was shown there that the squished paths are responsible, in a certain sense, for producing the Airy kernel point process found in \cite{BCJ18} provided that $0<a<1/3$.   
    Part of the reason is that the squished paths are the level lines of the height function when restricted to certain faces of the Aztec diamond.  Furthermore, if $a = n^{-\gamma}$ for some $\gamma \in (2/3, 1)$, then  as $n \to \infty$ the squished path separating rough and smooth interfaces converges to the Airy process \cite{JM23}.  This result relies on the fact that the squished path, although combinatorially is a two-dimensional path, becomes a one-dimensional path with high probability if $a$ is sufficiently small given $n$, which makes analysis much easier than the setup considered here. We note that for all $a$, the squished and backbone paths should differ on a scale which is $o(n^{1/3})$ and hence converge to the same Airy limit. However, even our most optimistic ideas did not show this for $a$ greater than some small constant $a_0$. Indeed, we ran into many of the same issues around crude loop and path counting which lead to the upper bound of $1/3$ in \cite{BCJ22}, and significant new ideas would be needed to prove a version of \cref{T:custom2} for squished paths in the full range of $a$.

\subsection{Boundaries beyond tilings and vertex models}
\label{S:beyond-tilings}
Compared to the frozen-rough boundary, the rough-smooth boundary in random tilings is much closer in character to interfaces that appear in less integrable random surface models. In particular, there are many models where we expect Airy line ensemble convergence at noisy interfaces and where natural families of boundary paths are undirected, e.g.\ contour lines in the 3D Ising model or the $(2+1)$-dimensional SOS model and variants in certain parameter ranges, see \cite{ioffe2015interaction, ioffe2016low,  ferrari2023airy2} and references therein. Note that in many settings, the predicted limit laws are slightly different objects 
(e.g. Dyson-Ferrari-Spohn diffusions and related objects \cite{ioffe2018dyson, basu2025characterizing}) which exhibit their own limit transition to the Airy line ensemble \cite{ferrari2023airy2, dimitrov2025uniform}.

Because of the many extra difficulties in these and related models, establishing Airy line ensemble convergence in any regime seems incredibly difficult. However, there has nonetheless been substantial recent progress studying level lines and establishing limit laws in non-Airy regimes, e.g.\, see \cite{caputo2016scaling, caddeo2024level, chen2025limiting}. It would be interesting to see if ideas used in those settings could have relevance in random tiling or vertex models, where there is extra integrability but rough-smooth boundaries are still difficult to understand.

\subsection{Outline of the paper} \label{S:outline-intro} 
The next two sections are essentially an extended introduction. \cref{S:def-main-results} gives precise definitions and statements of our two main theorems.
In \cref{S:Two-periodicAztecdiamond}, we introduce the two-periodic Aztec diamond, in \cref{S:face-heights} we introduce the height function and in \cref{S:temperley-intro}, we introduce Temperley's bijection.  We state our main results in \cref{S:mainresults-intro}. \cref{sec:proof-sketch} contains a detailed proof sketch and includes two auxiliary results on the model which we believe may be of particular interest. 

The remainder of the paper is devoted to the proofs.
    Very roughly, after an initial combinatorial section (\cref{S:trees}) containing a proof of Temperley's bijection for the Aztec diamond and some auxiliary results, the paper is divided into an essentially probabilistic part (\cref{S:basic-tools} through to \cref{S:cvg-airy}) which contains the meat of the paper and an integrable part (\cref{subsec:KasteleynTheory} to \cref{subsec:Kinvasympproofs}) whose proofs are mostly standard asymptotic analysis or modifications of previous arguments for the model (the exception here is the symmetry in \cref{L:expected-middle-height-intro}).
    
    \cref{S:basic-tools} records the basic tools that we need for working with the model. Many results in this section have integrable proofs for this section are postponed until \cref{subsec:KasteleynTheory} and beyond. In \cref{S:proof-t1} we prove \cref{T:custom1}. 
    In \cref{S:global-control}, \cref{S:regularity}, and \cref{S:cvg-airy} build to the proof of \cref{T:custom2}. \cref{S:global-control} establishes global path control and one of our auxiliary results (\cref{T:main-3}), \cref{S:regularity} establishes the local path regularity and our other auxiliary result (\cref{T:overhang-thm}), and \cref{S:proof-t1} completes the proof of \cref{T:main-2}.

	\subsection*{Acknowledgments}
	S.C. and T.F. were supported by EPSRC EP\textbackslash T004290\textbackslash 1. D.D. was supported by an NSERC Discovery grant and a Sloan fellowship.
S.C. would like to thank both Maurice Duits and Kurt Johansson for the numerous conversations and fruitful discussions on this model over many years. 
	
	\section{Definitions and main results}
    \label{S:def-main-results}
	
	\subsection{The two-periodic Aztec diamond} \label{S:Two-periodicAztecdiamond}
	
	First, define vertex and face lattices in $\Z^2$:
	\begin{align*}
		\tt{\tilde V} = \{(i, j) \in \Z^2 : (i+j)\text{ mod } 2=1\}, \\
		\tt{\tilde F} = \{(i, j) \in \Z^2 : (i+j)\text{ mod } 2=0\}.
	\end{align*}
	We let the edge set $\tt{\tilde D}$ consist of all pairs $\{v, w\} \subset \tt{\tilde V}$ such that $v-w \in \{\pm e_1, \pm e_2\}$ where $e_1=(1,1)$ and $e_2=(-1,1)$. This gives an infinite bipartite graph $\tt{\tilde G} := (\tt{\tilde V}, \tt{\tilde D})$, consisting of white and black vertices
	\begin{align*}
		\tt{White} &= \{(i, j) \in \tt{\tilde V} : i\text{ mod } 2=1, j\text{ mod } 2=0\}, \\
		\tt{Black} &= \{(i, j) \in \tt{\tilde V} : i\text{ mod } 2 =0, j\text{ mod } 2=1\}.
	\end{align*}
	The set $\tilde{F}$ gives the faces of $\tilde{G}$, where a face is labelled by its centre. Next, fix a parameter $a \in (0, 1)$. If $(i,j) \in \tt{\tilde F}$, $i, j \in 2 \Z + 1$, and $i + j \in 4 \Z + 2$, then we label $(i,j)$ as an \emph{a-face}. If $i, j \in 2 \Z + 1$, and $i + j \in 4 \Z$, then we label $(i,j)$ as an \emph{b-face}. If an edge is incident to an $a$-face it has weight $a$ and if it is incident to a $b$-face it has weight $1$. We call the resulting weighted graph the \emph{two-periodic lattice} with parameter $a$. This weighting splits the white and black vertex sets into two pieces, which in total partitions $\tt{\tilde V}$ into four pieces, which we label as $\tt{\tilde N}, \tt{\tilde S}, \tt{\tilde E}, \tt{\tilde W}$ and are defined as follows:
	\begin{itemize}[nosep]
		\item First, every vertex $v \in \tt{\tilde V}$ is adjacent to exactly one $a$ face $f(v)$.
		\item If $v$ is south of $f(v)$ (i.e. $v - f(v) = (0, -1)$), let $v \in \tt{\tilde S}$. Similarly, if $v$ is north, east, or west of $f(v)$ (i.e. $v - f(v) = (0, 1), (1, 0), (-1, 0)$, respectively), let $v \in \tt{\tilde N}, \tt{\tilde E}, \tt{\tilde W}$. 
	\end{itemize}
	It is straightforward to check that $\tt{White} = \tt{\tilde N} \cup \tt{\tilde S}$ and $\tt{Black} = \tt{\tilde E} \cup \tt{\tilde W}$. 
	
	For $n \in 4\N$, the \emph{two-periodic Aztec diamond} $\tt{A}_n$ of size $n$ is the subgraph of the two-periodic lattice induced by the vertices in the box $[-n, n]^2$. We write $\tt{V} = \tt{\tilde V} \cap [-n, n]^2$ and similarly define $\tt{N}, \tt{S}, \tt{E}, \tt{W}$ for vertex classes in the two-periodic Aztec diamond. We let $\tt{D}$ be the edge set of $\tt{A}_n$. The bounded faces of the Aztec diamond are given by $\tt{F} = \tt{\tilde F} \cap (-n, n)^2$. It will also be useful to consider a larger set which includes boundary faces:
	$$
	\tt{\bar F} = \{f \in \tt{\tilde F} : \text{ there exists } v \in \tt{V} \text{ with } \|v - f\|_2 = 1\}. 
	$$
	The size parameter $n$ is suppressed from the above notation; its value will always be clear from context.
	See  \cref{fig:ttwoperiodicweights} for an illustration of the two-periodic Aztec diamond.
	\begin{figure}
		\centering
		\includegraphics[height=6cm]{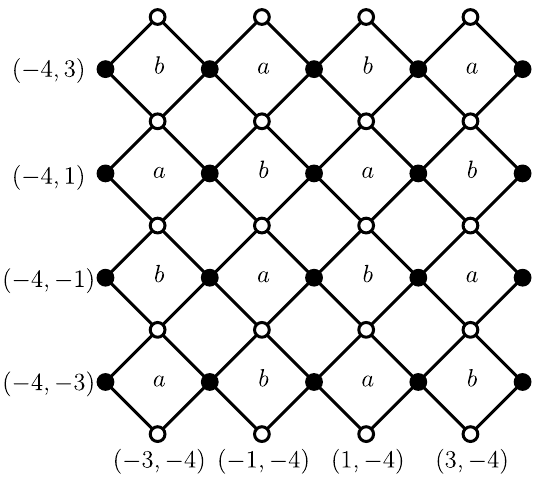}
		\caption{The two-periodic Aztec diamond of size $4$. The edges around each face labeled $a$ have weight $a$ while the edges around each face labeled $b$ have weight $1$.}
		\label{fig:ttwoperiodicweights}
	\end{figure}
	A \emph{dimer configuration} $D$ on $\tt{A}_n$ is a perfect matching of $\tt{V}$, i.e.\ a subset of $\tt{D}$ so that each vertex is incident to exactly one edge. Edges in a dimer configuration are called \emph{dimers}. The law of the {\it two-periodic Aztec diamond with parameter $a$} is the probability measure $\mathbb P_{a, n}$ on dimer configurations of $\tt{A}_n$, where each configuration $D$ has probability proportional to the product of the edge weights.
	Note that if we had started with any weighting that assigned an edge weight $a$ to $a$-faces and $b$ to $b$-faces, then as long as $a \ne b$, the resulting measure would be equivalent to $\mathbb P_{a', n}$ for some $a' \in (0, 1)$, possibly after a reflection.

    \begin{figure}
		\centering
		\includegraphics[height=6cm]{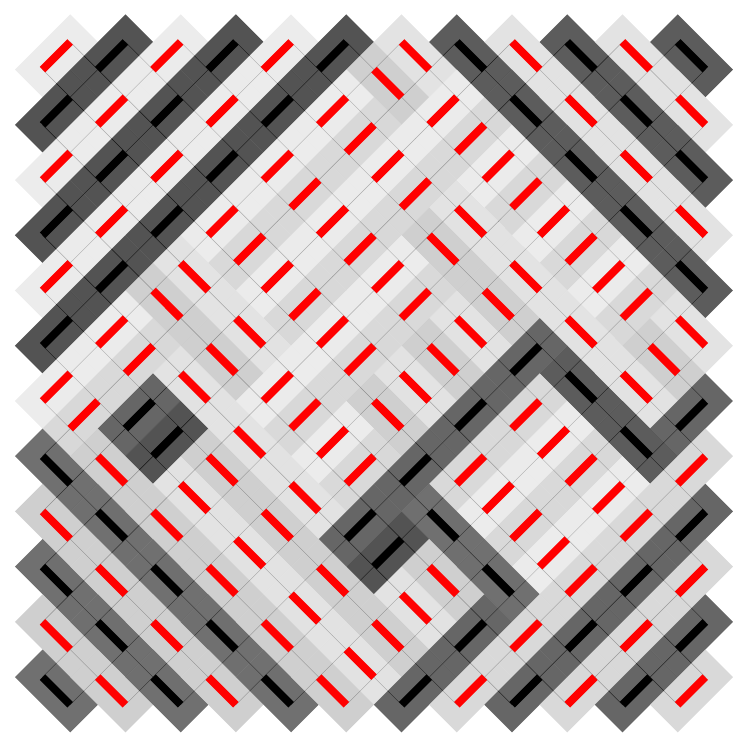}
		\includegraphics[height=6cm]{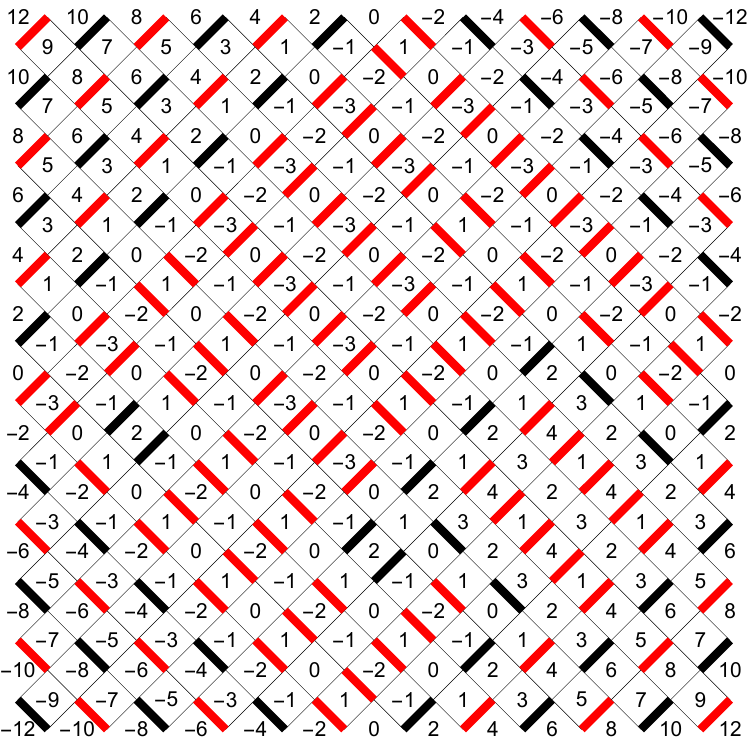}
		\caption{The relationship between dimers and dominoes is shown on the figure on the left while the figure on the right shows the height function for this dimer covering.}
		\label{fig:dimersandheights}
	\end{figure}
	
	\subsection{Height functions}
	\label{S:face-heights}
	Dimer coverings of bipartite graphs have an alternative \emph{height function} representation, an idea dating back to Thurston \cite{Thu:90}. For a dimer configuration $D$ on $\tt{A}_n$, we define its height function $h = h(D):\mathtt{\bar F} \to \Z$ with the following rules:
	\begin{itemize}
		\item $h(-n, -n) = -n$.
		\item As we traverse counterclockwise on the faces around each white vertex, we increase the height by $1$ provided that there is no dimer covering the shared edge between faces.  If there is a dimer covering the shared edge between faces, we decrease the height by 3.
		\item As we traverse clockwise on the faces around each black vertex, we increase the height by $1$  provided that there is no dimer covering the shared edge between faces.  If there is a dimer covering the shared edge between faces, we decrease the height by 3.
	\end{itemize}
	The height function on the boundary $\partial \mathtt{F} := \mathtt{\bar F} \setminus \mathtt{F}$ is deterministic since there are no dimers outside of $\mathtt{V} \cap [-n, n]^2$. We can easily check that:
	\begin{itemize}[nosep]
		\item If $f \in \partial \mathtt{F}$ is of the form $(i, -n)$ or $(i, -n-1)$, then $h(f) = i$.
		\item If $f \in \partial \mathtt{F}$ is of the form $(-n, i)$ or $(-n-1, i)$, then $h(f) = i$.
		\item If $f \in \partial \mathtt{F}$ is of the form $(i, n)$ or $(i, n+1)$, then $h(f) = -i$.
		\item If $f \in \partial \mathtt{F}$ is of the form $(n, i)$ or $(n+1, i)$, then $h(f) = -i$.
	\end{itemize}
	In other words, as we traverse the boundary $\partial \mathtt{F}$ clockwise, we increase along the west boundary, decrease along the north boundary, increase along the east boundary, and then decrease along the south boundary. See \cref{fig:dimersandheights} for an example. In the two-periodic Aztec diamond, the limit shape of the height function has maximal gradient in the frozen regions, and is flat throughout the smooth region. The rough region transitions between these two extremes.
	
	The above correspondence gives a bijection between the set of dimer configurations on $\tt{A}_n$ and the set of allowable height functions. These are functions $h:\mathtt{\bar F} \to \Z$ with the above boundary values, and such that around every vertex $w \in \mathtt{White}$ heights increase by $1$ as we move around counterclockwise except for between a single pair of faces, and for every vertex $w \in \mathtt{Black}$ heights increase by $1$ as we move around clockwise except for between a single pair of faces.
	\subsection{Temperley's bijection}
	\label{S:temperley-intro}
	
	Dimer coverings of $\tt{A}_n$ can also be understood through a version of Temperley's bijection. The classical Temperley's bijections maps dimer coverings of bipartite graphs to spanning trees of an associated graph in the case when the boundary of the graph has constant height, see \cite{Tem74, KPW00}. In the Aztec diamond, the height changes as we move around the boundary, and so if we construct a Temperley-type correspondence in this setting we get a bijection with certain classes of spanning forests. There are four versions of this bijection, which correspond to the north, south, east, and west vertices, but in the present paper we will only use the north and south bijections.

    	\begin{figure}
		\centering
		\includegraphics[height=6cm]{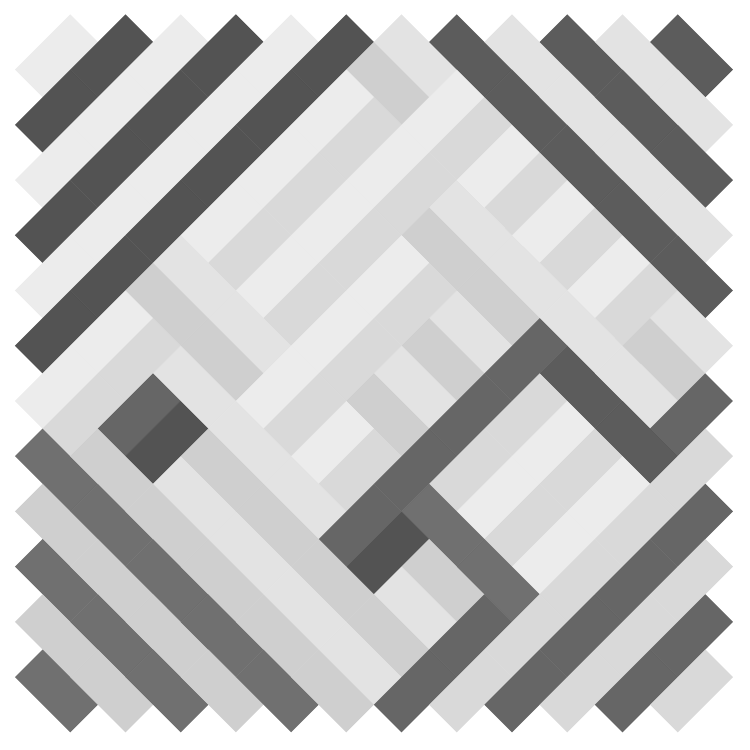}
		\includegraphics[height=6cm]{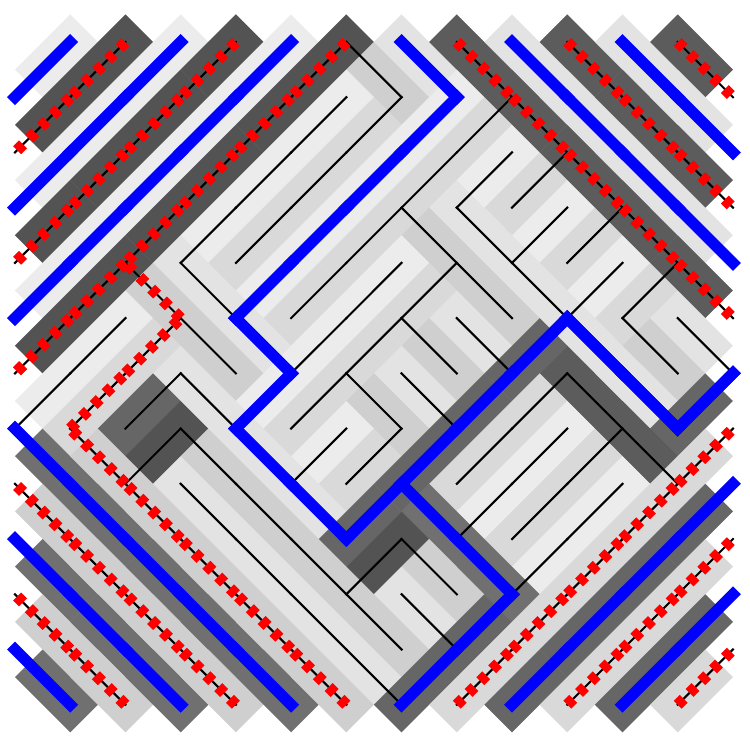}
		\caption{The left figure shows a dimer tiling, the right shows both the north and south Temperleyan forests. Backbone paths in the north and south forests are drawn in red and blue lines, respectively, as is our convention for all figures.}
		\label{fig:dimersandtreepaths}
	\end{figure}
	
	To define the south bijection, define a graph $\tt{G}(\mathtt{S})$ whose vertex set is a small extension of $\mathtt{S}$. Define
	\begin{equation}
		\label{E:tbW}
		\mathtt{\bar S} = \mathtt{\tilde S} \cap ([-n - 1, n + 1] \times [-n, n]).
	\end{equation}
	We let $\mathtt{\bar S}$ be the vertex set of $\tt{G}(\mathtt{S})$, and let $\{v, w\} \subset \mathtt{\bar S}$ be an edge whenever $v-w \in \{\pm 2 e_1, \pm 2 e_2\}$. Now, given a dimer configuration $D$ on $\tt{A}_n$, we can associate to it a subgraph $F_D \subset \tt{G}(\mathtt{S})$ with the same vertex set $\mathtt{\bar S}$, by including an edge $\{v, v + e\}$ in $F_D$ for $e \in \{\pm 2 e_1, \pm 2 e_2\}$ whenever $\{v, v + e/2\}$ is a dimer in $D$. We call $F_D$ the \textit{south Temperleyan forest} for $D$. See \cref{fig:dimersandtreepaths} for an example of this correspondence.
	
	The map $D \mapsto F_D$ is one-to-one, and we can identify its image as the set of \textit{dimer-compatible forests} of $\tt{G}(\mathtt{S})$, which we now define. First, divide the outer boundary of $\tt{G}(\mathtt{S})$ into four pieces:
	\begin{equation}
		\label{E:cardinal-boundaries}
		\begin{split}
			\partial_E \tt{S} &= (\{n + 1\} \times [-n, n]) \cap \tt{\bar S}, \quad 	\partial_W \tt{S} =  (\{-n - 1\} \times [-n, n]) \cap \tt{\bar S}, \\
			\partial_S \tt{S} &= ([-n, n] \times \{-n\}) \cap \tt{\bar S}, \quad 	\partial_N \tt{S}=  ([-n, n] \times \{n\}) \cap \tt{\bar S}.
		\end{split}
	\end{equation}
	We call $\partial^{\tt{x}} \tt{S} := \partial_E \tt{S} \cup \partial_W\tt{S} = \tt{\bar S} \setminus \tt{S}$ the \textbf{sink boundary} for $\tt{G}(\tt{S})$, we call $\partial^{\tt{o}} \tt{S} = \partial_N \tt{S} \cup \partial_S \tt{S}$ the \textbf{source boundary} for $\tt{G}(\tt{S})$, and let $\partial \tt{S}$ be the union of all four parts of the boundary.  
	\begin{defn}
		\label{D:dc-forest}
		A subgraph $F \subset \tt{G}(\mathtt{S})$ is a \textbf{dimer-compatible forest (DCF)} if:
		\begin{enumerate}[label=(\roman*)]
			\item $F$ is a spanning forest, i.e. its vertex set is $\tt{\bar S}$.
			\item Every component of $F$ contains exactly one element of $\partial^{\tt{x}} \mathtt{S}$.
			\item For every vertex $v \in \partial^{\tt{o}} \mathtt{S}$, let $\theta_F(v)$ denote the unique vertex in $\partial^{\tt{x}} \mathtt{S}$ in the same component as $v$. Then $\theta_F(v)$ is contained in the cross-shaped region $\{v + (n, \pm n) : n \in 2 \Z \}$. This region intersects $\partial^\tt{x} \tt{S}$ at exactly two points, once in $\partial_E \tt{S}$ and once in $\partial_W \tt{S}$.
		\end{enumerate} 
	\end{defn}

	We will typically think of all edges in each component of a DCF $F$ as being oriented towards the sink boundary  $\partial^{\tt{x}} \tt{S}$. It is easy to include this orientation in the map $D \mapsto F_D$ by including the oriented edge $(v, v + e)$ in $F_D$ whenever $\{v, v + e/2\}$ is a dimer in $D$. 

    The \textbf{north Temperleyan forest} is essentially the planar dual of the south forest, see \cref{S:trees} for the precise setup. Alternately, it can be obtained through the same south forest procedure after rotation by $\pi$.
	
	We can reconstruct of a dimer configuration the height function directly from any of its Temperleyan forests. Very roughly, the height of a vertex $v$ in a Temperleyan forest $F$ is equal to the height of the sink vertex $s_v$ in the component containing $v$ plus four times the winding number of the path from $v$ to $s_v$. Here the height of a vertex is the average height of its incident faces. Because of this correspondence, we can think different components of a Temperleyan forest as wanting to live at distinct heights, and Definition \ref{D:dc-forest}(iii) as a height-matching condition.
	
	The two-periodic Aztec diamond plays particularly well with Temperley's bijection. Indeed, if we let $\mathbb{Q}_{\tt{S}, a, n}$ be the pushforward of the measure $\P_{a, n}$ under the Temperley map $D \mapsto F_D$, then the probability of a DCF $F$ under $\mathbb{Q}_{\tt{S}, a, n}$ is simply proportional to $a^2$ times the number of directed edges which point northwest or northeast. In other words, we bias towards edges which point south. Given this structure, we can imagining sampling such a forest $F \sim \mathbb{Q}_{\tt{S}, a, n}$ in two steps:
	\begin{itemize}
		\item First, sample all of the paths connecting the source and sink boundaries $\partial^{\tt{o}} \mathtt{S}$ and $\partial^{\tt{x}} \mathtt{S}$. We call these the \textbf{backbone paths}.
		\item Given the first step, the corridors between backbone paths are conditionally independent, and each region can individually be sampled using Wilson's algorithm for sampling spanning trees on a weighted graph. In particular, the law of the path from any vertex until it hits a backbone path is just that of the loop-erasure of a south-biased random walk.
	\end{itemize}
	Because of this two-step process, most interesting information in the forest $F_D$ is carried by the backbone paths, which exhibit long-range correlations. The biased loop-erased random walks which fill in the corridors have correlations which decay exponentially fast. Moreover, it is unlikely that the loop-erasure of a biased random walk has a high winding number, and so components of $F_D$ are essentially flat. 
    
    Figure \ref{fig:placeholder} is a simulation of the north and south Temperleyan forests for a two-periodic Aztec diamond with $n = 1200, a = 0.5$. South backbone paths are given in blue, and north in red; this is our convention through the paper. A selection of paths off of the backbone moving through the smooth region are also included. These paths show a strong north/south drift prior to hitting the backbone, as we expect from the bias.

	\subsection{Main results} \label{S:mainresults-intro}

      \begin{figure}
        \centering
        \includegraphics[width=1\linewidth]{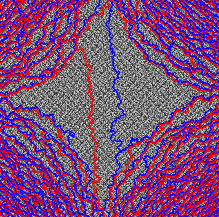}
        \caption{A simulation of the south forest in the two-periodic Aztec diamond with $n = 800$ and $a = 0.7$, with only a portion of the Aztec diamond around the smooth region pictured. The $\cS$-backbone paths are pictured in blue, and the (dual) $\cN$-backbone paths are red.}
        \label{fig:800a07}
    \end{figure}

     Recall from the overview that our first main result concerns convergence of the height functions to the Airy surface after rescaling the underlying lattice. This perspective avoids having to directly work with paths.

	To state \cref{T:custom1} precisely, let $\cH_n$ denote the height function for the two-periodic Aztec diamond of size $n$, let $\zeta_n$ be the set of all faces with $\|f\|_\infty \le n^{3/4}$, and set
	\begin{equation}
	    \label{E:central-height}
        H'_n = \frac{1}{|\zeta_n|} \sum_{f \in \zeta_n} \cH_n(f).
	\end{equation}
	Define the \textbf{central height} $H_n$ to be the unique integer in the interval $(H'_n - 1/2, H_n' + 1/2]$. The central height $H_n$ is the height at which the flat smooth region sits on average. In the definition, it is not important the exact set we are averaging over to define this height, as long as it grows to infinity with $n$ and is contained in the smooth region. The random variable $H_n$ itself is converging to a discrete Gaussian random variable, see \cite{BN25} and Remark \ref{R:discrete-gaussian}.
	
	Next, we define scaling parameters. Following \cite{BCJ22}, define
	\begin{equation} \label{eq:parameterc}
		c=\frac{a}{(1+a^2) },
	\end{equation}
and fix
	$\xi_c=-\frac{1}{2}\sqrt{1-2c}$. The point $(\xi_c, \xi_c)$ is the unique point in the third quadrant on the rough-smooth limit curve. Set
	\begin{equation} \label{eq:scalingparameters}
		c_0=\frac{(1-2c)^{\frac{2}{3}}}{(2c(1+c))^{\frac{1}{3}}}, \hspace{5mm} 
		\lambda_1 =  \frac{\sqrt{1-2c}}{2c_0} \hspace{5mm} \mbox{and} \hspace{5mm}
		\lambda_2=\frac{(1-2c)^{\frac{3}{2}}}{2c c_0^2}.\footnote{Note that there is a typo in both \cite{BCJ18} and \cite{CJ16} for $c_0$ which is fixed here.}
	\end{equation}
	For this theorem, some of the limiting objects come from the \textbf{full-plane smooth phase} $\P_a$. There are multiple ways of defining this object (see Section \ref{S:smooth-phase-def}), but for now, we simply note that the measure $\P_a$ on dimer configurations of $\Z^2$ is the limit of the measures $\P_{a, n}$ as we take $n \to \infty$, i.e. the local limit of the two-periodic Aztec diamond in the center of the smooth region. As there are no boundary conditions in the full-plane smooth phase, the height function is naturally defined only up to constant shift. We choose this shift in the most natural way, so that the height function is ergodic and mean $0$ on $a$- and $b$-faces.
	
	\begin{thm}
		\label{T:main-1}
		For $(t, x) \in \R^2$, define $[t, x]_n$ to be the nearest $a$-face in $\tt{\tilde F}$ to the point
		$$
		n (\xi_c, \xi_c) + t \cdot 2^{1/3} \lambda_2 n^{2/3} (1, -1) + (x - t^2) \cdot 2^{5/3} \lambda_1 n^{1/3} (1, 0).
		$$
		Then for any finite set $S \subset \R^2$ we have the convergence
		\begin{equation}
			\label{E:height-extra-intro}
			\begin{split}
				(H_n - \cH_n[u]_n : u \in S)
				\implies (4 \cA(u) + X_u : u \in S),
			\end{split}
		\end{equation}
		where $\cA$ is the Airy surface, and the random variables $X_u, u \in S$ are i.i.d. and independent of $\cA$. Each of these random variables individually is distributed as the height of (any) $a$-face in the full-plane smooth phase with parameter $a$.
		
		We can write $X_u \eqd 4 Y$, where $Y$ is the winding number of the loop-erasure of a biased nearest-neighbour random walk on $\Z^2$, which takes up and right steps with probability $a^2/(2 + 2 a^2)$ and left and down steps with probability $1/(2 + 2 a^2)$.
	\end{thm}
	
	\begin{rem}
		The law of each of the random variables $X_u$ should be a discrete Gaussian, but we have not pursued this direction here. More remarkably, we believe that the law of these random variables is the same as the limit law of the central height. This latter fact seems to be a special symmetry of the two-periodic Aztec diamond, and we do not expect it to hold for smooth phases and rough-smooth boundaries in other tiling models.
	\end{rem}

	\begin{rem}
		\label{R:main-thm-variants}
		\cref{T:main-1} has the following two variants, which lend more credence to the idea that the limit of the height function is an Airy line ensemble sitting in a smooth background.
		
		First, with the same setup as in that theorem, we have that
		\begin{equation*}
			\begin{split}
				(H_n - \cH_n(v - (1, 1) + [u]_n) : u \in S, v \in \tt{\tilde F})
				\implies (4\cA(u) + \cH^u(v) : u \in S, v \in \tt{\tilde F}),
			\end{split}
		\end{equation*}
		where now each function $\cH^u:\tt{\tilde F} \to \Z$ is an independent copy of the height function for the full-plane smooth phase. The convergence above is in the product topology.
		
		Moreover, if we consider a sequence of sets $\phi_n = [-r_n, r_n]^2 \cap \Z^2$ for some $r_n = o(n^{1/24})$, then letting $D \sim \P_{a, n}$, we can couple $D|_{[u]_n + v}, u \in S, v \in \phi_n$ with the full-plane smooth phase so that the two objects are equal with probability $1 - o(1)$. Because of the presence of the Airy lines in an $n^{2/3} \times n^{1/3}$ scaling window, the optimal result here would allow for any sets $\phi_n$ that are lower order than this scale. 
		
		The above results, together with ergodicity of the full-plane smooth phase, says that if we take small local averages of the height function then we converge to the Airy surface. Setting up some notation, with $\phi_n$ as above, define the mollified height function
		$$
		\cH_n^{\phi_n}([u]_n) := \frac{1}{|\phi_n|} \sum_{v \in \phi_n} \cH_n(v + [u]_n).
		$$
		Then
		\begin{equation*}
			\begin{split}
				\tfrac{1}{4} (H_n - \cH_n^{\phi_n}([u]_n) : u \in S)
				\implies (\cA(u) : u \in S).
			\end{split}
		\end{equation*}
		Both \cref{T:main-1} and the various strengthenings above are shown in Section \ref{S:proof-t1}, where they all follow as corollaries of the more technical \cref{T:main-1-restatement}.
	\end{rem}

    Next, we state a precise version of \cref{T:custom2} on backbone path convergence. While this theorem is stated for south backbone paths only, the proof will show the result for the south and north paths together. Setting some notation, for a two-periodic Aztec diamond of size $n$, we let $\cS_1^-, \cN_1^-, \cS_2^-, \cN_2^-,  \dots, \cS_{n/2}^-, \cN_{n/2}^-$ denote the $n/2$ backbone paths in the associated south and north Temperleyan forest that start on the lower boundary, with the ordering from left to right.  Then there is a unique split point $I = I_n \in \{0, \dots, n/2\}$ where for $i \le I_n$, the path $\cS_i^-$ ends on the west boundary, and for $i > I_n$, the path $\cS_i^-$ ends on the east boundary. To help parse the definitions, it may be illustrative to consider the third panel in \Cref{fig:dimersandheights} where $n = 12$, and we have six south forest paths $\cS_i^-$ and six north forest paths $\cN_i^-$. In that example, $I_n = 3$. We similarly use notation $\cS_i^+, \cN_i^+$ for south and north backbone paths starting on the north boundary.
    
    By considering simulations (e.g.\ see \cref{fig:intro-paths}) we should expect that the path $\cS_I^- := \cS_{I_n}^-$ delineates the rough-smooth boundary. 
	Now, the Airy paths are one-dimensional in the limit whereas the paths $\cS_i^-$ are plane curves which may backtrack (e.g.\ see \cref{fig:800a07}). Therefore in order to consider the convergence of the paths $\cS_i^-$ to the Airy line ensemble, we should consider the smallest and largest points where these curves intersect each vertical line. First, let $R_\theta$ denote counterclockwise rotation by $\theta$, and define the set
    $$
    \tt{Meso}_n = n(\xi_c, \xi_c) + R_{-\pi/4} ([-n^{3/4}, n^{3/4}] \times [- n^{1/2}, n^{1/2}]),
    $$
    which is a mesoscopic enlargement of the $O(n^{2/3} \times n^{1/3})$ scaling window at the rough-smooth boundary.
    
    Define paths $\cA_i^{n, \pm}:\R \to \R$ by
	\begin{align*}
		\cA_i^{n, +}(t) &= \sup \{ x \in \R: [t, x]_n \in \cS_{I + 1 - i} \cap \tt{Meso}_n \}, \\
		\cA_i^{n, -}(t) &= \inf \{ x \in \R : [t, x]_n \in \cS_{I + 1 -i}  \cap \tt{Meso}_n\}.
	\end{align*}
	\begin{thm}
		\label{T:main-2}
		First, $$
		\P(H_n = 4I_n - n - 1) = 1 - o(1)
		$$
		as $n\to \infty$. In other words, with high probability the path split point is equal to the central height (after shifting the indexing appropriately). 
		
		Next, the two line ensembles $\{\cA_i^{n, \pm} : i \in \N\}$ converge to the parabolic Airy line ensemble $\{\cA_i : i \in \N\}$ in the sense of finite dimensional distributions (FDDs). That is, for any finite set $T \subset \R$ we have
		\begin{equation}
			\label{E:Ain-intro}
			(\cA_i^{n, \pm}|_T : i \in \N) \cvgd (\cA_i|_T : i \in \N).
		\end{equation}
		Moreover, with all assumptions as in \cref{T:main-1} and Remark \ref{R:main-thm-variants}, we have the joint FDD convergence
		\begin{equation}
			\label{E:Ain-joint-intro}
			\begin{split}
				(H_n - \cH_n^{\phi_n}[u]_n, \;  \cA_i^{n, \pm}(t))
				&\implies (-4 \cA(u), \; \cA_i(t)).
			\end{split}
		\end{equation}
        	Here the joint convergence is over $u \in S, t \in T$, and $i \in \N$, where $S \subset \R^2, T \subset \R$ are arbitrary finite or countable sets. 
		\end{thm}

\begin{rem}
It is not particularly important that we intersect with the enlarged boundary region $\tt{Meso}_n$ in the definitions $\cA_i^{n, \pm}$. This could be removed with some extra asymptotic formula analysis, but we have not pursued this here for brevity. The exact scaling of $n^{3/4} \times n^{1/2}$ in the definition of $\tt{Meso}_n$ is essentially arbitrary, and only important in that it is larger than the $n^{2/3} \times n^{1/3}$ Airy scale.
\end{rem}

	\begin{rem}
		\label{R:variants-main-thm-2}
		A version of \cref{T:main-2} also holds using the north Temperleyan backbone paths, see \cref{T:main-2-restatement}. Indeed, an important part of the proof is showing that the naturally paired north and south backbone paths stay close together as they move through the rough-smooth boundary region. Our exact estimate shows that the vertical distance between these paths is $O(n^{1/4} \log^2 n)$, whereas the fluctuation scale is $O(n^{1/3})$, see Corollary \ref{C:close-ties} for details. We expect that the true distance between these pairs of paths is of much lower order, i.e. logarithmic (or at least polylogarithmic) in $n$. 
		
		We can also give a slight strengthening \cref{T:main-2} by replacing the paths $\cA_i^{n, \pm}(t)$ with versions of the form
		$$
		t \mapsto \sup_{|u| \le n^{-1/3 - \eps}} \cA_i^{n, \pm}(t + u), \qquad t \mapsto \inf_{|u| \le n^{-1/3 - \eps}} \cA_i^{n, \pm}(t + u),
		$$
		where $\eps > 0$ is arbitrary, see \cref{T:main-2-restatement}. Note that if $n^{-1/3}$ is replaced by $n^{-2/3}$, then this result becomes essentially immediate from \cref{T:main-2}. Of course, the optimal result in this direction would be to show that the paths $\cA_i^{n, \pm}$ actually converge uniformly to the Airy line ensemble. We state this as a conjecture.
	\end{rem}
	
	\begin{conj}
		\label{conj:Gibbs}
		The two line ensembles $\{\cA_i^{n, \pm} : i \in \N\}$ converge in the uniform-on-compact topology to the  Airy line ensemble $\{\cA_i : i \in \N\}$. 
	\end{conj}
	
	Almost all approaches towards understanding uniform convergence to the Airy line ensemble go through studying a Gibbs property in the prelimit that converges to the Brownian Gibbs property in the limit. This approach was pioneered by Corwin and Hammond \cite{CH11} in the setting of non-intersecting Brownian motions, and extended and refined to prove uniform tightness for a range of one-dimensional non-intersecting or softly non-intersecting path models, e.g.\ see \cite{corwin2016kpz, dauvergne2023uniform, wu2023convergence, barraquand2023spatial, dimitrov2021tightness, corwin2018transversal, aggarwal2024scaling, serio2023tightness, das2025half, dimitrov2025uniform}. The Brownian Gibbs property for the Airy line ensemble also leads to a complete probabilistic characterization of the object, recently proven in \cite{aggarwal2023strong} (see also \cite{dimitrov2021characterization} for an earlier result), which presents the tantalizing possibility of proving convergence to the Airy line ensemble without an appeal to formulas.
    
    In all of these papers, an underlying Gibbs property for non-intersecting paths is crucial for propagating pointwise path control to control on entire sets. In our setting, a natural Gibbs property for the backbone paths should exist. Indeed, because the backbone paths are non-crossing paths coming out of a weighted spanning forest, we expect that there is a tractable Gibbs property related to the study of non-crossing loop erased random walks, coming out of Fomin's theory \cite{fomin2001loop}. However, even given this property, we expect that proving uniform convergence of the paths $\cA_i^{n, \pm}$ would require a new suite of ideas, since in the present setting our paths are two-dimensional, whereas all of the prior Gibbsian line ensemble literature deals with one-dimensional paths.
	
	\section{Proof sketch and some auxiliary results}
	\label{sec:proof-sketch}
	As the paper is quite long, we have opted to give a detailed proof sketch in order to help guide the reader. Along the way, we will present two auxiliary results about the behaviour of the backbone paths in the two-periodic Aztec diamond.

	\subsubsection*{Two challenges and three tools} \qquad
	As discussed previously, there are two central challenges in studying the rough-smooth boundary that are not present at the rough-frozen boundary in tiling models, or in other models where Airy line ensemble convergence has been previously established.
	\begin{enumerate}[label=\Roman*.]
		\item The presence of a noise (in our case, from the smooth phase) which blurs the exact boundary. 
		\item The natural collections of paths defining the boundary are undirected. 
	\end{enumerate}
   There are many models where we expect Airy line ensemble convergence in the presence of effects I and II, see \cref{S:beyond-tilings} for more details. Of these, the two-periodic Aztec diamond is particularly tractable, and we have many refined integrable and probabilistic tools available for understanding the rough-smooth boundary. In the present paper, the tools we use fit into three main categories: 
	\begin{enumerate}[label=\arabic*.]
		\item \textbf{Height function estimates.} \qquad While we cannot understand absolutely everything about the height function directly from formulas, we can essentially compute the expectation of the height function throughout the Aztec diamond with precise asymptotics. We can also gain some control over how the height function concentrates and how it decorrelates. Finally, a more refined analysis at the rough-smooth boundary reveals that certain averages of the height function converge the Airy surface; this was established in \cite{BCJ18}. 
		\item \textbf{Smooth phase couplings.} \quad The full-plane smooth phase measure $\P_a$ with parameter $a \in (0, 1)$ is equivalent (through Temperley's bijection) to a biased spanning forest on $\Z^2$, where edges pointing in the $(1, 1)$ and $(1, -1)$ directions have weight $a^2$ and edges pointing in the $(-1, -1)$ and $(-1, 1)$ have weight $1$. In \cite{BCJ22}, the authors used the determinantal structure of the two-periodic Aztec diamond measure $\P_{a, n}$ to prove the following smooth phase coupling result: for a box of with side length $L$ at the rough-smooth boundary, the total variation distance between $\P_a$ and $\P_{a, n}$ is $O(n^{-1/3} L^8)$. We use both this estimate, and a couple of variants which give smooth phase couplings on larger sets as we move further away from the rough-smooth boundary and into the central smooth region. The smooth phase itself can be understood through Wilson's algorithm for sampling spanning trees.
		\item \textbf{Spanning tree estimates and interlacing.} \qquad  If we condition on the backbone paths in one of the four Temperleyan forests, then we can fill in the remainder using Wilson's algorithm. In particular, individual paths in any of the Temperleyan forests simply follow a biased loop-erased random walk path until they hit the backbone, and so remain in a certain parabola with high probability. This observation starts to give a lot of structure to the spanning forests when we combine it with the fact that the south and north forests are dual to each other. For example, in the corridor between two adjacent south backbone paths, the trajectory of the dual north backbone path cannot position itself in a way that blocks a south loop-erased random walk path from reaching either of the south backbone paths, see \Cref{fig:pathblocking}.
	\end{enumerate}
    \begin{figure}
        \centering
        \includegraphics[width=0.5\linewidth]{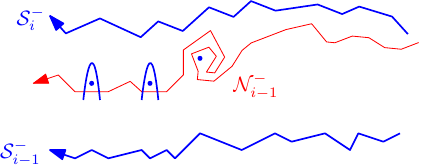}
        \caption{A low-likelihood configuration of backbone paths. The dual north backbone path $\mathcal N_{i-1}^-$ is positioned relative to the south backbone paths $\cS_{i-1}^-, \cS_i^-$ in such a way that makes it very unlikely that a south-drifted random walk from any of the three blue sites will hit $\cS_{i-1}^-$ or $\cS_i^-$ before hitting $\cN_{i-1}^-$. Parabolic random walk range constraints are included for two of the blue sites.}
        \label{fig:pathblocking}
    \end{figure}
	The first two categories above are highly reliant on the specific integrable structure of the two-periodic Aztec diamond, whereas the third is based on more a general probabilistic/combinatorial framework and holds for two-periodic tilings on more general domains. Beyond establishing the tools above, we do not use the exact formulas for the two-periodic Aztec diamond and all arguments are probabilistic or geometric in nature. 

    Most of the tools described above have standard proofs and will not surprise readers familiar with the model. The only result we highlight here is the following  symmetry, recorded again in the body of the text as \cref{L:expected-middle-height}.
	\begin{lem}
		\label{L:expected-middle-height-intro}
        Let $a \in (0, 1]$, and let $(i, j) \in \tt{F}$ be such that $i, j \in 2 \Z$. Then
		$$
		\E \mathcal{H}_n(i, j) = -\E \mathcal{H}_n(-i, j) = - \E \mathcal{H}_n(i, -j) = \E \mathcal{H}_n(-i, -j)
		$$
		In particular, $\E \mathcal{H}_n(2i,0) = \E \mathcal{H}_n(0,2i)= 0$ for all $-n/2 \le i \le n/2$.
	\end{lem}
This lemma is an obvious symmetry when $a = 1$, but beyond this case we do not have a combinatorial explanation. We executed a by-hand verification of the equality $\E \mathcal{H}_n(0,0) = 0$ for small values of $n$, but could not discern a pattern. Note that the expected height at the origin is not an integer in the full-phase smooth phase, so even this symmetry is not coming from a simpler property of the infinite-volume model. 
	
	\subsubsection*{The overview of proof of \cref{T:main-1}} \qquad \cref{T:main-1} is significantly easier than \cref{T:main-2}, and in the context of the present paper can be viewed as a warm-up to the latter statement. Indeed, by analyzing convergence to the Airy surface rather than finding paths that converge to the Airy line ensemble, we circumvent effect II described above. Moreover, the framework of random surface convergence in \cref{T:main-1} can easily incorporate the idea of an Airy line ensemble sitting in a background noise, whereas convergence of line ensembles struggles to capture this directly.
	
	\cref{T:main-1} follows uses the main theorem of \cite{BCJ18}, together with a few of the auxiliary tools described in $1$ and $2$ above and some probabilistic understanding of the smooth phase. Translated into our language, the main theorem of \cite{BCJ18} essentially says the following; a stronger and more precise version is recorded as \cref{T:BCJ18-weak}.
    
    Consider a finite collection of times $T \subset \R$ and a finite collection of disjoint intervals $(x_i, y_i), i \in I$. Then for a particular choice of mollifying sequence $\phi_n$ (see Remark \ref{R:main-thm-variants}), we have:
	\begin{equation}
		\label{E:BCJ-intro}
		\begin{split}(\cH^{\phi_n}_n[t, x_i]_n &- \cH^{\phi_n}_n[t, y_i]_n \;;\; (i, t) \in I \times T) \\
			&\implies (4\cA(t, y_i) - 4\cA(t, x_i) \;;\;  (i, t) \in I \times T).
		\end{split}
	\end{equation}
	There are two key differences between \eqref{E:BCJ-intro} and \eqref{E:height-extra-intro}: the presence of the mollifying sequence in \eqref{E:BCJ-intro} versus the presence of the i.i.d.\ random variables $X_u$ in \eqref{E:height-extra-intro}, and the fact that \eqref{E:BCJ-intro} looks at differences of the height function, rather than the height function itself. Note that in \cref{T:BCJ18-weak}, the scaling of the plane with $n$ differs slightly, but this is a minor effect and easily resolved.
	
	In \cite{BCJ18}, statement \eqref{E:BCJ-intro} is proven by asymptotic analysis of exact formulas. In that context, the mollifying sequence $\phi_n$ is crucial since it renders the limiting object determinantal. The asymptotic analysis yields a limit result for height differences rather than for the heights themselves because the underlying exact formulas are for the dimer configuration, which is a discrete derivative of the height function. To remove the mollifiers $\phi_n$, we will use that locally around each point $[t, x]_n$, the dimer configuration is well-coupled to a smooth phase (tool $2$). Fast decay of correlations in the smooth phase guarantees that the mollifying sequence must have had the effect of simply averaging out this smooth phase. It also guarantees that any noise we see from this smooth phase at a fixed point is independent of the limiting heights. By making this argument rigorous, we can replace \eqref{E:BCJ-intro} with a statement of the following form:
	\begin{equation}
		\label{E:BCJ-intro-prime}
		\begin{split}
			(\cH^{\phi_n}_n[t, x_i]_n &- \cH^{\phi_n}_n[t, y_i]_n \;;\; (i, t) \in I \times T) \\
			&\implies (4\cA(t, y_i) + Y_{i, t} - 4\cA(t, x_i) - X_{i, t} \;;\;  (i, t) \in I \times T),
		\end{split}
	\end{equation}
	where the $X_{i, t}, Y_{i, t}$ are i.i.d.\ random variables, independent of $\cA$,  each equal in law to the height at any $a$-face in the full-plane smooth phase. Next we remove the height differences. First, it is not difficult to reduce this problem to the case when $|I| = 1$, and it will be easier to work in the mollified framework \eqref{E:BCJ-intro}. If we take the point $y_1$ in \eqref{E:BCJ-intro-prime} to be very large then with high probability the right-hand side \eqref{E:BCJ-intro} is simply equal to $\{\cA(x_1, t) : t \in T)$. This suggests that it will be useful to allow $y_1$ to tend to $\infty$ with $n$. In \cref{T:BCJ18-strong}, we extend the analysis of \cite{BCJ18} to prove a version of their main theorem that allow us to do this with $y_1 = \log^{3} n$. Given this, it is enough to show that for fixed $t$, we have
	$$
	\lim_{n \to \infty} \P(\cH_n^{\phi_n}[t, \log^3 n]_n = H_n) = 1.
	$$
	Heuristically, this follows from the fact that the two-periodic Aztec diamond should agree with the full-plane smooth phase on the entire smooth region if we are sufficiently far away from the boundary. In practice, we prove this by patching together a mesoscopic smooth phase coupling near the boundary of the smooth phase that deals with regions at least distance $n^{1/3} \log^2 n$ away from the rough-smooth boundary, and a macroscopic smooth phase coupling that kicks in the moment we are distance $n^{6/7}$ away from the boundary.
	
	\subsubsection*{The overview of the proof of \cref{T:main-2}} \cref{T:main-2} is significantly more involved, and the bulk the paper is devoted to its proof. Before discussing the proof, recall that to recover the height at a vertex $v$ along a path in one of the Temperleyan forests, we simply take the height at the source or sink vertex and add four times the winding number of the path connecting $v$ to its source or sink. The upshot of this is that if paths do not typically wind (as is suggested by the simulations) then the height function should stay relatively stable as we move along a backbone path. Moreover, the different components of the Temperleyan forests can be viewed as essentially flat, decorated with local height fluctuations which come from small windings of the tree.
	
	Focusing now on the south Temperleyan forest, recall also that off the backbone paths, Wilson's algorithm guarantees that the forest paths are simply south-biased loop-erased random walks. In summary, a simplistic picture of the components in the south forest near the rough-smooth boundary suggests that each component sits almost entirely \textit{above} its corresponding backbone path (up to a logarithmic error) and has height function equal to a deterministic shift of the path index (again, up to a logarithmic error). To reconcile this picture with \cref{T:main-1}, we expect that there is a sequence of south backbone paths indexed by $J_n, J_n - 1, J_n - 2, \dots$ with $J_n := \lceil H_n/4 + n/2 \rceil$ which follow the Airy lines. See \Cref{fig:airydoodle} for a cartoon, and \Cref{fig:800a07} for a simulation.

    \begin{figure}
        \centering
        \includegraphics[width=0.9\linewidth]{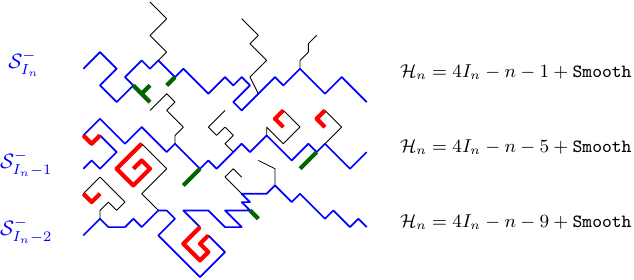}
        \caption{A cartoon of the rough-smooth boundary. The top three prelimiting Airy paths $\mathcal S_{I_n - i}^-$ are pictured, with associated typical corridor heights. Two types of forest paths that can change heights within a corridor are highlighted: red whorls, which change height by winding, and green paths, which connect up to upper backbone path despite the southward random walk drift.}
        \label{fig:airydoodle}
    \end{figure}
	
	\subsubsection*{Controlling global path behaviour} \qquad To establish this picture rigorously, we need to gain a strong level of control over the backbone paths. The principal difficulties arise from the fact that we are working with undirected, rather than directed, paths.

    Before we can address the rough-smooth boundary, we must first rule out pathological \textit{global} behaviour. 
	For purposes of illustration, consider the single path $\cS_J^- = \cS^n_{J_n}$. In order to later conduct an analysis of the rough-smooth boundary, we will need to rule out the following four global phenomena:
	\begin{enumerate}[label=(\roman*)]
		\item $\cS_J^-$ travels macroscopically into the rough or frozen regions.
		\item $\cS_J^-$ travels macroscopically into the smooth region $\tt{Smooth}$, rather than staying confined to a small window around the rough-smooth boundary $\tt{RS}$.
		\item $\cS_J^-$ stays in a small window around the rough-smooth boundary, but winds around the smooth region or backtracks, rather than simply following the rough-smooth boundary from south to west.
		\item $\cS_J^-$ traces the rough-smooth boundary from south to west, but it is not the last path to do so. In other words, $J_n \ne I_n$, where $I_n$ is the split point defined prior to \cref{T:main-2}.
	\end{enumerate} 
    See \Cref{fig:cartoon1} for a picture of the regions above.
	The tools $1$-$3$ introduced above are sufficient to rule out each of these scenarios.

     \begin{figure}
\centering
\begin{subfigure}{0.45\textwidth}
    \includegraphics[width=\textwidth]{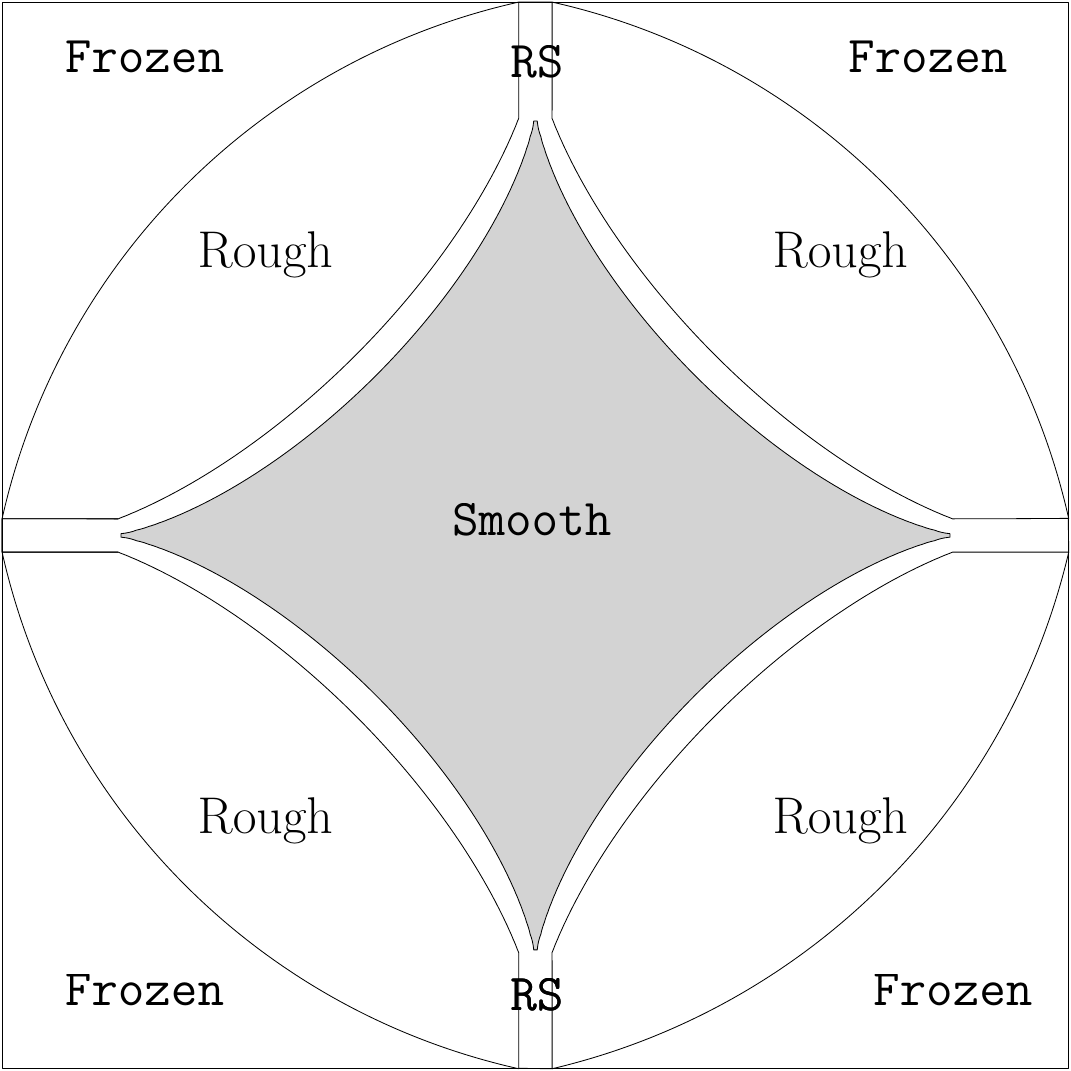}
    \caption{Regions in the two-periodic Aztec diamond.}
    \label{fig:cartoon1}
\end{subfigure}
\hfill
\begin{subfigure}{0.45\textwidth}
    \includegraphics[width=\textwidth]{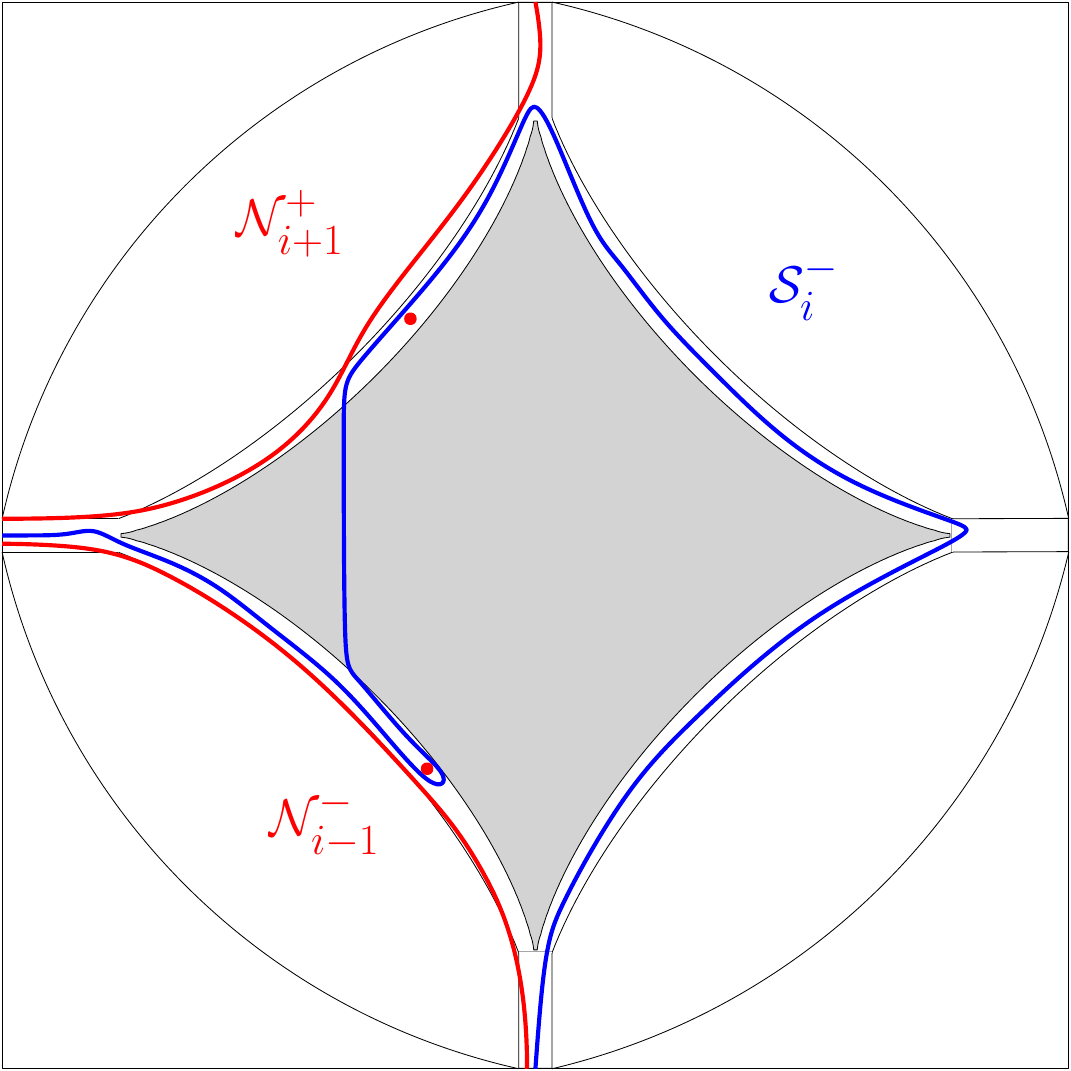}
    \caption{A potentially bad trajectory for a path $\mathcal S_i$ starting in the region $\tt{RS}$.}
    \label{fig:cartoon2}
\end{subfigure}
\end{figure}

	We can address scenario (i) with expectation and concentration estimates on the height function (tool $1$). Such estimates give that $H_n = O(1)$, and outside of a small, simply connected window around the central smooth region and the rough-smooth boundary $\tt{RS}$, the height deviates by more than $O(1)$ away from $0$. Since the path $\cS_{J}$ starts out at a height of $H_n + O(1)$ by definition, it cannot exit this window without winding. On the other hand, a path that starts on the 
	boundary of a simply connected domain must have winding number zero upon first exit. This prevents $\cS_J$ from exiting this window, ruling out (i).

    We address (ii) and (iii) together. Consider a single path $\cS_i^-$, which starts on the part of the south boundary in the set $\tt{RS}$, and ends on the west boundary. First, if we can couple the whole smooth region with the full-plane smooth phase, then this path behaves as a south-drifted random walk whenever it is in the region $\tt{Smooth}$. At a global scale, this means that it can only move directly south while in this region. This greatly restricts the kind of trajectories that the path can have, see \Cref{fig:cartoon2} for an example. From here, the rough idea for proving (ii) and (iii) is that if the path $\mathcal S_i^-$ does not simply follow the rough-smooth boundary curve from south to west, then we can find vertices in the north forest from which a loop-erased random walk would have to behave atypically, i.e.\ breaking the parabolic bounds in \Cref{fig:pathblocking}. 
    
    These points are easy to find if only a single path $\cS_i^-$ has a bad trajectory. Two such points are shown in red in \Cref{fig:cartoon2}. If multiple paths have bad trajectories, then we can create nested \textit{onions} of bad paths. 
	~\cref{L:onion-control} systematically extends the path-blocking idea in \Cref{fig:pathblocking} to give us quantitative control on onions in terms of their heights, widths, and number of layers. Together with the height estimates used in establishing (i), and the drift estimates in the smooth region, this  gives sufficient control to eliminate (ii) and (iii).

	Once (i)--(iii) are ruled out for $\cS_J^-$ and its neighboring south backbone 
	paths, we can conclude that there exists a topmost path $\cS_I^-$ tracing the 
	rough-smooth boundary from south to west within the region $\operatorname{RS}$. 
	By standard random walk estimates, all south forest paths that start above 
	$\cS_I^-$ within the smooth region $\operatorname{Smooth}$ must eventually coalesce with 
	$\cS_I^-$. By the smooth phase coupling, these south forest paths start at height $H_n$ on average. Therefore the starting height 
	of $\cS_I^-$ must coincide with the central height $H_n$, unless $\cS_I^-$ itself 
	winds around the smooth region—a possibility ruled out by (i)--(iii). 
	This forces $I_n = J_n$, resolving (iv).
	
	While a full implementation of the argument eliminating (i)--(iv) is possible using the techniques developed in this paper, doing so would require height function estimates on the whole Aztec diamond and smooth phase couplings that hold throughout the entire smooth region. To avoid some of these technicalities, we will instead prove a weaker statement which still suffices for our study of the rough-smooth boundary. The weak version replaces each region with a kind of topological skeleton and focuses on whether paths cross barriers imposed by this skeleton. A precise version is as follows.
	
	First, let $\alpha > 0$ be the unique point such that $(-\alpha, 0)$ and $(0, -\alpha)$ lie on the rough-smooth boundary curve. Let $\xi_0:[-\alpha, \alpha] \to \R^2$ be the rough-smooth boundary curve in the third quadrant, parametrized so that for all $a \in [-\alpha, \alpha]$, the point $\xi_0(a)$ is on the line $\{ {\bf v} \in \R^2 : \mathbf{v} \cdot e_2 = a \}$ (here recall that $e_1 = (1, 1), e_2 = (-1, 1)$). Define the coordinate change $\beta_n:[-\alpha n,  \alpha n] \times \R \to \R^2$ by
	$$
	\beta_n(t, x) = n\xi_0(t/n) + e_1 x.
	$$
	In words, $\beta_n(t, x)$ gives a point at location $t$ along the rough-smooth boundary curve, at distance $x$ away from the boundary. Note that the spatial scaling in $\beta_n$ does not change with $n$; the scaling $n \xi_0(\cdot/n)$ simply puts the limit curve back in unscaled coordinates. Next, define regions
		\begin{align*}
		\tt{RS}_n^* &:= \beta_n([-2n^{3/4}, 2n^{3/4}] \times [- n^{1/2} \log^{3/2} n, 2n^{1/3} \log^2 n]), \\
		\tt{PRS}_n^* &:= \beta_n([-2n^{3/4}, 2n^{3/4}] \times [- n^{1/2} \log^{2} n, n^{3/4}]),
	\end{align*}
	We view the set $\tt{RS}_n^*$ as an enlarged version of the region of the rough-smooth boundary where we will eventually see Airy fluctuations,
	and $\tt{PRS}_n^*$ as a padded version. The exact exponents in the definitions of $\tt{RS}_n^*, \tt{PRS}_n^*$ are not so important. What is important is that the $O(n^{2/3} \times n^{1/3})$-region where we are studying fluctuations is lower order compared to $\tt{RS}_n$, whose height is lower order compared to $\tt{PRS}_n$. Next, let $\tt{X} = \{{\bf v} \in \R^2 : v_1 = \pm v_2\}$.  See \Cref{fig:introtheorem} for a sketch of these regions, and an illustration of the upcoming theorem.
	\begin{thm}
		\label{T:main-3}
		Recall that $I_n$ is the label of the last south backbone path $\cS_i^-$ to move from the south boundary to the west boundary. Then with probability $1 - o(1)$ as $n \to \infty$, the following events hold:
		\begin{enumerate}[label=\arabic*.]
			\item (Boundary paths avoid the skeleton and are funnelled through the boundary) The paths $\cS_i^-$ with $I_n - n^{1/4} \le i \le I_n$ do not intersect the region $(\tt{X} \cup \tt{PRS}_n^*)\setminus \tt{RS}_n^*$.
			\item ($\cS_{I_n}^-$ is the top path) The path $\cN_{I_n}^-$ does not intersect the region $\tt{PRS}_n^*$. By interlacing, neither do any of the paths $\cS_i, i > I_n$.
			\item (Height matching) $H_n = 4I_n - n - 1$. 
		\end{enumerate}		
	\end{thm}

     \begin{figure}
        \centering
        \includegraphics[width=0.5\linewidth]{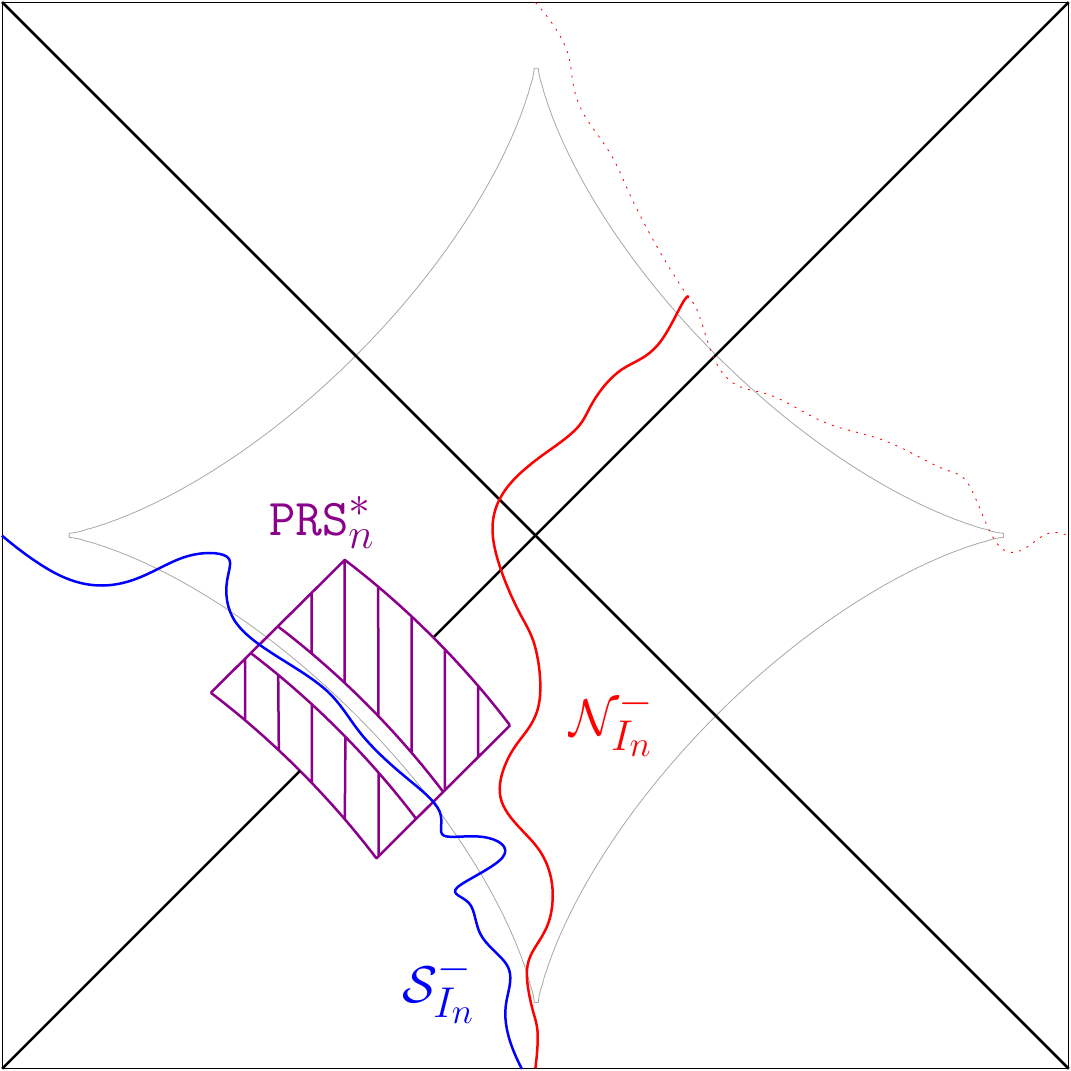}
        \caption{Theorem \ref{T:main-3} gives global control of backbone paths. One of the main upshots of the theorem is that the path $\cS_{I_n}^-$ avoids the region $(\tt{X} \cup \tt{PRS}_n^*)\setminus \tt{RS}_n^*$, whereas the path $\cN_{I_n}^-$ avoids the region $\tt{PRS}_n^*$ entirely before meeting a path started on the north side of the Aztec diamond. This establishes that $\cS_{I_n}^-$ is the top path moving through the rough-smooth boundary}
        \label{fig:introtheorem}
    \end{figure}
 
	Point $1$ in the above theorem should be viewed as a weak resolution to problems (i)-(iii) above, point $2$ a resolution to problem (iv), and point $3$ an upshot. Note that the precise powers of $n$ in the definition of $\tt{RS}_n$ are not particularly important, beyond the fact that $n^{5/6} \gg n^{2/3}$ and $n^{1/2} \gg n^{1/3}$. \cref{T:main-3} is restated with more detail in the sequel as \cref{T:main-3-restatement}.
	
	\subsubsection*{Path regularity at the rough-smooth boundary.} \qquad Given the weak global control in \cref{T:main-3}, we have a starting point for understanding the paths $\cS_i$ at the rough-smooth boundary. If we can additionally prove a kind of weak path regularity locally, then we can deduce \cref{T:main-2}. To this end, we will prove the following estimate as part of \cref{P:backtrack-estimate}.
	\begin{thm}
		\label{T:overhang-thm}
        The following estimates holds simultaneously with probability $1 - o(1)$ as $n \to \infty$. Let $k \in \{0, \dots, \lfloor n^{1/4} \rfloor\}$, and suppose that $v = (v_1, v_2), w = (w_1, w_2)$ are two vertices in $\tt{RS}_n^* \cap \cS_{I-k}^-$ where $v$ occurs \textbf{before} $w$ on the path $\cS_{I-k}^-$. Then
			\begin{equation}
            \label{E:backtrack-est}
				w_1 - v_1 < (2k + 1) n^{1/4} \log^{2} n.
			\end{equation}
	\end{thm}   
\cref{T:overhang-thm} is a kind of \textit{backtrack estimate}, which gives quantitative control on how far the undirected paths $\cS_i$ are from being directed as they move through the rough-smooth boundary. The estimate shows that for paths of index $I -k$ with $k = O(1)$, the maximum size of a backtrack in the rough-smooth boundary region is of order $O(n^{1/4} \log^2 n)$, which is lower order than the Airy scaling. \cref{T:overhang-thm} is not expected to be sharp, and indeed, the true backtrack bound above should be polylogarithmic. 

In order to illustrate how \cref{T:overhang-thm} feeds into \cref{T:main-2}, let us temporarily assume we can show \eqref{E:backtrack-est} with a stronger bound of size $\alpha_n = o(n^{1/24})$. 

Fix $x, t \in \R$, and let $v = [x, t]_n$. Let $\pi$ be the south forest path starting incident to $v$, and ending at a sink vertex $s$. Let $w$ be the first vertex on a backbone path $\cS_i$ visited by the path $\pi$, and let $w'$ be the first vertex visited by $\pi$ on $\cS_i$ whose height $\cH_n(w')$ equals the sink height $\cH_n(s)$. The height $\cH_n(v)$ can be written as a sum of the sink height $\cH_n(s) = \cH_n(w')$, and the relative height $\cH_n(v) - \cH_n(w')$. Now, we can use a local smooth phase coupling around $v$ (tool $2$) to say that in a box $B$ of width $4 \alpha_n$ around $v$, the dimer configuration is close to the smooth phase (since $n^{-1/3} \cdot \alpha_n^8 = o(1)$). 
	Because any south forest path entering a smooth region experiences a strong downward drift, this makes it highly unlikely that $\cS_i$ will enter a slightly smaller box $B'$ of width $2 \alpha_n$ centered around $v$. In particular, we expect $w'$ to lie outside of $B'$, and the height difference $\cH_n(v) - \cH_n(w')$ should be given by the height of $v$ in the smooth phase coupling on $B'$, plus a contribution if the path $\pi$ ever winds around the box $B'$. By our backtrack estimate, winding cannot happen around boxes of width larger than $\alpha_n$, and so $\pi$ cannot wind around the box $B'$. This allows us to identify $\cH_n(v) - \cH_n(w')$ as the smooth term $X$ in \cref{T:main-1}. The sink height $\cH_n(w')$ (equivalently, the path index $i$ after a deterministic shift) then gives the Airy term. Applying this idea on a fine mesh of vertices $v$ and using some of the weak global path control from \cref{T:main-3} then translates \cref{T:main-1} into \cref{T:main-2}. 

    Now, the above argument fails if we directly apply the estimate \eqref{E:backtrack-est}, since $n^{1/4} \log^2 n$ is much larger than the $o(n^{1/24})$-scale on which we have a high-probability smooth phase coupling. In order to work around this issue, rather than executing the above argument with a single smooth phase coupling, we construct a ring of boxes of width $\log^3 n$ which will act as a path barrier. By tool $2$ above, the smooth phase coupling succeeds each box with probability at least $1 - O(n^{-1/3} \log^{24} n)$, and so by a union bound, we can create a ring containing $n^{1/3 - \eps}$ overlapping boxes, and couple each box to a different copy of the smooth phase so that all couplings simultaneously succeed with probability at least $1 - n^{-\eps/2}$. Now, the box width we have chosen is large enough so that drift estimates on smooth phase imply that the probability of a backbone path passing through any box in this ring is smaller than any polynomial in $n$. Therefore with high probability, we will only have a winding contribution to $\cH_n(v) - \cH_n(w')$ if one of the paths $\cS_i$ winds around the entire ring, which has width at least $n^{1/3 - \eps}$.  For $\eps < 1/12$,  we have $n^{1/3 - \eps} \gg n^{1/4} \log^2 n$, and so \cref{T:overhang-thm} can guarantee that this does not happen.
	
\subsubsection*{The overview of the proof of \cref{T:overhang-thm}.} \qquad We start with the bound on the top path $\cS_I^-$. From \cref{T:main-3}, $\cS_I^-$ moves through the region $\tt{RS}_n^*$ without entering the region $\tt{X} \cup \tt{PRS}_n^*$, whereas $\cN_I^-$ never enters the region $\tt{PRS}_n^*$. In other words, the picture will resemble \Cref{fig:introtheorem}. Now consider a north forest path $\pi$ starting just above $\cS_I^-$ in the region $\tt{RS}_n^*$. By Wilson's algorithm, the path $\pi$ is simply the loop-erasure of biased random walk, stopped when it hits $\cN_I^-$. In particular, it will move as a biased loop-erased walk until it exits $\tt{RS}_n^*$, and so with high probability, it will stay in a parabola opening up vertically while it sits in this region. Since the region $\tt{RS}_n^*$ has height $O(n^{1/2} \log^{3/2} n)$, the parabola has maximum width given by the square root of this: $O(n^{1/4} \log^{3/4} n)$. Now, if the path $\cS_I$ has backtracks which are much larger than this scale, it will trap some of these loop-erased walk paths, contradicting Wilson's algorithm. 
	
	Now, we can apply a similar idea to bound the backtrack size of the top north backbone path $\cN_{I-1}^-$. Unlike the path $\cS_I^-$, as this path moves through $\tt{RS}_n^*$ it in a corridor between $\cS_{I-1}^-$ and $\cS_I^-$, rather than in a region which is unbounded above. However, the backtrack estimate on $\cS_I^-$ guarantees that the upper boundary of this corridor is still regular, and so a version of the argument still applies. We can continue moving through the lines $\cS_{I - 1}^-, \cN_{I-2}^-, \dots$ inductively, sacrificing only an additive factor with each successive line index.

	\section{Trees, forests, and Temperley's bijection}
	\label{S:trees}
	In this section, we describe and prove Temperley's bijection for the Aztec diamond, and describe its relationships to the height function and dimer representations. Along the way, we will prove some basic combinatorial and probabilistic properties of these correspondences (e.g. symmetries, interlacing, Wilson's algorithm).
	
	Before discussing Temperley's bijection, we record three basic symmetries of the two-periodic Aztec diamond. Here we let $f_* \mu$ denote the pushforward of a measure $\mu$ under a map $f$. The following lemma is immediate. 
	
	\begin{lem}
		\label{L:sym}
		The measure $\P_{a, n}$ is invariant under the following three operations sending the south vertex set $\tt{S}$ to $\tt{E}, \tt{W}$, and $\tt{N}$, respectively.
		\begin{itemize}
			\item $\P_{a, n} = (\rho_\tt{E})_* \P_{a, n}$ where $\rho_\tt{E}(x, y) = (-y, -x)$. The map $\rho_\tt{E}$ sends $\tt{S} \mapsto \tt{E}, \tt{E} \mapsto \tt{S}$ and $\tt{N} \mapsto \tt{W}, \tt{W} \mapsto \tt{N}$.
			\item $\P_{a, n} = (\rho_\tt{W})_* \P_{a, n}$ where $\rho_\tt{W}(x, y) = (y, x)$. The map $\rho_\tt{W}$ sends $\tt{S} \mapsto \tt{W}, \tt{W} \mapsto \tt{S}$ and $\tt{N} \mapsto \tt{E}, \tt{E} \mapsto \tt{N}$.
			\item $\P_{a, n} =(\rho_\tt{N})_* \P_{a, n}$ where $\rho_\tt{W}(x, y) = (-x, -y)$. The map $\rho_\tt{N}$ sends $\tt{S} \mapsto \tt{N}, \tt{N} \mapsto \tt{S}$ and $\tt{W} \mapsto \tt{E}, \tt{E} \mapsto \tt{W}$.
		\end{itemize}
	\end{lem}

	\subsection{A generalization of Temperley's bijection} \qquad
	
	In this subsection, we prove the generalized Temperley's bijection introduced in Section \ref{S:temperley-intro}, along with a handful of other combinatorial properties of this correspondence. First, recall from Section \ref{S:temperley-intro} the definition of the extended south-vertex set $\tt{\bar S}$, the south-vertex graph $\tt{G(S)}$, the source and sink boundaries  $\partial^{\tt{o}} \tt{S}, \partial^{\tt{x}} \tt{S}$, and the notion of a dimer-compatible forest $F$ on $\tt{G(S)}$. 
	
	We give similar definitions for the north vertices using the rotation $\rho_\tt{N}$ is \cref{L:sym}. Define $\tt{\bar N} = \rho_\tt{N} \tt{\bar S}, \tt{G}(\tt{N}) = \rho_\tt{N} (\tt{G}(\tt{S}))$ and say that a spanning forest $F$ of $\tt{G}(\tt{N})$ is dimer-compatible if $F = \rho_\tt{N} F'$ for a dimer-compatible spanning forest $F'$ on $\tt{G}(\tt{S})$. We define  $\partial_N \tt{N}, \partial_S \tt{N}, \partial_E \tt{N}, \partial_W \tt{N}$ exactly as in \eqref{E:cardinal-boundaries} but with $\tt{N}$ in place of $\tt{S}$, and set $\partial^\tt{x} \tt{N} = \partial_E \tt{N} \cup \partial_W \tt{N}, \partial^\tt{o} \tt{N} = \partial_N \tt{N} \cup \partial_S \tt{N}$, and $\partial N = \partial^\tt{x} \tt{N} \cup \partial^\tt{o} \tt{N}$.
	
	Note that we could give similar definitions for the east and west vertices using the other two symmetries in \cref{L:sym}, and develop Temperley's bijection in these settings as well. We will not need these bijections in the present paper.
	
	The wiring structure for paths starting on the source boundary in a DCF is quite rigid. We specify this structure in the next lemma. We state it for DCFs on $\tt{G}(\tt{S})$, but a symmetric version holds for $\tt{G}(\tt{N})$ by applying the map $\rho_\tt{N}$. See \cref{fig:temperley-labelled} for an example of \cref{L:tree-paths}. 

    \begin{figure}
		\centering
        \includegraphics[scale=1]{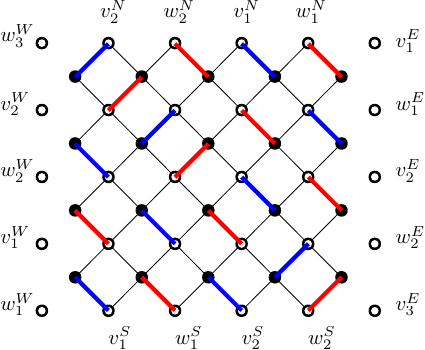} 
    \vspace{2em}

		\includegraphics[scale=1]{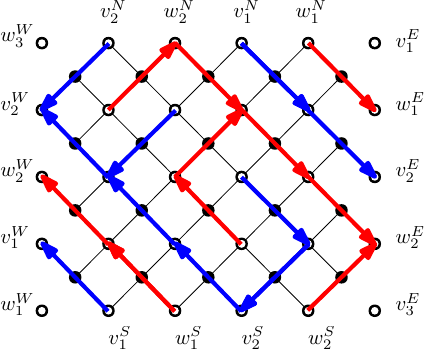}
        \caption{A dimer configuration on an Aztec diamond with $n = 4$, together with the associated pair of north (red) and south (blue) Temperleyan forests. Boundary vertex labelling is as in \cref{L:tree-paths}, \cref{L:dimer-compatible-duals}. Here $I = 2$.}
        \label{fig:temperley-labelled}
	\end{figure}
	
	\begin{lem}
		\label{L:tree-paths}
		Label the vertices of $\partial \tt{S}$ as follows, moving clockwise in a cycle around the boundary from the northeast corner $(n+1, n)$:
		$$
		v_1^E, \dots, v_{n/2+1}^E, v^S_{n/2}, \dots, v^S_1, v^W_1, \dots, v^W_{n/2}, v^N_{n/2}, \dots, v^N_1. 
		$$
		Now, consider a DCF $F$ on $\tt{G}(\tt{S})$, and recall the map $\theta_F:\partial^\tt{o} \tt{S} \to \partial^\tt{x} \tt{S}$ from Definition \ref{D:dc-forest}(iii). Then
		\begin{enumerate}[label=\arabic*.] 
			\item There is a split point $I = I(F) \in \{0, \dots, n/2\}$ such that for $j \le I$ we have $\theta_F(v_j^S) \in \partial_W\tt{S}$, and for $j > I$ we have $\theta_F(v_j^S) \in \partial_E\tt{S}$.
			\item Given $I = I(F)$, the map $\theta_F$ satisfies:
			$$
			\theta_F(v_j^S) = \begin{cases}
				v_j^E, \quad &j > I \\
				v_j^W, \quad &j \le I 
			\end{cases}, \qquad \theta_F(v_j^N) = \begin{cases}
				v_{j+1}^E, \quad &j < I \\
				v_j^W, \quad &j > I 
			\end{cases}
			$$
			If $I > 0$, then $\theta_F(v^N_I) \in \{v^W_I, v^E_{I+1}\}$.
		\end{enumerate}
	\end{lem}
	
	\begin{proof}
		Point $1$ holds by planarity. For point $2$, observe that we can translate Definition \ref{D:dc-forest}(iii) into the condition that for $j \in \{1, \dots, n/2\}$ we have
		$$
		\theta_F(v^S_j) \in \{v_j^E, v^W_j\}, \qquad \theta_F(v^N_j) \in \{v_{j+1}^E, v^W_j\}.
		$$ 
		The choice of $\theta_F$ is determined as in point $2$ by this constraint, the definition of $I$, and planarity.
	\end{proof}
	
	\begin{rem}
		\label{R:labelling}
		The labelling on the boundary vertices match the height function, up to an affine shift. For example, for every south vertex $v = v^S_i \in \partial_S \tt{S}$, let $f_v = v - (1, 0)$. The face $f_v$ is always in $\partial \tt{F}$ so it has a deterministic height, which satisfies
		$$
		h(f_{v_i^{*}})/4 + n/4 = i - 1.
		$$
		Similar affine relationships hold for the vertices on the other three boundaries. Note the similarity between the relationship here and the relation between $I_n$ and $H_n$ in \cref{T:main-2}. The takeaway in all settings is that \textit{path index equals height-over-$4$} (up to shifting the labels). However, here the relationship is straightforward and deterministic, whereas in \cref{T:main-2} the relationship only holds with high probability, and is a consequence of our global control on path behaviour. 
	\end{rem}

	Now, the graphs $\tt{G}(\tt{S}), \tt{G}(\tt{N})$ are duals of each other, in the following sense. Let $E_\tt{S}$ and $E_\tt{N}$ denote the edge sets of $\tt{G}(\tt{S})$ and $\tt{G}(\tt{N})$, represented as line segments in the plane. Then the map taking every edge in $E_\tt{S}$ and rotating by $\pi/2$ about its midpoint is a bijection from $E_\tt{S}$ to $E_\tt{N}$.  
	In particular, every edge in $E_\tt{S}$ crosses exactly one edge in $E_\tt{N}$, and we refer to such edges as duals.
	This dual relationship extends to dimer-compatible forests. 
	
	\begin{lem}
		\label{L:dimer-compatible-duals}
		Consider a DCF $F$ on $\tt{G}(\tt{S})$. Define the dual $\check F$ of $F$ to be the subgraph of $\tt{G}(\tt{N})$ with vertex set $\tt{\bar N}$ and edge set consisting of all edges in $E_\tt{N}$ that do not cross an edge in $F$. Then $\check F$ is a DCF on $\tt{G}(\tt{N})$. 
		
		Moreover, if the split point $I$ is given as in \cref{L:tree-paths}, we have the following description for $\theta_{\check F}$. Extend the labelling of $\partial \tt{S}$ in \cref{L:tree-paths} to all vertices of $\partial \tt{N} \cup \partial \tt{S}$ as follows, so that moving clockwise in a cycle around the boundary from the northeast corner $(n+1, n)$ we have:
		\begin{align*}
			&v_1^E, w_1^E, v_2^E, \dots, v^E_{n/2}, w_{n/2}^E, v_{n/2+1}^E, w^S_{n/2}, v^S_{n/2}, \dots, w^S_1, v^S_1, \\ &w^W_1, v^W_1,  w^W_2, \dots, w^W_{n/2}, v^W_{n/2}, w^W_{n/2 + 1}, v^N_{n/2}, w^N_{n/2}, \dots, v^N_{1}, w^N_{1}. 
		\end{align*}
		Then 
		$$
		\theta_{\check F}(w_j^S) = \begin{cases}
			w_j^E, \quad &j > I \\
			w_{j+1}^W, \quad &j < I 
		\end{cases}, \qquad \theta_{\check F}(w_j^N) = \begin{cases}
			w_j^E, \quad &j \le I \\
			w_j^W, \quad &j > I 
		\end{cases}
		$$
		If $I > 0$, then $\theta_{\check F}(w^S_I) = w_I^E$ if $\theta_F(v^N_I) = v^W_I$ and $\theta_{\check F}(w^S_I) = w_{I+1}^W$ if $\theta_F(v^N_I) = v^E_{I+1}$.
	\end{lem}
	
	\cref{fig:temperley-labelled} may help with parsing the labelling in \cref{L:dimer-compatible-duals}.
	
	\begin{proof}
		First observe that $\check F$ cannot contain a cycle. Indeed, any cycle in $\check F$ would surround a component in $F$. On the other hand, every component in $F$ contains a vertex from $\partial^\tt{x} \tt{S}$, and there are no cycles in $\tt{G}(\tt{N})$ which surround a vertex in $\partial^\tt{x} \tt{S}$. Hence $\check F$ is a forest, giving Definition \ref{D:dc-forest}(i). Next, 
		$$
		|E_\tt{N}| = |\tt{E}| + |\tt{W}| =|\tt{S}| +  |\tt{N}|.
		$$
		This follows since the edge set $E_\tt{N}$ is in bijection with the east/west vertex set $\tt{E} \cup \tt{W}$ under the map taking an edge $e$ to its midpoint. Now, $|E(F)| = |\tt{S}|$, as $F$ is a forest with $|\partial^\tt{x} \tt{S}|$-many components and vertex set $\partial^\tt{x} \tt{S} \cup \tt{S}$. Therefore $|E(\check F)| = |\tt{N}|$ since $|E(F)| + |E(\check F)| = |E_\tt{N}|$ by duality. Since $\check F$ is a forest with vertex set $\tt{N} \cup \partial^\tt{x}\tt{N}$, it contains $|\partial^\tt{x}\tt{N}|$-many components. 
		
		Suppose a component $C$ of $\check F$ contains no vertices in $\partial^\tt{x}\tt{N}$. Then there must be two components $D, D'$ (possibly the same) of $F$ that separate $C$ from the east and west boundaries $\partial_E \tt{N}, \partial_W \tt{N}$, respectively. Component $D$ must contain vertices in $\partial_E \tt{S}, \partial_S \tt{S},\partial_N \tt{S}$, whereas component $D'$ must contain vertices in $\partial_W \tt{S}, \partial_S \tt{S},\partial_N \tt{S}$. Part 2 of \cref{L:tree-paths} shows that it is not possible for both these components to exist. Together with the component count for $\check F$, this yields Part (ii) of \cref{D:dc-forest}.
		
		Given Parts (i) and (ii) of \cref{D:dc-forest}, we can define the map $\theta_{\check F}$ mapping vertices in $\partial^\tt{o} \tt{N} \to \partial^\tt{x} \tt{N}$. The planar dual relationship between $F, \check F$ allows us to determine $\theta_{\check F}$ from $\theta_F$ (given by \cref{L:tree-paths}) as in the statement of the lemma. This also guarantees Part (iii) of \cref{D:dc-forest}.
	\end{proof}

	Next, we state Temperley's correspondence for the Aztec diamond. Given a dimer configuration $D$ on $\tt{A}_n$, for any choice of $\tt{C} \in \{\tt{N}, \tt{S}, \tt{E}, \tt{W}\}$ we can construct a directed subgraph $F_{\tt{C}}(D)$ of $\tt{G}(\tt{C})$ as follows. For every vertex $w \in \tt{C}$, there is a unique dimer in $D$ of the form $\{w, w + e\}$ where $e \in \{\pm e_1, \pm e_2\}$. Add the directed edge $(w, w + 2 e)$ to the graph $F_{\tt{C}}(D)$. 
	
	\begin{prop}
		\label{P:temperley-bij}
		For every even value of $n \in 4 \N$ and $\tt{C} \in \{\tt{N}, \tt{S}\} $, the map $D \mapsto F_{\tt{C}}(D)$ is a bijection between dimer configurations on $\tt{A}_n$ and dimer-compatible forests on $\tt{G}(\mathtt{C})$, directed towards the sink boundary $\partial^\tt{x} \mathtt{C}$. 
		
		The inverse map $D_\tt{C}$ of $F_{\tt{C}}$ is given as follows. Consider a DCF $F$ on $\tt{G}(\mathtt{C})$, and construct the dual forest $\check F$ on $\tt{G}(\mathtt{\check C})$, where $\tt{\check S} = \tt{N}, \tt{\check N} = \tt{S}$ by including an edge $e$ in $\check F$ if and only if $e$ does not cross an edge in $F$. Then $\check F$ is a DCF in $\tt{G}(\mathtt{\check C})$. Orient $F$ towards $\partial^\tt{x} \tt{C}$ and $\check F$ towards $\partial^\tt{x} \tt{\check C}$, and let
		$$
		D_\tt{C}(F) = \{(v, v + e) : (v, v + 2e) \text{ is a directed edge in } F \text{ or } \check F \}.
		$$
		Finally, for a dimer configuration $D$, the dual forest $\check F_{\tt{C}}(D)$ to $F_{\tt{C}}(D)$ is $F_{\tt{\check C}}(D)$.
	\end{prop}
	To prove \cref{P:temperley-bij}, we will need to track height changes along forest paths. This is taken care of by the next two results.
	
	\begin{lem}
		\label{L:tree-path-heights}
		Let $D$ be a dimer configuration on $\tt{A}_n$ and let $P = (s_0, s_1, \dots, s_k)$ be a directed path of vertices in the graph $F_{\tt{S}}(D)$. For $\eps \in \{(0, \pm 1), (\pm 1, 0)\}$, define the face-path 
		$$
		F^\eps(P) = (f_0^\eps, \dots, f_k^\eps) = (s_0 + \eps, \dots, s_k + \eps).
		$$
		Then the height function $h$ for $D$ evolves along $F^v(P)$ as follows. Here for two points $a, b \in \R^2$ we write $[a, b]$ for the line segment from $a$ to $b$, oriented in that way.
		\begin{itemize}
			\item For $i \in \{1, \dots, k-1\}$, if $[s_i, s_{i+1}]$ does not cross $[f_{i-1}^\eps, f^\eps_i]$, then $h(f_{i-1}^\eps) = h(f_i^\eps)$.
			\item For $i \in \{1, \dots, k-1\}$, if $[s_i, s_{i+1}]$ crosses $[f_{i-1}^\eps, f^\eps_i]$ from right-to-left, then $h(f_{i-1}^\eps) + 4 = h(f_i^\eps)$.
			\item For $i \in \{1, \dots, k-1\}$, if $[s_i, s_{i+1}]$ crosses $[f_{i-1}^\eps, f^\eps_i]$ from left-to-right, then $h(f_{i-1}^\eps) - 4 = h(f_i^\eps)$.
		\end{itemize}
		The same height evolution holds for $F_{\tt{N}}(D)$.
	\end{lem} 
	
	\cref{L:tree-path-heights} is readily verified from the definition of the height function.
	
	\begin{cor}
		\label{C:vertex-heights}
		For a vertex $v$ in a dimer configuration $D$, let its height $h(v)$ be the average height of the four incident faces. Then in the south-forest setup of \cref{L:tree-path-heights}, we have that
		$$
		h(s_{k-1}) - h(s_0) = L - R,
		$$
		where $L$ is the number of left turns made by the path $P$, and $R$ is the number of right turns. The same holds for a north-forest path.
	\end{cor}

	\begin{rem}\label{rem:average-height}
		We remark that $h(v)$ as defined above uniquely specifies the dimer incident to $v$. For instance, if $v\in \mathtt{S}$ and fix the $a$-face $v+(0,1)$ to have height $1$. Then we have that $h(v)$ is $\frac{5}{2}$, $\frac{3}{2}$, $\frac{1}{2}$, and $-\frac{1}{2}$ for the dimers $(v,v+e_1)$, $(v,v-e_2)$, $(v,v-e_1)$, and $(v,v+e_2)$ respectively. Moreover, it is not hard to see that given the average height function for all vertices of the Aztec diamond, we can recover the tiling. 
	\end{rem}

	\begin{proof}[Proof of \Cref{C:vertex-heights}]
		It suffices to prove the result when $k = 2$, as the general case follows by concatenation. The $k=2$ case follows from \cref{L:tree-path-heights} and averaging out the height changes amongst all of the incident faces around two consecutive vertices on a path. 
	\end{proof}
	
	\begin{proof}[Proof of Proposition \ref{P:temperley-bij}]
		We consider the case $\tt{C} = \tt{S}$, as the $\tt{C} = \tt{N}$ case is symmetric.
		First we show that $F_{\tt{S}}(D)$ contains no (undirected) cycles.
		Suppose that $C = (s_1, \dots, s_k = s_1)$ is a cycle in $F_{\tt{S}}(D)$. First, we necessarily have that $C$ is also a directed cycle, since otherwise some vertex in $C$ would have two outgoing edges. As a directed cycle, $C$ either has four more left turns then right turns, or vice versa. By tracking height changes around $C$ using Corollary \ref{C:vertex-heights}, this gives two different values for the height of $s_1$, which is not possible.
		
		Therefore $F_{\tt{S}}(D)$ is a directed acyclic graph. By construction, the only sinks in $F_{\tt{S}}(D)$ are the elements of $\partial^\tt{x} \tt{S}$, so every component of $F_{\tt{S}}(D)$ must contain at least one element of $F_{\tt{S}}(D)$. On the other hand, the number of edges in $F_{\tt{S}}(D)$ in $|\tt{S}|$, whereas the number of vertices is $|\tt{S}| + |\partial \tt{S}|$ and so $F_{\tt{S}}(D)$ has exactly $n+1 = |\partial \tt{S}|$ components. Therefore each component has exactly one element of $\partial^\tt{x} \tt{S}$, as desired.
		
		Now consider a vertex $v = (i, -n) \in \partial_S \tt{S}$, and let $P = (v = v_1, \dots, v_k = u)$, where $u \in \partial^\tt{x} \tt{S}$. First assume $u \in \partial_W \tt{S}$. Define an augmented path $P' = (v_0, v_1, \dots, v_{k+1})$ by letting $v_0 = v_1 + 2e_2 = (i - 2, -n - 2)$ and $v_{k+1} = v_k - 2e_2$, adding in the dimers $\{v_0, v_0 + e_2\}, \{v_k, v_k + e_2\}$ to the dimer configuration $D$ and extending the height to include all incident faces to $v_0, v_{k+1}$. The new path $P'$ does not wind around its initial vertex, and its initial and final edges have the same direction. Therefore it has an equal number of left and right turns, and so $h(v_0) = h(v_{k+1})$ by Corollary \ref{C:vertex-heights}. On the other hand, we can compute $h(v_0)$ using the deterministic boundary values for the heights near $v = (i, n)$. This gives $h(v_0) = i + 1/2$. By similar reasoning, letting $u = (-n-1, j)$, we can compute that $h(u) = j - 1/2$. Therefore $j = i + 1$ and so $u = (-n- 1, i + 1)$. By a symmetric argument, we can show that if $u \in \partial_E \tt{S}$, then $u = (n+1, -i + 1)$. This gives Definition \ref{D:dc-forest}(iii) when $v \in \partial_S \tt{S}$. A similar argument gives the same condition when $v \in \partial_N \tt{S}$, and so $F_\tt{S}(D)$ is a dimer-compatible forest.
		
		Next, we show that given a DCF $F$ on $\tt{G}(\tt{S})$, the set $D_{\tt{S}}(F)$ is a dimer configuration. First, the dual forest $\check F$ is a DCF on $\tt{G}(\tt{N})$ by \cref{L:dimer-compatible-duals}, and so orienting both $F, \check F$ towards the sink boundaries is well-defined. Next, $D_\tt{S}(F)$ is a subgraph of $\tt{A}_n$ where every white (i.e. $\tt{N}/\tt{S}$)-vertex has degree one, since in the oriented forest $F, \check F$, all vertices have out-degree one. Moreover, every black vertex also has degree one, since a black (i.e. $\tt{E}/\tt{W}$)-vertex is incident to an edge in $D_\tt{S}(F)$ if that planar edge contained the black vertex at its midpoint in either $F$ or $\check F$. Exactly one of these possibilities holds for each black vertex by the dual relationship between $F$ and $\check F$. Hence $D_\tt{S}(F)$ is a dimer configuration.
		
		It is straightforward to check that the maps $D_\tt{S}, F_\tt{S}$ are inverses. Finally, for a dimer configuration $D$, for every edge $\{v, v + 2e\}$ in the edge set $E_\tt{S}$ for $\tt{G}(\tt{S})$, it is not possible that both $\{v, v + 2e\}$ and its dual are contained in both $F_\tt{S}(D)$ and $F_\tt{N}(D)$ since the black vertex $v + e$ has degree $1$ in $D$. Hence the planar embeddings of $F_\tt{S}(D)$ and $F_\tt{N}(D)$ do not cross, and so they must be duals.
	\end{proof}

	\subsection{Temperley's bijection and the two-periodic Aztec diamond} As discussed in the introduction, the bijection in \cref{P:temperley-bij} is important in our setting because the pushforward under $F_\tt{C}$ of the two-periodic Aztec diamond measure $\mathbb P_{a, n}$ is tractable. First, observe that $D_\tt{C}, F_\tt{C}$ interact well with path weights, as described in the next lemma. This is immediate from the definitions.
	
	\begin{lem}
		\label{L:path-weights}
		Fix $\tt{C} \in \{\tt{N}, \tt{S}\}$, and let $\tt{D}_\tt{C} \subset \tt{D}$ be the set of all edges incident to a $\tt{C}$-vertex in the Aztec diamond graph $\tt{A}_n$. Let $\operatorname{Dir}_\tt{C}$ be the set of directed edges in $\tt{G}(\tt{C})$ and let $\phi_{\tt{C}}:\tt{D}_\tt{C} \to \operatorname{Dir}_\tt{C}$ be the map sending the edge $\{v, v + e\}$ to the directed edge $(v, v + 2e)$, where $v \in \tt{C}$.
		
		Let $w:\tt{D} \to [0, \infty)$ be any weight function which equals $1$ on $\tt{D} \setminus \tt{D}_\tt{C}$, and let $\tilde w = w \circ \phi^{-1}$ be the corresponding weight function on $\operatorname{Dir}_\tt{C}$. Extend $w, \tilde w$ to subgraphs of $\tt{A}_n$ and directed subgraphs of $\tt{G}(\tt{C})$, by letting the weight of a subgraph be the product of its edge weights. Then with $F_\tt{C}, D_\tt{C}$ as in \cref{P:temperley-bij}, for any dimer configuration $D$ on $\tt{A}_n$ or any DCF $F$ on $\tt{G}(\tt{C})$ we have
		$$
		\tilde w \circ F_\tt{C}(D) = w(D) \qquad \text{ and } \qquad w \circ D_\tt{C}(F) = \tilde w(F).
		$$
	\end{lem}
	
	\cref{L:path-weights} is relevant because we can produce the dimer measure $\mathbb P_{a, n}$ by starting with a collection of edge weights that are $1$ on $\tt{D} \setminus \tt{D_C}$, for any choice of $\tt{C}$, as per the next lemma.
	
	\begin{lem}[Proposition 8.2, \cite{BCJ22}]
		\label{L:gauge} 
		Fix $a \in (0, 1)$, and consider the following collections of edge weights on $\tt{A}_n$.
		\begin{itemize}[nosep]
			\item $\tt{N}$-weights: edges of the form $\{v, v - (\pm 1, 1)\}$, where $v \in \tt{N}$ have weight $a^{2}$. All other edges have weight $1$.
			\item $\tt{S}$-weights: edges of the form $\{v, v + (\pm 1, 1)\}$, where $v \in \tt{S}$ have weight $a^{2}$. All other edges have weight $1$.

		\end{itemize}
		Then for either of the two edge weight choices above, the measure on dimer configurations on $\tt{A}_n$ given by choosing configurations proportional to their edge weight is the two-periodic Aztec diamond measure $\mathbb P_{a, n}$. 
	\end{lem}
	
	\begin{figure}
		\centering
		\includegraphics[height=4cm]{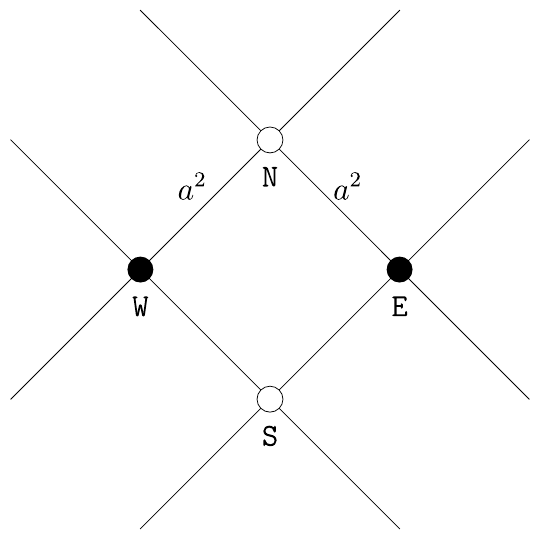}
		\qquad
		\includegraphics[height=4cm]{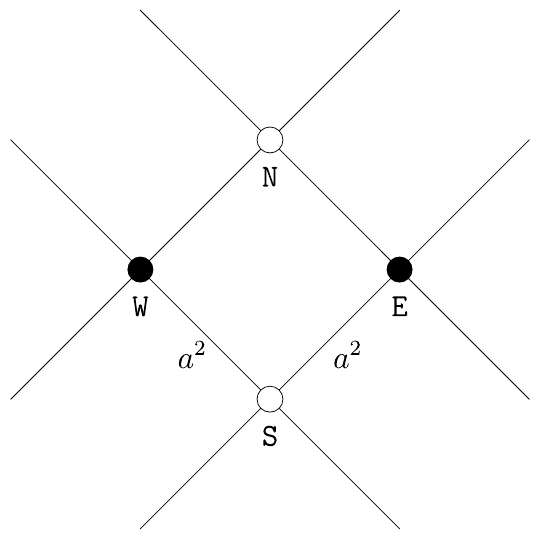}\\

		\includegraphics[height=4cm]{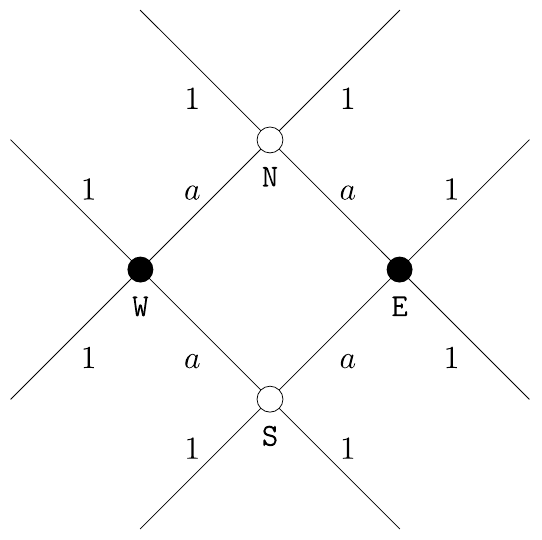}
		\qquad
		\includegraphics[height=4cm]{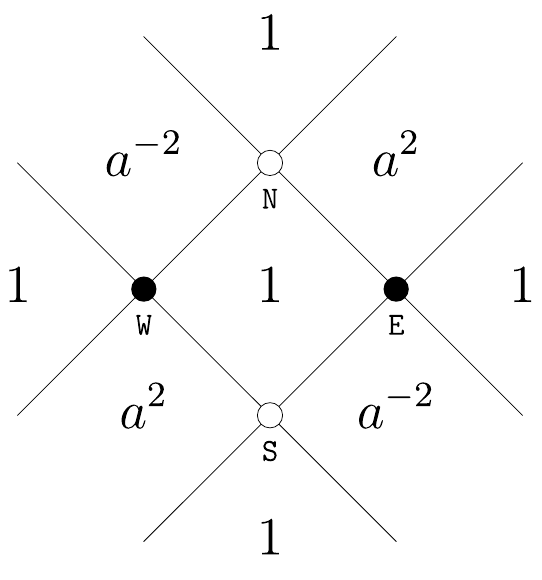}
		\caption{The north and south edge weights on $\mathtt{A}_n$ as considered in \cref{L:gauge}, along with the original two-periodic weights and the face weights.  }
		\label{fig:L:gauge} 
	\end{figure}
	
	The idea of \cref{L:gauge} is that a dimer model is determined by its face weights, i.e.\ the alternating product of edge weights around each face. Two dimer models with the same face weights are called \emph{gauge equivalent}, and produce the same measure, see Section 8 in \cite{BCJ22} for details. Combining Lemmas \ref{L:gauge} and \ref{L:path-weights}, we have the following result. For this corollary, we call a directed edge of the form $(v, v + 2 e)$ in $\tt{G}(\tt{C})$ a northeast (NE) edge if $e = (1, 1)$, a northwest (NW) edge if $e = (1, -1)$, a southeast (SE) edge if $e = (-1, 1)$ and a southwest (SW) edge if $e = (-1, -1)$. 
	
	\begin{cor}
		\label{C:pushforward}
		For $\tt{C} \in \{\tt{N}, \tt{S}\}$ and $a \in (0, 1)$, let $\Q_\tt{C, a, n}$ denote the pushforward measure of $\P_{a, n}$ under the map $F_\tt{C}$. Then $\Q_\tt{C, a, n}$ is a measure on dimer-compatible forests in $\tt{G(C)}$, where for a forest $F$ (oriented towards $\partial^\tt{x} \tt{C}$, $\Q_\tt{C, a, n}(F)$ is proportional to the weight of $F$, where edge weights on $\operatorname{Dir}_\tt{C}$ are given as follows:
		\begin{itemize}[nosep]
			\item $\tt{C} = \tt{N}$: NE/NW edges have weight $1$ and SE/SW edges weight $a^2$.
			\item $\tt{C} = \tt{S}$: SE/SW edges have weight $1$ and NE/NW edges weight $a^2$.

		\end{itemize}
	\end{cor}
	
	\begin{figure}
		\centering
		\includegraphics[height=3cm]{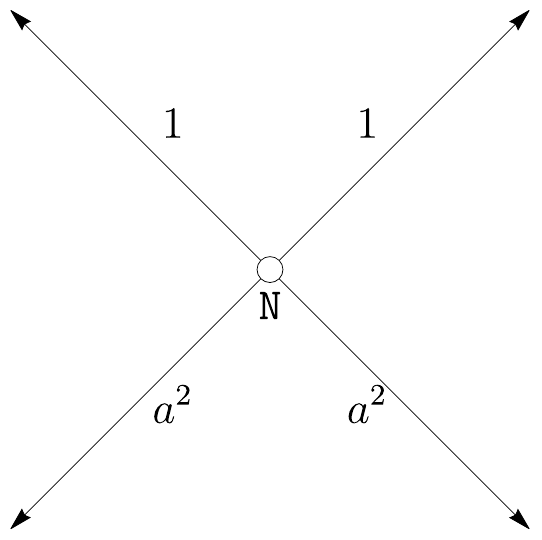}
		\qquad
		\includegraphics[height=3cm]{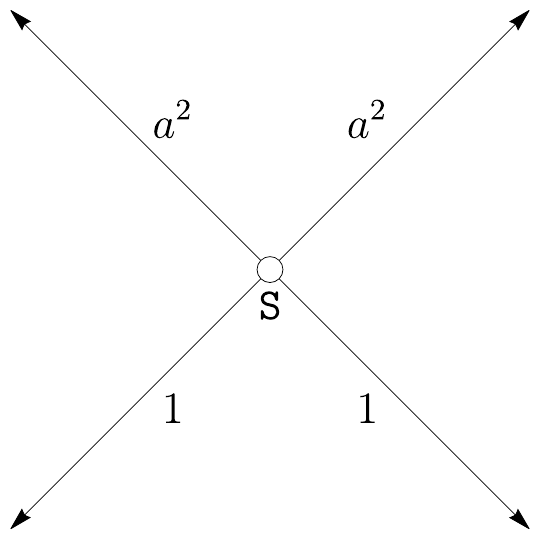}

		\caption{The random walk weights from \cref{C:pushforward}. }
		\label{fig:C:pushforward} 
	\end{figure}

	In other words, the pushforward of $\mathbb P_{a, n}$ under the map $F_\tt{C}$ is a measure on DCFs, where the weights bias towards edges that point in cardinal direction $\tt{C}$.

	\begin{proof}
		By \cref{L:gauge}, the measure $\mathbb P_{a, n}$ assigns weights to dimer configurations based on the $\tt{C}$-weights in that lemma, which equal $1$ off of $\tt{D_C}$. Therefore by \cref{L:path-weights} and the bijectivity of $F_\tt{C}$ (\cref{P:temperley-bij}), the pushforward $\Q_{\tt{C}, a, n}$ assigns measure to DCFs with weights described above.
	\end{proof}
	
	One of the main advantages of moving from the dimer measure $\P_{a, n}$ to the DCF measure $\Q_{\tt{C}, a, n}$ is the connection to loop-erased random walks via Wilson's algorithm \cite{Wil96}. 
	
	First, consider a finite weighted directed graph $G = (V, E)$, with a distinguished root vertex $v$, and suppose that for every $w \in V$ there is a directed path from $w$ to $v$. We can sample a spanning tree $T$ on $G$, oriented towards $v$, with probability proportional to the weight of $T$, as follows:
	\begin{enumerate}[label=\arabic*.]
		\item Choose any enumeration $v_1, v_2, v_3, \dots v_k$ of $V \setminus \{v\}$, and set $T_0 = \{v\}$.
		\item At step $i \in \{1, \dots, k\}$, run a random walk $X_i$ with step probabilities proportional to edge weights in $G$, until it hits a vertex in $T_{i-1}$. Set $T_i = T_{i-1} \cup \operatorname{LE}(P_i)$, where $\operatorname{LE}(X_i)$ is the loop-erasure of $X_i$ (i.e. we remove loops of $X_i$ in the order in which they appear to get a simple path). 
		\item Set $T = T_k$. 
	\end{enumerate}
	Wilson's algorithm has the following obvious extension. Suppose that $H$ is a subforest of $G$ with vertex set $W \subset V$. Then if we run Wilson's algorithm on $V \backslash W$ with the whole subgraph $H$ in place of $T_0$, the final graph $F = T_k$ has law $\mu$ given as follows:
	\begin{itemize}[nosep]
		\item $\mu$ assigns positive measure to spanning forests $F$ of $G$ such that $F|_W = H$, and every component of $F$ is oriented towards an element of $W$. Here we write $F|_W$ for the subgraph of $F$ induced by the vertex set $W$.
		\item Subject to this constraint, the probability of $F$ is proportional to its weight.
	\end{itemize}
	To see why this holds, observe that running Wilson's algorithm in this setting is equivalent to running Wilson's algorithm on a graph $\tilde G$, whose vertex set is given by $V$ together with an extra sink vertex $w$, where we remove all outgoing edges from $G|_W$ except those in $H$, and where we connect vertices in $H$ to the new sink vertex $w$. We refer to the measure described in the two bullets above as $\tt{Forest}(G \mid H)$.
	
	Now we examine the upshot of Wilson's algorithm for the south forest in the two-periodic Aztec diamond. First, for a DCF $F$ on $\tt{G}(\tt{S})$, define paths 
	$
	\cS_i^\pm = \cS_i^\pm(F), i = 1, \dots, n/2
	$
	where (with notation as in \cref{L:tree-paths}), the path $\cS_i^-$ is the path starting at $v_i^S$ and ending in $\partial^\tt{x} \tt{S}$ in $F$, and the path $\cS_i^+$ is the path starting at $v_i^N$ and ending at a sink vertex in $\partial^\tt{x} \tt{S}$. We let $\cS = \cS(F) = \bigcup_{i =1}^{n/2} (\cS_i^+ \cup \cS_i^-)$. We call $\cS$ the \textbf{backbone} of the DCF $F$, and the paths $\cS_i^\pm$ the \textbf{backbone paths}.

	\begin{lem}
		\label{L:simple-sample}
		In this lemma, $\tt{G}(\tt{S})$ should be understood as a weighted directed graph, with south-biased weights given as in Corollary \ref{C:pushforward}, and $F \sim \Q_{\tt{S}, a, n}$.
		\begin{enumerate}[label=\arabic*.]
			\item Conditional on $\cS = \cS(F)$, $F \sim \tt{Forest}(\tt{G}(\tt{S}) \mid \cS)$.
			\item Let $v \in \tt{S}$, and let $\Pi^\cS(v) = (v = v_0, v_1, v_2, \dots, v_T)$ denote the path starting at $v$ in $F$, stopped at the first vertex $v_T \in \cS$. Then conditional on $\cS$, 
			$$
			(v_0, \dots, v_T) \sim \operatorname{LE}(X_v),
			$$
			where $X_v$ is a random walk on $\tt{G}(\tt{S})$ started at $v$, stopped when it hits $\cS$.
		\end{enumerate} 
	\end{lem}
	We can similarly define north backbone paths $\cN_i^\pm$, and state a version of \cref{L:simple-sample} in this setting.
	
	\begin{proof}
		Conditional on $\cS$, the only constraint on $F$ required for Definition \ref{D:dc-forest} is that $F$ is a spanning forest of $\tt{G}(\tt{S})$ where every component is oriented towards an element of $\cS$. The first part of the lemma then follows from Corollary \ref{C:pushforward}. The second part follows from Wilson's algorithm for sampling measures of the form $\tt{Forest}(G \mid H)$.
	\end{proof}
	
	\cref{L:simple-sample} gives a precise framework for showing that paths in north/south Temperleyan forests tend to drift north/south, and gives us a probabilistic foothold for understanding the forest measures $\Q_{\tt{C}, a, n}$.
	
	\begin{rem}
		\label{R:fomin}
		\cref{L:simple-sample} says nothing about how we sample the backbone paths $\cS_i^\pm$ prior to filling in the remaining forest. This can again be understood through Wilson's algorithm. While we will not state a precise correspondence here, one can show that the paths $\cS_i^\pm$ are loop-erased random walk paths starting at $\partial^\tt{o} \tt{S}$, and conditioned to end at one of the prescribed points on $\partial^\tt{x} \tt{S}$, and on non-intersection. The notion of non-intersecting loop-erased random walks, and hence the law of $\cS$, can be understood through Fomin's determinant formula \cite{fomin2001loop}. Since the backbone paths $\cS_i^\pm $ are coming from \emph{south-drifted} random walks, by conditioning them to hit the east or west boundaries (which forces them to move against their drift), we should expect that they become essentially one-dimensional. Therefore we may loosely expect that the whole ensemble bears a resemblance to non-intersecting one-dimensional walks. This gives some intuition from the forest perspective for why we might expect to see Airy fluctuations. While this is an intriguing picture, it is not as tractable as models of one-dimensional non-intersecting random walks, and we were not to use this description to gain a foothold at the rough-smooth boundary.
	\end{rem}
	
	\subsection{The smooth phase.} \label{S:smooth-phase-def} \qquad 
	Near the center of the Aztec diamond, and in large regions within corridors, the Aztec diamond measure is close to a full-plane dimer measure called the \emph{smooth phase}. The smooth phase is the local limit away from the boundary of a $2$-periodic dimer model with flat boundary conditions. It is easiest to first describe this measure from the point of view of Temperleyan forests, where we can sample it using \emph{Wilson's algorithm rooted at $\infty$}, see \cite{BLPS01}. For this algorithm, let $\tt{C} \in \{\tt{N}, \tt{S}\}$.
	\begin{enumerate}[label=\arabic*.]
		\item Choose any enumeration $v_i,i \in \N$ of the full-plane vertex set $\tt{\tilde C}$.
		\item Run a random walk $X_1$ starting at $v_1$ with $\tt{C}$-biased edge weights as in Corollary \ref{C:pushforward}. Let $T_1 = \operatorname{LE}(X_1)$. Because of the drift, $X_1$ is transient and so $\operatorname{LE}(X_1)$ is well-defined.
		\item At steps $i = 2, 3, \dots$, run a random walk $X_i$ with $\tt{C}$-biased edge weights until it hits a vertex in $X_{i-1}$. Set $T_i = T_{i-1} \cup \operatorname{LE}(X_i)$. The random walk $X_i$ hits $T_{i-1}$ almost surely.
		\item Set $T = \bigcup_{i=1}^\infty T_i$.
	\end{enumerate}
	It is easy to check that the resulting graph $T$ is an infinite (spanning) tree on $\tt{\tilde C}$.  
	We let $\Q_{a, \tt{C}}$ be the law of $T$, and call $T$ a \emph{smooth phase $\tt{C}$-spanning tree} (with parameter $a$).
	As with the finite version of Wilson's algorithm, this law does not depend on the enumeration of $\tt{\tilde C}$. Now, the lattice $\tt{\tilde S}$ with edges of the form $\{v, v + 2 e\}$ is dual to the lattice $\tt{\tilde N}$ with edges of the same form, and given $T \sim \Q_{a, \tt{S}}$, we can construct a dual tree $\check T$ on $\tt{\tilde N}$ by letting a dual edge $\check e \in \check T$ if and only if the original edge $e \notin T$. The same duality relation holds between $\tt{\tilde E}, \tt{\tilde W}$.
	
	The results of \cite{BLPS01} (i.e. Theorem 12.2 therein; the versions of these lattices with oriented edge weights as in Corollary \ref{C:pushforward} are dual networks in the sense of that theorem) imply that if $T \sim \Q_{a, \tt{C}}$, then $\check T \sim \Q_{a, \tt{\check C}}$, where $\tt{\check C}$ denotes the dual direction. Together we can take the pair $(T, \check T)$ and produce a dimer configuration using an infinite analogue of Temperley's bijection. Orient both $T, \check T$ towards $\infty$ (i.e. so that there are no sink vertices), and define a perfect matching $D_{\tt{C}}(T)$ of $\tt{\tilde G} = (\tt{\tilde V}, \tt{\tilde D})$ by letting
	$$
	D_{\tt{C}}(T) = \{\{v, v + e\} : v \in \tt{C} \cup \tt{\check C} : (v, v + 2 e) \text{ is a directed edge in }  T \text{ or } \check T\}.
	$$
	Now, given a perfect matching $D$ of $\tt{\tilde G}$, we can define an (oriented) spanning forest $F_\tt{C}(D)$ on $\tt{\tilde C}$ by letting
	$$
	F_\tt{C}(D) = \{(v, v + 2 e) :  v \in \tt{C}, \{v, v + e\} \in D\}.
	$$
	The maps $D_{\tt{C}}$ and $F_\tt{C}$ are inverses. We let 
$$
\P_a = (D_\tt{S})_* \Q_{a, \tt{S}} = (D_\tt{N})_* \Q_{a, \tt{N}} 
$$
and call $\P_a$ the \textbf{smooth phase}.
	
	Just as we can define heights in the Aztec diamond, we can also define heights in the smooth phase. As with the Aztec diamond, this requires an anchor point. The natural anchor point is $\infty$, which is the unique choice that guarantees that the height function is ergodic. To define the height function anchored at $\infty$, let $D \sim \P_a$ and $T = F_\tt{S} (D)$. We first define the height function $h^f_D$ anchored at an $a$-face $f \in \tt{\tilde F}$ by using the rules in Section \ref{S:face-heights}, subject to the initial condition $h^f_D(f) = 0$. All these height functions are shifts of each other.
	
	Next, for a simple path $\Pi = (v_0, v_1, \dots)$ satisfying 
	\begin{equation}
		\label{E:nn}
		\lim_{n \to \infty} v_{n, 2} = -\infty,
	\end{equation} 
	we say that the (counterclockwise) winding number of $\Pi$ is the number times the path $\Pi$ crosses the line $v_0 + \{0\} \times (0, \infty)$ moving right-to-left minus the number of times it crosses moving left-to-right. The southward drift of the random walk and condition \eqref{E:nn} ensures that each of these counts is finite. We write
	$\operatorname{Wind}(\Pi)$ for the winding number of $\Pi$. For an $a$-face $f$, let $\Pi_f$ be the path in $T$ starting on the $\tt{S}$-vertex $f - (0, 1)$. All paths $\Pi_f$ satisfy \eqref{E:nn} almost surely since $T$ is built from south-biased random walks, and by this property, there are faces $f$ with $\operatorname{Wind}(\Pi_f) = 0$. Choose any such $a$-face $f$ and set $h_D = h_D^f$. This is the height function of $D$, anchored at $\infty$. 
	
	The next lemma shows that this procedure is well-defined and clarifies the connection between winding numbers and height functions.
	\begin{lem}
		\label{L:height-lem}
		Suppose that $D$ is a dimer configuration on $\tilde V$, and that $T = F_\tt{S} D$ is an infinite tree such that the infinite path $(v_0 = (0,0), v_1, \dots)$ in $T$ starting at the origin satisfies \eqref{E:nn}. 
		Then the function $h_D$ is well-defined, and for any $a$-face $f$ we have
		\begin{equation*}
			h_D(f) = -4 \operatorname{Wind}(\Pi_f).
		\end{equation*}
	\end{lem}
	
	\begin{proof}
		Both claims in the lemma follow if we can show that
		$$
		h_D^f - h_D^{f'} = 4 \operatorname{Wind}(\Pi_{f'}) - 4 \operatorname{Wind}(\Pi_{f}).
		$$
		For this, let $v(f) = f - (0, 1), v(f') = f' - (0, 1)$. Let $(v, v')$ be consecutive vertices contained in $\Pi_f \cap \Pi_{f'}$. By condition \eqref{E:nn}, we can choose $v'$ so that the paths $\Pi_f, \Pi_{f'}$ stay below the line $\R \times \{v_2'\}$ after hitting $v'$. By Corollary \ref{C:vertex-heights},
		$$
		h_D^f(v) = h_D^f(v(f)) + L(f) - R(f),
		$$
		where $L$ is the number of left turns of $\Pi_f$ prior to hitting $v'$ and $R$ is the number of right turns. A case-by-case analysis starting from the fact that $h_D^f(f) = 0$ shows that $h_D^f(v(f)) = (2/\pi) \theta_f$, where $\theta_f \in (-\pi, \pi)$ is the angle formed between the edge $e$ leaving $v(f)$ and the south-pointing vector $(0, -1)$.  Since $\Pi_f$ stays below the line $\R \times \{v_2'\}$ after hitting $v'$, 
		$$
		h_D^f(v) = (2/\pi) \theta_f + L(f)- R(f) = (2/\pi) \theta_f - 4\operatorname{Wind}(\Pi_{f}) - (2/\pi)\gamma_f,
		$$
		where $\gamma_f$ is the angle formed by the edges $e$ and $(v, v')$. A similar computation holds for $h_D^{f'}(v)$, and the lemma follows since $\theta_f - \gamma_f = \theta_{f'} - \gamma_{f'}$ is the angle formed by the edge $(v, v')$ and the vector $(0, -1)$.
	\end{proof}
	
	We end this section with a symmetry of the smooth phase that falls out of \cref{L:height-lem}.
	
	\begin{cor}
		\label{C:reflection-symmetry-smooth}
		If $D \sim \P_a$ and $f$ is any $a$-face or $b$-face, then 
		$$
		h_D(f+ (x, y)) \stackrel{d}{=} - h_D(f + (-x, y)),
		$$
		where the distributional equality is jointly over all $(x, y) \in \hat \Z^2$.
	\end{cor}
	
	\begin{proof}
		Let $L$ be any vertical line with $x$-coordinate $x_0 \in 2 \Z + 1$. Reflect the spanning tree $T \sim \Q_{a, S}$ across $L$ to get a new spanning tree $T'$. The trees $T, T'$ are equal in law by the symmetry in Wilson's algorithm, and moreover, letting $D = D_\tt{S} T, D' = D_\tt{S} T'$, by \cref{L:height-lem} we can see that
		$h_D(x_0 + x, y) = - h_{D'}(x_0 - x, y)$ for all $(x, y) \in \hat \Z^2$. Since $a$ and $b$ faces have $x$-coordinates in $2 \Z + 1$, the lemma follows.
	\end{proof}
	
	\begin{rem}
		Corollary \ref{C:reflection-symmetry-smooth} does not hold if $f$ is not an $a$- or $b$-face, and this corollary is why we choose to anchor our height function to be $0$ on an $a$-face, rather than at the origin, which will have height in $4 \Z + 1$. This convention differs from our convention for the finite Aztec diamond, where the origin has height in $4 \Z$ and $a$-faces have heights in $4 \Z - 1$. We have chosen the latter convention for the Aztec diamond to accommodate the remarkable symmetry $\E h_D (0,0) = 0$, where $D \sim \P_{a, n}$, see \cref{L:expected-middle-height}. In the smooth phase, no reasonable convention on the heights yields an integer value for $\E h_D (0,0)$, whereas our choice gives $\E h_D(f) = 0$ for any $a$- or $b$-face $f$ through Corollary \ref{C:reflection-symmetry-smooth}.
	\end{rem}
	
	\section{Basic tools} 
	\label{S:basic-tools}
	
	In this section, we gather together the three sets of basic tools discussed at the beginning of Section \ref{sec:proof-sketch}. Most of the height function and smooth phase coupling estimates (tools 1 and 2) in the upcoming Sections \ref{S:height-fn-est} and \ref{S:smooth-phase-couplings} have integrable proofs, and these are postponed until the latter sections of the paper. On the other hand, the straightforward probabilistic estimates on random walks and spanning trees in Section \ref{S:rw-estimates} are proven therein. A final section (Section \ref{S:compound-estimates}) combines all three tools to produce a few new height function estimates.
	
	\subsection{Height function estimates} 
	\label{S:height-fn-est}
	Throughout this section we write $\mathcal{H}_n:\tt{\bar F} \to \Z$ for the height function of a dimer configuration $D \sim \P_{a, n}$. In order to state some asymptotics it will also be convenient to extend $\mathcal{H}_n$ to all of $[-n, n]^2$ by letting $\cH_n(v) = \cH_n([v])$, where $[v]$ is the nearest face to the point $v$, breaking ties arbitrarily.
	
	We start with a remarkable symmetry of the two-periodic Aztec diamond. 
    
	\begin{lem}
		\label{L:expected-middle-height}
		For any face $F = (i, j) \in \tt{F}$ with $i, j \in 2 \Z$, we have
		$$
		\E \mathcal{H}_n(i, j) = -\E \mathcal{H}_n(-i, j) = - \E \mathcal{H}_n(i, -j) = \E \mathcal{H}_n(-i, -j)
		$$
		In particular, $\E \mathcal{H}_n(2i,0) = \E \mathcal{H}_n(0,2i)= 0$ for all $-n/2 \le i \le n/2$.
	\end{lem}
	
	The anti-symmetry $\E \mathcal{H}_n(i, j) = - \E \mathcal{H}_n(-i, j)$ in \cref{L:expected-middle-height} is completely surprising and comes from a delicate formula comparison. We do not have a combinatorial explanation. It is equivalent to a surprising symmetry in height expectation after changing the $a$-face weights from $a$ to $1/a$. \cref{L:expected-middle-height} is proven in \cref{S:Proof:expected-middle-height}. 
	We also need the following result from \cite{BN25}.
	
	\begin{lem}{\cite[special case of Theorem 3.1]{BN25}}
		\label{L:variance-middle-height}
		The central height $\cH_n(0,0)$ converges in law to a limiting random variable $X$ as $n \to \infty$. Moreover, $\E \cH_n(0,0)^k \to \E X^k \in \R$ as $n \to \infty$, for all fixed $k \in \N$.
	\end{lem}
	
	\begin{rem}
		\label{R:discrete-gaussian}
		The random variable $X$ can be written as a sum of independent random variables $H + Y$, where $H$ is the limit of the averaged central height $H_n$ defined prior to \cref{T:main-1}, and $Y$ is the height at the origin $(0,0)$ in the smooth phase. Theorem 1.1 in \cite{BN25} shows that $H$ is a discrete Gaussian random variable (as with \cref{L:variance-middle-height}, \cite{BN25} also prove this in much greater generality). The random variable $Y$, which is the winding number of a biased loop-erased random walk on $\Z^2$, should also be a discrete Gaussian random variable. We do not pursue this here. 
		
		More remarkably, we believe that $H \eqd - Y$, though we do not have a proof. One upshot of \cref{L:expected-middle-height} is that $\E X = 0$, and so $\E H = - \E Y$. We can also compute that $\E H^2 = \E Y^2$ and give explicit formulas for this second moment. This computation also gives an explicit evaluation for the right-hand side of \cite[Equation (61)]{BN25} for the special case of the two-periodic Aztec diamond.  Since these formulas are not relevant to our argument, we omit this computation as it adds considerable length. 
	\end{rem}

	In Section \ref{S:compound-estimates}, we will use global smooth phase couplings and estimates on the heights in the smooth phase to propagate the height estimates at the origin given by the previous two lemmas throughout the smooth region. 
	
	For now, we will move into the rough-smooth boundary and the rough phase. We record two expectation computations, to be proven in \cref{S:expectation-computations}.
	
	To state the next lemma, recall from the preamble to \cref{T:main-3} that $\xi_0:[-\alpha, \alpha] \to \R^2$ is the rough-smooth boundary curve in the third quadrant, parametrized so that for all $t \in [-\alpha, \alpha]$, the point $\xi_0(t)$ is on the line $\{ {\bf v} \in \R^2 : \mathbf{v} \cdot e_2 = t \}$. Recall the coordinate change $\beta_n:[-\alpha n,  \alpha n] \times \R \to \R^2$ by
	$$
	\beta_n(t, x) = n\xi_0(t/n) + e_1 x.
	$$
	In words, $\beta_n(t, x)$ gives a point at location $t$ along the rough-smooth boundary curve, at distance $x$ away from the boundary. Note that the transformation $\beta_n$ does not scale space; the scaling $n \xi_0(\cdot/n)$ simply puts the limit curve back in unscaled coordinates. 
	
	\begin{lem}
		\label{L:expectation-computation}
		There exists $\eta_a > 0$ such that the following holds for all large enough $n$. For all $t \in [-n^{5/6}, n^{5/6}]$ and $x \in [-n^{8/15}, 2n^{1/3} \log^2 n]$, we have
		\begin{align*}
			|\E \mathcal H_n(\beta_n(t, x)) + \eta_a |x|^{3/2} n^{-1/2}| &\le n^{1/6}.
		\end{align*}
	\end{lem}
	
	\begin{lem}
		\label{L:expectation-computation-2}
		There exists a constant $\mu_a > 0$ such that the following holds for all large enough $n$. For all $x \le 2n^{1/3} \log^2 n$ we have
		$$
		\E \mathcal H_n(\beta_n(0, x)) \le -\mu_a |x|^{3/2} n^{-1/2} + n^{1/6}
		$$
		whenever $\beta_n(0, x) \in [-n, n]^2$.
	\end{lem}

	We can also establish concentration around the expectations in \cref{L:expectation-computation} and \ref{L:expectation-computation-2}. However, we defer this computation to Section \ref{S:compound-estimates}, since it requires a connection to the smooth region via a smooth phase coupling.
	\subsection{Smooth phase couplings}
	\label{S:smooth-phase-couplings}
	
	In this section, we collect coupling results between the smooth phase and the Aztec diamond. We start with a local coupling result at the rough-smooth boundary from \cite{BCJ22}. 
	
	\begin{lem}(Local smooth coupling, Proposition 4.5, \cite{BCJ22})
		\label{L:local-coupling}
		Fix $L > 0$, and consider a subset 
		$
		\Lambda \subset \beta_n([-L n^{2/3}, Ln^{2/3}]\times [-L n^{1/3}, L n^{1/3}]).
		$
		Then there exists a constant $C = C(L) > 0$ such that
		$$
		d_{\operatorname{TV}}(\P_{a, n}|_\Lambda, \P_a|_\Lambda) \le C n^{-1/3} |\Lambda|^4.
		$$
		Here $d_{\operatorname{TV}}$ denotes the total variation distance between the measures.
		
	\end{lem}
	In \cite{BCJ22}, this lemma is stated only for boxes. However, the proof goes through verbatim regardless of the geometry of the set $\Lambda$, see Section 7 in \cite{BCJ22} or the similar proofs of \cref{L:global-coupling-1} and \cref{L:global-coupling-2} in \cref{S:smoothcouplingproofs}. The set $\beta_n([-L n^{2/3}, Ln^{2/3}]\times [-L n^{1/3}, L n^{1/3}])$ appearing in the lemma statement simply corresponds to a region of fixed size at the rough-smooth boundary after we pass to the limiting coordinates. We note that it is not difficult to use the Gibbs property for dimers and replace the factor $|\Lambda|^4$ by $|\partial \Lambda|^4$ in the above lemma, where $\partial \Lambda$ is the inner boundary of $\Lambda$. However, this improvement does not help us so we have not included this strengthening here.
	
	If we move away from the rough-smooth boundary and into the smooth region, then we can improve the total variation bound in \cref{L:local-coupling} and couple on larger sets. The next two lemmas give two couplings on specific larger sets. The choice of parameters in the sets are somewhat arbitrary, and chosen for ease of proof.

    \begin{figure}
        \centering
        \includegraphics[width=0.8\linewidth]{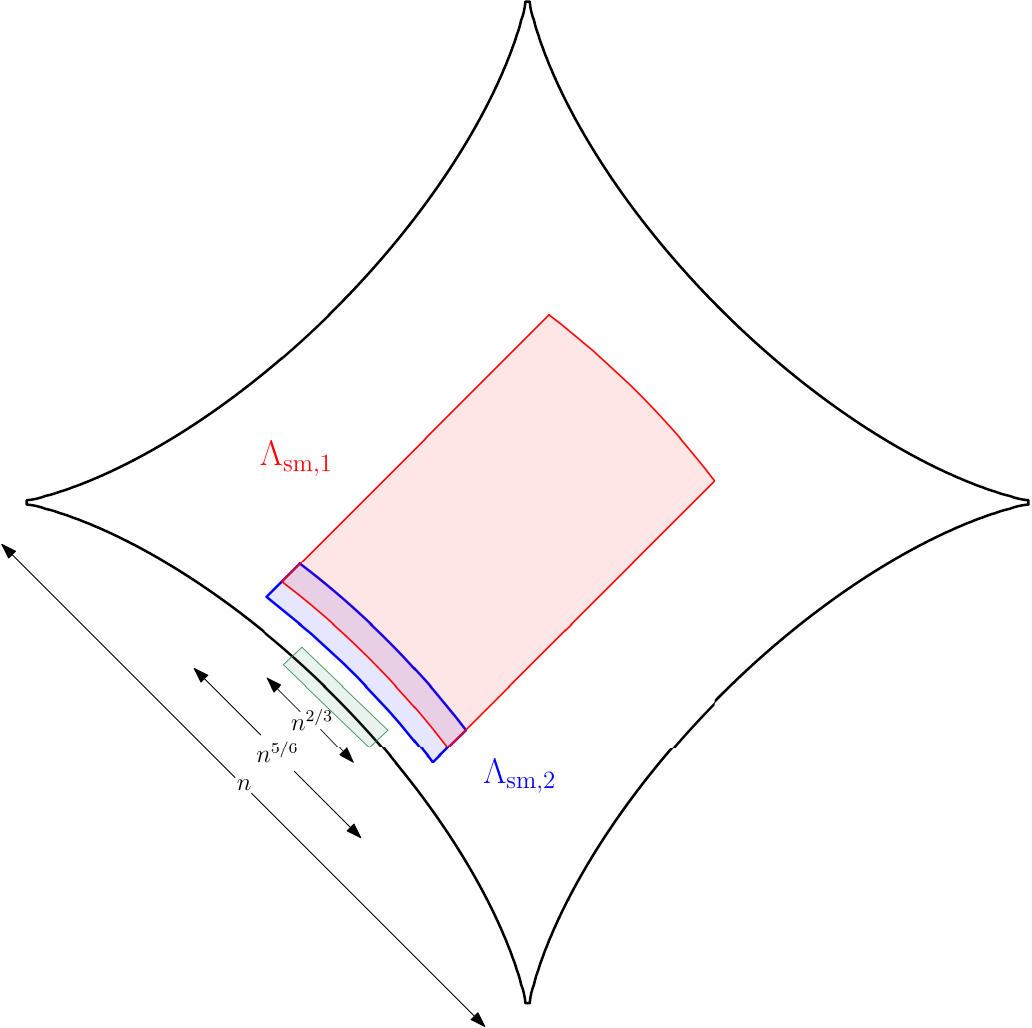}
        \caption{The coupling sets $\Lambda_{\operatorname{sm}, 1}, \Lambda_{\operatorname{sm}, 2}$ depicted within the smooth region. The green box is the rough-smooth boundary where we will study Airy fluctuations, and where \cref{L:local-coupling} holds. The figure is not to scale, e.g.\ the gap between $\Lambda_{\operatorname{sm}, 2}$ and the Airy region is of size $n^{1/3} \log^2 n$, and the overlap between $\Lambda_{\operatorname{sm}, 1}, \Lambda_{\operatorname{sm}, 2}$ is of order $n^{6/7}$. }
        \label{fig:smoothregions}
    \end{figure}
	
	\begin{lem} (Global smooth coupling)
		\label{L:global-coupling-1} 
	There exist $a$-dependent constants $\gamma_a, d_a > 0$ such that the following holds. First define
		$$
		\Lambda_{\operatorname{sm}, 1} = \beta_n([-n^{5/6} , n^{5/6}] \times [n^{6/7}, n(|\xi_0(0)| + \gamma_a)). 
		$$
		Then
		$$
		d_{\operatorname{TV}}(\P_{a, n}|_{\Lambda_{\operatorname{sm}, 1}}, \P_a|_{\Lambda_{\operatorname{sm}, 1}}) \le 3 e^{-d_a n^{6/7}}.
		$$
		The same bound holds with $\Lambda_{\operatorname{sm}, 1}$ replaced by any of its rotations $R_{\pi/2} \Lambda_{\operatorname{sm}, 1}$, $R_{\pi} \Lambda_{\operatorname{sm}, 1}$, or $R_{3\pi/2} \Lambda_{\operatorname{sm}, 1}$, where $R_\theta$ denotes counterclockwise rotation by $\theta$.
	\end{lem}
	
	\begin{lem} (Mesoscopic smooth coupling)
		\label{L:global-coupling-2}
		Define the region
		$$
		\Lambda_{\operatorname{sm}, 2} = \beta_n([-n^{5/6} , n^{5/6}] \times [n^{1/3} \log^2 n, 2 n^{6/7}]). 
		$$
		Then there exists an $a$-dependent constant $d_a > 0$ such that
		$$
		d_{\operatorname{TV}}(\P_{a, n}|_{\Lambda_{\operatorname{sm}, 2}}, \P_a|_{\Lambda_{\operatorname{sm}, 2}}) \le 3 e^{- d_a \log^2 n}.
		$$
		The same bound holds with $\Lambda_{\operatorname{sm}, 2}$ replaced by any of its rotations $R_{\pi/2} \Lambda_{\operatorname{sm}, 2}$, $R_{\pi} \Lambda_{\operatorname{sm}, 2}$, or $R_{3\pi/2} \Lambda_{\operatorname{sm}, 2}$.
	\end{lem}
	{\begin{rem}
			Due asymptotics for $K^{-1}_{a,1}(x,y)$, the inverse Kasteleyn matrix entry evaluated at $x$ and $y$ (see \cref{subsec:KasteleynTheory} for definitions) given in  \cite[Lemma 5.3]{BN25} for two vertices $x$ and $y$ in the interior of the smooth region, it is not hard to show by mirroring the same proofs as the ones given in \cref{L:global-coupling-1} and \cref{L:global-coupling-2}, that a smooth phase coupling for the entire smooth region as long as we are an appropriate mesoscopic distance away from the boundary.  Indeed, such a result follows from the fact that the error between $K^{-1}_{a,1}(x,y)$ and its smooth phase counterpart is exponentially small, as established in \cite[Lemma 5.3]{BN25}.  An immediate consequence is that the mesoscopic average of the height function considered in our main theorems is equal with high probability to the macroscopic average height function considered in \cite{BN25}. 
		\end{rem}}	
	The proof of these results can be found in \cref{S:smoothcouplingproofs}. Note that a smooth phase coupling with the same order of error as in \cref{L:global-coupling-2} should hold on the union $\Lambda_{\operatorname{sm}, 1} \cup \Lambda_{\operatorname{sm}, 2}$. The reason we prove couplings separately on $\Lambda_{\operatorname{sm}, 1}$ and $\Lambda_{\operatorname{sm}, 2}$ instead has to do with technical considerations around a choice of conjugation for the kernel for the determinantal measure $\P_{a, n}$.
	See \cref{fig:smoothregions} for an illustration of Lemmas \ref{L:local-coupling}, \ref{L:global-coupling-1}, and \ref{L:global-coupling-2}.

	\subsection{Random walk and spanning tree estimates}
	\label{S:rw-estimates}
	
	Next, we record a few basic estimates coming from standard random walk bounds. Here and throughout the paper, we say that a random walk on $\tt{\tilde S}$ is \textbf{south-biased} (or north-biased) if it moves according to the edge weights in Corollary \ref{C:pushforward}.
	
	We start with a straightforward estimate on the path of a biased random walk. For this lemma and through the section, for a path $X = (X_0, X_1, X_2, \dots)$ we write $\operatorname{Range}(X)$ for the set of vertices visited by $X$.
	
	\begin{lem}
		\label{L:lerw-estimate-basic}
		Let $X = (X_0, X_1, X_2, \dots)$ be a south-biased random walk on $\tt{\tilde S}$ started at $X_0 = (1, 0)$. For $\alpha \ge 2$, define the (roughly) parabolic region
		$$
		P^\alpha_\tt{S} = \{(x, y) \in \R^2 : y \le \alpha, |x| \le (\alpha + \log(y^- + 1))\sqrt{y^- + 1}\},
		$$
		where $y^- = \max(-y, 0)$.
		Then for an $a$-dependent constant $d_a > 0$, we have
		$$
		\P(\operatorname{Range}(X) \not\subset P^n_\tt{S}) \le 2 \exp (-d_a \alpha).
		$$
	\end{lem} 
	Moving forward, we will also use the notation $P^\alpha_\tt{N} = \rho_\tt{N} P^\alpha_\tt{S}$.
	
	\begin{proof}
		Write $X_{i, 1}, X_{i, 2}$ for the $x$- and $y$-coordinates of $X$. Then $X_{\cdot, 1}$ is a simple symmetric random walk on $\Z$, and $X_{\cdot, 2}$ is a biased random walk on $\Z$ with step probabilities
		$$
		\P(X_{i, 2} - X_{i-1, 2} = 1) = a^2/(1 + a^2), \qquad \P(X_{i, 2} - X_{i-1, 2} = - 1) = 1/(1 + a^2).
		$$
		Let $\mu_a = \E(X_{i, 2} - X_{i-1, 2}) = (a^2-1)/(1 + a^2)$. Then by Hoeffding's inequality,
		\begin{align*}
			\P&(X_\ell \notin P^\alpha_\tt{S}) \\
			&\le \P\Big(|X_{\ell,1}| > (\alpha + \log (\ell |\mu_a|/2 + 1)) \sqrt{\ell |\mu_a|/2 + 1}\Big) + \P(X_{\ell,2} > \ell \mu_a/2) \\
			&\le 2 \exp(- |\mu_a|^2 (\alpha + \log(\ell |\mu_a|/2 + 1))^2/16) + \exp(- \ell \mu_a^2/8).
		\end{align*}
		On the other hand, by a union bound
		$$
		\P(\exists i \in \N \text{ s.t. } X_i \notin P^\alpha_\tt{S}) \le \sum_{\ell = \lfloor \alpha \rfloor }^\infty \P(X_\ell \notin P^\alpha_\tt{S}),
		$$
		where the sum starts at $\lfloor \alpha \rfloor$, since prior to time $\lfloor \alpha \rfloor$, $X_\ell$ is deterministically contained in $P^\alpha_\tt{S}$. Combining this with the previous bound gives the result.
	\end{proof}
	
	The above random walk estimate yields a simple bound for the behaviour of south forest paths in the Aztec diamond off of the backbone. 
	
	\begin{lem}
		\label{L:regular-paths}
		Let $F \sim \Q_{\tt{S}, a, n}$, let $\cS = \cS(F)$ be the backbone of $F$, and let $\Pi^\cS(v)$ denote the path in $F$ from a vertex $v$ stopped when it first hits the backbone, as in \cref{L:simple-sample}. Define the event
		$$
		\mathsf{Para}_\tt{S}(\alpha) = \{\operatorname{Range}(\Pi^\cS(v)) \subset P^\alpha_\tt{S} + v \text{ for all } v \in \tt{S}\}.
		$$
		Then
		$$
		\P(\mathsf{Para}_\tt{S}(\alpha) \mid \cS) \ge 1 - 2 n^2 \exp(-d_a \alpha).
		$$
        Here we are conditioning on the entire backbone of $F$, i.e.\ all south backbone paths.
	\end{lem}
	
	\begin{proof}
		This follows from Part (2) of \cref{L:simple-sample}, \cref{L:lerw-estimate-basic}, and a union bound. 	
	\end{proof} 
	
	We end this section by establishing decorrelation and concentration estimates for heights in the smooth phase.
	
	\begin{lem}
		\label{L:smooth-phase-heights-decor}
		Let $D \sim \P_a$ and let $f \in \tt{\tilde F}$ be any face. Then
		\begin{equation}
			\label{E:height-subgauss}
			\P(|h_D(f)| > m) \le 2\exp(-d_a m^2)
		\end{equation}
		for all $m > 0$. 
		Moreover, for any two faces $f, f'$ we have the covariance estimate
		$$
		\operatorname{Cov}(h_D(f), h_D(f')) \le c_a \exp(-d_a\|f - f'\|_\infty).
		$$
	\end{lem}
	
	\begin{proof}
		We first prove the lemma in the setting where all faces are $a$-faces.
		
		Let $f$ be any $a$-face. We have $h_D(f) = - 4 \operatorname{Wind}(\Pi_f)$, with notation as in \cref{L:height-lem}. From Wilson's algorithm, $\Pi_f = \operatorname{LE}(X)$, where $X$ is a south-biased random walks started at $v_f = f - (0, 1)$.
		
		For $m \in \N$, on the event $\operatorname{Wind}(\Pi_f) \ge m$, the random walk $X$ must satisfy $X_{\ell, 2} \ge (v_f)_2$ for some $\ell \ge m^2/4$. Using Hoeffding's inequality and a union bound as in proof of \cref{L:lerw-estimate-basic}
		$$
		\P(\exists \ell \ge m^2/4 \text{ s.t. } X_{\ell, 2} \ge (v_f)_2) \le \sum_{\ell = m^2/4}^\infty \exp(-\ell \mu_a^2/4) \le 2 \exp(- d_a m^2),
		$$ 
		giving \eqref{E:height-subgauss}.
		
		We move on to the covariance estimate, starting with the case when $f, f'$ are $a$-faces. We may assume $f_2 \le f'_2$, and set $\alpha = \|f - f'\|_\infty$ and $v := f - (0, 1), v' := f' - (0, 1)$. We have that
		\begin{equation}
			\label{E:parab-bd}
			[P^{\alpha/6}_\tt{N} + v'] \cap [P^{\alpha/6}_\tt{S} + w] = \emptyset, \qquad \text{ for all } w \in \Z^2, w \in P^{\alpha/6}_\tt{S} + v.
		\end{equation}
		Here recall that $P^{\alpha/6}_\tt{N}$ is a north-facing parabola given by rotating $P^{\alpha/6}_\tt{S}$ by $\pi$.
		Let $X, X'$ be independent south-biased walks starting at $v, v'$. Then by Wilson's algorithm, we have $$
		(\Pi_f, \Pi_{f'}) \stackrel{d}{=} (\operatorname{LE}(X), Y)
		$$
		where $Y$ is given by $\operatorname{LE}(X'|_{[0, T]})$, where $T$ is the first time $X'$ hits $\operatorname{LE}(X)$, concatenated with the remainder of the path to infinity in  $\operatorname{LE}(X)$ at the hitting location of $X'$. To prove the estimate in the lemma, we will show that
		\begin{equation}
			\label{E:winding-estimate}
			\P(\operatorname{Wind}(Y) \ne \operatorname{Wind}(\operatorname{LE}(X'))) \le 4 \exp(-d_a \alpha).
		\end{equation}
		To see how this proves the covariance bound, observe that gives a coupling between the independent pair $(\frac{1}{4} (\operatorname{Wind}(\operatorname{LE}(X)), \frac{1}{4} (\operatorname{Wind}(\operatorname{LE}(X')))$ and the height vector $(h_D(f), h_D(f'))$ 
		such that the two disagree with probability at most $c_a\exp(-d_a \alpha)$. The covariance estimate follows by using this coupling together with the tail bound in \eqref{E:height-subgauss}.
		
		To prove \eqref{E:winding-estimate}, observe that by \cref{L:lerw-estimate-basic},
		\begin{equation}
			\label{E:parabol}
			\operatorname{LE}(X) \subset P^{\alpha/6}_\tt{S} + v
		\end{equation}
		with probability at least $1 - 2 \exp(-d_a \alpha)$. Now, let $w$ be the final vertex of $Y$ that is contained in $P_{\tt{N}}^{\alpha/6} + v'$, and let $\bar Y$ be the segment of $Y$ ending at $w$. Similarly, let $w'$ be the final vertex on $\operatorname{LE}(X')$ that is in $P_{\tt{N}}^{\alpha/6}+ v'$, and let $\bar X'$ be the segment of $\operatorname{LE}(X')$ ending at $w'$. If $\bar X' = \bar Y$, then $Y$ and $\operatorname{LE}(X')$ have the same winding number, so it suffices to estimate the probability that $\bar X' \ne \bar Y$.
		
	In order for $\bar X'$ and $\bar Y$ to disagree on the event \eqref{E:parabol}, the path $X'$ would have to re-enter $P_{\tt{N}}^{\alpha/6}+ v'$ after entering the set $P^{\alpha/6}_\tt{S} + v$. By the relation \eqref{E:parab-bd}, if $T$ denotes the time when $X'$ hits $\operatorname{LE}(X)$, then this would require that
		$$
		X'|_{[T, \infty)} - X'(T) \not \subset P_{\tt{S}}^{\alpha/6},
		$$
		which has probability bounded above by $2 \exp(-d_a \alpha)$ by another application of \cref{L:lerw-estimate-basic} (and the fact that $T$ is a stopping time, so $X'|_{[T, \infty)} - X'(T)$ is another south-biased random walk). This completes the proof of the coupling.

		The general case of the lemma when some faces are not $a$-faces follows from a similar argument, this time decomposing the height of a face as a winding number plus a term coming from the initial direction of the loop-erased random walk.
	\end{proof}

	\subsection{Compound estimates}
	\label{S:compound-estimates}
	
	We end Section \ref{S:basic-tools} with two height estimates that combine results from the three sections above. The first estimate gives expectation and concentration estimates on the height function throughout the portion of the smooth region where we have smooth phase couplings. With notation as in \cref{L:global-coupling-1} and \ref{L:global-coupling-2}, define the cross-shaped set
	$$
	\tt{Cross} = \tt{Cross}_n := \bigcup_{\theta \in \{0, \pi/2, \pi, 3 \pi/2\}} R_{\theta} (\Lambda_{\operatorname{sm}, 1} \cup \Lambda_{\operatorname{sm}, 2}).
	$$
	For this lemma, a face $f = (i, j) \in \tt{F}$ is a $c$-face if $i, j \in 2 \Z$ and $i + j \in 4 \Z$.
	\begin{cor}
		\label{C:expected-sm-height}
		For all large enough $n$, for all $c$-faces $f \in \tt{Cross}$ we have that
		$$
		|\E \mathcal{H}_n(f)| \le 2 \exp(-d_a n^{6/7}), \qquad \E |\mathcal{H}_n(f)|^2 \le c_a
		$$ 
		for constants $c_a, d_a > 0$.
	\end{cor}
	
	\begin{proof}
		We prove the lemma when $f \in \Lambda_{\operatorname{sm}, 1} \cup \Lambda_{\operatorname{sm}, 2}$, as the claim on the rest of $\tt{Cross}$ follows by a symmetric argument. First assume $f \in \Lambda_{\operatorname{sm}, 1}$. By \cref{L:global-coupling-1} we can couple the height function $\mathcal{H}_n$ to the smooth phase height function $\mathcal H$ so that with probability at least $1 - 3 e^{-d_a n^{6/7}}$ we have
		\begin{equation}\label{E:CH}
			\mathcal H_n(f) - \mathcal H_n(0,0) = \mathcal H(f) - \mathcal H(0,0).
		\end{equation}
		Importantly, the coupling in \cref{L:global-coupling-1} only couples the increments of the height function, since the anchor points for the two height functions are different. Therefore letting $A$ be the event where \eqref{E:CH} holds, and letting $X = \mathcal H_n(0,0) + \mathcal H(v) - \mathcal H(0,0)$,  we have
		\begin{equation*}
			\mathcal H_n(f) = X + \mathbf{1}(A^c)(\mathcal H_n(f) -X)
		\end{equation*}
		Now, $\E X = 0$ since $\mathcal H(f) \stackrel{d}{=} \mathcal H(0,0)$ by the translation invariance of $\P_a$, and $\E \mathcal H_n(0,0) = 0$ by \cref{L:expected-middle-height}. Moroever, by the Cauchy-Schwarz inequality, we have
        $$
        |\E \mathbf{1}(A^c)(\mathcal H_n(f) -X)| \le \sqrt{\P(A^c)} \sqrt{\E (\mathcal H_n(f) -X)^2} \le c_a n \exp(-d_a n^{-6/7}),
        $$
        where the final inequality uses that $\P(A^c) \le 3 e^{-d_a n^{6/7}}$, the tail bound on heights in the smooth phase (\cref{L:smooth-phase-heights-decor}), and the fact that $|\mathcal H_n(f)| \le 4n$ deterministically since $\mathcal H_n$ has increments bounded by $4$ on adjacent faces, and fixed boundary conditions. Combining these results gives the expectation bound, after changing $d_a$. The second moment bound follows similarly, this time using that $\E X^2 \le c_a$ by \cref{L:variance-middle-height} and \cref{L:smooth-phase-heights-decor}.
		
		Now, for a $c$-face $f \in \Lambda_{\operatorname{sm}, 2}$, consider any $c$-face $f' \in \Lambda_{\operatorname{sm}, 2} \cap  \Lambda_{\operatorname{sm}, 1}$. Then by \cref{L:global-coupling-2} we can couple $\mathcal H_n$ to the smooth phase height function $\mathcal H$ such that
		\begin{equation*}
			\mathcal H_n(f) - \mathcal H_n(f') = \mathcal H(f) - \mathcal H(f')
		\end{equation*}
		with probability at least $1 - 2 \exp(-d_a \log^2 n)$. We can then proceed similarly to bound $\E \mathcal H_n(f)$ and $\E \mathcal H_n(f)^2$, this time using as input the estimates on the moments of $\mathcal H_n(f')$ from this lemma, rather than those for $\mathcal H_n(0,0)$.
	\end{proof}
	
	Next, we aim to upgrade the expectation estimates in Lemmas \ref{L:expectation-computation} and \ref{L:expectation-computation-2} to concentration estimates. We will use the following standard concentration inequality.
	
	\begin{lem}
		\label{L:concentration-lem}
		For any faces $f, f' \in \tt{F}$ and any $r > 0$ we have
		$$
		\P(|\mathcal{H}_n(f) - \E (\cH_n(f) \mid H_n(f'))| \ge r \sqrt{\|f - f'\|_1}) \le 2\exp(- r^2/18).
		$$
		Moreover, the function 
		$$
		x \mapsto \E (\cH_n(f) \mid H_n(f') = x), \qquad x \in 4 \Z
		$$
		is a non-decreasing $1$-Lipschitz function of $x$.
	\end{lem}
	
	\begin{proof}
		The concentration bound follows from a martingale argument and Azuma's inequality, see \cite[Theorem 5.8]{Gor21} for details. The fact that the function $\E (\cH_n(f) \mid H_n(f') = x)$ is $1$-Lipschitz and non-decreasing follows from height monotonicity for random tilings. The details of the argument are again contained in the proof of \cite[Theorem 5.8]{Gor21}.
	\end{proof}

	\begin{cor}
		\label{C:rough-phase-heights}
		Define the region
		$$
		\tt{Cap} = \beta_n([-n^{5/6}, n^{5/6}] \times [-n^{1/2} \log^2 n, 2n^{1/3} \log^2 n]). 
		$$
		Let $A_n$ be the event where for all $v = \beta_n(t, x) \in \tt{Cap}$ we have
		\begin{align*}
			|\mathcal H_n(R_\theta  v) + \eta_a |x|^{3/2} n^{-1/2}| &\le n^{1/4} \log^2 n, \quad \theta = 0, \pi, \\
			|\mathcal H_n(R_\theta  v) - \eta_a |x|^{3/2} n^{-1/2}| &\le n^{1/4} \log^2 n, \quad \theta = \pm \pi/2.
		\end{align*}
		Here $\eta_a$ is the constant from \cref{L:expectation-computation}.
		Then $\P(A_n) = 1 - o(n^{-1/6})$.
		
		Similarly, let $B_n$ be the event where for all $v = \beta_n(0, x) \in [-n, n]^2$ with $x \le 2 n^{1/3} \log^2 n$ we have
		\begin{align*}
			\cH_n(R_\theta v) &\le -\frac{\mu_a}{2} |x|^{3/2} n^{-1/2} + n^{1/4} \log^2 n, \quad \theta = 0, \pi.\\ \cH_n(R_\theta v) &\ge \frac{\mu_a}{2} |x|^{3/2} n^{-1/2} + n^{1/4} \log^2 n, \quad \theta = \pm \pi/2.
		\end{align*}
		Here $\mu_a$ is as in \cref{L:expectation-computation-2}. Then $\P(B_n) = 1 - o(n^{-1/6})$.
	\end{cor}
	
	\begin{proof}
		Let $A_n^0 \supset A_n, B_n^0 \supset B_n$ be the versions of these events where we only ask the inequalities to holds when $\theta = 0$. We first prove that $\P(A_n^0), \P(B^n_0) = 1 - o(n^{-1/6})$. Throughout the proof we focus on bounding heights on the face set $\tt{F}$. The extension to all general heights in $\R^2$ follows by approximation through the nearest face.
		
		\textbf{Bounding $\P(A_n^0)$.} \qquad For $i \in [-n^{1/3}, n^{1/3}] \cap \Z$, let $f_i$ be the $c$-face closest to the point $\beta_n (n^{1/2} i, n^{1/2})$. Then for every face $f \in \tt{Cap}$, we can find an index $I(f) \in [-n^{1/3}, n^{1/3}] \cap \Z$ with $\|F - F_{I(f)}\|_1 \le 2n^{1/2} \log^2 n$. Therefore by applying \cref{L:concentration-lem} with $r = \log n /2$, taking a union bound over the $o(n^2)$-many faces contained in $\tt{Cap}$, with probability at least $1 - o(n^2) \cdot \exp(-\log^2 n/72)$ we have
		\begin{equation}
			\label{E:union-1}
			|\mathcal{H}_n(f) - \E (\cH_n(f) \mid \cH_n(f_{I(f)}))| \le n^{1/4} \log^2 n/\sqrt{2}.
		\end{equation}
		Next, we claim that there exists a constant $c_a > 0$ such that for all faces $f, f_i$ and all $j_0 \in 4 \Z$ we have
		\begin{equation}
			\label{E:expected-1}
			|\E (\cH_n(f) \mid \cH_n(f_i) = j_0) - \E \cH_n(f)| \le c_a + j_0.
		\end{equation}
		Indeed, we can write
		\begin{align*}
			&|\E \cH_n(f) - \E (\cH_n(f) \mid \cH_n(f_i) = j_0)| \\
			&\le \sum_{j \in 4 \Z } \Big|\E (\cH_n(f) \mid \cH_n(f_i) = j) - \E (\cH_n(f) \mid \cH_n(f_i) = j_0)\Big| \P(\cH_n(f_i) = j) \\
			&\le \sum_{j \in 4 \Z } |j-j_0| \P(\cH_n(f_i) = j) \\
			&= \E|\cH_n(f_i) - j_0| \le c_a + j_0. 
		\end{align*}
		Here the second inequality uses that $\E (\cH_n(f) \mid \cH_n(f_i) = j)$ is a $1$-Lipschitz function of $j$ (\cref{L:concentration-lem}), and the final inequality uses Corollary \ref{C:expected-sm-height}. Finally, Corollary \ref{C:expected-sm-height}, Chebyshev's inequality, and a union bound give that
		\begin{equation}
			\label{E:Hnfi}
			\begin{split}
				&\P(\text{ There exists } i \in [-n^{1/3}, n^{1/3}] \cap \Z \text{ s.t. } |\mathcal H_n(F_i)| \ge \tfrac{1}{4} n^{1/4} \log^2 n) \\
				&\le c_a n^{-1/6} \log^{-4} n.  
			\end{split}
		\end{equation}
		Together with \eqref{E:union-1}, \eqref{E:expected-1}, and the estimate from \cref{L:expectation-computation}, this implies $\P(A_n^0) = 1 - o(n^{-1/6})$.
		
		\textbf{Bounding $\P(B_n^0)$.}\qquad This estimate is similar. Indeed, similarly to \eqref{E:union-1}, with probability at least $1 - 2n^2 \exp(-\log^2 n/72)$, for all faces of the form $f = \beta_n(0, x)$ for some $x \le 2 n^{1/3} \log^2 n$ we have
		\begin{equation}
			\label{E:union-2}
			|\mathcal{H}_n(f) - \E (\cH_n(f) \mid \cH_n(f_0))| \le \frac{1}{2}(n^{1/2} - x)^{1/2} \log n,
		\end{equation}
		where $f_0 = \beta_n(0, n^{1/2})$ as before. Moreover, from Corollary \ref{C:expected-sm-height} and Chebyshev's inequality we have
		\begin{equation*}
			\P(|\cH_n(f_0)| \ge n^{1/10}) \le c_a n^{-1/5}.  
		\end{equation*}
		Combining this bound, \eqref{E:expected-1}, \eqref{E:union-2}, and \cref{L:expectation-computation-2} gives that with probability $1 - o(n^{-1/6})$, for all faces $f = \beta_n(0, x)$ with $x \le 2 n^{1/3} \log^2 n$ we have
		\begin{align*}
			\mathcal{H}_n(f)
			\le &- \mu_a |x|^{3/2} n^{-1/2} + \frac{1}{2}(n^{1/2} - x)^{1/2} \log n + c_a + n^{1/10} + n^{1/6}.
		\end{align*}
		yields the desired estimate on $\P(B_n^0)$ after simplification.
		
		\textbf{Bounding $\P(A_n), \P(B_n)$.} \qquad To extend the height estimates on $\P(A_n^0), \P(B_n^0)$ to all of $A_n, B_n$, observe that all steps go through verbatim, except we cannot apply Lemmas \ref{L:expectation-computation} and \ref{L:expectation-computation-2} directly since these lemmas do not pertain to points of the form $R_\theta v$ where $v \in \tt{Cap}$ or $v = \beta_n(0, x) \in [-n, n]^2$  and $\theta \ne 0$. Rather, we need to apply these lemmas in conjunction with the symmetries in expectation from \cref{L:expected-middle-height} (in the case when $\theta = \pi$, we could also use the stronger symmetries from \cref{L:sym}, but this is not useful for $\theta = \pm \pi/2$).
	\end{proof}

	\section{An Airy surface in a smooth background}
	\label{S:proof-t1}
	
	In this section, we prove \cref{T:main-1}, following the proof sketch outlined in Section \ref{sec:proof-sketch}. We will also prove a few variants of this result which offer different perspectives on the idea that at the rough-smooth boundary we see an Airy surface sitting in a smooth background. As discussed there, our starting point is the main result of \cite{BCJ18}, which shows that Airy statistics arise after taking small averages of the height function around the rough-smooth boundary. 
	
	\subsection{The Airy line ensemble, the Airy surface, and the main result of \cite{BCJ18}}
	\qquad
	
	We start by giving a formal definition of the Airy line ensemble.
	\begin{defn}
		\label{D:ALE-formal}
		The (stationary) \textbf{Airy line ensemble} is a random sequence of continuous functions $\cA = (\cA_i:\R \to \R, i \in \N)$ whose law is uniquely characterized by the following two properties:
		\begin{itemize}
			\item For all $x \in \R$ and $i \in \N$, $\cA_i(x) > \cA_{i+1}(x)$.
			\item For any finite set of times $T = \{t_1, \dots, t_k\}$, the point process 
			$$
			\bigcup \{(t_j, \cA_i(t_j)) : i \in \N, j \in \II{1, k}\}
			$$
			is a determinantal point process on $T \times \R$ with kernel
			\begin{equation}\label{eq:extendedAiry}
				{\mathcal{A}}(\tau_1,\zeta_1;\tau_2,\zeta_2)=\tilde{\mathcal{A}}(\tau_1,\zeta_1;\tau_2,\zeta_2)-\phi_{\tau_1,\tau_2} (\zeta_1 ,\zeta_2),
			\end{equation}
			where
			\begin{equation} \label{eq:Airymod}
				\begin{split}
					\tilde{\mathcal{A}}(\tau_1,\zeta_1;\tau_2,\zeta_2)&= 
					\int_0^\infty e^{-\lambda (\tau_1-\tau_2) } \mathrm{Ai} (\zeta_1 +\lambda) \mathrm{Ai} (\zeta_2+\lambda) d\lambda\\
				\end{split}
			\end{equation}
			and
			\begin{equation}\label{eq:Airyphi}
				\phi_{\tau_1,\tau_2} (\zeta_1 ,\zeta_2) =
				\mathbbm{I}_{\tau_1<\tau_2} \frac{1}{\sqrt{4 \pi (\tau_2-\tau_1)}} e^{-\frac{(\zeta_1-\zeta_2)^2}{4(\tau_2-\tau_1)}-\frac{(\tau_2-\tau_1)(\zeta_1+\zeta_2)}{2}+\frac{(\tau_2-\tau_1)^3}{12}}.
			\end{equation}
			The function $\cA$ is known as the \textbf{extended Airy kernel}.
		\end{itemize}
	\end{defn}
    In the above definition and in the sequel we use the integer interval notation $\II{a, b} := \{a, \dots, b\}$.
	Next, as in the introduction we define the \textbf{Airy surface} $\cA:\R^2 \to \{0, 1, 2, \dots\}$ by letting 
	$$
	\cA(x, t) = \# \{i \in \N : \cA_i(t) \ge x\}.
	$$
	This function is well-defined since $\cA_k \to -\infty$ as $k \to \infty$, uniformly on compact sets. This follows, for example, from \cite[Theorem 1.6]{DauvergneVirag2021}. The Airy surface is the analogue of the height function for random tilings, and so it is a natural object to introduce when studying height function convergence in the absence of a tractable family of non-intersecting paths. 
	
	Next, we recast the main theorem of \cite{BCJ18} in terms of height function convergence to the Airy surface.  To set up some notation, for a finite set $\phi$ and a point $v \in \R^2$, define the mollified height function
	\begin{equation}
		\label{E:mollified-height-fun}
		\cH_n^\phi(v) = \frac{1}{|\phi|} \sum_{u \in \phi} \cH_n(v + u),
	\end{equation}
	with ties broken using the lexicographic order. Also, define the scaling transformation
	$$
	\tilde \beta_n(t, x) = [\beta_n(2^{1/3}\lambda_2 n^{2/3} t, 2^{2/3}\lambda_1 n^{1/3} x)]_a,
	$$
	where the constants $\lambda_1, \lambda_2$ are as in \eqref{eq:scalingparameters}. Here the outer bracket $[\cdot]_a$ maps a point in $\R^2$ to the nearest $a$-face, where ties are broken with the lexicographic order. The transformation $\tilde \beta_n$ maps limiting coordinates at the rough-smooth boundary to unscaled coordinates in the two-periodic Aztec diamond. This is the scaling transformation used in \cite{BCJ18}, and differs from the scaling transformations $\gamma_n, \hat \gamma_n$ we defined in the introduction by a shear that becomes negligible in the $n \to \infty$ limit. For most of this section, we will work with the transformation $\tilde \beta_n$, in order to more easily work with the results of \cite{BCJ18}. We translate to the scaling $\hat \gamma_n$ in a straightforward final step. Next, define
	\begin{equation}
	\label{E:tilde-Hn}
    	\tilde \cH_n := \cH_n \circ \tilde \beta_n, \qquad \tilde \cH_n^\phi := \cH_n^\phi \circ \tilde \beta_n 
	\end{equation}
	for the height function and mollified height function defined in the limiting coordinates. 
	\begin{thm}[Theorem 1.1, \cite{BCJ18}]
		\label{T:BCJ18-weak}
		Suppose we have the following data:
		\begin{itemize}
			\item A collection of times $t_1 \le t_2 \le \cdots \le t_k$ and spatial locations $x_1 \le y_1 \le x_2 \le y_2 \cdots \le x_k \le y_k$.
			\item Mollifiers $\phi_n, n \in \N$ of the form
			$$
			\phi_n = \{2\lfloor \ell \log^2 n \rfloor e_2 : \ell \in \{0, \dots, a_n\}\},
			$$
			where the sequence $a_n \to \infty$ as $n \to \infty$ and satisfies $a_n = o(n^{1/3}\log^{-2} n)$ as $n \to \infty$.
			\item Sequences $t_{i, n} \to t_i, x_{i, n} \to x_i, y_{i, n} \to y_i$ for all $i = 1, \dots, k$ such that for every $i \in \{1, \dots, k - 1\}$ we have that
			\begin{align*}
	\lambda_2 n^{2/3} t_{i, n} + 2 a_n \log^2 n + 4 &< \lambda_2 n^{2/3} t_{i + 1, n} + 2 a_n \log^2 n, \\
\lambda_1 n^{1/3} x_{i, n} + 8 <	\lambda_1 n^{1/3} y_{i, n} + 4 &< \lambda_1 n^{1/3} x_{i + 1, n}.
			\end{align*} 
		\end{itemize}
		Then we have the following Fourier transform convergence. There exists a constant $\eps_0 > 0$, such that for any $\lambda = (\lambda_{i, j})  \in \mathbb{C}^{k \times \ell}, \|\lambda\|_\infty < \eps_0$ we have
		\begin{equation}
			\label{E:height-with-a-shift}
			\begin{split}
				\lim_{n \to \infty} &\E \exp \left (\sum_{i, j=1}^k \lambda_{i, j} [\tilde \cH_n^{\phi_n}(t_{j, n}, y_{i, n}) - \tilde \cH_n^{\phi_n} (t_{j, n}, x_{i, n})]\right) \\
				=\;\; &\E \exp \left (-4 \sum_{i, j=1}^k \lambda_{i, j} [\cA(t_j, y_i) - \cA (t_j, x_i)]\right).
			\end{split}
		\end{equation}
	\end{thm}
	This theorem is a stronger version of the statement \eqref{E:BCJ-intro} described in the introduction. Note that in the original statement in \cite{BCJ18}, the sequences $t_{i, n}, x_{i, n}, y_{i, n}$ are fixed, equal to $t_i, x_i, y_i$, and the prelimiting inequalities in the third bullet point above are replaced with the stronger assumptions that $t_i < t_{i+1}, y_i < x_{i+1}$ for all $i \in \{1, \dots, k-1\}$. The proof goes through verbatim in the more general setting described above, since the inequalities $t_i < t_{i+1}, y_i < x_{i+1}$ are only used to ensure that all of the line segments from
	$$
	\tilde \beta_n(t_i, x_j) + 2\lfloor \ell \log^2 n \rfloor e_2 \qquad \text{to} \qquad \tilde \beta_n(t_i, y_j) + 2\lfloor \ell \log^2 n \rfloor e_2
	$$ 
	are disjoint for any choice of $i, j \in \{1, \dots, k\}, \ell \in \{0, \dots, a_n\}$. The $+4$ and $+8$ terms in the inequalities in the third bullet ensure that we do not run into issues from rounding to the nearest $a$-face. We translate \cref{T:BCJ18-weak} to a statement that begins to look more like \cref{T:main-1}. 
	\begin{cor}
		\label{C:BCJ-translation}
		Let $t_i, x_i, \phi_n$ be as in \cref{T:BCJ18-weak}. Then there exist random variables $H_{i, j, n}, i, j \in \II{1, k}, n \in \N$ such that we have the following joint convergence in law:
		\begin{equation}
			\label{E:joint-easy}
			\begin{split}
	&(H_{i, j, n} - \tilde \cH_n^{\phi_n} (t_{i, n}, x_{j, n}), H_{i, j, n} - \tilde \cH_n^{\phi_n} (t_{i, n}, y_{j, n}): i, j \in \II{1, k}) \\
	\implies &(4\cA(t_i, x_j), 4 \cA(t_i, y_j) : i, j \in \II{1, k}).
			\end{split}
		\end{equation}
	\end{cor}
	
	\begin{proof}
	The Laplace transform convergence in \eqref{E:height-with-a-shift} implies the following convergence in distribution.
		\begin{equation}
			\label{E:joint-1}
			\begin{split}
				(\tilde \cH_n^{\phi_n} (t_{i, n}, y_{j, n}) &- \tilde \cH_n^{\phi_n} (t_{i, n}, x_{j, n}): i \in \II{1, k}, j \in \N) \\
				&\implies (\cA(t_i, x_j) - \cA(t_i, y_j) : i \in \II{1, k}, j \in \N).
			\end{split}
		\end{equation}
		We can translate this statement to involve height shifts $H_{i, j, n}$. Indeed, by Skorokhod's representation theorem we can realize the above distributional convergence as almost sure convergence on some probability space, so that
		$$
		\tilde \cH_n^{\phi_n} (t_{i, n}, y_{j, n}) - \tilde \cH_n^{\phi_n} (t_{i, n}, x_{j, n}) \to \cA(t_i, x_j) - \cA(t_i, y_j)
		$$
		almost surely for all $i = 1, \dots, k, j \in \N$. In this coupling, define $H_{i, j, n} = \tilde \cH_n^{\phi_n} (t_i, y_j) + \cA(t_i, y_j)$. This realizes \eqref{E:joint-easy} as almost sure convergence, yielding the result.
	\end{proof}
	\subsection{The proof of \cref{T:main-1}}
	
	We can now clearly see the two obstacles to proving \cref{T:main-1}:
	\begin{itemize}[nosep]
		\item Corollary \ref{C:BCJ-translation} requires a specific mollification of increasing size with $n$, which has the effect of eliminating fluctuations from the smooth background.
		\item Corollary \ref{C:BCJ-translation} requires a different (and abstractly defined) random height shift for every statistic, whereas \cref{T:main-1} gives a comparison with the central height.
	\end{itemize} 
	We deal with both of these by appealing to the smooth phase couplings in Section \ref{S:smooth-phase-couplings}. We first show that the height shift can be taken independently of the coordinate $j$.
	
	\begin{lem}
\label{L:j-independent}
Consider the setting of Corollary \ref{C:BCJ-translation}, where the averaging parameter $a_n = o(n^{1/24} \log^{-2} n)$. In this setting we can replace $H_{i, j, n}$ with the height shift $H_{i, n}:=H_{i, 1, n}$ and the statement still holds. 
	\end{lem}
	
	\begin{proof}
	First, by potentially filling in the sequence $x_1,y_1, \dots, x_k, y_k$ with new intermediate pairs and associated new approximation sequences, we may assume that
		\begin{align}
				x_i = y_{i-1}, \quad \text{ and } \quad \lambda_1 n^{1/3} x_{i, n} \ge \lambda_1 n^{1/3} y_{i-1, n} + 5.
			\end{align}
	We claim that with this assumption, any choice of random variables $H_{i, j, n}$ satisfying \eqref{E:joint-easy} will satisfy
	$$
	H_{i, j, n} - H_{i, j+1, n} \cvgp 0
	$$
	as $n \to \infty$ for all $i \in \II{1, k}, j \in \II{1, k-1}$. This is equivalent to the statement that 
	\begin{equation}
		\label{E:avg}
	\tilde \cH_n^{\phi_n}(t_{i, n}, y_{j, n}) - 	\tilde \cH_n^{\phi_n}(t_{i, n}, x_{j + 1, n}) \cvgp 0
	\end{equation}
	for all $i \in \II{1, k}, j \in \II{1, k-1}$. To prove \eqref{E:avg}, consider the set 
			$$
			\Lambda = \beta_n(t_{j, n}, y_{i, n}) + [-100 a_n \log^2 n, 100a_n \log^2 n]^2.
			$$
			Since $a_n \log^2 n = o(n^{1/24})$, by \cref{L:local-coupling} we can couple the Aztec diamond measure $D_n \sim \P_{a, n}$ with the smooth phase $D \sim \P_a$ such that $
			D_n|_{\Lambda} = D|_{\Lambda}$ with probability $1 - o(1)$. Let $\cH$ denote the height function in $D$, and define the functions $\tilde \cH, \tilde \cH^{\phi_n}$ as they were defined for $\cH_n$. Then on the event where $
			D_n|_{\Lambda} = D|_{\Lambda}$, we have that
			$$
				\tilde \cH_n^{\phi_n}(t_{i, n}, y_{j, n}) - 	\tilde \cH_n^{\phi_n}(t_{i, n}, x_{j + 1, n}) = 	\tilde \cH^{\phi_n}(t_{i, n}, y_{j, n}) - 	\tilde \cH^{\phi_n}(t_{i, n}, x_{j + 1, n}),
				$$
			so it suffices to show that the right-hand side above converges to $0$ in probability with $n$. This follows from the covariance bound \cref{L:smooth-phase-heights-decor}, the fact that $\phi_n$ only averages over $a$-faces, and the symmetry in Corollary \ref{C:reflection-symmetry-smooth}, which implies $\E \cH(f) = 0$ on any $a$--face $f$.
	\end{proof}
	
	We can push the smooth phase coupling idea from \cref{L:j-independent} a bit more in order to eliminate the need for the mollifying sequence $\phi_n$.
	
	\begin{lem}
		\label{L:arbitrary-mollifiers-new}
		In Corollary \ref{C:BCJ-translation}/\cref{L:j-independent}, we can allow $\phi_n \subset (2 \Z)^2$ to be any sequence of sets of satisfying $|\phi_n| \to \infty$ as $n \to \infty$, and $\max_{x \in \phi_n} \|x\|_\infty = o(n^{1/24})$. The random variables $H_{i, n}$ do not depend on the choice of $\phi_n$.
		
		Moreover, under the assumption that
		\begin{equation}
		\label{E:beta-n-assn}
		\inf_{(i, j) \ne (i', j')} \|\tilde \beta_n(t_{i, n}, x_{j, n}) - \tilde \beta_n(t_{i', n}, x_{j', n})\|_\infty \to \infty,
		\end{equation}
		we have the following joint convergence:
		\begin{equation}
			\label{E:height-with-a-shift2}
			\begin{split}
				(H_{i, n} - \cH_n(\tilde \beta_n(t_{i, n}, x_{j, n}) + v)&: i, j \in \II{1, k}, v \in \hat \Z^2) \\
				\implies (4\cA(t_i, x_j) - \cH_{i, j}((1, 1) + v) &: i, j \in \II{1, k}, v \in \hat \Z^2),
			\end{split}
		\end{equation}
		where $\cA$ is the Airy surface, and $(\cH_{i, j}, i, j \in \II{1, k})$ is a collection of height functions drawn from the smooth phase $\P_a$, independent of $\cA$ and of each other.
	\end{lem}
	
	\begin{rem}
		\label{R:mollifier-extent}
		The upper bound on the diameter of the mollifying sequence is not optimal, and comes from a technical constraint on the smooth coupling. The machinery we develop in later sections could improve this bound to $\max_{x \in \phi_n} \|x\|_\infty = o(n^{1/3})$, but we do not pursue this improvement here. Note that the condition that $\phi_n \subset (2 \Z)^2$ is purely so that we average over only $a$- and $b$-faces. This could easily be loosened, but we would need to account for the fact that the expected height in the smooth phase at other faces is non-zero.
		
		Note also that the assumption \eqref{E:beta-n-assn} simply ensures that the heights $\cH_{i, j}$ are independent in the limit. If there were a pair of points $(i, j), (i', j')$ (or multiple pairs) where $\tilde \beta_n(t_{i, n}, x_{j, n}) - \tilde \beta_n(t_{i', n}, x_{j', n})$ was converging to a point in $\hat \Z^2$, rather than $\infty$, then a version of \eqref{E:height-with-a-shift2} would hold where $\cH_{i, j}$ and $\cH_{i', j'}$ were spatial shifts of each other.
	\end{rem}
	
	To prove \cref{L:arbitrary-mollifiers-new}, we need a simple lemma related to tail-triviality of the smooth phase. For this lemma, for a dimer configuration $D \sim \P_a$ we let $T_x D$ be the translation of $D$ by $x \in \Z^2$, i.e. 
	$
	T_x(D) = \{ e + x : e \in D\}.
	$
	The smooth phase $\P_a$ is invariant under a translation $T_x$ if and only if $x$ is a $c$-face (i.e. if $x = (i, j)$ where $i, j \in 2 \Z$ and $i + j \in 4 \Z$).
	\begin{lem}
		\label{L:smooth-phase-mixing}
		Let $A_1, \dots, A_k$ be events in the space of all dimer configurations on $\tt{\tilde V}$ (i.e. the space of all perfect matchings of $\tt{\tilde V}$). Let $x_{n, 1}, \dots, x_{n, k} \in \tt{\tilde F}, n \in \N$ be $c$-faces, and let $((U_{n, 1}, \dots, U_{n, k}) : n \in \N)$ be a sequence of disjoint subsets of $\Z^2$ such that $x_{n, i} \in U_{n, i}$ for all $i$ and 
		$$
		\inf \{\|x_{n, i} - y\|_\infty : y \notin U_{n, i}\} \to \infty
		$$
		with $n$ for all $i = 1, \dots, k$. Let $\cF_n$ be the $\sigma$-algebra generated by $D \sim \P_a$ restricted to $V_n^c$ where $V_n := \bigcup_{i=1}^k U_{n, i}$. Then
		$$
		\lim_{n \to \infty} \P_a( D \in T_{x_{n,i}} A_i, i = 1, \dots, k \; \mid \; \cF_n) = \prod_{i=1}^k \P(D \in A_i),
		$$
		where the convergence holds in $L^1$.
	\end{lem}
	
	\begin{proof}
		We proceed by induction, with the $n = 0$ base case being tautological. Now suppose that the limiting claim holds for some $n = 0, 1, \dots$. By translation invariance of $\P_a$ we have
		$$
		\P_a( D \in T_{x_{n,i}} A_i, i = 1, \dots, n \mid \cF_n) \eqd \P_a( D \in T_{x_{n,i} - x_{n, 1}} A_i, i = 1, \dots, n \mid \cF_n^*),
		$$
		where $\cF_n^*$ is the $\sigma$-algebra generated by $D$ restricted to $T_{-x_{n, 1}} V_n$. Therefore it suffices to prove the claim when $x_{n, 1} = (0,0)$. Now, let $\mathcal G_n \supset \mathcal F_n$ be the $\sigma$-algebra generated by $D|_{U_{n, 1}^c}$ together with the events $T_{x_{n,i}} A_i, i = 2, \dots, n$. Then
		\begin{align*}
			&\P_a( D \in T_{x_{n,i}} A_i, i = 1, \dots, n \mid \cF_n) \\
			=\; &\E (\P_a(D \in A_1 \mid \mathcal G_n) \mathbf{1}(D \in T_{x_{n,i}} A_i, i = 2, \dots, n) \mid \cF_n)
		\end{align*}
		where the equality uses the tower property of conditional expectation. Now, 
		$$
		\lim_{n \to \infty} \P_a(D \in A_1 \mid \mathcal G_n) = \P_a(D \in A_1)
		$$
		in probability by Kolmogorov's $0$-$1$ law. (To apply Kolmogorov's $0$-$1$ law we can use Wilson's algorithm to construct the smooth phase $D$ as a function of an IID sequence of random variables. In this construction, any limit points of the left-hand side above are measurable with respect to the tail $\sigma$-algebra). Moreover, by the inductive hypothesis,
		$$
		\P (T_{x_{n,i}} D \in A_i, i = 2, \dots, n) \mid \cF_n) = \prod_{i=2}^n \P(D \in A_i)
		$$
		in probability. Combining the previous three displays yields the result.
	\end{proof}
	
	\begin{proof}[Proof of \cref{L:arbitrary-mollifiers-new}]
		Define
		$$
		\psi_n = \{2\lfloor \ell \log^2 n \rfloor e_2 : \ell \in \{0, \dots, \lfloor n^{1/100} \rfloor \} \},
		$$
		and let $\phi_n$ be an arbitrary mollifying sequence satisfying the conditions of the corollary.
		Let $b_n \in \N$ be any sequence with $\phi_n \cup \psi_n \subset [-b_n, b_n]^2$ and $b_n = o(n^{1/24})$. Let 
		$$
		U_n = \tt{\tilde V} \cap \bigcup_{i, j = 1}^k ([-b_n, b_n]^2 + u_{i, j}),
		$$
		where $u_{i, j} := \tilde \beta_n(t_{i, n}, x_{j, n})$.
		Then $|U_n| \le 4 k^2 b_n^2$, and there exists $L > 0$ such that for all $n$, we have $U_n \subset \beta_n([-L n^{2/3}, L n^{2/3} \times [-L n^{1/3}, L n^{1/3}])$. Therefore by \cref{L:local-coupling}, we can couple the two-periodic Aztec diamond $D_n \sim \P_{a, n}$ and the smooth phase $D \sim \P_a$ so that 
		\begin{equation}
			\label{E:coupling-equal}
			\P(D|_{U_n} = D_n|_{U_n}) \ge 1 - C n^{-1/3} k^8 b_n^8 \ge 1 - o(1).
		\end{equation}
		Let $\cH$ denote the height function in $D$.  On the event where $D|_{U_n} = D_n|_{U_n}$,
		\begin{equation}
			\label{E:split-hnti}
			\begin{split}
				\cH^{\psi_n}(u_{i, j}) - \cH^{\phi_n}(u_{i, j}) &= \cH_n^{\psi_n}(u_{i, j}) - \cH^{\phi_n}_n(u_{i, j}).
			\end{split}
		\end{equation}
		Now, arguing as in the proof of \cref{L:j-independent}, as $n \to \infty$ we have that
		\begin{equation}
			\label{E:hnpsin}
			\cH^{\psi_n}(u_{i, j}) \cvgp 0, \qquad \cH^{\phi_n}(u_{i, j}) \cvgp 0.
		\end{equation}
		Putting this together with \eqref{E:coupling-equal}, \eqref{E:split-hnti}, and \cref{L:j-independent} for the mollifiers $\psi_n$ then gives
		\begin{equation*}
			(H_{i, n} - \cH^{\phi_n}(u_{i, j}): i, j \in \II{1, k}) \implies (\cA(t_i, x_j) : i, j \in \II{1, k}).
		\end{equation*}
		where the random variables $H_{i, n}$ are the same as in \eqref{E:joint-easy}. This is the first claim in the lemma.
		
		We move on to the display \eqref{E:height-with-a-shift2}. We will work in a version of the coupling above, but with 
		$$
		b_n = \min\Big(n^{1/100} ,\inf_{(i, j) \ne (i', j')} \|\tilde \beta_n(t_{i, n}, x_{j, n}) - \tilde \beta_n(t_{i', n}, x_{j', n})\|_\infty/3\Big).
		$$
		This ensures that the sets $u_{i, j} + [-b_n, b_n]^2, i , j \in \II{1, k}$ are disjoint. We also set
		$$
		\phi_n := (2 \Z)^2 \cap ([-b_n, b_n]^2 \setminus [-b_n/2, b_n/2]^2).
		$$
		The bound \eqref{E:coupling-equal} still holds here, as does the second convergence in \eqref{E:hnpsin} and on the event where $D|_{U_n} = D_n|_{U_n}$, for any $i, j$ and any point $v \in u_{i, j} + [-b_n, b_n]^2$ we have
		$$
			\cH(v) -\cH^{\phi_n}(u_{i, j}) = \cH_n(v) - \cH^{\phi_n}_n(u_{i, j}).
			$$
			By this equality and the first part of the lemma, to prove \eqref{E:height-with-a-shift2} it suffices to show that
		\begin{equation}
			\label{E:AXX}
			(H_{i, n} - \cH_n^{\phi_n}(u_{i, j}); \cH(u_{i, j} + v))\implies (4\cA(t_i, x_j); \cH_{i, j}((1, 1) + v)),
		\end{equation}
		where the convergence in law is joint over all $i, j \in \II{1, k}$ and all $v \in \hat \Z^2$.
		To prove \eqref{E:AXX}, define the set 
		$$
		V_n = \tt{V}  \cap \bigcup_{i, j =1}^k ([-b_n/2, b_n/2]^2 + u_{i, j}).
		$$
		On the event in \eqref{E:coupling-equal}, given $D|_{V_n^c}$ and $D_n|_{V_n^c}$, resample the dimer configuration $D|_{V_n} = D_n|_{V_n}$ according to the $2$-periodic weights to give a new configuration $D'|_{V_n} = D_n'|_{V_n}$. Extend $D', D_n'$ to equal $D, D_n$ on $V_n^c$, and let $\cJ$, $\cJ_n$ denote the height functions in these configurations on the full plane and the Aztec diamond, respectively. By the Gibbs property for the dimer laws $\P_a, \P_{a, n}$, we have that $(\cH, \cH_n) \eqd (\cJ, \cJ_n)$ so it suffices to show that \eqref{E:AXX} holds for the pair $(\cJ, \cJ_n)$.
	Observe that by the definition of $V_n, \phi_n$, we have $\cJ_n^{\phi_n}(t_{i, n}, x_{j, n}) = \cH_n^{\phi_n}(t_{i, n}, x_{j, n})$ for all $i, j$. Therefore to prove \eqref{E:AXX} it suffices to show that conditional on $\cH_n^{\phi_n}$ (and the $H_{i, n}$, which can be taken as measurable functions of $H_n^{\phi_n}$)  we have
		\begin{equation}
			\label{E:conditional-limit}
			(\cJ(u_{i, j} + v))_{i, j \in \II{1, k}, v \in\hat \Z^2} \implies (\cH_{i, j}((1, 1) + v))_{i, j \in \II{1, k}, v \in\hat \Z^2}.
		\end{equation}
		For this, observe that if we let $\cF_n$ be the $\sigma$-algebra generated by $D|_{V_n^c}$ and let $\mathcal G_n$ be $\sigma$-algebra generated by both $\cF_n$ and $\cH_n^{\phi_n}$ we have
		$$
		\P(D' \in \cdot \mid \cF_n) = \P(D' \in \cdot \mid \mathcal G_n).
		$$
		Now, the sets $V_n$ fit the assumptions of \cref{L:smooth-phase-mixing} since $b_n \to \infty$ with $n$.
		Therefore by that lemma, for any sets $A_{i, j}, i, j = 1, \dots, k$ we have
		$$
		\P(\cJ(u_{i, j} + \cdot) \in A_{i, j} : i, j \in \II{1, k} \mid \mathcal F_n) \cvgp \prod_{i, j =1}^k \P(\cJ((1, 1) + \cdot) \in A_{i, j}).
		$$
		Combining the previous two displays gives \eqref{E:conditional-limit}.
	\end{proof}
	Next, we recall the definition of the central height. Using the mollifier notation introduced above and letting $\zeta_n$ be the set of \textit{all} faces with $\|f\|_\infty \le n^{3/4}$, the central height $H_n$ is the unique integer in the interval 
	$$
	(\cH_n^{\zeta_n}(0,0) - 1/2, \cH_n^{\zeta_n}(0,0) + 1/2].
	$$
	To complete the proof of \cref{T:main-1}, we need to show that all of the height shifts $H_{i, n}$ can be replaced by $H_n$. Up to a coordinate change, this will immediately imply the theorem.
	
	\begin{lem}
		\label{L:height-shift}
		In the setting of \cref{L:arbitrary-mollifiers-new}, we can take $H_{i, n} = H_n$ for all $i$, and the lemma still holds.
	\end{lem}

	To prove \cref{L:height-shift}, we require a strengthening of \cref{T:BCJ18-weak} in the $k = \ell = 1$ case that allows the endpoint $y_1$ to tend to $\infty$ with $n$. The proof in this setting is a similar to the proof of \cref{T:BCJ18-weak} in \cite{BCJ18} with a few minor tweaks, and is included in \cref{S:BCJ18}. 
	
	\begin{thm}
		\label{T:BCJ18-strong}
		Suppose we have:
		\begin{itemize}
			\item Sequences $t_n \to t$ and $x_n \le y_n = \log^3 n$ with $x_n \to x \in \R \cup \{\infty\}$.
			\item Mollifiers $\phi_n, n \in \N$ of the form
			$$
			\phi_n = \{2\lfloor k \log^2 n \rfloor e_2 : k \in \{0, \dots, a_n\}\},
			$$
			where the sequence $a_n \to \infty$ as $n \to \infty$ and satisfies $a_n = o(n^{1/3}\log^{-2} n)$ as $n \to \infty$.
		\end{itemize}
		Then there exists a constant $\eps > 0$ such that for any $\lambda \in[-\eps,\eps]$ we have
		\begin{equation}
			\label{E:height-with-a-shift'}
			\begin{split}
				\lim_{n \to \infty} \E \exp \left (\lambda [\tilde \cH_n^{\phi_n}(t_n, y_n) - \tilde \cH_n^{\phi_n} (t_n, x_n)]\right) =\E \exp \left (4\lambda \cA(x, t)\right),
			\end{split}
		\end{equation}
		where if $x = \infty$, we set $\cA(x, t) = 0$.
	\end{thm}
	
	\begin{proof}[Proof of \cref{L:height-shift}]
		First, in the setting of \cref{L:arbitrary-mollifiers-new}, we can extend the collection of points $x_i, i =1, \dots, k$ and sequences $x_{i, n}$ to infinite collections $x_i, x_{i, n}, i, n \in \N$ with $x_i \to \infty$ with $n$. The convergences in law in the lemma continue to hold by Kolmogorov's extension theorem. We claim that in this setting, any sequence of heights $H_{i, n}$ realizing the convergence
	\begin{equation}
		\label{E:joint-harder}
			(H_{i, n} - \tilde \cH_n^{\psi_n} (t_{i, n}, x_{j, n}) : j \in \N)
			\implies (4\cA(t_i, x_j) : j \in \N)
	\end{equation}
	with the explicit choice of
	$$
	\psi_n = \{2\lfloor k \log^2 n \rfloor e_2 : k \in \{0, \dots, \lfloor n^{1/100} \rfloor \} \},
	$$
		must satisfy
		\begin{equation}
			\label{E:45}
		|H_{i, n} - H_n| \cvgp 0
		\end{equation}
		as $n \to \infty$, which yields the desired result. To prove \eqref{E:45}, first note that for any fixed $i$, we have $\cA(t_i, x) \cvgp 0$ as $x \to \infty$, since the Airy point process has a largest point almost surely. Therefore
		\begin{equation*}
			\lim_{j \to \infty} \lim_{n \to \infty} H_{i, n} - \tilde \cH_n^{\psi_n} (t_{i, n}, x_{j, n}) = 0,
		\end{equation*}
		where the convergence is in probability.
		Similarly, by \cref{T:BCJ18-strong}, letting $y_n = \log^3 n$ as in that theorem, we have that in probability,
		\begin{equation*}
\lim_{j \to \infty} \lim_{n \to \infty} \cH_n^{\psi_n}(t_{j, n}, y_n) - \tilde \cH_n^{\psi_n} (t_{i, n}, x_{j, n}) = 0.
		\end{equation*}
		Subtracting together these two equations gives that in probability,
		\begin{equation*}
	\lim_{j \to \infty} \lim_{n \to \infty} \cH_n^{\psi_n}(t_{j, n}, y_n) - H_{i, n} = 0.
		\end{equation*}
	Given this, to prove the lemma we just need to show that
		\begin{equation}
			\label{E:cvg0-in-prob}
			\tilde \cH_n^{\psi_n}(t_{j, n}, y_n) - H_n \cvgp 0.
		\end{equation}
		Recall the sets $\Lambda_{\operatorname{sm}, 1}, \Lambda_{\operatorname{sm}, 2}$ defined in the global smooth phase couplings, \cref{L:global-coupling-1} and \cref{L:global-coupling-2}. Then by applying these two lemmas, with probability 
		$1 - o(1)$ we can couple our dimer configuration $D \sim \P_{a, n}$ with two other dimer configurations $D_1, D_2 \sim \P_a$ such that
		\begin{equation}
			\label{E:successful-coupling}
			D|_{\Lambda_{\operatorname{sm}, i}} = D_i|_{\Lambda_{\operatorname{sm}, }i}, \qquad \text{ for } i = 1, 2.
		\end{equation}
		Now, let $f_1$ be an $a$-face chosen so that $f_1 + \psi_n \subset \Lambda_{\operatorname{sm}, 1} \cap \Lambda_{\operatorname{sm}, 2}$, and let $f_2 = \tilde \beta_n(t_{j, n}, y_n)$. Since $y_n \to \infty$ much faster than $\log^2 n$, we have that $f_2 + \psi_n \subset \Lambda_{\operatorname{sm}, 1}$. Then letting $h_i$ be the height function for $D_i$, in the above coupling we have
		\begin{equation}
			\label{E:diff-c}
			\begin{split}
				\cH^{\psi_n}_n(f_2) - \cH^{\psi_n}_n (f_1) &=  h_{2}^{\psi_n}(f_2) - h_{2}^{\psi_n}(f_1), \\
				\cH^{\psi_n}_n(f_1) - \cH^{\zeta_n}_n (0,0) &=  h_{1}^{\psi_n}(f_2) - h_{1}^{\zeta_n}(0,0).
			\end{split}
		\end{equation}
		Here $\zeta_n$ is the set of all faces with $\|f\|_\infty \le n^{3/4}$.
		Then arguing as in \eqref{E:hnpsin}, all four terms on right-hand side above tend to $0$ with $n$ in probability. To handle the averaging in $h_{1}^{\zeta_n}(0,0)$, which also includes faces that are not $a$- or $b$-faces, we use the reflection symmetry $\E h_1(x, y) = -\E h_1(1-x, 1 + y)$ from Corollary \ref{C:reflection-symmetry-smooth} to get that $\E h_{1}^{\zeta_n}(0,0) = O(n^{-3/4})$. Therefore putting together the above discussion and the displays \eqref{E:successful-coupling}, \eqref{E:diff-c}, we have that
		$$
			\tilde \cH_n^{\psi_n}(f_2) - \cH^{\zeta_n}_n (0,0) = \tilde \cH_n^{\psi_n}(t_{j, n}, y_n) - \cH^{\zeta_n}_n (0,0) \cvgp 0.
		$$
		To complete the proof, observe that the height difference $\cH^{\zeta_n}_n (0,0) - h^{\zeta_n}_n (0,0)$ is always an integer. Since $h^{\zeta_n}_n (0,0)$ converges to $0$ in probability, this implies that $\cH^{\zeta_n}_n (0,0) - H_n$ also converges to $0$ in probability, since $H_n$ is the nearest integer to $\cH^{\zeta_n}_n (0,0)$. Putting this together with the previous display gives \eqref{E:cvg0-in-prob}.
	\end{proof}
	
	\cref{T:main-1} is almost immediate from the previous few lemmas, and a minor change of coordinates. We restate a stronger version of that theorem here, together with the variants discussed in Remark \ref{R:main-thm-variants}. Note that the upcoming \cref{T:main-1-restatement} (along with all statements in the body of the paper) is stated in terms of coordinates centered around the rough-smooth limit curve whereas \cref{T:main-1} is set up using the second order Taylor approximation to the limit curve. The two versions are equivalent since the limit curve is smooth in a neighbourhood $\xi_0(0)$ and the lower order Taylor terms are negligible in our scaling. 
    
    First, define the scaling transformation $\gamma_n:\R^2 \to \R^2$ given by
	\begin{equation}
		\label{E:unscaled-gamma}
\gamma_n(t, x) = n \xi_0(t/n) + (x, 0),
	\end{equation}
	and let
	$$
	\hat \gamma_n(t, x) = [\gamma_n(2^{1/3}\lambda_2 n^{2/3} t, 2^{5/3}\lambda_1 n^{1/3} x)]_a - (1, 1).
	$$
	\begin{thm}
		\label{T:main-1-restatement} Consider $k$ (potentially equal) points $u_1, \dots, u_k \in \R^2$, and sequences $u_{i, n} \to u_i, n \to \infty$. Then for any mollifying sequence $\phi_n \subset (2 \Z)^2$ satisfying $|\phi_n| \to \infty$ as $n \to \infty$ and $\max_{x \in \phi_n} \|x\|_\infty = o(n^{1/24})$, we have that
		\begin{equation}
			\label{E:height-mollify-main}
				(\cH_n^{\phi_n} (\hat \gamma(u_{i, n})) - H_n : i \in \II{1, k})
				\implies (- 4 \cA(u_i): i \in \II{1, k}).
		\end{equation}
		Moreover, under the additional assumption that
		$$
		\lim_{n \to \infty} \inf_{1 \le i < j \le k} \|\hat \gamma_n(u_{i, n}) - \hat \gamma_n(u_{j, n})\|_\infty = \infty,
		$$
		we have the convergence
		\begin{equation}
			\label{E:height-bonus-main}
			\begin{split}
				&(\cH_n (\hat \gamma_n(u_{n, i}) + v) - H_n : i \in \II{1, k}, v \in \hat \Z^2) \\
				\implies &(\cH^i(v) - 4 \cA(u_i): i \in \II{1, k}, v \in \Z^2),
			\end{split}
		\end{equation}
		where $\cA$ is the Airy surface, and the random fields $\cH^i, i \in \II{1, k}$ are i.i.d. copies of the height function for $\P_a$, independent of $\cA$. 
	\end{thm}
	
	\begin{proof}
	If we replace $\hat \gamma_n$ with $\tilde \beta_n$, then \cref{T:main-1-restatement} is \cref{L:arbitrary-mollifiers-new} together with \cref{L:height-shift}. There are minor notational differences between the statements. To move from the coordinates $\tilde \beta_n$ to $\hat \gamma_n$, we can simply observe that
	$
	\hat \gamma_n(t, x) = \tilde \beta_n(\zeta_n(t, x)),
	$
	where the change-of-coordinate function $\zeta_n:\R^2 \to \R^2$ satisfies 
	$$
	\max_{u \in K} \|\zeta_n(u) - u \|_\infty = O (n^{-1/3})
	$$
	for every compact set $K$. Since the theorem statement for $\tilde \beta_n$ allows for arbitrary sequences converging to the points $u_i$, pre-composing $\tilde \beta_n$ with $\zeta_n$ does not affect convergence statements.
	\end{proof}

	\section{Global path control}
	\label{S:global-control}
	
	In this section, we establish global control of backbone path locations, culminating in a stronger version of \cref{T:main-3}. In particular, we will show that there is a \textit{last path} passing through the rough-smooth boundary, which will be our candidate path for convergence to the Airy process. We start by recalling some definitions from earlier sections and stating a slightly more detailed version of \cref{T:main-3}.
	
	Let $D \sim \P_{a, n}$ be a sample from the two-periodic Aztec diamond with parameter $a$, let $\cH_n$ be its height function, let $F_\tt{S}(D), F_\tt{N}(D)$ be the south and north Temperleyan forests associated to $D$ and let $\cN_i^\pm, \cS_i^\pm. i = 1, \dots, n/2$ be the north and south backbone paths in these forests. Here recall that the paths $\cS_i^-, \cS_i^+$ start at the source vertices $v^S_i, v^N_i$ respectively, and the north backbone paths $\cN_i^-, \cN_i^+$ start at the boundary vertices $w^S_i, w^N_i$, respectively. We think of these paths as plane curves, formed by connecting adjacent vertices with line segments, and parametrizing the resulting paths by arc-length, starting at the source vertex. For example, we will write $\cS_i^-:[0, t_i^-] \to \R^2$ where $t_i^-$ is the length of $\cS_i^-$, $\cS_i^-(0) = v^S_i$, and $\cS_i^-(t_i^-) \in \partial_{E} \tt{S} \cup \partial_{W} \tt{S}$.
	
	Since the south and north paths cannot cross each other, these backbone paths interlace, with the order of the interlacing determined by the start point order, as in \cref{L:dimer-compatible-duals}. For example, moving left-to-right, we see the following backbone paths starting on the south boundary:
	$$
	\cS_1^- \prec \cN_1^- \prec \cS_2^- \prec \cdots \prec \cS_{n/2}^- \prec \cN_{n/2}^-.
	$$
	We let $I = I_n \in \{0, \dots, n/2\}$ be the split point described in \cref{L:tree-paths}, so that for $j \le I$, the path $\cS_j^-$ ends on the west boundary $\partial_W \tt{S}$, and for $j > I$, the path $\cS_j^-$ ends on the east boundary $\partial_E \tt{S}$. By the interlacing above (alternately, by \cref{L:tree-paths}), the paths $\cN_i^-, i < I$ end on the west boundary and the paths $\cN_i^-, i > I$ end on the east boundary. The path $\cN_I^-$ may end on the east or west boundary.
	
	Next, recall that the transformation $\beta_n$ maps the point $n \xi_0(0)$ where the rough-smooth boundary intersects the line $\{x = y\}$ to $(0,0)$. This transformation also flatten outs and rotates the rough-smooth boundary (without scaling) so that it becomes a horizontal line. Define the regions
	\begin{align*}
		\tt{RS}_n^* &:= \beta_n([-2n^{3/4}, 2n^{3/4}] \times [- n^{1/2} \log^{3/2} n, 2n^{1/3} \log^2 n]), \\
		\tt{PRS}_n^* &:= \beta_n([-2n^{3/4}, 2n^{3/4}] \times [- n^{1/2} \log^{2} n, n^{3/4}]),
	\end{align*}
	which surround the rough-smooth boundary, together with the set
	and the set $\tt{X} = \{{\bf v} \in \R^2 : v_1 = \pm v_2\}$.
	
	Finally, recall the event $\mathsf{Para}_\tt{S}(\alpha)$ from \cref{L:regular-paths}, which constrains south forest paths to parabolas away from the south backbone. We define $\mathsf{Para}_\tt{N}(\alpha)$ as the analogous parabolic constraint event for north forest paths (with north-opening parabolas). More precisely, we let $\mathsf{Para}_\tt{N}(\alpha)$ be the set of DCFs $F$ on $\tt{G(N)}$ such that $\rho_N F \in \mathsf{Para}_\tt{S}(\alpha)$, where $\rho_N(x, y) = (-x, -y)$, as in \cref{L:sym}.
	\begin{thm}
		\label{T:main-3-restatement}
		With probability $1 - o(1)$ as $n \to \infty$, the following events hold for $D_n \sim \P_{a, n}$ and the associated Temperleyan forests:
		\begin{enumerate}[label=\arabic*.]
			\item $|H_n| \le \log n$.
			\item The events $\mathsf{Para}_\tt{S}(\log^2 n), \mathsf{Para}_\tt{N}(\log^2 n)$ hold.
			\item The paths $\cS_i^-$ with $I - n^{1/4} \le i \le I$ do not intersect the region $(\tt{X} \cup \tt{PRS}^*_n)\setminus \tt{RS}^*_n$.
			\item The paths $\cN_i^-, i \ge I$ and $\cN_i^+, i \in \II{1, n/2}$ do not enter the region $\tt{PRS}^*_n$.
			\item The paths $\cS_i^-, i \le I - n^{1/4} \log^4 n$ do not intersect the region $\tt{PRS}^*_n$, and the path $\cS_{I}^+$ does not enter the region $\tt{PRS}^*_n$ prior to coalescing with another path.
			\item $H_n = 4I - n - 1$.
		\end{enumerate}		
	\end{thm}
	
	Note that the proof of \cref{T:main-3-restatement} will actually yield bounds on slightly larger regions, which we will define as we move through this section. We start with a simple lemma that converts height estimates from Section \ref{S:compound-estimates} to bounds on path locations. For this lemma, define three overlapping regions:
	\begin{align*}
		\tt{A}_{SW} &:= \beta_n([-n^{5/6}, n^{5/6}] \times [n^{1/3} \log^2 n, n(|\xi_0(0)| + \gamma_a)) \\
		\tt{B}_{SW} &:= \beta_n([-n^{5/6}, n^{5/6}] \times [- n^{1/2} \log^2 n, 2n^{1/3} \log^2 n]) \\
		\tt{C}_{SW} &:= \tt{C1}_{SW} \cup \tt{C2}_{SW}, \quad \text{ where } \\
		\tt{C1}_{SW} &:=
		\beta_n([-n^{5/6}, n^{5/6}] \times [- n^{1/2} \log^{2} n/2, - n^{1/2} \log^{3/2} n]) \\
		\tt{C2}_{SW} &:= \beta_n(\{0\} \times (-\infty, - n^{1/2} \log^2 n]).
	\end{align*}
    Here $\gamma_a$ is as in \cref{L:global-coupling-1}.
	For $S \in \{\tt{A}, \tt{B}, \tt{C}, \tt{C1}, \tt{C2} \}$, we also define
	\begin{align*}
		S_{NW} &= R_{-\pi/2} S_{SW}, \qquad S_{SE} = R_{\pi/2} S_{SW}, \qquad S_{NE} = R_{\pi} S_{SW},\\
		\qquad S_{\times} &= S_{NW} \cup S_{SW} \cup S_{SE} \cup S_{NE}. 
	\end{align*}
	where here recall that $R_\theta$ denotes a counterclockwise rotation by $\theta$. See \cref{fig:lemma52} for an illustration of these regions.

    \begin{figure}
        \centering
        \includegraphics[width=0.7\linewidth]{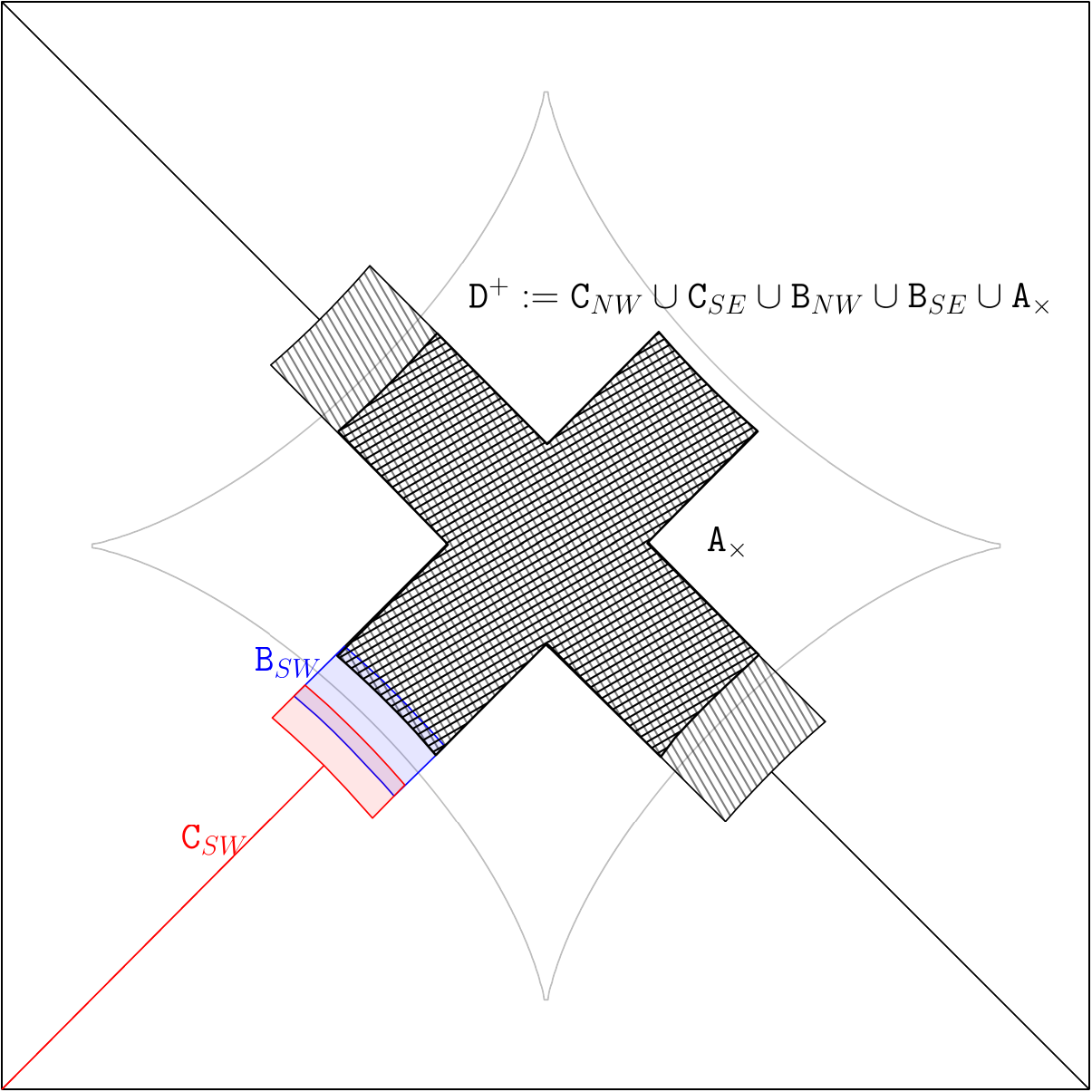}
        \caption{Some of the important forbidden sets in \cref{L:basic-path-ctrl}.}
        \label{fig:lemma52}
    \end{figure}

	\begin{lem}
		\label{L:basic-path-ctrl}
		The following events hold with probability $1 - o(1)$ as $n \to \infty$.
		\begin{enumerate}[label=\arabic*.]
			\item The paths $\cS_i^\pm, \cN_i^\pm, i \ge n/4 + n^{1/4} \log^2 n$ do not enter the region 
			$$
			\tt{D^-} :=\tt{C}_{NE} \cup \tt{C}_{SW} \cup \tt{B}_{NE} \cup \tt{B}_{SW} \cup \tt{A}_\times.
			$$ 
			Similarly, the paths $\cS_i^\pm, \cN_i^\pm, i \le n/4 - n^{1/4} \log^2 n$ do not enter $\tt{D}^+ := \tt{C}_{NW} \cup \tt{C}_{SE} \cup \tt{B}_{NW} \cup \tt{B}_{SE} \cup \tt{A}_\times$. 
			\item The paths $\cS_i^\pm, \cN_i^\pm, |i - n/4| \le 2 n^{1/4} \log^2 n$ do not enter the set $\tt{C}_\times$.
			\item The paths $\cS_i^\pm, \cN_i^\pm, |i - n/4| \ge  n^{1/4} \log^4 n/2$ do not enter the set $\tt{A}_\times \cup \tt{B}_\times$.
			\item The split point $I_n$ satisfies $I_n - n/4 \in [-n^{1/4} \log^2 n-1, n^{1/4} \log^2 n]$. 
			
		\end{enumerate} 
	\end{lem}
	
	\begin{proof}
		We first translate the height bounds from Corollaries \ref{C:expected-sm-height} and \ref{C:rough-phase-heights}. These corollaries imply that with probability $1- o(1)$, we have that
		\begin{itemize}[nosep]
			\item $|\cH_n(v)| \le \log n$ for all $v \in \tt{A}_\times$ by Corollary \ref{C:expected-sm-height}.
			\item $\cH_n(v) \le n^{1/4} \log^2 n$ for all $v \in \tt{B}_{SW} \cup \tt{B}_{NE}$ and $\cH_n(v) \ge -n^{1/4} \log^2 n$ for all $v \in \tt{B}_{SE} \cup \tt{B}_{NW}$ by the first estimate in Corollary \ref{C:rough-phase-heights}.
			\item $|\cH_n(v)| \le n^{1/4} \log^4 n$ for all $v \in \tt{B}_\times$ by the first estimate in Corollary \ref{C:rough-phase-heights}.
			\item $\cH_n(v) \le - 100 n^{1/4} \log^2 n$ for all $v \in \tt{C}_{SW} \cup \tt{C}_{NE}$ and $\cH_n(v) \ge 100 n^{1/4} \log^2 n$ for all $v \in \tt{C}_{SE} \cup \tt{C}_{NW}$.  This uses the first estimate in Corollary \ref{C:rough-phase-heights} on the $\tt{C1}$-pieces and the second estimate on the $\tt{C2}$-pieces.
		\end{itemize}
		On the event where these four bounds hold, we will show that the four points in the lemma hold deterministically.
		
		For part $1$, by Remark \ref{R:labelling}, the paths $\cS_i^\pm, \cN_i^\pm, i \ge n/4 + n^{1/4} \log^2 n$ all start at vertices of height at least $4 n^{1/4} \log^2 n + O(1)$. They can only change their height by winding (\cref{L:tree-path-heights}). On the other hand, the set $\tt{D}^-$ is connected to the boundary of the Aztec diamond, so the first time any of these paths enters this set, its height would have to be within $4$ of its initial height. Therefore by the second and third bullets above, none of these paths can enter $\tt{D}^-$. The second piece of part $1$ follows similarly, as does part $2$.
		
		For part $3$, observe that the sets $\tt{A}_\times \cup \tt{B}_\times \cup \tt{D}^-$ and $\tt{A}_\times \cup \tt{B}_\times \cup \tt{D}^+$ are connected to the boundary of the Aztec diamond, and by the above bullet points, with probability $1- o(1)$ the heights on these sets are bounded above by $n^{1/4} \log^4 n/100$ (for the first set) and below by $-n^{1/4} \log^4 n/100$ (for the second set). Therefore by a similar argument as for parts $1$ and $2$, any path whose initial height has absolute value at least $n^{1/4} \log^4 n/99$ cannot enter at least one of these sets, and so it cannot enter $\tt{A}_\times \cup \tt{B}_\times$. Part $3$ then follows by translating path indices to heights via Remark \ref{R:labelling}.
		
		For part $4$, observe that $\tt{D}^-$ cuts off the south and west boundaries of the Aztec diamond. Therefore by second part $1$, we have $I_n - n/4 \le n^{1/4} \log^2 n$. Similarly, the set $\tt{D}^+$ cuts off the south and east boundaries of the Aztec diamond, and so $I_n - n/4 \ge -n^{1/4} \log^2 n - 1$.
	\end{proof}
	
	To move beyond the height restrictions in \cref{L:basic-path-ctrl} to get the full suite of restrictions in \cref{T:main-3-restatement}, we will need to control the possibility that paths backtrack or create spirals, which will change their height via \cref{L:tree-path-heights}. This is a particular challenge when paths do this in a nested fashion, creating a kind of \emph{onion}. We can control the presence of onions by using the parabolic path estimates from the $\mathsf{Para}$ events defined in \cref{L:regular-paths}. The next lemma gives a precise statement.
	
	For this lemma, for $r_1, r_2 > 0$ and $\mathbf{v} \in \R^2$ define
	\begin{align*}
		B_{r_1, r_2}(\mathbf{v}) &= \mathbf{v} + [-r_1, r_1] \times [-r_2, r_2], \qquad \\
		\partial^\pm B_{r_1, r_2}(\mathbf{v}) &= \mathbf{v} + \{ \pm r_1\} \times [-r_2, r_2].
	\end{align*}
	\begin{lem}
		\label{L:onion-control}
		For $r_1, r_2 \ge 1$ and $\ell \in \N$ let 
		$$
		\mathsf{Onion}(r_1, r_2, \ell)
		$$
		be the event where there exists a path $\cC_{i_0}^*:[0, t] \to \R$ for some $i_0 \in \{1, \dots, n/2\}$, $* \in \{+, -\}, \mathcal C \in \{\cN, \cS\}$ and path segments $\pi_1 = \cC_{i_0}^*|_{[s_1, t_1]}, \pi_2 = \cC_{i_0}^*|_{[s_2, t_2]}$ for some $t_1 \le s_2$ satisfying the following two conditions:
		\begin{itemize}
			\item There exists $\mathbf{v} \in \R^2$ such that $\pi_1, \pi_2 \subset B_{r_1, r_2}(\mathbf{v})$.
			\item The points $\pi_1(s_1), \pi_2(t_2)$ sit on the same boundary $\partial^\pm B_{r_1, r_2}(\mathbf{v})$ and the points $\pi_1(s_2), \pi_2(t_1)$ sit on the opposite boundary $\partial^\mp B_{r, R}(\mathbf{v})$.
			\item Let 
			$
			M
			\subset B_{r_1, r_2}(\mathbf{v})$ be the union of $\pi_1, \pi_2$ and the region in $B_{r_1, r_2}(\mathbf{v})$ bounded between these two paths. Then there are at most $\ell$ paths of the form $\cS_j^\pm$ or $\cN_j^\pm$ that enter $M$.
		\end{itemize}
		Informally, the event $\mathsf{Onion}(r_1, r_2, \ell)$ says that there exists an \textbf{onion} of width $r_1$, height $r_2$, and $\ell$ layers.
		Then assuming that $r_2 \ge 1$ and $r_1 \ge 10 \ell \sqrt{r_2} \log^2 n$, we have 
		\begin{equation}
			\label{E:Onion-r1-empty}
			\mathsf{Onion}(r_1, r_2, \ell) \cap \mathsf{Para}_\tt{S}(\log^2 n) \cap \mathsf{Para}_\tt{N}(\log^2 n) = \emptyset.
		\end{equation}
	\end{lem}

    See \cref{fig:onion} for an illustration of \cref{L:onion-control}, and a sketch of the proof.

    \begin{figure}
        \centering
        \includegraphics[width=1\linewidth]{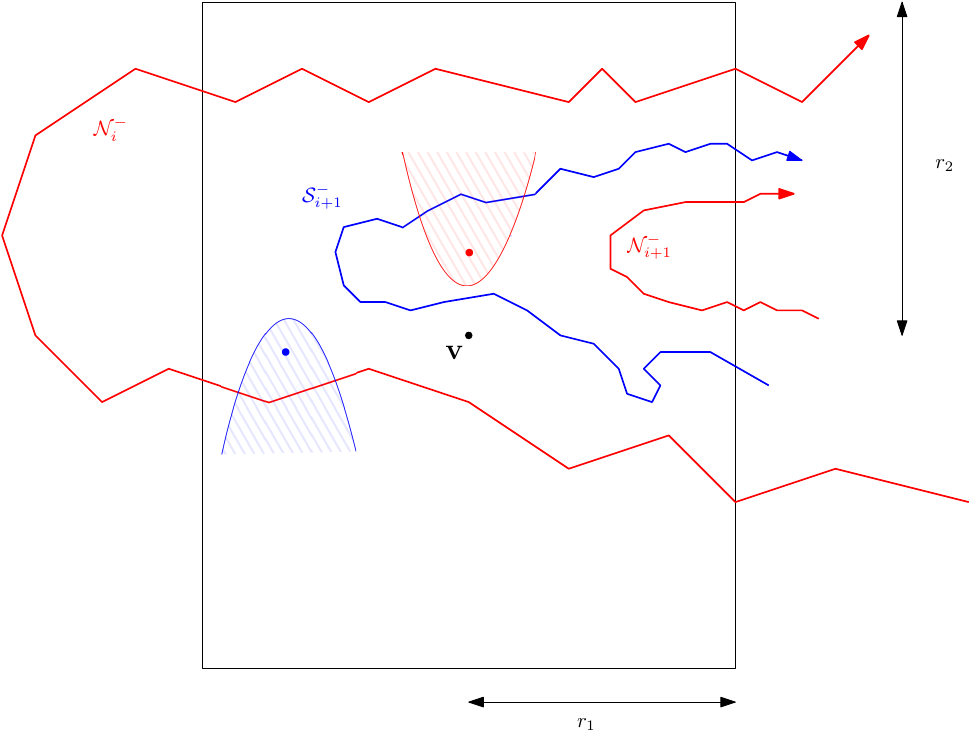}
        \caption{An example of the event $\mathsf{Onion}(r_1, r_2, 3)$. If $\ell, r_2$ are too small relative to the width $r_1$, then in any onion configuration, we can find a north or south vertex $w$ from which loop-erased random walk paths can only join the backbone by exiting the parabolic set $w + P^{\log^2 n}_\cdot$. Two such vertices (one north vertex, one south vertex), are shown in the above figure.}
        \label{fig:onion}
    \end{figure}

	\begin{proof}
		Throughout the proof, we assume we are on the event defined by the left-hand side of \eqref{E:Onion-r1-empty}. Working from this starting point, we will derive the inequality between $r_1, r_2$, and $\ell$.
		
		Without loss of generality, we assume that the path $\cC_{i_0}^*$ realizing the event $\mathsf{Onion}(r_1, r_2, \ell)$ is of the form $\cN_{i_0}^-:[0, t] \to \R$, and that $\pi_1(s_1), \pi_2(t_2) \in \partial^+ B_{r_1, r_2}(\mathbf{v})$. The other cases have symmetric arguments. We may also assume $r_1, r_2 \le n$ since the Aztec diamond has height and width $2n$. This is the setting of \cref{fig:onion}.
		
		We proceed by induction on $\ell$, starting with the base case when $\ell = 1$. In this case, the only $\cS$- or $\cN$-path that enters the sets $M$ is the path $\cN_{i_0}^-$. In particular, no $\cN$-path enters the set $M$. Now consider a vertex $\mathbf{w} \in \tt{S}$ that lies in the set $M$, and satisfies $|w_1 - v_1| \le 2$. The south forest path starting from $\mathbf{w}$ must exit $M$ at a vertex $\mathbf{u}$ on the boundary set $\mathbf{v} + \{-r_1, r_1\} \times [-r_2, r_2]$. On the other hand, since we are on the event $\mathsf{Para}_\tt{S}(\log^2 n)$,  the exit point must satisfy
		$$
		\mathbf{u} \in \mathbf{w} + P^{\log^2 n}_\tt{S},
		$$
		which implies that
		$$
		r_1 \le |u_1-w_1| + 2 \le (\log^2 n + \log(|u_2 - w_2| + 1))\sqrt{|u_2 - w_2| + 1} + 2.
		$$
		Using the bounds $|u_2 - w_2| \le 2 r_2$, $1 \le r_2 \le n$, and simplifying gives
		\begin{equation}
			\label{E:strong-ell-1}
			r_1 < 10 \log^2 n \sqrt{r_2},
		\end{equation}
		which implies the lemma for $\ell = 1$.

		Now, let $r_2 \ge \ell \ge 2$ and assume that the lemma holds for $\ell - 1$. Consider the path $\cN_{i_0}^-:[0, t] \to \R$. Let $[a_1, b_1], [a_2, b_2], \dots, [a_k, b_k] \subset [0, t]$ be an enumeration of all the intervals for which 
		$$
		\cN_{i_0}^-|_{[a_i, b_i]} \subset B_{r_1, r_2}(\mathbf{v}),
		$$
		and the points $\cN_{i_0}^-(a_i), \cN_{i_0}^-(b_i)$ lie on opposite boundaries $\partial^\pm B_{r_1, r_2}(\mathbf{v})$ and $\partial^\mp B_{r_1, r_2}(\mathbf{v})$. In other words, we have enumerated all left-to-right and right-to-left crossings of $[-r_1, r_1] \times [-r_2, r_2]$. Since we know $\cS_{i_0}^-$ has at least one left-to-right crossing and one right-to-left crossing of this box, we can find an index $i \in \{1, \dots, k-1\}$ such that the intervals $[a_i, b_i]$ and $[a_{i+1}, b_{i+1}]$ represent opposite types of crossings. We may take these two intervals as the intervals $[s_1, t_1]$ and $[s_2, t_2]$ realizing the event $\tt{Onion}(r_1, r_2, \ell)$.
		
		Now, since the crossings are consecutive and we have assumed that $
		\pi_1(s_1), \pi_2(t_2) \in \partial^+ B_{r_1, r_2}(\mathbf{v})$, each of the $\ell - 1$ other $(\cS \cup \cN)$-paths that enter the region $M$ must enter through the boundary $\partial^+ B_{r_1, r_2}(\mathbf{v})$, as in \cref{fig:onion}. At this point, there are two cases:
		\begin{itemize}[nosep]
			\item At least one of these paths entering $M$ also enters the smaller box 
			$$
			B := \mathbf{v} + (r_1 - 20 \log^2 n \sqrt{r_2}, r_1] \times [-r_2, r_2].
			$$ In this case, this path crosses from left-to-right and then back from right-to-left across the set $M \setminus B$. Let $M' \subset M \setminus B$ be the region between these crossings. By the ordering/disjointness of the north and south backbone paths, $\cS_{i_0}^-$ does not enter $M'$. Therefore the set $M'$ realizes the event $\tt{Onion}(r_1 - 10 \log^2 n \sqrt{r_2}, r_2, \ell - 1)$, which implies the inequality in the lemma by the inductive hypothesis.
			\item No paths enter the set $M \cap B$. In this case, the set $M \cap B$ and the path $\cS_{i_0}^-$ realizes the event $\tt{Onion}(10 \log^2 n \sqrt{r_2} -\eps, r_2, 1)$ for every $\eps > 0$, contradicting \eqref{E:strong-ell-1}. \qedhere
		\end{itemize}
	\end{proof}
	
	We can combine the previous two lemmas to upgrade the path control in Part (1) of \cref{L:basic-path-ctrl} to include all paths above (or below) the split point $I = I_n$. This will imply Part (3) of \cref{T:main-3-restatement}. The caveat is that we need to slightly shrink the sets $\tt{D}^\pm$. For this next lemma, define
	\begin{align*}
		\tt{D}^+_0 &= \{v  \in \tt{D}^+ : |v_1 + v_2| \wedge |v_1 - v_2| \le n^{5/6}/2\} \setminus (\tt{B}_{SW} \cup \tt{B}_{NE}), \\
		\tt{D}^-_0 &= \{v  \in \tt{D}^- : |v_1 + v_2| \wedge |v_1 - v_2| \le n^{5/6}/2\} \setminus (\tt{B}_{SE} \cup \tt{B}_{NW}).
	\end{align*}
	\begin{lem}
		\label{L:QSE-pass}
		With probability $1 - o(1)$, the paths $\cS_i^-, i \le I$ do not enter $\tt{D}^+_0$. Similarly, the paths $\cS_i^-, i \ge I + 1, \cN_i^+, i \ge I + 1$ do not enter $\tt{D}^-_0$, and the paths $\cN_i^+, i \le I$ do not enter $\tt{D}^+_0$.
	\end{lem}
	
	\begin{proof}
		For the proof, we first define an event $E$ of probability $1 - o(1)$, and then show that the lemma holds deterministically on $E$. We will only verify the lemma for the paths $\cS_i^-, i \le I$, as the claim for the other three sets of paths follows by a symmetric argument.

		First, let $E_1$ be the probability $(1 - o(1))$-event where the four points in \cref{L:basic-path-ctrl} hold. Next, let $E_2$ be the event where for any path segment $\pi$ in the south forest $F_\tt{S}(D)$ started at a vertex $v$ and with $\pi \subset \tt{A}_\times$, we have
		\begin{equation}
			\label{E:pi-P}
			\pi \subset P_\tt{S}^{\log^2 n},
		\end{equation}
		and similarly for any path segment $\gamma$ in the north forest $F_\tt{N}(D)$ started at a vertex $v$ and with $\gamma \subset \tt{A}_\times$, we have	$\gamma \subset P_\tt{N}^{\log^2 n}$. The event $E_2$ has probability $1 - o(1)$ in the smooth phase by the standard estimate in \cref{L:lerw-estimate-basic} and a union bound. Therefore it also has probability $1 - o(1)$ in the Aztec diamond by the smooth phase couplings, \cref{L:global-coupling-1}, \ref{L:global-coupling-2}.
		We now define
		$$
		E = E_1 \cap E_2 \cap \mathsf{Para}_\tt{S}(\log^2 n) \cap \mathsf{Para}_\tt{N}(\log^2 n),
		$$
		which has probability $1 - o(1)$ by \cref{L:regular-paths},  \cref{L:basic-path-ctrl}, and the above discussion regarding $E_2$. In the remainder of the proof, we work on $E$. We will only prove the claim for $\tt{D}^+_0$, as the proof for $\tt{D}^-_0$ is symmetric. 	For the sake of contradiction, suppose that one of the paths $\cS_i^-, i \le I$ enters the set $\tt{D}^+_0$. By path ordering, the rightmost path $\cS_I^-$ must then also enter $\tt{D}^+_0$. 

          \begin{figure}
            \centering
            \includegraphics[width=0.7\linewidth]{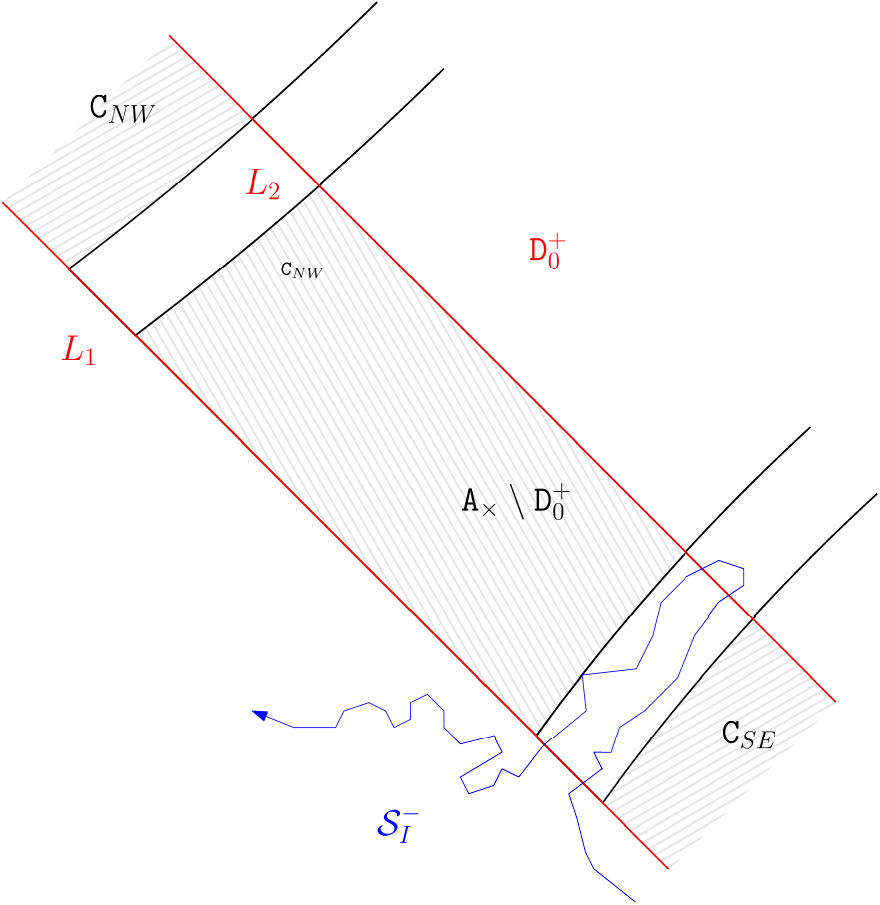}
            \caption{An illustration of the first part of the proof of \cref{L:QSE-pass}. For the the path $\tt{S}_I^-$ to cross over the line $L_2$, it must either cross through the sets $\tt{A}_\times$, or else use one of the corridors between $\tt{A}_\times$ and $\tt{C}_{SE}$ or $\tt{C}_{NW}$. The former is essentially rules out since south paths $\tt{A}_\times$  have a strong southward drift. The latter event is illustrated above, and results in an unfavourable $\tt{Onion}$ event.}
            \label{fig:lemma541}
        \end{figure}
        We first split $\tt{D}^+_0$ into two pieces,
		$$
		\tt{D}_1 = \{v \in \tt{D}^+_0 : v_1 + v_2 \ge - n^{5/6}/2\}, \qquad \tt{D}_2 = \tt{D}^+_0 \setminus \tt{D}_1.
		$$ 
		We first rule out that the paths $\cS_I^-$ enter $\tt{D}_1$. The argument is sketched in \cref{fig:lemma541}. Setting up some notation, define the lines
		$$
		L_i = \{(x, y) : x + y = - n^{5/6}/i \}, \quad i = 1, 2.
		$$
		Let $R$ denote the region bounded between the lines $L_1$ and $L_2$, and the sets $\tt{C}_{SE}, \tt{C}_{NW}$.
		
		For the sake of contradiction, suppose that one of the paths the paths $\cS_i^-, i \le I$ enters the set $\tt{E}_1$. By path ordering, the rightmost path $\cS_I^-$ must then also enter $\tt{E}_1$.	By Part (4) of \cref{L:basic-path-ctrl}, we have that $|I - n/4| \le n^{1/4} \log^2 n + 1$. Therefore by Part (2) of \cref{L:basic-path-ctrl}, the path $\cS_I^-$ cannot enter the sets $\tt{C}_{SE} \cup \tt{C}_{NW}$,	so the only way $\cS_I^-$ can enter $\tt{D}_1$ is if it crosses through the region $R$, first crossing the line $L_1$ and then later crossing the line $L_2$. Let $\pi_1 = \cS_I^-|_{[s, t]} \subset R$ denote the segment of $\cS_I^-$ given by taking the \textit{first} crossing of $R$ from $L_1$ to $L_2$. Let $v = \cS_I^-(s), w = \cS_I^-(t)$ denote the first and last points on $\pi_1$. We claim that
		\begin{equation}
			\label{E:n5610}
			w_1 - v_1 \ge n^{5/6}/100.
		\end{equation}
		Suppose that \eqref{E:n5610} fails. Then $w_2 - v_2 \ge 0.4 n^{5/6}$, since $L_1$ and $L_2$ are line segments parallel to the line $y = - x$, vertically shifted by $n^{5/6}/2$. On the other hand, the set $\tt{B}_{NW} \cap R$ is a thin strip contained in a set of the form
		$$
		\{(x, y) + v_{NW} : |x - y| \le n^{3/4}, |x+y| \le n^{5/6}/2\}
		$$
		for some $v_{NW} \in \R^2$. The same holds true of $\tt{B}_{SE} \cap R$ (with the vertex $v_{NW}$ replaced by a vertex $v_{SE}$). In particular, for $n$ large enough we have the estimate
		\begin{equation}
			\label{E:Cap-spread}
			\max \{|a_2 - a_2'| : a, a' \in \tt{B}_*\cap R \} < 0.39 n^{5/6}\}, \quad  * \in \{\text{NE}, \text{SW}\}.
		\end{equation}
		The negation of \eqref{E:n5610} together with \eqref{E:Cap-spread} implies that there exists a path $\pi' \subset \pi_1 \cap \tt{A}_\times$ starting and ending at vertices $a, a'$ and with $a'_2 - a_2 \ge n^{5/6}/100$. This contradicts the parabolic path estimate \eqref{E:pi-P}, and so \eqref{E:n5610} must hold. 

        \begin{figure}
            \centering
    \includegraphics[width=0.6\linewidth]{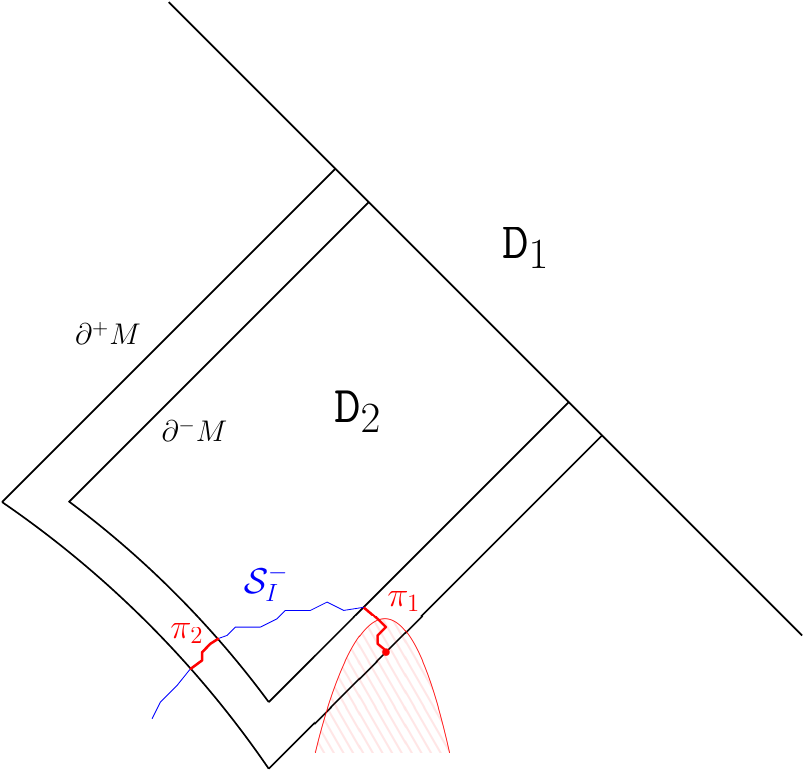}
            \caption{The setup of the second part of the proof of \cref{L:QSE-pass}. Given that the path $\mathcal S_I^-$ does not enter $\tt{D}_1$, for it to enter $\tt{D}_2$ it must cross through the boundary strip $M$ in a way that contradicts a parabolic drift estimate. One example of this is given above.}
            \label{fig:lemma542}
        \end{figure}
		
		Now, since the path $\cS_I^-$ started on the south side of the Aztec diamond and finishes on the west side, there must exist a path $\pi_2 \subset \cS_I^-$ contained in the subset $S'$ of the strip $[v_1, w_1] \times \R$ which is bounded below by $\pi_1$:
		$$
		S' = \{(x, y) \in [v_1, w_1] \times \R : \text{ There exists } x \in \pi_1, v_2 \le y \},
		$$
		and which crosses $S'$ from right to left. Let $S''$ be the region in $S'$ bounded between $\pi_1$ and $\pi_2$. 
		
		By the ordering on paths, the only paths of the form $\cS^\pm_i$ that can enter the region $S''$ are the paths $\cS^-_i, i \le I$. Additionally, only those paths with index $i \ge n/4 - n^{1/4} \log^2 n$ can enter $S''$.	Indeed, the region $S''$ is contained in the union of the set $\tt{D}^+$ together with the component of the Aztec diamond that is connected to the north boundary. Therefore any path $\cS^-_i, i \le I$ that enters $S'$ must enter the region $\tt{D}^+$, and therefore must have index $i \ge n/4 - n^{1/4} \log^2 n$ by Part (1) of \cref{L:basic-path-ctrl}. 
		
		In summary, we have constructed a region in a strip $[v_1, w_1] \times \R$ of width at least $n^{5/6}/100$, bounded between two $\cS_i^-$ paths crossing from left to right and right to left, which contains at most $2 n^{1/4} \log^2 n \le 2 n^{1/4} \log^n$ $\cS_i$-paths. By $\cS$- and $\cN$-path interlacing, this region also contains at most $3n^{1/4} \log^2 n$ $\cN$-paths. That is, we are on the onion event
		$
		\mathsf{Onion}(n, n^{5/6}/100, 5n^{1/4}\log^2 n).
		$
		By \cref{L:onion-control}, this event cannot hold simultaneously with $E$. Having arrived at a contradiction, we can conclude that $\cS_I^-$ does not enter the set $\tt{D}_1$.
		
		To finish the proof, we show that $\cS_I^-$ does not enter $
        \tt{D}_2 = \tt{D}^+_0 \setminus \tt{D}_1$. The argument is illustrated in \cref{fig:lemma542}. Let $M$ be the closure of the set
		$$
		\{v \in \R^2 \setminus \tt{D}^+ : d(v, \tt{D}_2) \le n^{1/4}\}.
		$$
		Note that $M \subset \tt{A}_\times$, since when defining $\tt{D}_0^+$ we removed a boundary strip around the perimeter of $\tt{A}_\times$ in the region containing $\tt{D}_2$. Let $\partial^- M$ be the portion of the boundary of $M$ contained in $\tt{D}_2$, and let $\partial^+ M$ be the portion of the boundary not contained in $\tt{D}_0^-$.
		
		Since $\cS_I^-$ does not hit any vertex in $\tt{D}_1$, for $\cS_I^-$ to enter $\tt{D}_2$, there must be two disjoint path segments $\pi_1 = \cS_I^-|_{[s_1, t_1]}, \pi_2 = \cS_I^-|_{[s_2, t_2]}$ with $t_1 \le s_2$ contained in $M$, such that:
		\begin{itemize}
			\item $\cS_I^-(s_1), \cS_I^-(t_2) \in \partial^+ M$ and $\cS_I^-(s_2), \cS_I^-(t_1) \in \partial^- M$.
			\item The vertex $\cS_I^-(t_2)$ has a smaller polar coordinate in $(0, 2 \pi)$ than $\cS_I^-(s_1)$. We can guarantee this path ordering by choosing $[s_1, t_1]$ minimally, $[s_2, t_2]$ maximally, and using that $\cS_I^-$ moves from the south boundary to the west boundary of the Aztec diamond.
		\end{itemize}
		We will use these two points to derive a contradiction,  repeatedly using that $M \subset \tt{A}_\times$, so the estimate \eqref{E:pi-P} is in effect.
		
		First, by \eqref{E:pi-P} we must have $\cS_I^-(t_1) \in P_\tt{S}^{\log^2 n + 4} + \cS_I^-(s_1)$ (the $+ 4$ here is to accommodate the fact that the points $\cS_I^-(s_i), \cS_I^-(t_i)$ may not be vertices in $\tt{S}$). This forces the vertex $\cS_I^-(s_1)$ to lie on the portion of $\partial^{+} M$ where $y = x + n^{5/6}/2 + n^{1/4}$; call this part of the boundary $\partial^{++} M$. By the ordering in the second bullet, the point $\cS_I^-(t_2)$ must also lie on $\partial^{++} M$. However, again appealing to condition \eqref{E:pi-P}, we have that $\cS_I^-(t_2) \in P_\tt{S}^{\log^2 n + 4} + \cS_I^-(s_2)$. Since $\cS_I^-(s_2) \in \partial^- M$, this is a contradiction.
	\end{proof}
	
	By \cref{L:QSE-pass}, every path $\cS_i^\pm, \cN_i^\pm$ is forbidden from entering either $\tt{D}_0^-$ or $\tt{D}_0^+$, with the exception of the two paths $\cS_I^+, \cN_I^-$. These are the free paths which appear in \cref{fig:800a07}. While they will typically cross through the smooth region, we can still control the possibility that they approach the rough-smooth boundary. This is the goal of the next lemma, which will imply Part (4) of \cref{T:main-3-restatement}. For this next lemma, we define one last region, $\Theta$, given as follows. First, let $L$ denote the vertical line containing the point $\beta_n(n^{5/6}/5, -n^{1/2}\log^2 n)$, and let $L_W \subset L$ be the largest segment of $L$ whose endpoints are contained in $\tt{C}_{SW}, \tt{C}_{NW}$. Let $L_E = R_\pi L_W$, and let $\Theta$ be the connected component of 
	$$
	[-n - 1, n + 1]^2 \setminus (\tt{C}_\times \cup L_W \cup L_E)
	$$
	containing the origin, see \cref{fig:theta}.

    \begin{figure}
        \centering
\includegraphics[width=0.5\linewidth]{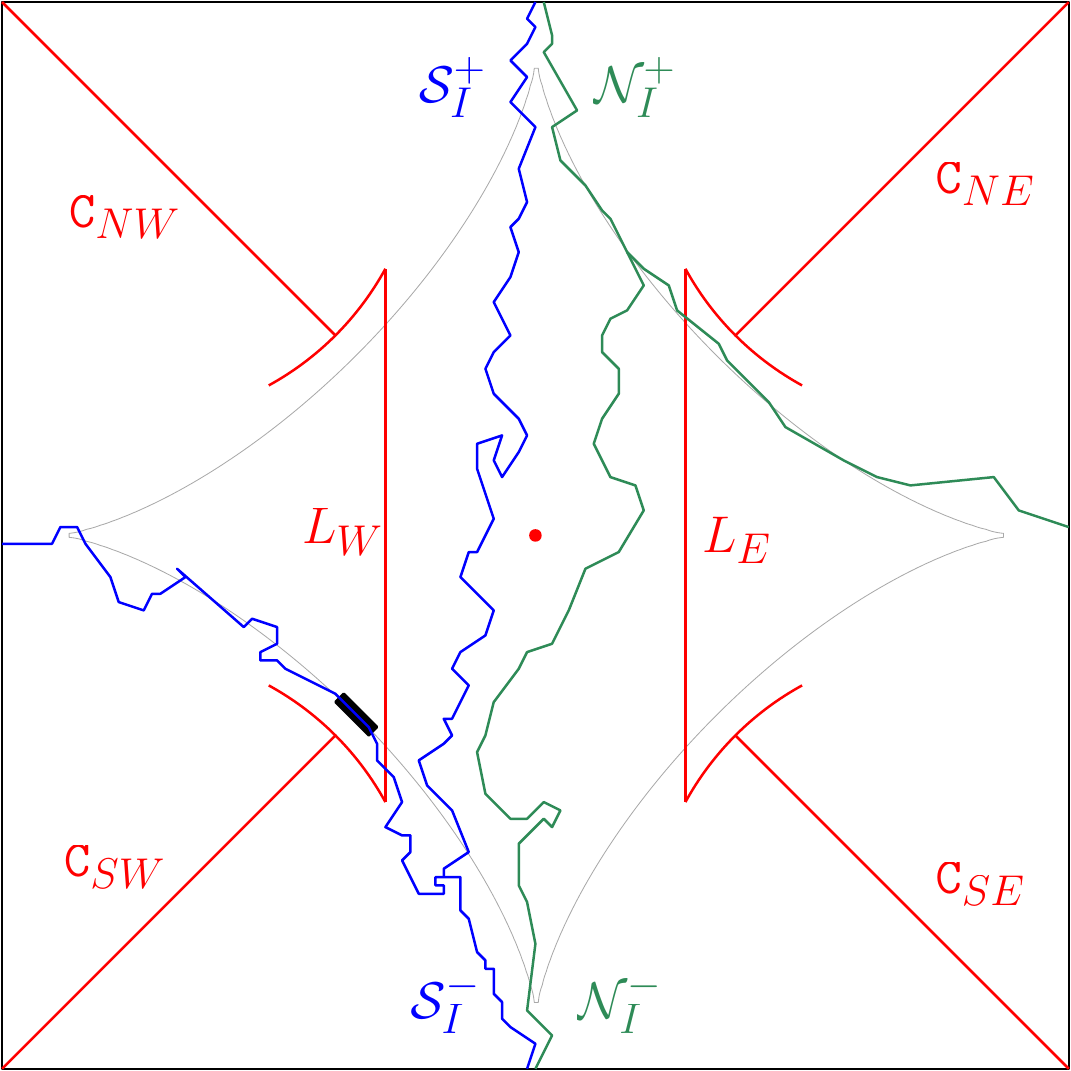}
        \caption{A sketch of \cref{L:SN-path-control}. The region containing the origin is $\Theta$. It is well-separated by the red curves from the black region, where we will eventually study Airy fluctuation.}
        \label{fig:theta}
    \end{figure}
	\begin{lem}
		\label{L:SN-path-control}
		The following holds with probability $1 - o(1)$ as $n \to \infty$.
		Either $\cS_I^+$ coalesces with $\cS_I^-$ and $\cN_I^-$ coalesces with $\cN_I^+$, or else $\cS_I^+$ coalesces with $\cS_{I+1}^-$ and $\cN_I^-$ coalesces with $\cN_{I+1}^+$. Prior to the coalescence points, the paths $\cS_I^+,\cN_I^-$ are contained in the set $\Theta$.
	\end{lem}
	
	\begin{proof}
		The coalescence claims are deterministic given the definition of the split point $I$ by \cref{L:dimer-compatible-duals}, and from this point forward in the proof we write $\cS_I^+, \cN_I^-$ for just the portion of these paths prior to the coalescence point with another $\cS$- or $\cN$-path. We must check that the paths $\cS_I^+, \cN_I^-$ are contained in the set $\Theta$. We may assume that $\cS_I^+$ eventually coalesces with $\cS_I^-$ and $\cN_I^-$ coalesces with $\cN_I^+$, as the other case follows a symmetric argument. Throughout, we work on the high probability event $E$ defined in the proof of \cref{L:QSE-pass}. The lemma in this case is deterministic on this event.
		
		For $i = 3, 4, 5$, let $L_i$ be the vertical line that crosses the boundary $\partial^-\tt{Cap}_{SW}$ at the point $\beta_n(n^{5/6}/i, -n^{1/2}\log^2 n)$. Let $L_{i, W} \subset L_i$ be the largest line segment whose endpoints lie on $\tt{C}_{SW}, \tt{C}_{NW}$, and let $L_{i, E} = R_\pi L_{i, W}$. 
		
		We first rule out the possibility that $\cS_I^+$ crosses $L_{4, E}$. First, by Part (4) of \cref{L:basic-path-ctrl}, we have that $|I - n/4| \le n^{1/4} \log^2 n + 1$, and so by Part (2) of \cref{L:basic-path-ctrl}, the path $\cS_I^+$ does not enter $\tt{C}_\times$. Therefore the only way for $\cS_I^+$ to cross $L_{4, E}$ is if it first crosses $L_{3, E}$. Moreover,  $\cS_I^+$ later coalesces with $\cS_I^-$, and $\cS_I^-$ never crosses $\tt{C}_\times \cup L_{3, E}$, since this would require it to enter $\tt{D}_0^+$, contradicting \cref{L:QSE-pass}. Therefore $\cS_I^+$ must cross $L_{3, E}$ again after crossing $L_{4, E}$ prior to joining $\cS_I^-$.

		Now, we can find subpaths $\pi_1 = \cS_I^+|_{[s_1, t_1]}, \pi_2 = \cS_I^+|_{[s_2, t_2]}, t_1 \le s_2$ that stay in the region $R$ bounded by the four sets
		$$
		L_{3, E}, \quad L_{4, E}, \quad \tt{C}_{SE}, \quad \tt{C}_{NE}
		$$
		such that $\pi_1$ starts on $L_{3, E}$ and ends on $L_{4, E}$ and $\pi_2$ starts on $L_{4, E}$ and ends on $L_{3, E}$. Let $R' \subset R$ denote the subregion bounded between the paths $\pi_1, \pi_2$. Again, the path $\cS_I^-$ cannot enter $R'$ by \cref{L:QSE-pass}, since it cannot cross $\tt{D}_0^+$. Similarly, the path $\cN_{I+1}^+$ cannot enter $R'$ by \cref{L:QSE-pass} since it cannot cross $\tt{D}_0^-$. Therefore by the ordering on the $\cS$- and $\cN$-paths, no path can enter the set $R'$ other than $\cS_I^+$ itself, and so since the width of $R'$ is $n^{5/6}/12$ we can conclude that we are on the event $\mathsf{Onion}(n, n^{5/6}/12, 1)$. This event has an empty intersection with $E$ by \cref{L:onion-control}.
		
		Now, we will use the fact that $\cS_I^+$ does not cross $L_{4, E}$ to rule out the possibility that $\cN_I^-$ crosses $L_E = L_{5, E}$. Consider the region $M$ of the Aztec diamond bounded between the three $\cN$-paths $\cN_I^-, \cN_{I-1}^-, \cN_{I+1}^+$ (including the piece of $\cN_I^-$ after it coalesces with $\cN_I^+$). There are two $\cS$-paths that lie within this region: $\cS_I^+$ and $\cS_I^-$, and for every point $v \in \partial M$, there is an $\tt{S}$-vertex in $M$ that is within distance $4$ of $v$. Since we are on the event $\tt{Para_S}(\log^2 n)$, there must be a path of south vertices ending on $\cS_I^+ \cup \cS_I^-$ which is contained in the connected component of $C$ of 
		$
		(v + P_\tt{S}^{\log^2 n + 4}) \cap M
		$
		containing $v$. In other words, $C \cap (\cS_I^+ \cup \cS_I^-) \ne \emptyset$. Now, if $v \in L_{5, E}$, then the only two ways that this can happen are if:
		\begin{itemize}[nosep]
			\item Either $\cS_I^+$ or $\cS_I^-$ hits $L_{4, E}$, or
			\item One of the three paths $\cN_I^-, \cN_{I-1}^-, \cN_{I+1}^+$ hits $\tt{C}_\times$.
		\end{itemize}
		The first of these is ruled out by the previous argument, and the second by \cref{L:basic-path-ctrl}. Therefore $\partial M \cap L_{5, E} = \emptyset$, and so $\cN_I^- \cap L_{5, E} = \emptyset$.
		
		By a symmetric argument we can ensure that $\cN_I^+$ does not cross $L_{4, W}$ and that $\cS_I^-$ does not cross $L_{5, W} = L_W$. Putting these estimates together with the fact that $\cS_I^+, \cN_I^-$ do not enter $\tt{C}_\times$ (\cref{L:basic-path-ctrl}) gives the lemma.
	\end{proof}
	
	Given the path control above, we can confirm that the split point and the central height are related deterministically. This is the final piece of \cref{T:main-3-restatement}. 
	
	\begin{cor}
		\label{C:height-path-equality}
		With probability $1 - o(1)$, we have $H_n = 4I_n - n - 1$.
	\end{cor}
	
	\begin{proof}
		Consider any $a$-face $f$, and let
		$$
		\pi = (v_0, \dots, v_k), \qquad v_k \in \partial^\tt{x} \tt{S}
		$$
		be the path in $F_\tt{S}(D)$ starting from the $\tt{S}$-vertex $v_0 := f - (0, 1)$. By Corollary \ref{C:vertex-heights}, the height $\cH_n(v_0)$ is given by $R - L + \cH_n(v_{k-1})$, where $R$ and $L$ denote the number of right and left turns along the path $\pi$. We first translate this to a statement about face heights on the event where $v_k$ joins the path $\cS_I^-$.
		First, using that the boundary heights in the Aztec diamond are deterministic, we can compute that
		\begin{equation*}
\cH_n(v_{k-1}) = 4I - n - 5/2 + \mathbf{1}(v_k - v_{k-1} = (2, -2)).
		\end{equation*}
Also,
	$$
		\cH_n(v_0) = \cH_n(f) + 1 + 2\theta /\pi,
		$$
		where $\theta \in (-3\pi/2, \pi/2)$ is the polar coordinate of the vector $v_1 - v_0$. Finally, 
		$$
		R - L = 4 \operatorname{Wind}(\pi) - \mathbf{1}(v_k - v_{k-1} = (2, -2)) + 2\theta/\pi + 5/2. 
		$$
		where $\operatorname{Wind}(\pi)$ is the number of times that the path $\pi$ crosses the vertical ray $f + \{0\} \times [0, \infty)$ moving right-to-left minus the number of times it crosses moving left-to-right. Putting all this together gives that
		\begin{equation}
		\label{E:face-ht}
		\cH_n(f) = 4 \operatorname{Wind}(\pi) + 4 I - n - 1.
		\end{equation}
		Now, let $f$ be the $a$-face closest to the point $\beta_n(0, n^{3/4})$. To complete the proof, given \eqref{E:face-ht} it is enough to show that with probability $1 - o(1)$, we have that
		\begin{itemize}[nosep]
			\item The south forest path $\pi$ started at $v_0 :=m \beta_n(0, n^{3/4}) - (0, 1)$ joins the backbone path $\cS_I^-$, and
			\item $\cH_n(f) = 4 \operatorname{Wind}(\pi) + H_n$.
		\end{itemize}
		We show that the first bullet holds deterministically on the event $E$ from \cref{L:QSE-pass}. 
		
		Indeed, the path $\pi$ starts in the set $\tt{D}_0^+$, so by \cref{L:QSE-pass}, it starts in a region which is above the path $\cS_I^-$, and so by backbone path ordering, the path $\pi$ cannot join any path $\cS_i^-, i < I$. Next, suppose that the path $\pi$ enters the region $\tt{C}_{SW}$ before coalescing with any other path, and let $\pi_0$ be the segment of $\pi$ up until the point when $\pi$ hits $\tt{C}_{SW}$.  In this case, Part (2) of \cref{L:basic-path-ctrl} and \cref{L:QSE-pass} imply that the path $\cS_I^-$ does not enter the region $\tt{D}_0^+ \cup \pi_0 \cup \tt{C}_{SW}$. On the other hand, this region cuts off the vertex $v^S_{I_n}$ from the sink boundary of the Aztec diamond, which is a contradiction. Therefore $\pi$ coalesces with another path before entering the region $\tt{C}_{SW}$. 
		
		Next, prior to coalescing with another path, the path $\pi$ stays in the region $P = v_0 + P^{\log^2 n}_\tt{S}$. Now, the connected component $Q$ of $R \setminus\tt{C}_{SW}$ containing $v_0$ is itself contained in $\tt{D}_0^-$, and also in the union of $\tt{D}_0^+$ and the component of $[-n, n]^2 \setminus \tt{D}_0^+$ which is cut off from the north boundary.
		Therefore by \cref{L:QSE-pass} and path ordering again, the path $\pi$ cannot join any path of the form $\cS_i^-, i \ge I +1$ or $\cS_i^+, i \ne I$. Finally, $C$ is disjoint from the set $\Theta$ in \cref{L:SN-path-control}, and so $\pi$ cannot join $\cS_I^+$. Therefore $\pi$ must join $\cS_I^-$, as desired.
		
		For the second bullet point, we use the smooth phase coupling in \cref{L:global-coupling-2}. We also continue to work on the high-probability event $E$. By \cref{L:global-coupling-2}, letting $D \sim \P_a$ be a dimer configuration with height functions $\cH, \cH_n$, we can couple this configuration with our Aztec diamond configuration $D_n$ so that with probability $1 - o(1)$ we have $D|_{\Lambda_{\mathrm{sm},2}} = D_n|_{\Lambda_{\mathrm{sm},2}}$. Also, let $\pi'$ be the south forest path started at $v_0$ in $D$, and let $\pi_0$ be the portion of $\pi$ stopped when the path first leaves the set $\Lambda_{\mathrm{sm},2}$. Note that $\pi, \pi'$ agree on the segment $\pi_0$. Now, we claim that with probability $1 - o(1)$, we have
		\begin{equation}
			\label{E:careful-wind}
\cH(f) = 4\operatorname{Wind}(\pi') = 4\operatorname{Wind}(\pi_0) = 4\operatorname{Wind}(\pi).
		\end{equation}
		The first equality is by \cref{L:height-lem} and the second equality uses the parabolic path estimate in \cref{L:lerw-estimate-basic}, together with the fact that the distance between $v_0$ to the boundary of $\Lambda_{\mathrm{sm},2}$ is at least $n^{3/4}/100$ for large $n$. The third equality is deterministic on $E$. Indeed, $\pi$ will not wind around $v_0$ after exiting $\Lambda_{\mathrm{sm},2}$ and before joining $\cS_I^-$, since we are on the event $\mathsf{Para}_\tt{S}(\log^2 n)$, and the path $\cS_I^-$ will not wind around $v_0$ since this would entail it entering $\tt{D}_0^+$. Finally, with high probability we have that
		$$
		\cH_n(f) = \cH(f) + H_n.
		$$
		This follows by the same reasoning in the proof of \cref{L:height-shift}, specifically the argument for the convergence \eqref{E:cvg0-in-prob}. Putting this display together with \eqref{E:careful-wind} yields the second bullet point.
	\end{proof}
	
	\begin{proof}[Proof of \cref{T:main-3-restatement}]
	Part $1$ follows from \cref{L:variance-middle-height}.	Part $2$ follows from \cref{L:regular-paths}. For Part $3$, Part (4) of \cref{L:basic-path-ctrl} guarantees that with probability $1 - o(1)$,
	 $$
	 [I_n - n^{1/4}, I_n] \subset [n/4 - 2 n^{1/4} \log^2 n, n/4 + n^{1/4} \log^2 n].
	 $$ 
	 Therefore by Part (2) of \cref{L:basic-path-ctrl} and \cref{L:QSE-pass}, no paths $\cS_i^-, i \in \II{I_n - n^{1/4}, I_n}$ enter the set $\tt{C}_\times \cup \tt{D}_0^+$ with probability $1 - o(1)$. Finally,
	 $$
	 (\tt{X} \cup \tt{PRS}^*_n) \setminus \tt{RS}^*_n \subset \tt{C}_\times \cup \tt{D}_0^+,
	 $$ 
	yielding the result. For Part $4$, \cref{L:QSE-pass} guarantees that with probability $1 - o(1)$, none of the paths $\cN_i^+$ enter the region $\tt{PRS}^*_n$, since in order to do so, these paths would need to first enter both $\tt{D}_0^-$ and $\tt{D}_0^+$. The path $\cN_I^-$ does not enter the region $\tt{PRS}^*_n$ by \cref{L:SN-path-control}, since $\Theta \cap \tt{PRS}^*_n = \emptyset$ and none of the paths $\cN_i^+$ enter $\tt{PRS}^*_n$. By path ordering, none of the paths $\cN_i^-, i \ge I$ enter $\tt{PRS}^*_n$ either.
	
	For part $5$, the path $\cS_I^+$ does not enter the region $\tt{PRS}^*_n$ prior to coalescing with another path by \cref{L:SN-path-control}, again since $\Theta \cap \tt{PRS}^*_n = \emptyset$.  The paths $\cS_i^-, i \le I_n - n^{1/4} \log^4 n$ do not enter $\tt{PRS}^*_n$ by Part (3) of \cref{L:basic-path-ctrl} and the lower bound on $I_n$ in Part (4) of \cref{L:basic-path-ctrl}, since $\tt{PRS}^*_n \subset \tt{B}_{SW} \cup \tt{A}_{SW}$. Part $6$ follows from Corollary \ref{C:height-path-equality}.
	\end{proof}
	
	\section{Local path regularity at the rough-smooth boundary}
	\label{S:regularity}
	In this section, we prove a more refined path regularity estimate in the rough-smooth boundary region $\tt{RS}_n^*$,  using the global path control in \cref{T:main-3-restatement} as the main input. 
	
	Throughout this section, we will work on the high probability event where the six points of \cref{T:main-3-restatement} hold. Call this event $\tt{Reg}_n$. The main content of this section is the following proposition. We also write $I = I_n$ for the split point, and throughout we think of the backbone paths $\cS_i^\pm, \cN_i^\pm$ as piecewise linear plane curves, parametrized by arc length, and directed from source to sink. Finally, in this section and the next section, for a backbone path $\gamma \in \{\cS_i^-, \cN_i^- : i = 1, \dots, n/2\}$ ending on the west boundary we write $U(\gamma), L(\gamma)$ for the two components of $[-n, n]^2 \setminus \gamma$, where $U(\gamma)$ contains $(n, n)$ and $L(\gamma)$ contains $(-n, -n)$.
	
	\begin{prop}
		\label{P:backtrack-estimate}
		For all large enough $n$, the following estimates hold on the event $\tt{Reg}_n$.
		\begin{enumerate}[label=\arabic*.]
			\item Let $k \in \{0, \dots, \lfloor n^{1/4} \rfloor\}$, and suppose that $v = (v_1, v_2) = \cS_{I-k}^-(s), w = (w_1, w_2) = \cS_{I-k}^-(t)$ are two points in $\tt{RS}_n^*$ for some $s < t$. Then
			\begin{equation}
				\label{E:backtrack-ineq}
				w_1 - v_1 < (2k + 1) n^{1/4} \log^{2} n.
			\end{equation}
			\item Similarly, suppose $k \in \{1, \dots, \lfloor n^{1/4} \rfloor\}$, and that $v = \cN_{I-k}^-(s), w = \cN_{I-k}^-(t)$ are two points in $\tt{RS}_n^*$ for some $s < t$. Then
			\begin{equation}
				\label{E:backtrack-ineq-induct-N}
				w_1 - v_1 < 2 k n^{1/4} \log^{2} n.
			\end{equation}
		\end{enumerate}
	\end{prop}

    \begin{figure}
        \centering
        \includegraphics[width=0.8\linewidth]{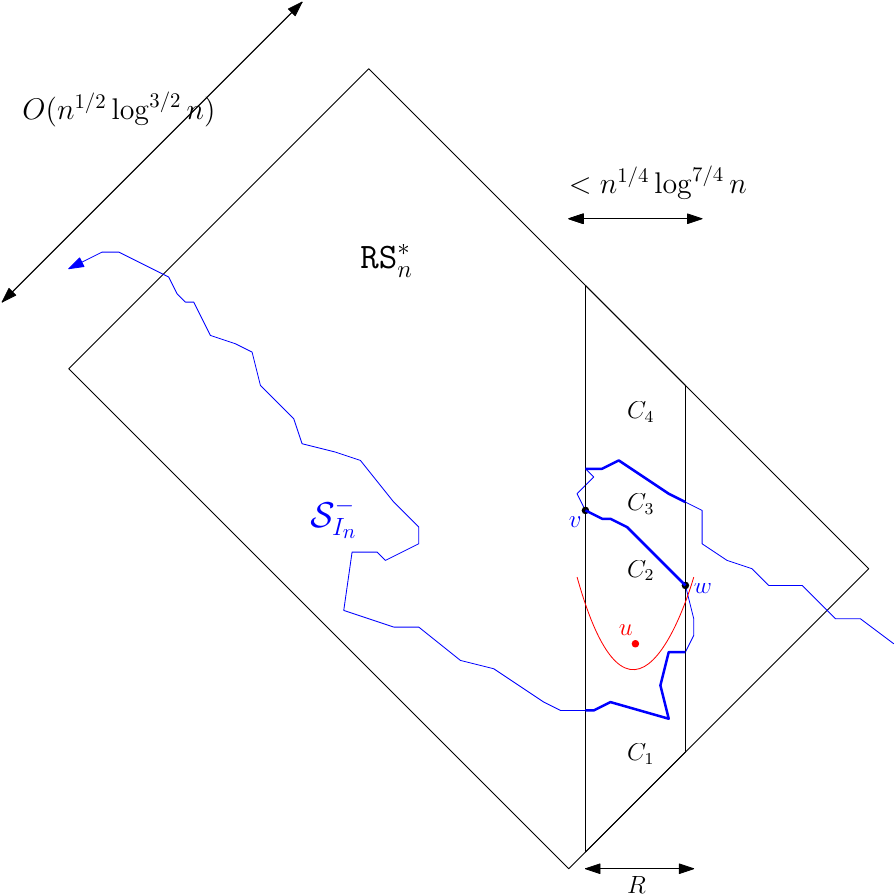}
        \caption{The base case of \cref{P:backtrack-estimate}. If the path $\cS_{I_n}^-$ has an overhang which is much larger than $\sqrt{n^{1/2} \log^{3/2} n}$, then this will trap a north path by virtue of the parabolic path estimates, since the region above $\cS_{I_n}^-$ in $\tt{LG}_n^*$ does not contain any backbone paths. The inductive step is similar, except instead of having a free region above a path $\cS_i^-$ or $\cN_i^-$, we have a path whose overhang depth is controlled by the inductive hypothesis.}
        \label{fig:proposition61}
    \end{figure}
	
	\begin{proof}
		We prove both parts of the proposition simultaneously by induction on the paths, using the ordering
		$$
	\cS_I^- , \cN_{I-1}^-, \cS_{I-1}^-, \cN_{I-2}^-, \dots.
	$$
	The base case is the estimate for the path $\cS_I^-$. See \cref{fig:proposition61} for an illustration of this case. Consider points $v = \cS_{I}^-(s), w = \cS_{I}^-(t)  \in \tt{RS}_n^*, s < t$. We assume that $v_1 < w_1$, since otherwise \eqref{E:backtrack-ineq} is immediate. Consider the strip $R = ([v_1, w_1] \times \R) \cap \tt{RS}_n^*$. We can recursively define times $s_1 < t_1 \le s_2 < t_2 \le \cdots \le s_k < t_k$ as follows. For this recursion, set $s_0 = t_0 = 0$.
    \begin{enumerate}[label=\arabic*.]
        \item For odd $i \in \N$, let $t_i = \inf \{r \ge t_{i-1} : \cS_I^-(r) \in \tt{RS}_n^* \cap  (\{v_1\} \times \R)$ and then let $s_i = \sup \{r \le t_i : \cS_I^-(r) \in \tt{RS}_n^* \cap  (\{w_1\} \times \R)$.
        \item For even $i \in \N$, let $t_i = \inf \{r \ge t_{i-1} : \cS_I^-(r) \in \tt{RS}_n^* \cap  (\{w_1\} \times \R)$ and then let $s_i = \sup \{r \le t_i : \cS_I^-(r) \in \tt{RS}_n^* \cap (\{v_1\} \times \R)$.
    \end{enumerate}
    We can only perform this recursion finitely many times, stopping at some $k \in \N$. Moreover, since the path $\cS_I^-$  starts to the left of $[v_1, w_1] \times \R$, ends to the right of this strip and never intersects the set $(\tt{X} \cup \tt{PRS}^*_n)\setminus \tt{RS}^*_n$ (Part (3) of \cref{T:main-3-restatement}), we have that $k$ is odd. Since $v = \cS_{I}^-(s), w = \cS_{I}^-(t)  \in \tt{RS}_n^*$ for some $s < t$, we have that $k \ge 3$.
    
    Now, the set
	$$
	R \setminus (\cS_I^-|_{[s_1, t_1]} \cup \cdots  \cup \cS_I^-|_{[s_k, t_k]})
	$$
	has $k+1$ connected components, $C_1, \dots, C_{k+1}$. We can index these components so that $C_1$ sits below $C_2$, which sits below $C_3$, etc. With this indexing, $C_2 \cap U(S_I^-)$ contains two path segments $\cS_I^-|_{[s_i, t_i]}, \cS_I^-|_{[s_{j}, t_{j}]}$ on its boundary.
In particular, for every $r \in [v_1, w_1]$ we can find a north vertex $u \in C_2 \cap U(S_I^-)$ with $|u_1 - r| \le 1$. We will consider such a vertex when $r = (v_1 + w_1)/2$, and let $\gamma$ be the north forest path starting at $u$. By Part (4) of \cref{T:main-3-restatement} and path ordering, no $\cN$-backbone paths enter the region $U(\cS_I^-) \cap \tt{RS}_n^*$, and so the path $\gamma$ does not join a backbone path before exiting $C_2$. Therefore by Part (2) of  \cref{T:main-3-restatement}, the path $\gamma$ must be contained in the region $u + P^{\log^2 n}_\tt{N}$ prior to exiting $C_2$. Since $C_2$ is bounded above and below by $\cS$-paths, the path $\gamma$ must exit through the set $\{v_1, w_1\} \times \R$, and so
	\begin{equation}
		\label{E:important-disjoint}
(u + P^{\log^2 n}_\tt{N}) \cap (\{v_1, w_1\} \times \R) \cap \tt{RS}_n^* \ne \emptyset.
	\end{equation}
Now, $|u - (v_1 + w_1)/2| \le 1$, so for this to hold, $(u + P^{\log^2 n}_\tt{N})$ must intersect $R$ in a location where the half-width of the parabola is at least $(w_1 -v_1)/2 - 1$. We have the inequality
$$
\operatorname{ht}(R) := \max \{|x_2 - y_2| : x, y \in R\} \le (2n^{1/3} \log^2 n + n^{1/2} \log^{3/2} n) + 2 (w_1 - v_1),
$$
on the height of $R$ (which is almost a parallelogram). Therefore by the definition of $P^{\log^2 n}_\tt{N}$, we have
\begin{align}
	\nonumber
(w_1 -v_1)/2 - 1 &\le \log^2 n + \log(\operatorname{ht}(R) + 1)\sqrt{\operatorname{ht}(R) + 1} \\
\label{E:paraben}
&\le 2n^{1/4} \log^{7/4} n + (w_1 -v_1)/4,
\end{align}
where the second inequality holds for large enough $n$. This yields \eqref{E:backtrack-ineq} in the base case.

We move on to the inductive step. The key difference in the base case and the inductive step is that in the base case, we used that $U(\cS_I^-) \cap\tt{RS}_n^*$ contains no north backbone paths. For any $k \ne 0$, the corridors $U(\cS_{I-k}^-), U(\cN_{I-k}^-)$ will contain dual backbone paths within the set $\tt{RS}_n^*$. However, if the paths $\cS_{I-k}^-$ (or $\cN_{I-k}^-$) backtrack too much, then we can use the inductive hypothesis to rule out the presence of such paths deep into a backtrack. 

Fix $k \ge 1$. We will only present the case when we prove the bound \eqref{E:backtrack-ineq-induct-N} on $\cN_{I-k}^-$ given the bound \eqref{E:backtrack-ineq} on $\cS_{I-k-1}^-$. The proof of the bound  \eqref{E:backtrack-ineq} on $\cS_{I-k}^-$ given the bound \eqref{E:backtrack-ineq-induct-N} on $\cN_{I-k}^-$ follows a symmetric argument (and is more similar to the base case).

Consider points $v = \cN_{I-k}^-(s), w = \cN_{I-k}^-(t)  \in \tt{RS}_n^*, s < t$. We assume that 
\begin{equation}
	\label{E:weak-assn}
w_1 - v_1 > (2k - 1) n^{1/4} \log^{2} n,
\end{equation} 
We define points $s_1 < t_1 \le s_2 < \cdots \le s_k < t_k$ and regions $C_1, \dots, C_{k+1}$ exactly as before, with $\cN_{I-k}^-$ in place of $\cS_I^-$. Again, $k$ must be odd and at least equal to $3$, and we work with the region $C_2$. 

Now, suppose $\cS_{I-k+1}^-$ enters $C_2$, and let $v' = \cS_{I-k+1}^-(s')$, $w' = \cS_{I-k+1}^-(t')$, $s' < t'$, denote the first and last times when $\cS_{I-k+1}^-$ is contained in $C_2$. The region $C_2$ is bounded above and below by portions of the north paths $\cN_{I-k}^-$, and so either both $v', w'$ are contained in $\{v_1\} \times \R$, or else both are contained in $\{w_1\} \times \R$. Now, for every point $u' \in C_2$ on the curve $\cS_{I-k+1}^-|_{[s', t']}$ we have
\begin{equation}
	\label{E:IH}
 |u_1' - v_1'|\le \max(u'_1 - v_1', w_1' - u_1') < (2k - 1) n^{1/4} \log^{2} n.
\end{equation}
Here the first inequality uses that $v_1' = w_1'$, and the second inequality uses the inductive hypothesis. Here we can apply the inductive hypothesis since $u', v', w' \in C_2 \subset \tt{RS}_n^*$. At this point, we assume that $v_1' = v_1$, as the case when $v_1' = w_1$ is symmetric. Define the set
$$
C_3' = C_3 \cap ([v_1, w_1^*] \times \R), \quad w_1^* := w_1 - (2k - 1) n^{1/4} \log^{2} n.
$$
which contains no points on any south backbone paths by the above discussion, and is non-empty by \eqref{E:weak-assn}. Now, similarly to the base case we can find a south vertex $u \in C_2'$ with $|u_1 - (v_1 + w_1^*)/2| \le 1$. Letting $\gamma$ be the south forest path starting from $u$, as in \eqref{E:important-disjoint} we have that
\begin{equation*}
	(u + P^{\log^2 n}_\tt{S}) \cap (\{v_1, w_1^*\} \times \R) \cap \tt{RS}_n^* \ne \emptyset.
\end{equation*}
At this point, we can proceed as in the base case following \eqref{E:weak-assn} to prove that
$$
(w_1^* - v_1)/4 \le 2 n^{1/4} \log^{7/4} n + 1,
$$
which gives \eqref{E:backtrack-ineq-induct-N} after substituting in $w_1$ for $w_1^*$ and simplifying.
\end{proof}

We can use \cref{P:backtrack-estimate} to show that north and south paths are close in the region $\tt{RS}_n^*$, see \cref{fig:corollary62}. For this corollary, we introduce the smaller set 
$$
\tt{RS}_n := \beta_n([-n^{3/4}, n^{3/4}] \times [- n^{1/2} \log^{3/2} n, 2n^{1/3} \log^2 n]).
$$
\begin{cor}
	\label{C:close-ties}
	On the event $\tt{Reg}_n$ we have the following estimates for $k \in \{0, \dots, \lfloor n^{1/4} \rfloor - 1\}$.
	\begin{enumerate}[label=\arabic*.]
		\item For any vertex $w \in \cN_{I-k-1}^- \cap \tt{RS}_n$ we can find a vertex $w' \in \cS_{I-k}^- \cap \tt{RS}_n^*$ with $\|w - w'\|_\infty \le (2k + 2) n^{1/4} \log^{2} n$.
		\item For any vertex $v \in \cS_{I-k}^- \cap \tt{RS}_n$ we can find a vertex $w \in \cN_{I-k-1}^- \cap \tt{RS}_n^*$ with $\|v - v'\|_2 \le (2k + 1) n^{1/4} \log^{2} n$.		
	\end{enumerate}
\end{cor}

\begin{figure}
    \centering
    \includegraphics[width=0.8\linewidth]{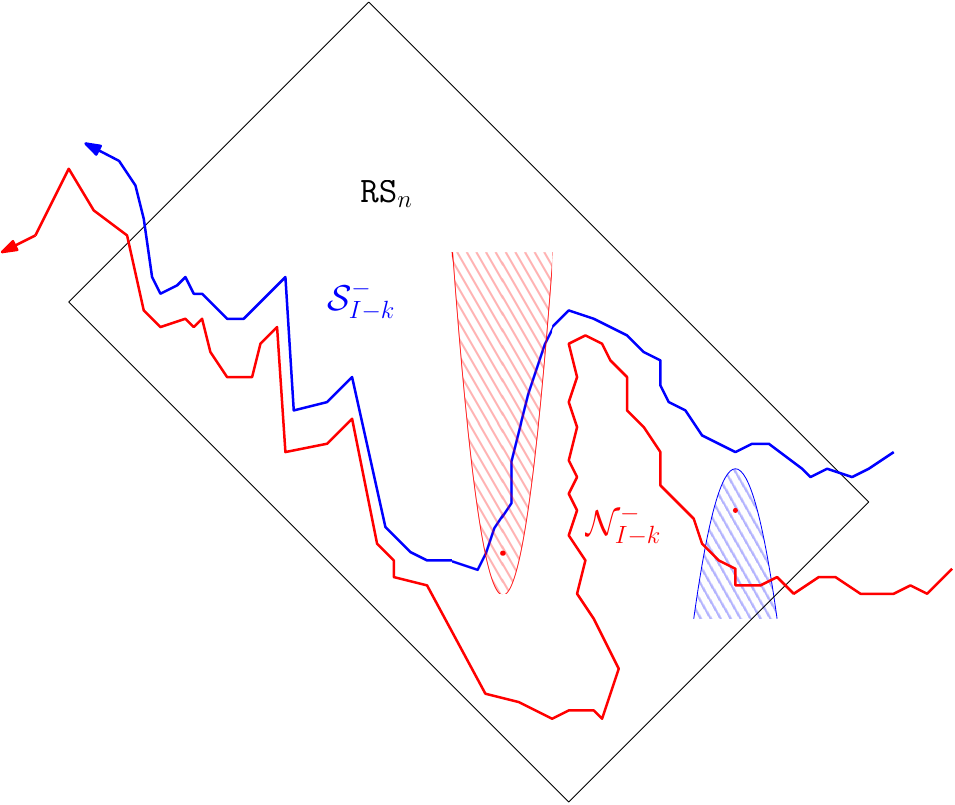}
    \caption{The idea behind Corollary \ref{C:close-ties} is that in the absence of large path overhangs (\cref{P:backtrack-estimate}), if a pair of paths $\cS_{I-k}^-, \cN_{I-k-1}^-$ deviate too far from each other, then some vertex $\mathbf{v}$ cannot be connected to a backbone path without exiting its bounding parabola $\mathbf{v} + P^{\log^2 n}_\cdot$.}
    \label{fig:corollary62}
\end{figure}
\begin{proof}
For the proof, it is useful to define an intermediate set $\tt{RS}_n'$ with $\tt{RS}_n \subset \tt{RS}_n' \subset \tt{RS}_n^*$, whose definition is the same as that of $\tt{RS}_n$ but with the interval $[-3 n^{3/4}/2, 3 n^{3/4}/2]$ in place of $[-n^{3/4}, n^{3/4}]$.
	
Fix $k$, and consider a south vertex $v\in \tt{RS}_n'$ contained in the corridor $L(\cS_{I-k}^-) \cap U(\cS_{I-k - 1}^-)$ between $\cS_{I-k}^-$ and $\cS_{I-k-1}^-$ (here notation is as in the proof of \cref{P:backtrack-estimate}). Let $r = d_\infty(v, \cS_{I-k}^- \cap \tt{RS}_n^*)$, and let $\gamma$ be the south forest path starting at $v$. We first show that
\begin{equation}
	\label{E:r-implication}
\text{$\gamma$ \text{ joins } $\cS_{I-k}^-$} \qquad \implies \qquad r \le (2k + 3/2) n^{1/4} \log^{2} n.
\end{equation}
Let $\gamma_0$ be the segment of $\gamma$ stopped at the first vertex along a south backbone path. By Part (2) of \cref{T:main-3-restatement}, we have that
$\gamma_0 \subset v + P_\tt{S}^{\log^2 n}$, and so
$$
\gamma_0 \subset A := (v + P_\tt{S}^{\log^2 n}) \cap L(\cS_{I-k}^-) \cap U(\cS_{I-k - 1}^-).
$$ 
The set $A$ is necessarily contained in $\tt{RS}_n^*$ by Part (3) of \cref{T:main-3-restatement}. Here we have used that any vertex $v$ in $\tt{RS}_n'$ is far away from the left and right boundaries of the larger set $\tt{RS}_n^*$, so the set $v + P_\tt{S}^{\log^2 n}$ only crosses $\tt{RS}_n^*$ on the part of the boundary where it would immediately enter $\tt{PRS}_n^*$. 

Now, suppose that $\gamma$ joins $\cS_{I-k}^-$. Then $\cS_{I-k}^-$ has non-empty intersection with the set $(v + P_\tt{S}^{\log^2 n}) \cap \tt{RS}_n^*$; let $w$ denote the intersection point. Now, since $v \in L(\cS_{I-k}^-)$, the path $\cS_{I-k}^-$ also intersects the vertical ray $v + \{0\} \times [0, \infty)$ within the set $\tt{RS}_n^*$. Call $w'$ the intersection point. Finally, let $w''$ be a point on the horizontal line $v + \R \times \{0\}$, on the segment of $\cS_{I-k}^-$connecting the vertices $w, w'$. We have that
\begin{align*}
w'_1 = v_1', \qquad |w_1 - v_1| \le 2 n^{1/4} \log^{7/4} n, \qquad |w''_1 - v_1| \ge r.
\end{align*}
Here the equality is by definition, the first inequality follows as in \eqref{E:paraben}, and the third inequality is from the definition of $r$. 
On the other hand, the backtrack estimate in \cref{P:backtrack-estimate} implies that
$$
\min(|w_1'' - w_1|, |w_1'' - w_1'|) \le (2k + 1) n^{1/4} \log^{2} n.
$$
Putting these results together and simplifying gives the implication \eqref{E:r-implication}.

Now, consider any vertex $w \in \cN_{I-k-1}^- \cap \tt{RS}_n$. We can find a south vertex $v \in  L(\cS_{I-k}^-) \cap U(\cN_{I-k - 1}^-)$ with $\|v - w\|_\infty \le 1$. This vertex is necessarily in $\tt{RS}_n'$ (again, this uses that the paths $\cS_{I-k}^-, \cN_{I-k - 1}^-$ do not intersect the upper and lower boundaries of $\tt{RS}_n^*$). Since $v \in U(\cN_{I-k - 1}^-)$, the south forest path from $v$ must join $\cS_{I-k}^-$, so we can apply \eqref{E:r-implication} to bound $d_\infty(v, \cS_{I-k}^- \cap \tt{RS}_n^*)$, and then in turn bound $d_\infty(w, \cS_{I-k}^- \cap \tt{RS}_n^*)$ by adding $1$. This gives part $1$ of the corollary. Part $2$ follows a symmetric proof.
\end{proof}

	\section{Convergence of the backbone to the Airy line ensemble}
    \label{S:cvg-airy}
	
	In the section, we prove the main theorem of the paper, \cref{T:main-2}, which shows that the backbone paths $\cS_I^-, \cS_{I-1}^-, \cS_{I-2}^-$ in the south Temperleyan forest converge to the Airy line ensemble near the rough-smooth boundary. Corollary \ref{C:close-ties} will show that the same is also true of the north backbone paths. We begin by stating a more detailed version of the main theorem, which also encompasses Remark \ref{R:variants-main-thm-2}. First, recall the scaling transformation $\gamma_n:\R^2 \to \R^2$ defined prior to \cref{T:main-1-restatement}:
	\begin{equation*}
		\gamma_n(t, x) = n \xi_0(t/n) + (x, 0).
	\end{equation*}
	Define functions $\cA_i^{n, \pm}:\R \to \R, i \in \{1, \dots, I_n\}$ by
	\begin{align*}
		\cA_i^{n, +}(t) &= \sup \{x \in \R : \gamma_n(2^{1/3}\lambda_2 n^{2/3} t, 2^{5/3}\lambda_1 n^{1/3} x) \in \cS_{I_n + 1- i}^- \cap \tt{RS}_n\} \\
		\cA_i^{n, -}(t) &= \inf \{x \in \R : \gamma_n(2^{1/3}\lambda_2 n^{2/3} t, 2^{5/3}\lambda_1 n^{1/3} x) \in \cS_{I_n + 1- i}^- \cap \tt{RS}_n\},
	\end{align*}
	where we let the above supremum/infimum be $-\infty/+\infty$ if the corresponding set is empty. We similarly define $\cA_i^{n, \tt{N}+}, \cA_i^{n, \tt{N}-}$ with the north path $\cN_{I_n - i}^-$ in place of $\cS_{I_n + 1- i}^-$.

	\begin{thm}
		\label{T:main-2-restatement}
		Fix $\ast \in \{+, -, \tt{N}+, \tt{N}-\}$. Then the line ensemble $\{\cA_i^{n, *} : i \in \N\}$ converges to the parabolic Airy line ensemble $\{\cA_i : i \in \N\}$ in the sense of finite dimensional distributions. That is, for any finite set $T \subset \R$ we have
		\begin{equation}
			\label{E:Ain}
			(\cA_i^{n, *}|_T : i \in \N) \cvgd (\cA_i|_T : i \in \N).
		\end{equation}
		Moreover, the limiting Airy line ensemble is the same as the line ensemble in the height function limit from
		\cref{T:main-1-restatement}. More precisely, with $\phi_n, \cH_n, \hat \gamma_n$ as in \cref{T:main-1-restatement} we have the joint FDD convergence
		\begin{equation}
			\label{E:Ain-joint}
			\begin{split}
			(\cH_n^{\phi_n}(\hat \gamma_n(u)) - H_n, \cA_i^{n, *}(t)) \implies (-4 \cA(u), \cA_i(t)).
			\end{split}
		\end{equation}
	Here the joint convergence is over $u \in S, t \in T$, and $i \in \N$, where $S \subset \R^2, T \subset \R$ are arbitrary finite or countable sets. 
		Finally, the convergence in \eqref{E:Ain} and \eqref{E:Ain-joint} is uniform on certain \textit{mesoscopic} sets. That is, in both \eqref{E:Ain} and \eqref{E:Ain-joint} we may replace $\cA_i^{n, *}(\cdot)$ with 
			\begin{equation}
				\label{E:suppinf}
			\sup_{|u| \le n^{-1/3 - \eps}} \cA_i^{n, *}(\cdot + u) \qquad \text{ or } \qquad \inf_{|u| \le  n^{-1/3 - \eps}} \cA_i^{n, *}(\cdot + u)
			\end{equation}
			for any $\eps > 0$, without affecting the result.
	\end{thm}
	
	\begin{rem}
		In the final mesoscopic strengthening of the theorem, the trivial estimate allows us to take suprema and infima over ranges of size $\Delta n^{-2/3}$, which is the scale on which we see the discrete lattice. Of course, the optimal result in this direction would be uniform convergence to the Airy line ensemble, as discussed in the introduction.
		
		Moreover, as with \cref{T:main-1-restatement}, we can strengthen the above theorem to allow the finite sets $S, T$ to vary with $n$. We could also give a version of the above theorem with $\cH_n^{\phi_n}$ replaced by $\cH_n$, where in the limit we also see i.i.d.\ copies of the full-plane smooth phase (as in \eqref{E:height-bonus-main}). For brevity, we have omitted these versions.
		\end{rem}

	The main step in proving \cref{T:main-2-restatement} is the following proposition. For this proposition and for the remainder of the section, fix a mollifying sequence $\phi_n$ satisfying the conditions of \cref{T:main-1-restatement}. For a fixed mesh size $\delta > 0$, a line index $i \in \N$ and a time $t \in \R$, define random variables 
	$$
	A_{n, \delta}(i, t) = \max \{x \in [-\delta^{-1}, \delta^{-1}] \cap \delta \Z : [H_n - \cH_n^{\phi_n}(\hat \gamma_n(x, t))]/4 \ge i - \delta \}, 
	$$
	where we set this $\max$ to $-\infty$ if $[H_n - \cH_n^{\phi_n}(\hat \gamma_n(x, t))]/4 < i$ for all $x$ in the given range.
	
	The random variables $A_{n, \delta}$ are an approximation of the location of the $i$-th Airy line at a mesh size $\delta$. Indeed, by \cref{T:main-1-restatement}, 
	$$
	\lim_{\delta \to 0} \lim_{n \to \infty} A_{n, \delta}(i, t) \quad \implies \quad \cA_i(t),
	$$
	jointly over finitely many lines $i$ and times $t$.
	\begin{prop}
		\label{P:key-prop}
		Fix $\eps, \delta > 0$, a time $t \in \R$, and $i \in \N$. Then
		\begin{align*}
			\lim_{\delta \to 0} \lim_{n \to \infty} \P(|A_{n, \delta}(i, t) - \sup_{|u| \le n^{-1/3-\eps}} \cA_i^{n, +}(\cdot + u)| > \sqrt{\delta}) &= 0, \\
			\lim_{\delta \to 0} \lim_{n \to \infty} \P(|A_{n, \delta}(i, t) - \inf_{|u| \le n^{-1/3 -\eps}} \cA_i^{n, -}(\cdot + u)| > \sqrt{\delta}) &= 0.
		\end{align*}
	\end{prop}
	
	Before proving \cref{P:key-prop}, we use it to establish \cref{T:main-2-restatement}.
	
	\begin{proof}[Proof of \cref{T:main-2-restatement}]
		We first note that by Corollary \ref{C:close-ties}, for any fixed $i \in \N$ and $t \in \R$, we have that
		$$
		\P(|\cA_i^{n, +}(t) - \cA_i^{n, \tt{N}+}(t)| \ge n^{-1/12} \log^3 n) \to 0
		$$
		as $n \to \infty$, and similarly with $+ \mapsto -$,  $\tt{N}+\mapsto\tt{N}-$. Therefore it suffices to prove \cref{T:main-2-restatement} only when $* \in \{+, -\}$. Next, write
		$$
		\mathcal B_i^{n, -}(t) = \inf_{|u| \le n^{-1/3-\eps}} \cA_i^{n, -}(\cdot + u), \qquad \mathcal B_i^{n, +} = \sup_{|u| \le n^{-1/3-\eps}} \cA_i^{n, +}(\cdot + u).
		$$
		It furthermore suffices to prove \eqref{E:Ain-joint} with $\mathcal B_i^{n, \pm}$ in place of $\mathcal A_i^{n, \pm}$. 
		
		Now, fix a finite set $T \subset \R$ and a countable dense set $Q \subset \R$. Using \cref{T:main-1-restatement} and the Skorokhod representation theorem, we can consider a coupling of all of the Aztec diamond measures at each $n \in 4 \N$ with the Airy line ensemble so that for any $t, q \in Q$ we have
		\begin{equation}
			\label{E:coupling-quality}
			\cH_n^{\phi_n}(\hat \gamma_n(q, t)) - H_n \to -4\cA(q, t)
		\end{equation}
		almost surely. It suffices to show that in this coupling, for any fixed $i \in \N, t \in Q$, we have
		$\mathcal B_i^{n, \pm}(t) \cvgp \cA_i(t)$. The statement \eqref{E:Ain-joint} then follows by taking $Q$ so that $S \cup  T^2 \subset Q^2$.

		Observe that by \eqref{E:coupling-quality}, we have
		$$
		\limsup_{n \to \infty} \mathbf{1}(|\cA_i(t)| \le \delta^{-1}-1) |A_{n, \delta}(i, t) - \cA_i(t)| \le \delta,
		$$
		and so by \cref{P:key-prop}, 
		$$
		\limsup_{n \to \infty} \P(|\mathcal B_i^{n, \pm}(t) - \cA_i(t)| > \delta + \sqrt{\delta}) \le \P(|\cA_i(t)| > \delta^{-1}).
		$$
		Since $\delta > 0$ was arbitrary and the random variables $\cA_i(t)$ are all finite, this gives the desired result.
	\end{proof}
	
	In the remainder of the section we prove \cref{P:key-prop}. 
	
	The main part of the proof involves showing that around any fixed point $\hat \gamma_n(x, t)$ we can define a box $O_\eps(x, t)$ size $n^{1/3 - \eps} \times n^{1/3-\eps}$ such that with high probability:
	\begin{itemize}[nosep]
		\item No south backbone paths enter the box $O_\eps(x, t)$
		\item We have
		\begin{equation}
		\label{E:O-containment}
		O_\eps(x, t) \subset U(\cS_{I_n-J}) \cap L(\cS_{I_n - J+1}),
		\end{equation}
		where 
		$
		J \approx (H_n - \cH_n^{\phi_n}(\hat \gamma_n(x, t)))/4.
		$
	\end{itemize}
	The dimensions of the box $O_\eps(x, t)$ are important here. The key point will be that the width of the box is larger than the size of the maximum backtrack of any south backbone path entering the region. The maximum backtrack size is $O(n^{1/4}\log^2 n)$ by \cref{P:backtrack-estimate}. The idea for constructing the box $O_\eps(x, t)$ is to use the local smooth phase coupling, \cref{L:local-coupling}, to rule out the potential presence of backbone paths. The difficulty with that lemma is that it only provides a coupling on a rectangle of width $o(n^{1/12})$, which is too small to make \cref{P:backtrack-estimate} effective. The workaround is to first build a ring of overlapping boxes of small size, each of which is individually coupled to the full-plane smooth phase. We build the rings in the next lemma, and then use that lemma together with the path regularity from Section \ref{S:regularity} to verify that a box inside this ring satisfies the two bullet points above.
	
	\begin{lem}
	\label{L:ring-construction}
	Let $D \sim \P_{a, n}$, and fix $\eps > 0, (x, t) \in \R^2$. For ease of notation, set $c_n = \log^3 n$. Define the set 
	$$
	S_\eps = \{v \in \Z^2 : \|v\|_1 = \lfloor n^{1/3 - \eps} \rfloor\},
	$$
	and for $v \in S_\eps$ let 
	$$
	B_v = \hat \gamma_n(x, t) + c_n(v + [-3/5, 3/5]^2).
	$$
	Then with probability $1 - o(1)$, the following claims hold:
	\begin{enumerate}[label=(\roman*)]
		\item We can couple $D$ with $|S_\eps|$-many copies of the smooth phase $D_v \sim \P_a, v \in S_\eps$ such 
		$$
		D_v|_{B_v} = D|_{B_v}, \qquad \text{ for all } v \in S_\eps.
		$$
		\item In each of the smooth phases $D_v, v \in S_\eps$, any path $\pi$ in the associated south forest, started at a point in $D_v$, satisfies
		$$
		\pi \subset P^{\log^2 n}_\tt{S}.
		$$
		Similarly, any path $\pi$ in the associated north forest, started at a point in $D_v$, satisfies
		$$
		\pi \subset P^{\log^2 n}_\tt{N}.
		$$
		\item Let $T = \bigcup_{v \in S_\eps} B_v$, let $\partial^- T$ be the part of the boundary of $T$ connected to the bounded component $U^-$ of $T^c$, and let $\partial^+ T$ be the part of the boundary of $T$ connected to the unbounded component $U^+$ of $T^c$. 
		
		 Consider any north or south backbone path $\pi$, and let $\pi_0:[s, t] \to \R$ be a segment of $\pi$ with $\pi_0|_{(s, t)}$ contained in the interior of $T$, and $\pi_0(s), \pi_0(t)$ on the boundary of $T$. Then:
		 \begin{itemize}
		 	\item If $\pi$ is a south backbone path, $\pi_0(s) \in \partial T^+$ and $\pi_0(t) \in \partial T^-$, then $\pi_0(s)$ is contained in a box $B_v$ with $v_2 \ge 1$.
		 	\item If $\pi$ is a south backbone path, $\pi_0(s) \in \partial T^-$ and $\pi_0(t) \in \partial T^+$, then $\pi_0(s)$ is contained in a box $B_v$ with $v_2 \le - 1$.
		 	\item If $\pi$ is a north backbone path, $\pi_0(s) \in \partial T^+$ and $\pi_0(t) \in \partial T^-$, then $\pi_0(s)$ is contained in a box $B_v$ with $v_2 \le -1$.
		 	\item If $\pi$ is a north backbone path, $\pi_0(s) \in \partial T^-$ and $\pi_0(t) \in \partial T^+$, then $\pi_0(s)$ is contained in a box $B_v$ with $v_2 \ge 1$.
		 \end{itemize} 
	\end{enumerate}
	\end{lem}

    \begin{figure}
        \centering
        \includegraphics[width=0.8\linewidth]{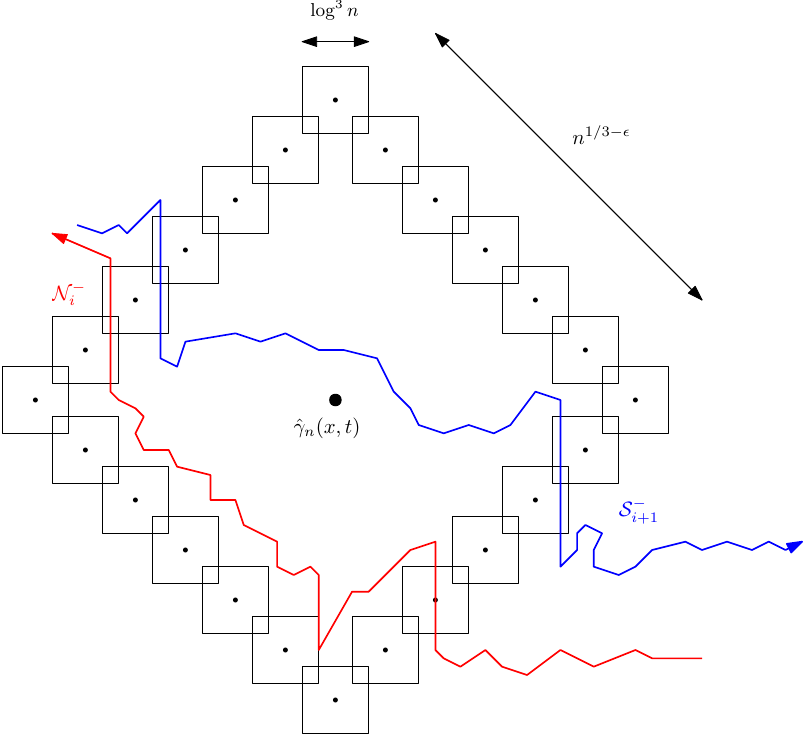}
        \caption{The setup of \cref{L:ring-construction}, together with the proof of (iii) given (ii). Essentially, the smooth phase couplings in part (ii) force strong drifts through the ring of boxes $\{B_v, v \in S_\eps\}$. This means $\cS$-paths can only enter through a box in the top half of the ring, and exit through a box in the bottom, whereas the opposite is true for $\cN$-paths.}
        \label{fig:lemma73}
    \end{figure}
	
	\begin{proof}
By \cref{L:local-coupling} and a union bound bound over $|S_\eps|$-many boxes, we can achieve all of the couplings in part (i) simultaneously with probability at least
$$
1 - c |S_\eps| n^{-1/3} \log^{24} n \ge 1 - c n^{-\eps/2},
$$
where the constant $c$ depends on $x, t$, and the inequality holds for large enough $n$.
The parabolic path control in part (ii) follows from \cref{L:lerw-estimate-basic} and a union bound, using that paths in the smooth phase forests are loop-erased random walks.

Part (iii) is deterministic given part (ii). See \cref{fig:lemma73} for an illustration. We only prove the claim about south backbone paths, as the claim about north backbone paths follows symmetrically. Consider a given path segment $\pi_0$, let $B_v$ be a box that $\pi_0$ intersects, and let $\tau:[s', t'] \to \R$ be a segment of $\pi_0$ contained in $B_v$. If we define $\tilde \tau:[s', t'] \to [-1, 1]^2$ by letting
\begin{align*}
\tau(r) = \hat \gamma_n(x, t) + c_n(v + \tilde \tau(r)),
\end{align*}
then by part (ii), for large enough $n$ we have that for any $r < r' \in [s', t']$:
\begin{align}
	\label{E:tau-bds}
|\tilde \tau(r')_1 - \tilde \tau(r)_1| \le 1/100, \qquad 
\tilde \tau(r')_2 - \tilde \tau(r)_2 \le 1/100.
\end{align}
Now, for any south backbone path segment $\pi_0$ contained in $T$, the inequalities above imply that $\pi_0$ can enter at most three distinct boxes, and that 
$$
|(\pi_0(s))_1 - (\pi_0(t))_1| \le 3c_n/100.
$$
Using this estimate, and the second bound in \eqref{E:tau-bds}, we get that if $\pi_0$ first enters $T$ through a box $B_v$ with $v_2 \le 0$ on the boundary $\partial^+ T$, then it cannot exit $T$ on the boundary $\partial^- T$. This gives the first bullet. Similarly, $\pi_0$ exits into $\partial^+ T$ through a box $B_v$ with $v_2 \ge 0$, then it cannot have started in the set $\partial^- T$. This gives the second bullet.
	\end{proof}
	
	\begin{lem}
		\label{L:ring-bd}
Fix $\eps > 0$ and define a set
$$
O_\eps = O_\eps(x, t) = \hat \gamma_n(x, t) + [-n^{1/3 - \eps}, n^{1/3 - \eps}]^2.
$$
Then with probability $1 - o(1)$, no backbone paths of the form $\cS_{I - k}^-, \cN_{I - k - 1}^-, k \in \{0, \dots, \lfloor \log^{50} n \rfloor\}$ enter the set $O_\eps$.
	\end{lem}

    \begin{figure}
        \centering
        \includegraphics[width=0.75\linewidth]{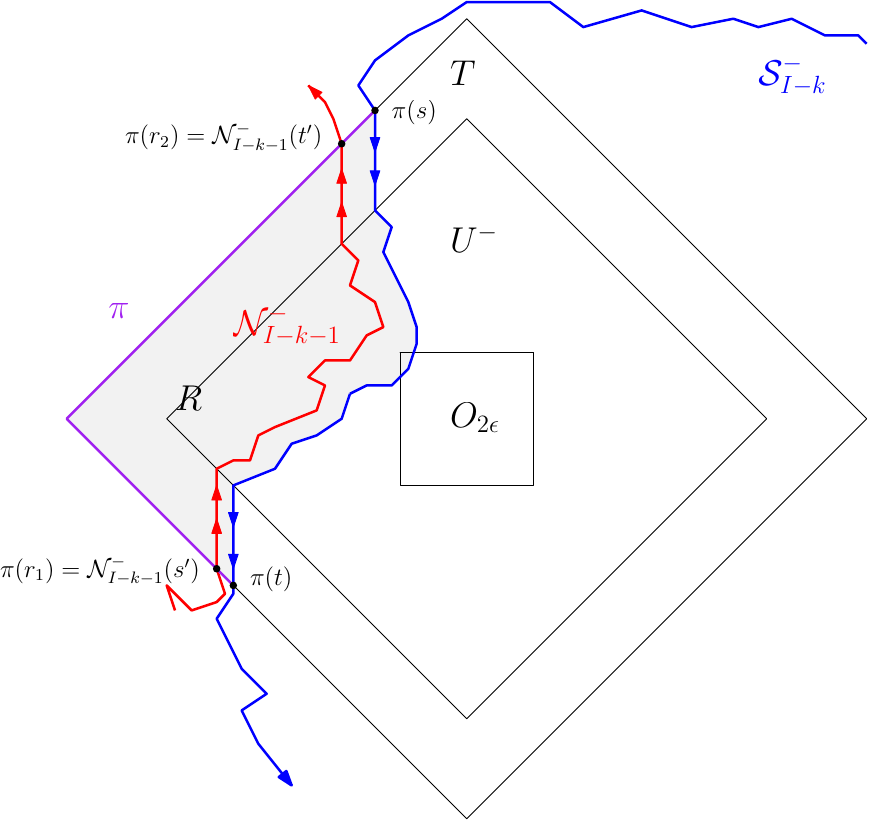}
        \caption{The proof setup for \cref{L:ring-bd}. }
        \label{fig:lemma7.4}
    \end{figure}
	\begin{proof}
	Fix $\eps \in (0, 1/12)$, and consider the probability $(1 - o(1))$-event where points (i)-(iii) of \cref{L:ring-construction} hold with $\eps$ and points $1, 2$ or Corollary \ref{C:close-ties} hold. On this event, we will show that deterministically, no backbone paths of the form $\cS_{I - k}^-, \cN_{I - k - 1}^-, k \in \{0, \dots, \lfloor \log^{50} n \rfloor\}$ enter the set $O_{2\eps}$ as long as $n$ is sufficiently large. 
	
	Fix $k \in \{0, \dots, \lfloor \log^{50} n \rfloor - 1\}$ and suppose that $\cS_{I - k}^-$ enters the set $O_{2 \eps}$. Let $T, \partial^\pm T, U^\pm$ be as in \cref{L:ring-construction}. The set $O_{2\eps}$ is contained in $U^-$, and moreover, every point in $T$ is $L^\infty$-distance at least $n^{1/3 - \eps} > n^{1/4} \log^{100} n$ away from $O_{2\eps}$. In particular, by Corollary \ref{C:close-ties}, the path $\cN_{I - k - 1}^-$ must also enter the set $U^-$. This path also stays in the set $L(\cS_{I-k}^-)$.
	
	Now, let $R$ be a connected component of $L(\cS_{I-k}^-) \cap (T \cup U^-)$ such that $R \cap U^- \ne \emptyset$. We aim to show that $\cN_{I - k - 1}^-$ cannot enter $R \cap U^-$, which will complete the proof since $R$ is arbitrary. The proof setup is illustrated in \cref{fig:lemma7.4}.
    
    The boundary of $R$ can be decomposed into two pieces: 
	\begin{itemize}[nosep]
		\item $\partial_1 R$, which is given by a path segment $\cS_{I - k}^-|_{[s, t]}$, where $\cS_{I - k}^-|_{(s, t)}$ is contained in the interior of $T \cup U^-$ and $\cS_{I - k}^-(s), \cS_{I - k}^-(t) \in \partial^+ T$.
		\item  $\partial_2 R$, which is a closed connected subset of $\partial^+ T$ between the points $\cS_{I - k}^-(s)$ and $\cS_{I - k}^-(t)$.
	\end{itemize}
	We view $\partial_2 R$ as a parametrized simple curve $\pi:[s, t] \to \R$ and choose the parametrization so that 
	$$
	\cS_{I - k}^-(s) = \pi(s), \qquad \cS_{I - k}^-(t) = \pi(t),
	$$
	and so that $\pi$ is one-to-one. 
Now at this point, suppose $\cN_{I - k - 1}^-$ intersects the set $R \cap U^-$, and let $\cN_{I - k - 1}^-|_{[s', t']}$ be a segment of $\cN_{I - k - 1}^-$ such that:
	\begin{enumerate}[nosep, label=\arabic*.]
		\item $\cN_{I - k - 1}^-|_{[s', t']}$ is contained in $R$, and intersects the set $U^-$.
		\item $\cN_{I - k - 1}^-(s'), \cN_{I - k - 1}^-(t')$ are on the boundary $\partial^+ T$.
		\item $t'$ is chosen maximally among all choices of $s', t'$ satisfying $1, 2$.
	\end{enumerate}
    Now, by the first two bullet points in \cref{L:ring-construction}(iii), there exists $v, w \in S_\eps$ such that
	$$
	\cS_{I - k}^-(s) \in B_{v}, \qquad \cS_{I - k}^-(t) \in B_{w}, \;\;\; \text{ and } v_{1} \ge 1, \;\; w_{1} \le -1,
	$$
	and by the third and fourth bullet points in \cref{L:ring-construction}(iii), there exists $v', w' \in S_\eps$ such that
	$$
	\cN_{I - k-1}^-(s') \in B_{v'}, \qquad \cN_{I - k-1}^-(t') \in B_{w'}, \;\;\; \text{ and } v_{1}' \le -1, \;\; w_{1}' \ge 1.
	$$
Now let $r_1, r_2 \in (s, t)$ be such that $\mathcal N_{I-k-1}^-(s') = \pi(r_1)$ and $\mathcal N_{I-k-1}^-(s') = \pi(r_2)$. From the previous two displays, we have that $s < r_2 < r_1 < t$ (see also \cref{fig:lemma7.4}).

On the other hand, consider the curve
$$
\tau = \cS_{I-k}^-|_{[0, t]} \cup \cN_{I-k-1}^-|_{[0, s']} \cup \pi_{[r_1, t]}.
$$
The set $\tau$ cuts off the curve $\cN_{I-k-1}^-|_{(s', t')}$ from the west boundary of the Aztec diamond, in the sense that $\cN_{I-k-1}^-|_{(s', t')}$ and $\{-n\} \times [-n, n]$ are in different components of the set $[-n, n]^2 \setminus \tau$. Therefore since the path $\cN_{I-k-1}^-$ finishes on the west boundary, there must be a point $r' \ge t'$ such that $\cN_{I-k-1}^-(r') \in \tau$. Since $r_2 < r_1$ and $\cN_{I-k-1}^-$ cannot intersect itself or $\cS_{I-k}^-$, the only way this can happen is if the path $\cN_{I-k-1}^-$ cuts back through the set $R \cap U^-$, (again, \cref{fig:lemma7.4} may help here). This contradicts the maximality of $t'$.
    
    Therefore $\cS_{I-k}^-$ cannot enter the set $O_{2\eps}$ for $\eps < 1/12$. A symmetric proof with the roles of the north and south paths reversed implies that $\cN_{I-k-1}^-$ cannot enter $O_{2\eps}$. Since $\eps > 0$ was arbitrary, the lemma follows.
	\end{proof}
	
	Our next goal is to prove \eqref{E:O-containment}. To do so, we first need an estimate on the minimum index of a path that enters a region near $O_\eps(x, t)$.	For this lemma, we will need to work with the scaling map $\tilde \beta_n$ used in \cref{T:BCJ18-strong}, introduced at the beginning of Section \ref{S:proof-t1}.
	
	\begin{lem}
	\label{L:large-paths}
	Fix $L > 0$, and consider the region
	$$
	\tt{A}_L = \tilde \beta_n([-L, L] \times [-L, \log^3 n]).
	$$
	Define
	$$
	K(L) = \max \{ i \in \N : \cS_{I-i}^- \text{ or } \cN_{I-i}^- \text{ enters the region } \tt{A}_L\}.
	$$
	Then
	$
	\P(K(L) \le \log^4 n) = 1 - o(1).
	$
	\end{lem}
	
	\begin{proof}
	First, with notation as in \cref{T:main-3-restatement}, let $K$ denote the connected component of
	$$
	\tt{A}_L \cup [(\tt{X} \cup \tt{PRS}^*_n)\setminus \tt{RS}^*_n]
	$$
	containing $\tt{A}_L$. The set $K$ is connected to the boundary of the Aztec diamond at the point $(n + 1, -n -1)$. Therefore by path ordering, if a north or south backbone path with index $I -i, i  > \log^4 n$ enters the region $\tt{A}_L$, then the same must be true of the path $\cN_{I - \lfloor \log^4 n \rfloor}^-$. By Part (3) of \cref{T:main-3-restatement}, this path can only enter the region $K$ through $\tt{A}_L$. Moreover, since $K$ is connected to the Aztec diamond boundary, if it enters the region $\tt{A}_L$, it first enters without winding around that region, and so its entry height is within $4$ of its initial height (\cref{L:tree-path-heights}). 
	
	Now, by parts $1$ and $6$ of \cref{T:main-3-restatement}, we have that $|I_n - n/4| \le \log n$ with probability $1 - o(1)$, and so $I_n - \lfloor \log^4 n \rfloor \le n/4 - \log^4 n/2$. By the relationship between path index and starting height, the starting height of the path $\cN_{I - \lfloor \log^4 n \rfloor}^-$ is bounded above by $-\log^4 n$. Therefore to complete the proof, it suffices to show that
	$$
	X = \min \{\cH_n(x) : x \in \tt{A}_L\}
	$$
	is bounded below by $-\log^4 n/2$ with probability $1 - o(1)$. We have that $X \ge X_0 - X_1 - X_2$, where
\begin{align*}
X_0 &= \cH_n(0,0) \\
X_1 &= \max \{|\cH_n(0,0) -\tilde \cH_n(t, \log^3 n)| : t \in [-L, L]\}, \\
X_2 &= \max \{|\tilde \cH_n(t, \log^3 n)- \tilde \cH_n(t, x)| : t \in [-L, L], x \in [-L, \log^3 n]\},
\end{align*}
and $\tilde \cH_n = \cH_n \circ \tilde \beta_n$.
First, $X_0$ converges in law as $n \to \infty$ by \cref{L:variance-middle-height} and so $X_0 \ge - \log^4 n/6$ with probability $1 - o(1)$. Next, as in the proof of \cref{L:height-shift} we can couple a dimer configuration $D \sim \P_{a, n}$ with two copies of the smooth phase measure $D_1, D_2 \sim \P_a$ using Lemmas \ref{L:global-coupling-1}, \ref{L:global-coupling-2} so that with probability $1 - o(1)$ we have $D_i|_{\Lambda_{\mathrm{sm},i}} = D|_{\Lambda_{\mathrm{sm},i}}$ for $i = 1, 2$. On the event where this coupling succeeds, using bounds on heights in the smooth phase (\cref{L:smooth-phase-heights-decor}), we have that
$$
\max \{|\cH_n(v) - \cH_n(0,0)| : v \in \Lambda_{\mathrm{sm},1} \cup \Lambda_{\mathrm{sm},2}\} \le  \log n
$$
with probability $1 - o(1)$, and so $X_1 \le \log^4 n/6$ with probability $1 - o(1)$. Finally, we show that $X_2 \le \log^4 n/6$ with probability $1 - o(1)$. To prove this, by Markov's inequality and a union bound over the $O(n \log^3 n)$-many lattice points in $\tt{A}_L$, it is enough to show that we can find $\lambda > 0$ such that
	\begin{equation}
		\label{E:lsnn}
\sup_{(t, x) \in [-L, L] \times [-L, \log^3 n]} \E \exp (\lambda |\tilde \cH_n(t, \log^3 n)- \tilde \cH_n(t, x)|) \le \exp(\log^3 n)
	\end{equation}
	for all large enough $n$.
	Suppose that \eqref{E:lsnn} fails. Then for all $\lambda > 0$, we can find sequences $t_{n,\lambda} \in [-L, L], x_{n,\lambda} \in [-L, \log^3 n]$ such that
	$$
\E \exp (\lambda |\tilde \cH_n(t_{n, \lambda}, \log^3 n) - \tilde \cH_n(t_{n, \lambda}, x_{n, \lambda})|) \ge \exp(\log^3 n)
	$$
for all $n$ in some infinite set $Y_\lambda$.
	We can apply diagonalization argument here to find $t_n, x_n, n \in Y$ that works simultaneously for all $\lambda > 0$, and then by passing to a further subsequence $Y' \subset Y$ we may assume that $t_n \to t \in [-L, L]$ and that $x_n \to x \in [-L, \infty]$. Now, letting 
	$$
	\phi_n = \{2\lfloor \ell \log^2 n \rfloor e_2 : \ell \in \{0, \dots, \lfloor \log^{1/2} n \rfloor \}\},
	$$
	then
	$$
\lim_{n \in Y'} \E \exp (\lambda |\tilde \cH_n^{\phi_n}(t_n, \log^3 n) - \tilde \cH_n^{\phi_n}(t_n, x_n)|) = \infty.
	$$
	for all $\lambda > 0$. This contradicts \cref{T:BCJ18-strong}.
	\end{proof}
	
	We now precisely state and prove \eqref{E:O-containment}.
	
	\begin{lem}
		\label{L:height-box-corridor}
		Fix $(x, t) \in \R^2, \eps > 0$, and let $O_\eps(x, t)$ be as in \cref{L:ring-bd}. Let $\phi_n$ be any mollifying sequence satisfying the conditions of \cref{T:main-1-restatement}. Then there exists a sequence $\sigma_n$ which tends to $0$ with $n$ (and depends on $\phi_n, x, t$) such that with probability $1 - o(1)$, we have that
		\begin{equation}
		\label{E:J-quation}
O_\eps(x, t) \subset U(\cS_{I_n-J}) \cap L(\cS_{I_n - J+1}),
		\end{equation}
		where $|J - (H_n - \cH_n^{\phi_n}(\hat \gamma_n(x, t)))/4| < \sigma_n.
		$
		In the special case when this forces $J = 0$, we remove the intersection with $L(\cS_{I_n - J+1})$.
	\end{lem}
	
	\begin{proof}
		All statements in the proof will hold with probability $1 - o(1)$, and we do not mention this further. We may also assume $\eps < 1/12$, and for ease of notation we set $u = \hat \gamma_n(x, t)$. 
		
	First, with notation as in \cref{L:large-paths}, the set $\tt{A}_{|t| + 1}$ has at most $\log^4 n$ south backbone paths that enter it: $\cS_{I-i}^-$, $i = 0, \dots, \lfloor \log^4 n \rfloor$. Here the upper bound on $i$ is from \cref{L:large-paths}, and the fact that no $\cS^+$-paths and no paths of the form $\cS_{I + i}, i > 0$ enter this region uses \cref{T:main-3-restatement}, parts 4 and 5, along with north and south path interlacing.
	
Next, by \cref{L:ring-bd}, none of the paths $\cS|_{I-i}^-$, $i = 0, \dots, \lfloor \log^4 n \rfloor$ enter the sets $O_\eps = O_\eps(x, t)$. Therefore there exists some index $J$ such that \eqref{E:J-quation} holds (where again, we omit the intersection with $L(\cS_{I_n - J+1})$ if $J = 0$). It remains to show that $J = (H_n - \cH_n^{\phi_n}(u)/4$. For this, consider the south forest path $\pi:[0, t] \to \R^2$ started at the south vertex adjacent to the $a$-face $u$, and let $\pi(s) \in \tt{S}$ denote the location when $\pi$ either first coalesces with a backbone path or exits the set $\tt{RS}_n^*$. Note that $\pi(0) \in U(\cS_{I- \lfloor \log^4 n \rfloor - 1}^-)$ since $\cS_{I- \lfloor \log^4 n \rfloor - 1}^-$ does not enter the set $\tt{A}_{|t| + 1}$ or the set $(\tt{X} \cup \tt{PRS}^*_n)\setminus \tt{RS}^*_n$ (the latter uses Part (3) of \cref{T:main-3-restatement}).

We first show that $\pi(s)$ is a coalescence point. Indeed, by the parabolic path estimate in Part (2) of \cref{T:main-3-restatement}, if $\pi(s)$ were in $\tt{RS}_n^*$, then it would be located on the lower boundary of $\tt{RS}_n^*$, at a point in the interior of $\tt{PRS}^*_n$. In particular, in this case the set
$$
T = \pi|_{[0, s]} \cup [(\tt{X} \cup \tt{PRS}^*_n)\setminus \tt{RS}^*_n] \cup \tt{A}_{|t| + 1}
$$
is connected, and $[-n - 1, n + 1]^2 \setminus T$ puts the south and west boundaries of the Aztec diamond in different components. As discussed above, the path  $\cS_{I- \lfloor \log^4 n \rfloor - 1}^-$ does not enter $T$, which contradicts that this path moves from the south to the west boundary.

Therefore $\pi(s)$ is a coalescence point. Moreover, since south forest paths enter $O_\eps$, $\pi(s)$ is not located in $O_\eps$, and again using the parabolic path estimate in Part (2) of \cref{T:main-3-restatement} together with the fact that the height of $\tt{RS}_n^*$ is less than $2n^{1/2} \log^{3/2} n$, we have that
$$
\pi(s) \in \tt{RS}_n^* \cap (u + [-n^{1/3 - \eps}/100, n^{1/3 - \eps}/100] \times (-\infty, -n^{1/3 - \eps}]).
$$
Here we have used that $\eps < 1/12$ to absorb the width of the parabola $P^{\log^2 n}_\tt{S}$ as it intersects $\tt{RS}_n^*$ (this width is $O(n^{1/4} \log^{3/4} n)$). Now, let $\cS_{I-K}, K \in \{0, \dots, \lfloor \log^4 n \rfloor + 1\}$ be the path containing $\pi(s)$. Arguing deterministically as in the proof of \eqref{E:face-ht} in Corollary \ref{C:height-path-equality}, we have that
$$
\cH_n(u) = 4 \operatorname{Wind}(\pi) + 4(I-K) - n - 1,
$$ 
and so by Part (6) of \cref{T:main-3-restatement} we have that
$$
H_n - \cH_n(u) = - 4 \operatorname{Wind}(\pi) + 4 K.
$$
Now, by \cref{L:local-coupling} we can couple our dimer configuration $D_n \sim \P_{a, n}$ with a dimer configuration $D \sim \P_a$ such that the two agree on the set $u + [-b_n, b_n]^2$, where $n^{1/100} \le b_n = o(n^{1/24})$ and $\phi_n \subset [-b_n, b_n]^2$. Let $\cH$ be the height function for $D$, and let $\pi_0$ be the south forest path in $D$ started at $\pi(0)$. On the event where this coupling holds, we have that 
$$
\cH_n(u) - \cH_n^{\phi_n} (u)= \cH(u) - \cH^{\phi_n} (u) = 4 \operatorname{Wind}(\pi_0) + o(1).
$$
Here the final equality uses \cref{L:height-lem}, and the $o(1)$ term goes to $0$ in probability as $n \to \infty$ (by Corollary \ref{C:reflection-symmetry-smooth}, \cref{L:smooth-phase-heights-decor}). Therefore to complete the proof, we must show that $\operatorname{Wind}(\pi_0) = \operatorname{Wind}(\pi)$ with high probability. First, $\pi_0$ will not cross vertical line 
$$
V = u + \{0\} \times [0, \infty)
$$
after it first exits the box $u + [-b_n, b_n]^2$ with high probability by the parabolic path estimate in \cref{L:lerw-estimate-basic}. Up until this point, $\pi_0$ and $\pi$ agree. Similarly, $\pi$ will not cross $V$ after it exits this box and before it joins the backbone path $\cS_{I-K}^-$. 

Next, if $\cS_{I-K}^-$ were to cross $V$ within the set $\tt{RS}_n$, then it would have a backtrack of size at least $n^{1/3 - \eps}/2$, since it cannot enter the set $O_\eps$. Since $\eps < 1/12$, this would contradict \cref{P:backtrack-estimate}. Therefore the barrier estimate in Part (3) of \cref{T:main-3-restatement} guarantees that the number of left-to-right crossings of the line $V$ by $\cS_{I-K}^-$ equals the number of right-to-left crossings, which finally implies that $\operatorname{Wind}(\pi_0) = \operatorname{Wind}(\pi)$, as desired.
	\end{proof}
	
At last, we are ready to prove \cref{P:key-prop}.

\begin{proof}[Proof of \cref{P:key-prop}]
Fix $\eps \in (0, 1/12), \delta \in (0, 1/2)$, and let 
$$
Z_\delta = [-\delta^{-1}, \delta^{-1}] \cap \delta \Z.
$$
Define $M:Z_\delta \to \N$ so that $M(x)$ is the nearest integer to the point $[H_n - \cH_n^{\phi_n}(\hat \gamma_n(x, t))]/4$, and define $J^\pm:Z_\delta \to \N$ so that 
\begin{align*}
J^-(x) = \min \{i = 0, 1, \dots: O_\eps(x, t) \subset U(\cS_{I_n - i}) \ne \emptyset\}, \\
J^+(x) = \max \{i = 1, \dots: O_\eps(x, t) \subset L(\cS_{I_n - i + 1}) \ne \emptyset\},
\end{align*}
where in the definition of $J^+$, we set the maximum to $0$ if the set above is empty.
By \cref{L:height-box-corridor}, $J^- = J^+$ and $|J^- - M| < \sigma_n$ with probability $1 - o(1)$ for a sequence $\sigma_n \to 0$ with $n$. 

Now, by the backtrack estimate in \cref{P:backtrack-estimate} and the fact that the width of $O_\eps(x, t)$ is of order $n^{1/3 - \eps}, \eps < 1/12$, the following implications hold simultaneously with probability $1 - o(1)$ for all $i \in \{0, \dots, \log n\}, x \in Z_\delta$:
\begin{align*}
O_\eps(x, t) &\subset U(\cS_{I_n + 1 - i}) \quad \implies \quad \sup_{|u| \le n^{-1/3-2\eps}} \cA_i^{n, +}(t + u) < x.  \\
O_\eps(x, t) &\subset L(\cS_{I_n + 1 - i}) \quad \implies \qquad \inf_{|u| \le n^{-1/3-2\eps}} \cA_i^{n, -}(t + u) > x.
\end{align*}
At this point, we fix $i$ and work deterministically on the event where the above implications hold, $J^- = J^+$, and $|J^- - M| < \sigma_n$ for $\sigma_n < \delta$. On this event, the two implications above give that
\begin{align}
	\label{E:suP}
\sup_{|u| \le n^{-1/3-2\eps}} \cA_i^{n, +}(t + u) &< \min \{x \in Z_\delta : J^- \le i -1 \}, \\
\label{E:inF}
\inf_{|u| \le n^{-1/3-2\eps}} \cA_i^{n, -}(t + u) &> \max \{x \in Z_\delta : J^+ \ge i \},
\end{align}
where the right-hand side of \eqref{E:suP} equals $-\infty$ if the set is empty, and the right-hand side of \eqref{E:inF} equals $-\infty$ if that set is empty. Using that $J^- = J^+$ and $|J^- - M| < \sigma_n < \delta$ then implies
	\begin{equation}
		\label{E:andede}
A_{n, \delta}(i, t) = \max \{x \in Z_\delta : J^+ \ge i \} = \min \{x \in Z_\delta : J^- \le i -1 \} - \delta
	\end{equation}
on the event $E = \{|A_{n, \delta}(i, t)| \le \delta^{-1} - 1\}$. The final event $E$ holds with probability $1 - o(1)$ as we first take $n \to \infty$ and then take $\delta \to 0$, since in this double limit, $A_{n, \delta}(i, t)$ converges in law to $\cA_i(t)$, which is a finite almost surely. Combining \eqref{E:suP}, \eqref{E:inF}, and \eqref{E:andede} gives the proposition.
\end{proof}

	\section{Correlation Kernel and Estimates}
	\label{subsec:KasteleynTheory}

	In this section, we introduce the Kasteleyn matrix as well as the relevance of its inverse.  We give a formula for the entries of the inverse of the Kasteleyn matrix of the two-periodic Aztec diamond, as well as formulas  which serve as a starting point for taking asymptotics of these entries. These results are refinements of the ones in \cite{CJ16}, and are proved in \Cref{subsec:Kinvasympproofs}. 
	
	\subsection{Inverse Kasteleyn matrix for the two-periodic Aztec diamond} 
	Suppose that $G$ is a planar weighted bipartite graph.  Recall that the Kasteleyn matrix, $K$, of the graph $G$ is a type of signed weighted adjacency matrix whose rows are indexed by the black vertices and whose columns are indexed by the white vertices of the graph, see \cite[Section 3]{Joh17} for more background. The importance of the Kasteleyn matrix is due to the following theorem. Here the underlying dimer measure assigns a dimer configurations a probability proportional to the product of edge weights.
	
	\begin{thm}[\cite{Ken97,Joh17}]\label{localstatisticsthm}
		Suppose that $E=\{\mathtt{e}_i\}_{i=1}^r$ is a collection of distinct edges on $G$ with $\mathtt{e}_i=(\mathtt{b}_i,\mathtt{w}_i)$, where $\mathtt{b}_i$ and $\mathtt{w}_i$ denote black and white vertices. The dimer configuration forms a determinantal point process on the edges of $G$ with correlation kernel $L$ meaning that the probability of observing edges $e_1,\dots, e_r$ is given by 
		$\det L(\mathtt{e}_i,\mathtt{e}_j)_{1 \leq i,j \leq r}$
		where
		$L(\mathtt{e}_i,\mathtt{e}_j) = K({\mathtt{b}_i,\mathtt{w}_i}) K^{-1}({\mathtt{w}_j,\mathtt{b}_i}).$
	\end{thm}
	More general versions of this theorem hold, but we do not need these here. The Kasteleyn matrix $K_{a,b}$ for the two-periodic Aztec diamond where $a$-faces and $b$-faces have weight $a, b > 0$ takes entries
	\begin{equation} \label{pf:K}\begin{split}
			&K_{a,b}(x,y)\\&=\left\{\begin{array}{ll}
				(b (1-j)+a j)(1-k)\mathrm{i}+(b j+a(1-j))k & \mbox{if } y=x+(-1)^j e_{k+1}, x \in \mathtt{E}, \\
				(b (1-j)+a j)k+(b j+a(1-j))(1-k)\mathrm{i} & \mbox{if } y=x+(-1)^j e_{k+1}, x \in \mathtt{W}, \\
				0 & \mbox{if $(x,y)$ is not an edge.}
			\end{array} \right.
		\end{split}
	\end{equation}
	where $\mathrm{i}^2=-1$ and $j,k\in \{0,1\}$. Here we only need to consider $K_{a, 1}$.  For historical reasons, e.g. see \cite{CJ16}, we set $n=4m$ for $m \in \mathbb{N}$, where $n$ is the size of the Aztec diamond, and will use this convention throughout and interchange between $n$ and $m$ for convenience. 
	
	We now proceed in stating formulas needed to give a suitable form for $K^{-1}_{a,1}$ as well as its asymptotics. 
	Let
	\begin{equation*} \label{ctilde}
		\tilde{c} (u_1,u_2)=2(1+a^2) + a(u_1+u_1^{-1})(u_2+u_2^{-1}),
	\end{equation*} 
	which is related to the so-called \emph{characteristic polynomial} for the dimer model~\cite{KOS03}; see~\cite[(4.11)]{CJ16} for an explanation. Write
	\begin{equation}
		h(\eps_1,\eps_2)=\eps_1(1-\eps_2)+\eps_2(1-\eps_1),
	\end{equation}
	and let  $\Gamma_R$ denote a positively oriented circle of radius $R$ around the origin, where we adopt this notation for the rest of the paper.  In what follows below, we set $\mathtt{W}_0=\mathtt{S}$,$\mathtt{W}_1=\mathtt{N}$, $\mathtt{B}_0=\mathtt{W}$, and $\mathtt{B}_1=\mathtt{E}$, along with the full-plane analogs using $\tilde{\mathtt{W}}_0$, etc.
	The full-plane smooth phase inverse Kasteleyn matrix is given by
	\begin{equation*}\label{smoothphaseeqn}
		\mathbb{K}^{-1}_{1,1}(x,y)=-\frac{\mathrm{i}^{1+h(\eps_x,\eps_y)}}{(2 \pi \mathrm{i})^2} 
		\int_{\Gamma_1} \frac{du_1}{u_1}  \int_{\Gamma_1} \frac{du_2}{u_2} \frac{a^{\eps_y} u_2^{1-h(\eps_x,\eps_y)} +a^{1-\eps_y} u_1 u_2^{h(\eps_x,\eps_y)}}{\tilde{c}(u_1,u_2) u_1^{\frac{x_1-y_1+1}{2}} u_2^{\frac{x_2-y_2+1}{2}}},
	\end{equation*}
	where 
	$x=(x_1,x_2)\in \tilde{\mathtt{W}}_{\eps_x}$ and $y=(y_1,y_2)\in \tilde{\mathtt{B}}_{\eps_y}$ with $\eps_x,\eps_y \in \{0,1\}$; see~\cite[Section 4]{CJ16} for details and connections with~\cite{KOS03}. Define
	\begin{equation*}
		\sqrt{\omega^2+2c} =e^{\frac{1}{2}\log (\omega + \mathrm{i} \sqrt{2c}) + \frac{1}{2}\log (\omega - \mathrm{i} \sqrt{2c})}
	\end{equation*}
	for $\omega \in \mathbb{C}\backslash \mathrm{i}[-\sqrt{2c},\sqrt{2c}]$, where the logarithm in the exponent has arguments in $(-\pi/2,3\pi/2)$. This is the same choice of branch cut given in \cite{CJ16}.  For $s\in \mathbb{Z}$, define
	\begin{equation}\label{eq:prev:Fs}
		F_s(w)=\frac{1}{2\pi \mathrm{i}} \int_{\Gamma_1} \frac{u^s}{1+a^2+a w(u+u^{-1})} \frac{\mathrm{d}u}{u}.
	\end{equation}
	Introduce 
	\begin{equation}
		G(\omega)=\frac{1}{\sqrt{2c}}(\omega-\sqrt{\omega^2+2c}) \label{eq:def:G}
	\end{equation}
	where $c=a/(1+a^2)$
	and let 
	\begin{equation}\label{eq:prev:mu_and_s}
		\mu(u)=1-\sqrt{1+c^2(u-u^{-1})^2} \hspace{5mm}\mbox{ and }s(u)=1-\mu(u)
	\end{equation}
	with the sign chosen so that $G(1/\omega)=-\mu(-u\mathrm{i})/(c(u+u^{-1}))\mathrm{i}$ if we set $u \mathrm{i}=G(\omega)$.  With the choice of branch cut, we have $\sqrt{(-w)^2+2c}=-\sqrt{w^2+2c}$ and so $G(-\omega)=-G(\omega)$ and for $t \in(\sqrt{2c},1/\sqrt{2c})$,
	\begin{equation*}\label{eq:def:branchcutati}
    \begin{split}
        \sqrt{(\mathrm{i} t)^2+2c}&=\mathrm{i}\sqrt{t^2-2c} \hspace{5mm}  \mbox{and} \\
        \hspace{5mm} \sqrt{\frac{1}{(\mathrm{i}t)^2}+2c}&=\sqrt{(-\mathrm{i}\frac{1}{t})^2+2c}=-\mathrm{i}\sqrt{t^{-2}-2c}.
    \end{split}
	\end{equation*}
	With this in mind, we use the notation
	\begin{equation}\label{eq:prev:theta}
		\theta(t)=\frac{t}{\sqrt{t^2-2c}}\hspace{5mm} \mbox{for} \hspace{5mm} t \in (\sqrt{2c},1/\sqrt{2c}).
	\end{equation}
	We now bring forward the following theorem from~\cite{CJ16}, which is the formula for $K^{-1}_{a,1}$ for the two-periodic Aztec diamond. 
	\begin{thm}[Theorem 2.3 in ~\cite{CJ16}] \label{thm:prev:Kinverse}
		For $n=4m$, $x=(x_1,x_2) \in \mathtt{W}_{\eps_1}$ and $y=(y_1,y_2) \in \mathtt{B}_{\eps_2}$ with $\eps_1,\eps_2 \in \{0,1\}$,
		the entries of $K^{-1}_{a,1}$ are given by 
		\begin{equation*} \label{thm:eq:prev:Kinverse}
			\begin{split}
				&K_{a,1}^{-1}(x,y)= \mathbb{K}^{-1}_{1,1} (x,y)
				-\Big( \mathcal{B}_{\eps_1,\eps_2}(a,n+x_1,n+x_2,n+y_1,n+y_2) 
				\\& -\frac{\mathrm{i}}{a} (-1)^{\eps_1+\eps_2} \big( \mathcal{B}_{1-\eps_1,\eps_2}(1/a,n-x_1,n+x_2,n-y_1,n+y_2) \\&   +\mathcal{B}_{\eps_1,1-\eps_2}(1/a,n+x_1,n-x_2,n+y_1,n-y_2) \big) \\
				&  +\mathcal{B}_{1-\eps_1,1-\eps_2}(a,n-x_1,n-x_2,n-y_1,n-y_2) \Big),
			\end{split}
		\end{equation*}
		where for $a<r<1/a$, we have
		\begin{equation}\label{eq:prev:B:first}
			\begin{split}
				&\mathcal{B}_{\eps_1,\eps_2}(a,x_1,x_2,y_1,y_2) =
				-\frac{\mathrm{i}^{\eps_1+\eps_2+1}}{(2 \pi \mathrm{i})^2} \int _{\Gamma_r} \frac{\mathrm{d}u_1}{u_1}\int_{\Gamma_r} \frac{\mathrm{d}u_2}{u_2} \frac{ F_{\frac{x_2}{2}}(\nu (u_1)) F_{\frac{y_1}{2}} (\nu(u_2))}{u_1^{\frac{x_1-1}{2}} u_2^{\frac{y_2-1}{2}}}
				\\
				&\times \left(\frac{1}{4 c^2} u_1^2 u_2^2 (2-\mu(-\mathrm{i}u_1))^2( 2-\mu(-\mathrm{i}u_2))^2 \right)^{m}
				\sum_{\gamma_1,\gamma_2=0}^1Y^{\eps_1,\eps_2}_{\gamma_1,\gamma_2}(u_1,u_2).
			\end{split}
		\end{equation}
		The function $Y_{\gamma_1,\gamma_2}^{\eps_1,\eps_2}$ is defined in~\cref{sec:appendix:Y}.
	\end{thm}
	Note that we will sometimes write $\mathcal{B}_{\eps_1,\eps_2}(a,x_1,x_2,y_1,y_2)=\mathcal{B}_{\eps_1,\eps_2}(a,(x_1,x_2),(y_1,y_2))$ for convenience.

	\subsection{Formulas for $K_{a,1}^{-1}$ for asymptotic analysis}
	
	In this subsection, we give suitable formulas for $K^{-1}_{a,1}$ that are amenable to asymptotic analysis.  Much of the notation is the same as the one used in \cite{CJ16}. Define
	\begin{equation} \label{eq:prev:H}
		{ H_{x_1,x_2}(\omega)}=\frac{\omega^{2m}  G \left(\omega \right)^{-\frac{x_1}{2}}}{G \left(\omega^{-1} \right)^{-\frac{x_2}{2}}},
	\end{equation}
	where $-n < x_k,x_2 < n$. 
	Let
	\begin{equation} \label{eq:prev:V}
		\begin{split}
			& V_{\eps_1,\eps_2}(\omega_1, \omega_2)=\frac{1}{2}\sum_{\gamma_1,\gamma_2=0}^1  (-1)^{\eps_1 +\eps_2 +\eps_1 \eps_2+ \gamma_1 (1+\eps_2) +\gamma_2(1+\eps_1) }
			G\left( \omega_1 \right) ^{\eps_1}G\left( \omega_2^{-1}  \right)^{\eps_2}  \\
			& \times s\left( G\left( \omega_1 \right)\right)^{\gamma_1}   s \left( G\left( \omega_2^{-1} \right)\right)^{\gamma_2}
			\left(
			\mathtt{x}^{\eps_1, \eps_2}_{\gamma_1,\gamma_2} \left(  \omega_1 , \omega_2^{-1}   \right)
			- \mathtt{x}^{ \eps_1, \eps_2}_{\gamma_1,\gamma_2} \left(  \omega_1 ,- \omega_2^{-1}   \right) \right),
		\end{split}
	\end{equation}
	\begin{equation} \label{eq:prev:s}
		s \left( G\left( \omega \right)\right)=\omega \sqrt{\omega^{-2}+2c} \hspace{2mm} \mbox{and} \hspace{2mm}
		s \left( G\left( \omega^{-1} \right)\right)=\frac{1}{\omega} \sqrt{\omega^{2}+2c},
	\end{equation}
	with 
	\begin{equation}\label{eq:prev:x}
		\begin{split}
			&\mathtt{x}^{\eps_1,\eps_2}_{\gamma_1,\gamma_2} \left(  \omega_1 , \omega_2   \right) =   \frac{ G\left( \omega_1 \right)G\left( \omega_2 \right)}{\prod_{j=1}^2 \sqrt{ \omega_j^2+2c} \sqrt{ \omega_j^{-2}+2c} }\\
			&\times \mathtt{y}_{\gamma_1,\gamma_2}^{\eps_1,\eps_2} \left( a,1, G\left( \omega_1 \right),G\left( \omega_2 \right)  \right) \left( 1- \omega_1^2 \omega_2^2 \right),
		\end{split}
	\end{equation}
	and $\mathtt{y}_{\gamma_1,\gamma_2}^{\eps_1,\eps_2}$ is given in~\cref{sec:appendix:Y}.  
	We will also define
	\begin{equation}\label{eq:T_G}
	T_G(\omega_1,\omega_2)=\frac{G(\omega_1)G(\omega_2^{-1})}{G(\omega_1^{-1})G(\omega_2)}.
	\end{equation}
	and notice that $T_G(\omega_1,\omega_2)=T_G(\omega_2,\omega_1)^{-1}$.

	As given in \cite[Lemma 3.2]{CJ16}, there is a remarkably simple formula for $V_{\eps_1,\eps_2}(\omega, \omega)$ which is given by 
	\begin{equation}\label{eq:prev:Vww}
		V_{\eps_1,\eps_2}(\omega, \omega) =\frac{ (-1)^{1+h(\eps_1,\eps_2)} a^{\eps_2}  G\left( \omega^{-1} \right)^{h(\eps_1,\eps_2)} + a^{1-\eps_2} G\left( \omega \right)  G\left( \omega^{-1} \right)^{1-h(\eps_1,\eps_2)} }{2(1+a^2) \sqrt{\omega^2+2c} \sqrt{\omega^{-2}+2c} }.
	\end{equation}
	Introduce $\omega_c =r_c e^{\mathrm{i} \theta_c}$, $\theta_c \in [0,\pi/2]$ for some $r_c>0$ an let $\Gamma_{\omega_c}$ be the contour 
	\begin{equation*}\label{eq:prev:Gammawc}
		\Gamma_{\omega_c}\,:\,[\theta_c,\pi-\theta_c]\cup[\pi+\theta_c,2\pi-\theta_c]\ni\theta\to r_c e^{\mathrm{i}\theta}.
	\end{equation*}
	For $x=(x_1,x_2) \in \mathtt{W}_{\eps_1}$ and $y=(y_1,y_2) \in \mathtt{B}_{\eps_2}$ we
	define
	\begin{equation}\label{eq:prev:Comega}
		C_{\omega_c}(x,y)=\frac{\mathrm{i}^{(x_2-x_1+y_1-y_2)/2}}{2\pi\mathrm{i}}\int_{\Gamma_{\omega_c}}
		V_{\eps_1,\eps_2}(\omega,\omega) G(\omega)^{\frac{y_1-x_1-1}{2}}G(\omega^{-1})^{\frac{x_2-y_2-1}{2}}\frac{\mathrm{d}\omega}{\omega}.
	\end{equation}
	
	The following lemma collects asymptotics for $K^{-1}_{a,1}(x, y)$ when the points $x, y$ are close to the diagonal in the third quadrant. It gives a starting point for the asymptotic analysis of $K^{-1}_{a,1}$ given in this paper.  It is pieced together from various statements in \cite{CJ16}, which we gather in a short proof in \cref{subsec:Kinvasympproofs}.  Note that we do not need an explicit form for $S(x,y)$ given below, only what is listed below.  We have kept the notation, up to a shift, similar to that used in \cite{CJ16}.

	\begin{lem}\label{lem:switchcontoursBtilde}
		For $x=(x_1,x_2)=(\lfloor 4m\eta_1 \rfloor +\overline{x}_1,\lfloor 4m\eta_1 \rfloor+\overline{x}_2)\in \mathtt{W}_{\eps_1}$ and $y=(y_1,y_2)=(\lfloor 4m\eta_2 \rfloor+\overline{y}_1,\lfloor 4m\eta_2 \rfloor +\overline{y}_2)\in \mathtt{B}_{\eps_2}$ with $\eps_1,\eps_2 \in\{0,1\}$ where $-1<\eta_1,\eta_2<0$ and $|\overline{x}_1|,|\overline{x}_2|,|\overline{y}_1|,|\overline{y}_2|  \le C m^{5/6}$ for some $C>0$ constant, we have for a constant $\mathtt{c}>0$
		\begin{equation}
        \label{E:first-eq-10.3}
			K^{-1}_{a,1}(x,y)=\mathbb{K}^{-1}_{1,1}(x,y)-\mathcal{B}_{\eps_1,\eps_2}(a,x+(n,n),y+(n,n))+O(e^{-\mathtt{c}n})
		\end{equation}
		where
		\begin{equation}\label{eq:prev:B}
			\begin{split}
				&\mathcal{B}_{\eps_1,\eps_2}(a,n+x_1,n+x_2,n+y_1,n+y_2)=\frac{\mathrm{i}^{\frac{x_2-x_1+y_1-y_2}{2}}}{(2\pi \mathrm{i})^2} \int_{\Gamma_p} \frac{\mathrm{d}\omega_1}{\omega_1} \int_{\Gamma_{1/p}} \mathrm{d}\omega_2 \\
				&\times \frac{V_{\eps_1,\eps_2}(\omega_1,\omega_2)}{\omega_2 - \omega_1} \frac{ H_{x_1+1,x_2}(\omega_1)}{H_{y_1,y_2+1}(\omega_2)}
			\end{split}
		\end{equation}
		where $V_{\eps_1,\eps_2}(\omega_1,\omega_2)$ is given in \eqref{eq:prev:V}. Moreover, we have
		\begin{equation} \label{eq:lem:switchcontours}
			\mathcal{B}_{\eps_1,\eps_2}(a,x+(n,n),y+(n,n))=C_1(x,y)+\tilde{\mathcal{B}}_{\eps_1,\eps_2}(a,x+(n,n),y+(n,n))
		\end{equation}
		where for $\sqrt{2c}<p<1$,
		\begin{equation*}
			\begin{split}\label{eq:prev:Btilde}
				&\tilde{\mathcal{B}}_{\eps_1,\eps_2}(a,x+(n,n),y+(n,n))=\frac{\mathrm{i}^{\frac{x_2-x_1+y_1-y_2}{2}}}{(2\pi \mathrm{i})^2} \int_{\Gamma_{1/p}} \frac{\mathrm{d}\omega_1}{\omega_1} \int_{\Gamma_{p}} \mathrm{d}\omega_2 \\
				&\times \frac{V_{\eps_1,\eps_2}(\omega_1,\omega_2)}{\omega_2 - \omega_1} \frac{ H_{x_1+1,x_2}(\omega_1)}{H_{y_1,y_2+1}(\omega_2)}.
			\end{split}
		\end{equation*}
		We also have that 
		\begin{equation}
        \label{E:KCS}
			\mathbb{K}^{-1}_{1,1}(x,y)=C_1(x,y)+S(x,y)
		\end{equation}
		where $S(x,y)$ is defined in \cite[Equation (2.18)]{CJ16}.  We have that $S(x,y)=0$ if $y_1-x_1>-1$ or $x_2-y_2>-1$. Also, if $f \in \{\pm e_1,\pm e_2\}$ and $x \in \mathtt{W}_{\eps_1}$, then $S(x,x+f)\not=0$ only if $f=e_2$ and then $S(x,x+e_2)=-\mathrm{i}a^{\eps_1-1}$.  
		
	\end{lem}

	Finally, we need the following lemma, which contains versions of the asymptotic equation \eqref{E:first-eq-10.3} when the points $x, y$ are close to the anti-diagonal or close to the diagonal in quadrants $1$. It is again essentially contained in \cite{CJ16}, and we give a brief proof in \cref{subsec:Kinvasympproofs}.
	
	\begin{lem}\label{L:Kinv:rotationestimate}
		For $x=(x_1,x_2)=(-\lfloor 4m\eta_1 \rfloor +\overline{x}_1, \lfloor 4m\eta_1 \rfloor +\overline{x}_2)\in \mathtt{W}_{\eps_1}$ and $y=(y_1,y_2)=(- \lfloor 4m\eta_2 \rfloor +\overline{y}_1, \lfloor 4m\eta_2 \rfloor +\overline{y}_2)\in \mathtt{B}_{\eps_2}$ with $\eps_1,\eps_2 \in\{0,1\}$ where $-1<\eta_1,\eta_2<0$ and $|\overline{x}_1|,|\overline{x}_2|,|\overline{y}_1|,|\overline{y}_2| \le C m^{5/6}$ for some constant $C>0$, we have for a constant $\mathtt{c}>0$
		\begin{equation*}
			K^{-1}_{a,1}(x,y)=\mathbb{K}^{-1}_{1,1}(x,y)+\frac{\mathrm{i}}{a}(-1)^{\eps_1+\eps_2}\mathcal{B}_{1-\eps_1,\eps_2}(a,n-x_1,n+x_2,n-y_1,n+y_2)+O(e^{-\mathtt{c}n}).
		\end{equation*}
		For $x=(x_1,x_2)=(\lfloor 4m\eta_1 \rfloor +\overline{x}_1,- \lfloor 4m\eta_1 \rfloor +\overline{x}_2)\in \mathtt{W}_{\eps_1}$ and $y=(y_1,y_2)=( \lfloor 4m\eta_2 \rfloor +\overline{y}_1,- \lfloor 4m\eta_2 \rfloor +\overline{y}_2)\in \mathtt{B}_{\eps_2}$ with $\eps_1,\eps_2 \in\{0,1\}$ where $-1<\eta_1,\eta_2<0$ and $|\overline{x}_1|,|\overline{x}_2|,|\overline{y}_1|,|\overline{y}_2|\leq Cm^{5/6}$ for some constant $C>0$, we have for a constant $\mathtt{c}>0$
		\begin{equation*}
			K^{-1}_{a,1}(x,y)=\mathbb{K}^{-1}_{1,1}(x,y)+\frac{\mathrm{i}}{a}(-1)^{\eps_1+\eps_2}\mathcal{B}_{\eps_1,1-\eps_2}(a,n+x_1,n-x_2,n+y_1,n-y_2)+O(e^{-\mathtt{c}n}).
		\end{equation*}
		For $x=(x_1,x_2)=(-\lfloor 4m\eta_1 \rfloor +\overline{x}_1,- \lfloor 4m\eta_1 \rfloor +\overline{x}_2)\in \mathtt{W}_{\eps_1}$ and $y=(y_1,y_2)=(-\lfloor 4m\eta_2 \rfloor +\overline{y}_1,- \lfloor 4m\eta_2 \rfloor +\overline{y}_2)\in \mathtt{B}_{\eps_2}$ with $\eps_1,\eps_2 \in\{0,1\}$ where $-1<\eta_1,\eta_2<0$ and $|\overline{x}_1|,|\overline{x}_2|,|\overline{y}_1|,|\overline{y}_2|\leq Cm^{5/6}$ for some constant $C>0$, we have for a constant $\mathtt{c}>0$
		\begin{equation*}
			K^{-1}_{a,1}(x,y)=\mathbb{K}^{-1}_{1,1}(x,y)-\mathcal{B}_{1-\eps_1,1-\eps_2}(a,n-x_1,n-x_2,n-y_1,n-y_2)+O(e^{-\mathtt{c}n}).
		\end{equation*}
	\end{lem}

	\subsection{Scalings}
	
	In this subsection, we give the scalings needed for asymptotic analysis. 
	Recall that  $\xi_c=-\frac{1}{2}\sqrt{1-2c}$.  Let $\xi = (\xi_1, \xi_2)$ and $(\tilde{\xi}_1, \tilde{\xi}_2)$ be on the rough-smooth limit shape curve (given in \eqref{eq:asymp:limitshape}) with $|\xi_1 -\xi_2|<m^{-\frac{1}{6}}$ and $|\tilde{\xi}_1 -\tilde{\xi}_2|<m^{-\frac{1}{6}}$ and $\xi_1,\xi_2,\tilde{\xi}_1,\tilde{\xi}_2<0$. We consider the following scalings for 
	\begin{equation}
		\begin{array}{l} \label{eq:def:scalingsxy}
			x=(\lfloor 4m\xi_1+4mX\rfloor+1,\lfloor 4m\xi_2+4mX \rfloor), \\
			y=(\lfloor 4m\tilde{\xi}_1+4mY \rfloor ,\lfloor 4m\tilde{\xi}_2+4mY \rfloor+1).
		\end{array}
	\end{equation}
	For this, fix $L>0$, $0<\delta <\frac{1}{12}$ and $\gamma >0$. 
	\begin{align}
		&\mbox{\emph{(Smooth region)}}  \nonumber\\
		&\hspace{30mm}m^{-\frac{1}{6}+\delta} \leq X,Y \leq -\xi_c + \gamma \label{eq:def:sm}\\
		&\mbox{\emph{(Smooth region close to rough-smooth boundary)}} \nonumber\\  
		&\hspace{30mm}m^{-\frac{2}{3}}L\leq X,Y < m^{-1/6+\delta}  \label{eq:def:smclosetoRS}\\
		&\mbox{\emph{(Rough-smooth boundary)}}  \nonumber \\
		&\hspace{30mm}-m^{-\frac{2}{3}}L\leq X,Y < m^{-\frac{2}{3}}L \label{eq:def:rs} \\
		&\mbox{\emph{(Rough region close to rough-smooth boundary)}}  \nonumber\\
		&\hspace{30mm}-m^{-\frac{1}{2}+\delta}\leq X,Y < -m^{-\frac{2}{3}}L \label{eq:def:rgclosetoRS}.
	\end{align}
	We will also consider a refined rough-smooth boundary scaling
	\begin{equation}
		\begin{split}
			\label{eq:def:rsrefined}
			&x=\big(\lfloor 4m\xi_c \rfloor +2 \lfloor \alpha_x \lambda_1 (2m)^{\frac{1}{3}}\rfloor)e_1\\
			&-\big(2\lfloor \hat{\beta}_x \lambda_2 (2m)^{\frac{2}{3}} +k_x \lambda_2 \log^2 m\rfloor \big)e_2 +f_x
		\end{split}
	\end{equation}
	where $|\hat{\beta}_x|,|f_x| <C$ and $-C<\alpha_x<\log^3 m$ for a constant $C>0$, $1\leq k_x \leq M$ where  $M=(\log m)^2 $. We have a similar scaling for $y$ with $\hat{\beta}_y,f_y, k_y, \alpha_y$.  We always choose $f_x, f_y$ so that $x$ and $y$ indicate will always be white and black vertices respectively. 
	\subsection{Asymptotics of the Correlation Kernel}\label{subsubsec:Asymptotics}
	
     We now give the refined asymptotics for $K^{-1}_{a,1}$ for the scalings given in the previous subsection.  To state our results, we need some details on the saddle point functions, giving a brief description of their origins.  In what follows below, we ignore integer parts since these are negligible. 
    
    First, if we set $x$ and $y$ as given in \eqref{eq:def:scalingsxy} we have that 
	\begin{equation}\label{eq:def:logH}
		\log H_{x_1+1,x_2}(\omega) =2m g_{\xi_1+X,\xi_2+X}(\omega) -\log G(\omega) 
	\end{equation}
	and 
	\begin{equation}\label{eq:def:logHy}
		\log H_{y_1,y_2+1}(\omega) =2m g_{\tilde{\xi}_1+Y,\tilde{\xi}_2+Y}(\omega) +\log G(\omega^{-1})
	\end{equation}
	where 
	\begin{equation} \label{eq:prev:gxi12}
		g_{\xi_1,\xi_2}(\omega) = \log \omega - \xi_1 \log G(\omega) +\xi_2 \log G(\omega^{-1})
	\end{equation}
	and the function $G$ is defined in \eqref{eq:def:G}. Recalling \eqref{eq:prev:B} and using \eqref{eq:def:logH} and \eqref{eq:def:logHy}, we have
	\begin{equation}
		\begin{split} \label{thmproof:aroundRSsm:B1}
			&\mathcal{B}_{\eps_1,\eps_2}(a,n+x_1,n+x_2,n+y_1,n+y_2)=\frac{\mathrm{i}^{\frac{x_2-x_1+y_1-y_2}{2}}}{(2\pi \mathrm{i})^2} \int_{\Gamma_p} \frac{\mathrm{d}\omega_1}{\omega_1} \int_{\Gamma_{1/p}} \mathrm{d}\omega_2 \\
			&\times \frac{V_{\eps_1,\eps_2}(\omega_1,\omega_2)}{\omega_2 - \omega_1} e^{2m g_{\xi_1+X,\xi_2+X}(\omega_1) -\log G(\omega_1) }e^{-2m g_{\tilde{\xi}_1+Y,\tilde{\xi}_2+Y}(\omega_2) -\log G(\omega_2^{-1})}.
		\end{split}
	\end{equation}
	To understand the asymptotic behvavior of this integral, we need to understand the saddle points of $g_{\xi_1,\xi_2}$. By differentiating~\eqref{eq:prev:gxi12} we arrive at 
	\begin{equation} \label{eq:def:gprime}
		\omega g'_{\xi_1,\xi_2}(\omega)=1+ \frac{\xi_1 \omega}{\sqrt{\omega^2+2c}} +\frac{\xi_2}{\omega \sqrt{\omega^{-2}+2c}}
	\end{equation}
	and the equation $\omega g'_{\xi_1,\xi_2}(\omega)=0$ has at most four roots in $\mathbb{C}\backslash(\mathrm{i}(-\infty,-1/\sqrt{2c}]\cup \mathrm{i} [-\sqrt{2c},\sqrt{2c}] \cup \mathrm{i} [1/\sqrt{2c},\infty))$; see \cite[Lemma 3.7]{CJ16}. 
    Note that the equation is symmetric under $\omega \mapsto -\omega$ and $\omega \mapsto \mathrm{i}\omega$ by the choice of square root.   We denote $\mathbb{H}$ to be the upper half-plane and set $\mathbb{H}_+=\{x+ y \mathrm{i}:x,y>0\}$ to be the first quadrant.

	The critical points of $g_{\xi_1,\xi_2}$ determine the macroscopic region which $(\xi_1,\xi_2)$ belongs to. The analysis in \cite{CJ16} was performed for $\xi_1=\xi_2$ to make the analysis explicit, showing that  the four roots of \eqref{eq:def:gprime} are given by $\omega_c, \omega^{-1}_c, -\omega_c$ and $-\omega_c^{-1}$ where $\omega_c \in \mathbb{H}_+\cup \mathbb{R}_{\geq 1} \cup [-\sqrt{2c},1]\mathrm{i}$.  In this case, they found that $(\xi_1,\xi_1)$ depends continously on $\omega_c$ with
   \begin{itemize}
\item  if $\omega_c \in (\sqrt{2c},1)\mathrm{i} $, then $(\xi_1,\xi_1)$ is in the smooth region (single critical point),
      \item if $\omega_c =\mathrm{i} $, then $(\xi_1,\xi_1)$ is at the smooth-rough boundary,  
 \item  if $\omega_c \in \mathbb{H}_+ $ with $|\omega_c|=1$, then $(\xi_1,\xi_1)$ is in the rough region (single critical point),
 \item if $\omega_c =1 $, then $(\xi_1,\xi_1)$ is at the frozen-rough boundary,
 \item if $\omega_c >1$, the $(\xi_1,\xi_1)$ is in the frozen region. 
   \end{itemize}

   Since we work slightly away from the main diagonal in this paper,  we adapt some of our notation accordingly.  Moreover, we use 
the notation that
	\begin{equation}\label{eq:def:omegacX}
		\begin{array}{lll}
			\omega_{c,X}\in \mathbb{H}_+\cup \partial \mathbb{H}_+ &\mbox{such that}& g'_{\xi_1+X,\xi_2+X}(\omega_{c,X})=0\\
			\tilde{\omega}_{c,Y}\in \mathbb{H}_+\cup \partial \mathbb{H}_+&\mbox{such that}& g'_{\tilde{\xi}_1+Y,\tilde{\xi}_2+Y}(\tilde{\omega}_{c,Y})=0,
		\end{array}
	\end{equation} 
   where $X$ and $Y$ are marking the distance from $(\xi_1,\xi_2)$ on the limit shape curve at the rough-smooth boundary with $|\xi_1-  \xi_2|< Cm^{-1/6}$. The point $\omega_{c, X}$ will be a perturbation of  $\omega_c$ at $(\xi_1 + X, \xi_1 + X)$, whereas $\tilde \omega_{c, Y}$ will be a perturbation of  $-\omega_c^{-1}$ at $(\xi_1+Y, \xi_1+Y)$.

	Our choice in notation means that $\omega_{c,0}$, that is where $X=0$, is such that  $g_{\xi_1,\xi_2}'(\omega_{c,0})=g_{\xi_1,\xi_2}''(\omega_{c,0})=0$ meaning that $(\xi_1,\xi_2)$ is on the limit shape curve. Similarly $\tilde{\omega}_{c,0}$, that is where $Y=0$, is such that  $g_{\tilde{\xi}_1,\tilde{\xi}_2}'(\tilde{\omega}_{c,0})=g_{\tilde{\xi}_1,\tilde{\xi}_2}''(\tilde{\omega}_{c,0})=0$ meaning that $(\tilde{\xi}_1,\tilde{\xi}_2)$ is on the limit shape curve.  
	
	Note that the analysis in \cite{DK21} holds for the whole Aztec diamond. We could have computed refined asymptotic
	formulas found in the latter here, due to the relation found in \cite{CD23} between the inverse Kasteleyn matrix entries and the particle system used in \cite{DK21}. For simplicity, we refine the asymptotics of the formulas used in \cite{CJ16}.

    We now state asymptotics for $\mathcal{B}_{\eps_1,\eps_2}$ when $x, y$ are scaled as in \eqref{eq:def:scalingsxy}. The proof of the upcoming theorem is contained in \cref{subsec:thm:KinversearoundRS}.

	\begin{thm} \label{thm:prev:KinversearoundRS}
		Assume that $x=(x_1,x_2)\in \mathtt{W}_{\eps_1}$ and $y=(y_1,y_2) \in \mathtt{B}_{\eps_2}$ with $\eps_1,\eps_2 \in \{0,1\}$. There exist positive constants $\mathtt{C}_1,\mathtt{C}_2,\mathtt{C}_3>0$ depending only on $a$, such that
		\begin{enumerate}
			\item If $x$ and $y$ have scaling given by ~\eqref{eq:def:sm} and $0 < \gamma < \mathtt{C}_2$, then 
			\begin{equation} \label{thm:eq:KinversearoundRS:sm}
				|\mathcal{B}_{\eps_1,\eps_2} (a,n+x_1,n+x_2,n+y_1,n+y_2) | \leq \mathtt{C}_1e^{-\mathtt{C}_2 m^{\frac{5}{6}+\delta}}. 
			\end{equation}
			
			\item If $x$ and $y$ have scaling given by ~\eqref{eq:def:smclosetoRS}, then 
			\begin{equation}\label{thm:eq:KinversearoundRS:smclosetoRS}
				\begin{split}
					&|\mathcal{B}_{\eps_1,\eps_2} (a,n+x_1,n+x_2,n+y_1,n+y_2) |\\& \leq \mathtt{C}_1 \bigg|\frac{H_{x_1+1,x_2}(\omega_{c,0})}{H_{y_1,y_2+1}(\tilde{\omega}_{c,0})}\bigg| e^{-m(\mathtt{C}_2X^{3/2}-\mathtt{C}_3Y^{3/2})} .
				\end{split}
			\end{equation}
			\item If $x$ and $y$ both have scaling given by~\eqref{eq:def:rs} and $|\xi_1-\tilde{\xi}_1|<m^{-1/3}$, then 
			\begin{equation}
				\begin{split}\label{thm:eq:KinversearoundRS:RS}
					&|\mathcal{B}_{\eps_1,\eps_2} (a,n+x_1,n+x_2,n+y_1,n+y_2)| \leq \mathtt{C}_1 (2m)^{-\frac{1}{3}}.
				\end{split}
			\end{equation}
			
			\item If $x$ and $y$ both have scaling given by ~\eqref{eq:def:rgclosetoRS}, $|\xi_1-\tilde{\xi}_1|<Cm^{-1}$, and $|X-Y|<Cm^{-1}$ for some $C>1$ constant, then
			\begin{equation}\label{eq:thm:KinversearoundRS:rough}
				\mathcal{B}_{\eps_1,\eps_2}(a,n+x_1,n+x_2,n+y_1,n+y_2)= C_{\omega_{c,X}}(x,y)+O(m^{-\frac{1}{2}})
			\end{equation}
			where $C_{\omega_{c,X}}$ is defined in \eqref{eq:prev:Comega} and $\mathrm{Re}\,\omega_{c,X}$ is order $\sqrt{X}$. 
		\end{enumerate}
		
	\end{thm}
	
	\begin{rem}\label{rem:mixedterm}
		If one of $x$ and $y$ have scaling ~\eqref{eq:def:sm}, then we also have the bound given in \eqref{thm:eq:KinversearoundRS:sm}. The proof of this follows by the same saddle point analysis for one of the variables in the proof of \eqref{thm:eq:KinversearoundRS:sm} (and another saddle point analysis for the other).  We omit this computation. 
	\end{rem}
	
	By relying on the analysis in \cite{CJ16}, we give a shortened proof of the above result in  \cref{subsec:Kinvasympproofs}.

Using the notation from \cite{CJ16}, let
	\begin{equation*}
		\mathtt{g}_{\eps_1,\eps_2} =
		\left\{
		\begin{array}{ll}
			\frac{\mathrm{i} \left(\sqrt{a^2+1}+a\right)}{1-a}
			&  \mbox{if } (\eps_1,\eps_2)=(0,0) \\
			\frac{\sqrt{a^2+1}+a-1}{\sqrt{2a} (1-a) }
			& \mbox{if } (\eps_1,\eps_2)=(0,1) \\
			-\frac{\sqrt{a^2+1}+a-1}{\sqrt{2a} (1-a) }
			&\mbox{if } (\eps_1,\eps_2)=(1,0)\\
			\frac{\mathrm{i}\left(\sqrt{a^2+1}-1\right)}{(1-a) a}
			&\mbox{if } (\eps_1,\eps_2)=(1,1).
		\end{array}
		\right.
	\end{equation*}
    
	The following theorem is an extension of \cite[Theorem 2.7]{CJ16}. The extension requires justification, which is done in \cref{S:thmpreviousRSasymptotics}.
	\begin{thm}\label{thm:previousRSasymptotics}
		If $x$ and $y$ have scaling defined by \eqref{eq:def:rsrefined}, then we have
		\begin{equation*}
			K_{a,1}^{-1}(x,y)=\mathbb{K}^{-1}_{1,1}(x,y)-\mathbb{K}_{A}(x,y)
		\end{equation*}
		where $\mathbb{K}_{A}(x,y)=\mathcal{B}_{\eps_1,\eps_2} (a,n+x_1,n+x_2,n+y_1,n+y_2) (1+o(1))$ as $n\to \infty$. Moreover, we have that as $n\to \infty$
		\begin{equation*}
			\begin{split}
				&\mathbb{K}^{-1}_{1,1}(x,y)=\mathrm{i}^{y_1-x_1+1} |G(\mathrm{i})|^{\frac{-2-x_1+x_2+y_1-y_2}{2}} c_0 \mathtt{g}_{\eps_x,\eps_y} 
				\\ &\times e^{\alpha_y \hat{\beta}_y -  \alpha_x \hat{\beta}_x -\frac{2}{3} (\hat{\beta}_x^3-\hat{\beta}_y^3)}
				(2m)^{-\frac{1}{3}} (\phi_{\hat{\beta}_x,\hat{\beta}_y}( \alpha_x+\hat{\beta}_x^2, \alpha_y+\hat{\beta}_y^2)+o(1))
			\end{split}
		\end{equation*}
		and 
		\begin{equation*}
			\begin{split}
				&\mathcal{B}_{\eps_1,\eps_2} (a,n+x_1,n+x_2,n+y_1,n+y_2) =\mathrm{i}^{y_1-x_1+1} |G(\mathrm{i})|^{\frac{-2-x_1+x_2+y_1-y_2}{2}} c_0 \mathtt{g}_{\eps_x,\eps_y} 
				\\ &\times e^{\alpha_y \hat{\beta}_y -  \alpha_x \hat{\beta}_x -\frac{2}{3} (\hat{\beta}_x^3-\hat{\beta}_y^3)}
				(2m)^{-\frac{1}{3}} (\tilde{\mathcal{A} }(\hat{\beta}_x , \alpha_x+\hat{\beta}_x^2; \hat{\beta}_y,\alpha_y+\hat{\beta}_y^2)+o(1)).
			\end{split}
		\end{equation*}
	\end{thm}
    Here the notation $\tilde \cA, \phi_{s, t}$ for the extended Airy kernel is as in \eqref{eq:Airymod}, \eqref{eq:Airyphi}.
	Note that the above theorem is stated in a stronger form than needed since it considers different values of $\hat{\beta}_x$ and $\hat{\beta}_y$ whereas in the sequel, we only need the case $\hat{\beta}_x=\hat{\beta}_y$.

	\section{Proofs of Lemmas \ref{L:global-coupling-1} and \ref{L:global-coupling-2}}\label{S:smoothcouplingproofs}

	In this section, we give the smooth phase coupling resuls, that is the proofs of \cref{L:global-coupling-1} and \cref{L:global-coupling-2}.  To prove these results, we introduce notation similar to the one used in \cite{BCJ22}.

	Consider $\Lambda_{\mathrm{sm},k}$ defined in \cref{L:global-coupling-1} for $k=1$ and \cref{L:global-coupling-2} for $k=2$.  We denote by $\partial\Lambda_{\mathrm{sm},k}$ to be vertices which share edges that cross the boundary of the box $\Lambda_{\mathrm{sm},k}$.  Letting $\mathtt{W}(\Lambda_{\mathrm{sm},k})$ denote all white vertices in $\Lambda_{\mathrm{sm},k}$, we write $\mathtt{W}(\Lambda_{\mathrm{sm},k})=\{w_1,\dots, w_R\}$.  Set $f_1=e_1, f_2=e_2, f_3=-e_1$, and $f_4=-e_2$ and $[N]=\{1,\dots, N\}$ for a positive integer $N$. A \emph{configuration} in  $\Lambda_{\mathrm{sm},k}$ is a  set of edges:
	\begin{equation} \label{E:smoothcouple:edges}
		({w}_1,{w}_1+f_{s_1}),\dots, ({w}_{R},{w}_{R}+f_{s_R})
	\end{equation}
	where $s_j \in[4], 1 \leq j \leq R$.  We can think of \eqref{E:smoothcouple:edges} as the event that all these edges are covered by dimers.   For $\overline{s} \in [4]^{R}$, we let $\overline{s}$ denote the configuration~\eqref{E:smoothcouple:edges}. Let
	\begin{equation*}
		A_{ij}(\overline{s})=K_{a,1}({w}_i,{w}_i)\mathbb{K}^{-1}_{1,1}({w}_i,{w}_j+f_{s_j})
	\end{equation*}
	and 
	\begin{equation*}\label{eq:defCij2}
		C_{ij}(\overline{s})=-K_{a,1}({w}_i,{w}_i)\mathcal{B}_{\eps_1,\eps_2}(a,n+x_1,n+x_2,n+y_1,n+y_2)(1+o(1));
	\end{equation*}
	where $w_i=(x_1,x_2)\in \mathtt{W}_{\eps_1}$ and ${w}_j+f_{s_j}=(y_1,y_2)\in \mathtt{B}_{\eps_2}$ and $\eps_1,\eps_2 \in \{0,1\}$. 
	
	Then, $\mathbb{P}_{a,n}$ induces the following measure on $[4]^R$:
	\begin{equation}\label{eq:inducedmeasureRS}
		p^{a,n}_{\mathrm{Az}}\big(\overline{s}|\Lambda_{\mathrm{sm},k}\big)=\det\big( A_{ij}(\overline{s})+ C_{ij}(\overline{s} )\big)_{1 \leq i,j \leq R } 
	\end{equation}
	while $\mathbb{P}_{a}$ induces the following measure on $[4]^R$ 
	\begin{equation}	
		p_{a}\big(\overline{s}|\Lambda_{\mathrm{sm},k}\big)=\det( A_{ij}(\overline{s}))_{1\leq i,j\leq R}.\label{smoothcouple:S}
	\end{equation}
	Note that if $w_j+f_{s_j}=w_k+f_{s_k}$ for $j\not = k$, then~\eqref{eq:inducedmeasureRS} and~\eqref{smoothcouple:S}  both give zero, so configurations with overlaps have probability zero.

	\begin{proof}[Proof of \cref{L:global-coupling-1}]
		The proof is the same as \cref{L:local-coupling} given in \cite[Proposition 4.5]{BCJ22} but replacing the bound for $|C_{ij}(\overline{s})|$ by the bound given in \eqref{thm:eq:KinversearoundRS:sm} for all $i,j$ and for all $\overline{s}$ such that all faces $x$ in $\Lambda_{\mathrm{sm},1}$ have scaling  \eqref{eq:def:sm}.    The result can be equivalently seen as replacing $n^{-1/3}$ in \cref{L:local-coupling} by the bound given in \eqref{thm:eq:KinversearoundRS:sm} and accounting for the size of the box (which is $o(n^{-1/2})$) in the statement of \cref{L:global-coupling-1}.
		
		The statement on the rotations follows from using \cref{L:Kinv:rotationestimate}, the same asymptotic bounds given in \eqref{thm:eq:KinversearoundRS:sm}, and the above argument. Indeed, these asymptotic bounds holds for any of the rotations considered here because the saddle point function from \cref{L:Kinv:rotationestimate} has the same form as the one considered in \cref{subsubsec:Asymptotics}.
	\end{proof}
	\begin{proof}[Proof of \cref{L:global-coupling-2}]
		Like in the proof of \cref{L:global-coupling-1} we give a shortened proof by relying on the proof of \cref{L:local-coupling} given in \cite[Proposition 4.5]{BCJ22}. The main difference to the proof this lemma compared with \cref{L:global-coupling-1} is that  we need to adjust for the additional factor seen in the estimate \eqref{thm:eq:KinversearoundRS:smclosetoRS}. We can remove this by showing this additional factor becomes a conjugation factor multiplied by another factor.
		
		Indeed, for $x=(x_1,x_2) \in \mathtt{W}_{\eps_1}$ and $y=(y_1,y_2) \in \mathtt{B}_{\eps_2}$ and $\eps_1,\eps_2 \in \{0,1\}$, we can write
		\begin{equation*}
			\begin{split}
				&K_{a,1}(x,x)(\mathbb{K}_{1,1}^{-1}(x,y)-\mathcal{B}_{\eps_1,\eps_2}(a,n+x_1,n+x_2,n+y_1,n+y_2)(1+o(1)))\\
				&= \frac{H_{x_1+1,x_2+1}(\omega_{c,0})}{H_{y_1+1,y_2+1}(\tilde{\omega}_{c,0})} \bigg(\frac{H_{y_1+1,y_2+1}(\tilde{\omega}_{c,0})}{H_{x_1+1,x_2+1}(\omega_{c,0})}K_{a,1}(x,x)\mathbb{K}_{1,1}^{-1}(x,y)\\ 
				&- \frac{H_{y_1+1,y_2+1}(\tilde{\omega}_{c,0})}{H_{x_1+1,x_2+1}(\omega_{c,0})}K_{a,1}(x,x)\mathcal{B}_{\eps_1,\eps_2}(a,n+x_1,n+x_2,n+y_1,n+y_2)(1+o(1))\bigg).
			\end{split}
		\end{equation*}
		We now proceed as before in the proof of \cref{L:local-coupling},  with $A_{ij}(\cdot)$ given by the first term on the right side of the above equation inside the parenthesis and with $C_{ij}(\cdot)$ given by the second term on the right side of the above equation inside the parenthesis.  The result follows by noting the additional factor $\frac{H_{x_1+1,x_2+1}(\omega_{c,0})}{H_{y_1+1,y_2+1}(\tilde{\omega}_{c,0})}$ outside the parenthesis becomes a conjugation factor, the bound from \eqref{thm:eq:KinversearoundRS:smclosetoRS} which gives an exponential bound for $C_{ij}(\cdot)$, the fact that  $\frac{H_{x_1+1,x_2+1}(\omega_{c,0})}{H_{y_1+1,y_2+1}(\tilde{\omega}_{c,0})}=G(\omega_{c,0})^{\frac{1}{2}}G(\tilde{\omega}_{c,0})^{-\frac{1}{2}}\frac{H_{x_1+1,x_2}(\omega_{c,0})}{H_{y_1,y_2+1}(\tilde{\omega}_{c,0})}$ from \eqref{eq:prev:H},and that $A_{ij}(\cdot)$ as chosen above, is the smooth phase correlation kernel with a conjugation factor of the form $\frac{H_{y_1+1,y_2+1}(\tilde{\omega}_{c,0})}{H_{x_1+1,x_2+1}(\omega_{c,0})}$.  Analogous to the proof of \cref{L:global-coupling-1}, the exponential bound is given by \eqref{thm:eq:KinversearoundRS:smclosetoRS} (after multiplying by $\frac{H_{y_1+1,y_2+1}(\tilde{\omega}_{c,0})}{H_{x_1+1,x_2+1}(\omega_{c,0})}$).
	\end{proof}

	\section{Proof of Expected Height Results}
    \label{S:sec10}
	In this section, we give the proofs  of the expected height function results, namely, the proofs of \cref{L:expected-middle-height}, \cref{L:expectation-computation} and \cref{L:expectation-computation-2}. 
	
	\subsection{Proof of \cref{L:expected-middle-height}} \label{S:Proof:expected-middle-height}
	
	To prove \cref{L:expected-middle-height}, we need the following result that is a surprising consequence of the explicit formulas for $K^{-1}_{a,1}$.  Once this result is established, \cref{L:expected-middle-height} follows readily. 
	
	\begin{lem} \label{L:expected-height-invariant}
		Let $k,\ell$ be {non-negative} integers with $-n+2k+2\ell<n$. Then $\mathbb{E}\mathcal{H}_n(-n+2k+2\ell,-n+2\ell)-\mathbb{E}\mathcal{H}_n(-n+2k,-n)$ is invariant under the map $a \mapsto a^{-1}$. 
	\end{lem}

	\begin{proof}
		Without loss of generality, we assume that $(-n+2k+1,-n+1)$ is an $a$-face. The same computation holds if $(-n+2k+1,-n+1)$ is a $b$-face.  
		Using the relationship between the height function and dimer coverings, we have
		\begin{equation*}
			\begin{split}
				&\mathbb{E}\mathcal{H}_n(-n+2k+2\ell,-n+2\ell)-\mathbb{E}\mathcal{H}_n(-n+2k,-n)=4\sum_{s=0}^{\ell-1}\sum_{\eps=0}^1 (-1)^\eps\\& \big( \mathbb{P}[\mbox{Dimer covers the edge}\\&((-n+2k+2s+1,-n+2s+2\eps),(-n+2k+2s+2\eps,-n+2s+1))]-1/4\big).
			\end{split}
		\end{equation*}
        
		We can write the above formula using the inverse Kasteleyn matrix using \cref{localstatisticsthm} and \eqref{pf:K} which gives
		\begin{equation}
        \label{E:Kasteleyn-sum}
			\begin{split}
				&\mathbb{E}\mathcal{H}_n(-n+2k+2\ell,-n+2\ell)-\mathbb{E}\mathcal{H}_n(-n+2k,-n)= 4 a \mathrm{i} \sum_{s=0}^{\ell-1}\sum_{\eps=0}^1(-1)^\eps \\& K^{-1}_{a,1}((-n+2k+2s+1,-n+2s+2\eps),(-n+2k+2s+2\eps,-n+2s+1)).
			\end{split}
		\end{equation}
		We substitute in the formula for $K_{a,1}^{-1}$ given in \cref{thm:prev:Kinverse} to get that the above equation equals
		\begin{equation} \label{E:expectationexpansion}
			\begin{split}
				&-4 a \mathrm{i} \sum_{s=0}^{\ell-1}\sum_{\eps=0}^1(-1)^\eps  \big(
				\mathcal{B}_{\eps,\eps}(a,2k+2s+1,2s+2\eps,2k+2s+2\eps,2s+1)\\
				&-\frac{\mathrm{i}}{{a}}\mathcal{B}_{1-\eps,\eps}(a^{-1},2n-2k-2s-1,2s+2\eps,2n-2k-2s-2\eps,2s+1)\\
				&-\frac{\mathrm{i}}{{a}}\mathcal{B}_{\eps,1-\eps}(a^{-1},2k+2s+1,2n-2s-2\eps,2k+2s+2\eps,2n-2s-1)\\
				&+\mathcal{B}_{1-\eps,1-\eps}(a,2n-2k-2s-1,2n-2s-2\eps,2n-2k-2s-2\eps,2n-2s-1)\big).
			\end{split}
		\end{equation}
		In the above equation, we have used 
		\begin{equation*}
			\sum_{\eps=0}^1(-1)^\eps \mathbb{K}^{-1}_{1,1}((-n+2k+2s+1,-n+2s+2\eps),(-n+2k+2s+2\eps,-n+2s+1))=0.
		\end{equation*}
		which follows by a computation (or since the full-plane smooth phase is flat in expectation). 
		
		We want to show that \eqref{E:expectationexpansion} is invariant under $a \mapsto a^{-1}$.  We show that 
		\begin{equation}\label{E:expectationexpansion1}
			-a \mathrm{i} \sum_{\eps=0}^1(-1)^\eps  
			\mathcal{B}_{\eps,\eps}(a,2k+2s+1,2s+2\eps,2k+2s+2\eps,2s+1)
		\end{equation}
		is invariant under $a \mapsto a^{-1}$ since the other three cases are equivalent and this will show that \eqref{E:expectationexpansion} is invariant under $a \mapsto a^{-1}$.    Note that each term in the  above summand is not invariant under the map $a \mapsto a^{-1}$. We use the formula for $\mathcal{B}_{\eps_1,\eps_2}$ given in~\eqref{eq:prev:B:first} to get that \eqref{E:expectationexpansion1} equals
		\begin{equation*}
			\begin{split}
				&-\frac{a}{(2\pi \mathrm{i})^2} \int_{\Gamma_r} \frac{\mathrm{d}u_1}{u_1} \int_{\Gamma_r}\frac{\mathrm{d}u_2}{u_2} \frac{F_{s} (\nu(u_1)) F_{s+k} (\nu(u_2))}{u_1^{k+s} u_2^{s}} \\ 
				&\times \Bigg( \frac{1}{4c^2} u_1^2 u_2^2 (2-\mu(-u_1 \mathrm{i}))^2 (2-\mu(-u_2 \mathrm{i}))^2\Bigg)^m \sum_{\gamma_1,\gamma_2=0}^1  Y_{\gamma_1,\gamma_2}^{0,0}(u_1,u_2)\\
				&-\frac{a}{(2\pi \mathrm{i})^2} \int_{\Gamma_r} \frac{\mathrm{d}u_1}{u_1} \int_{\Gamma_r}\frac{\mathrm{d}u_2}{u_2} \frac{F_{s+1} (\nu(u_1)) F_{k+s+1} (\nu(u_2))}{u_1^{k+s} u_2^s} \\ 
				&\times \Bigg( \frac{1}{4c^2} u_1^2 u_2^2 (2-\mu(-u_1 \mathrm{i}))^2 (2-\mu(-u_2 \mathrm{i}))^2\Bigg)^m \sum_{\gamma_1,\gamma_2=0}^1 Y_{\gamma_1,\gamma_2}^{1,1}(u_1,u_2).
			\end{split}
		\end{equation*}
		We combine both of the above integrals and factor. Since $\mu$ and $c$ are invariant under the map $a \mapsto a^{-1}$, to show that the above expression is invariant under $a \mapsto a^{-1}$, we only need to show that the following expression
		\begin{equation}\label{eq:expofFterms}
			\begin{split}
				&-\mathrm{i} a F_s(\nu (u_1)) F_{s+k} (\nu (u_2)) \sum_{\gamma_1,\gamma_2=0}^1 Y_{\gamma_1,\gamma_2}^{0,0}(u_1,u_2)  \\ & -\mathrm{i} a F_{s+1} (\nu (u_1)) F_{s+k+1} (\nu (u_2)) \sum_{\gamma_1,\gamma_2=0}^1 Y_{\gamma_1,\gamma_2}^{1,1}(u_1,u_2) 
			\end{split}
		\end{equation}
		is invariant under the map $a \mapsto a^{-1}$.  From \cite[Lemma 2.2]{CJ16}, we have that
        $$F_s\bigg(\frac{\mathrm{i} \omega}{\sqrt{2c}}\bigg)=\frac{\mathrm{i}^{|s|}}{(1+a^2)\omega \sqrt{\omega^{-2}+2c}} G(\omega^{-1})^{|s|}$$ 
       for $\omega \in  \mathbb{C}\backslash (\mathrm{i}(-\infty,1/\sqrt{2c}] \cup [1/\sqrt{2c},\infty))$ and $\mathrm{i} \omega=\sqrt{2c} \nu(u)$, we see that \eqref{eq:expofFterms} becomes 
		\begin{equation*}
			\begin{split}
				&-\mathrm{i} a F_{s} (\nu (u_1)) F_{s+k} (\nu (u_2)) \Bigg( \sum_{\gamma_1, \gamma_2=0}^1 Y_{\gamma_1,\gamma_2}^{0,0}(u_1,u_2) \\& -\mathrm{i}^2 \Bigg( - \frac{\mu(-\mathrm{i} u_1)}{c(u_1+u_1^{-1})} \Bigg)\Bigg( - \frac{\mu(-\mathrm{i} u_2)}{c(u_2+u_2^{-1})} \Bigg) \sum_{\gamma_1, \gamma_2=0}^1 Y_{\gamma_1,\gamma_2}^{1,1}(u_1,u_2) \Bigg)
			\end{split}
		\end{equation*}
		which we want to show is invariant under the map $a \mapsto a^{-1}$.  We simplify the above equation and use the definition of $\mu$ and $s$ found in \eqref{eq:prev:mu_and_s}, we arrive at
		\begin{equation*}
			\begin{split}
				&-\mathrm{i} \frac{F_s (\nu (u_1)) F_{s+k} (\nu (u_2))}{c^2(u_1+u_1^{-1})(u_2+u_2^{-1}) } a \Bigg( c^2(u_1+u_1^{-1})(u_2+u_2^{-1}) \sum_{\gamma_1,\gamma_2=0}^1 Y_{\gamma_1,\gamma_2}^{0,0}( u_1,u_2) \\&+(1-s(-\mathrm{i}u_1))(1-s(-\mathrm{i}u_1))\sum_{\gamma_1,\gamma_2=0}^1 Y_{\gamma_1,\gamma_2}^{1,1}( u_1,u_2) \Bigg),
			\end{split}
		\end{equation*}
		which we want to show is invariant under the map $a\mapsto a^{-1}$.  We use the  formula for $Y_{\gamma_1,\gamma_2}^{\eps_1,\eps_2}$ given in~\eqref{asfo:eq:Y} and note that $(1+a^2)F_s(\cdot)$, from~\eqref{eq:prev:Fs}, is invariant under the map $a \mapsto a^{-1}$. By these observations, to show that the above equation is invariant under the map $a \mapsto a^{-1}$, we need to show that 
		\begin{equation}
			\begin{split}\label{lemproof:expheights:bigcoeffexp}
				&ac^2(u_1+u_1^{-1})(u_2+u_2^{-1}) \sum_{\gamma_1,\gamma_2=0}^1(-1)^{\gamma_1+\gamma_2} s(-\mathrm{i} u_1)^{\gamma_1}s(\mathrm{i}u_2)^{\gamma_2} \mathtt{y}_{\gamma_1,\gamma_2}^{\eps_1,\eps_2}(-\mathrm{i}u_1 ,- \mathrm{i} u_2)\\
				&-a\sum_{\gamma_1,\gamma_2=0}^1s(-\mathrm{i} u_1)^{\gamma_1}s(-\mathrm{i} u_2)^{\gamma_2}(1-s(-\mathrm{i} u_1))(1-s(-\mathrm{i} u_2)^{1})\mathtt{y}_{\gamma_1,\gamma_2}^{\eps_1,\eps_2}(-\mathrm{i}u_1 ,- \mathrm{i} u_2)
			\end{split}
		\end{equation}
		is invariant under the map $a \mapsto a^{-1}$.  We expand out the above equation using the definition of $s(\cdot)$ given in~\eqref{eq:prev:mu_and_s} and noting that $s(-\mathrm{i} u_1)^2 = 1-c^2(u_1+u_1^{-1})^2$.  We can then expand out in terms of coefficients of $s(-\mathrm{i} u_1)$ and $s(-\mathrm{i}u_2)$.  The coefficient of $s(-\mathrm{i}u_1)^0s(-\mathrm{i}u_2)^0$ in~\eqref{lemproof:expheights:bigcoeffexp} is given by 
		\begin{equation}\label{lemproof:expheights:bigcoeffexp1}
			\begin{split}
				&a \mathtt{y}_{0,0}^{0,0}(-\mathrm{i}u_1,-\mathrm{i}u_2) c^2(u_1+u_1^{-1})(u_2+u_2^{-1}) - au_1u_2 \mathtt{y}_{0,0}^{1,1}(-\mathrm{i}u_1,-\mathrm{i}u_2)\\
				&+a (1-c^2(u_1+u_1^{-1})^2) (1-c^2(u_2+u_2^{-1})^2) u_1 u_2\mathtt{y}_{1,1}^{1,1}(-\mathrm{i}u_1,-\mathrm{i}u_2)\\
				&-a(1-c^2(u_1+u_1^{-1})^2)u_1u_2\mathtt{y}_{1,0}^{1,1}(-\mathrm{i}u_1,-\mathrm{i}u_2)\\
				&-a(1-c^2(u_2+u_2^{-1})^2)u_1u_2\mathtt{y}_{0,1}^{1,1}(-\mathrm{i}u_1,-\mathrm{i}u_2),
			\end{split}
		\end{equation}
		the coefficient of $s(-\mathrm{i}u_1)^1s(-\mathrm{i}u_2)^1$ in~\eqref{lemproof:expheights:bigcoeffexp} is given by 
		\begin{equation}
			\begin{split}\label{lemproof:expheights:bigcoeffexp2}
				&a \mathtt{y}_{1,1}^{0,0}(-\mathrm{i}u_1,-\mathrm{i}u_2) c^2(u_1+u_1^{-1})(u_2+u_2^{-1})-au_1 u_2\mathtt{y}_{0,0}^{1,1}(-\mathrm{i}u_1,-\mathrm{i}u_2)\\
				&-a u_1 u_2\mathtt{y}_{0,1}^{1,1}(-\mathrm{i}u_1,-\mathrm{i}u_2)-au_1 u_2\mathtt{y}_{1,0}^{1,1}(-\mathrm{i}u_1,-\mathrm{i}u_2)+au_1 u_2\mathtt{y}_{1,1}^{1,1}(-\mathrm{i}u_1,-\mathrm{i}u_2)
			\end{split}
		\end{equation}
		the coefficient of $s(-\mathrm{i}u_1)^1s(-\mathrm{i}u_2)^0$ in~\eqref{lemproof:expheights:bigcoeffexp} is given by 
		\begin{equation}
			\begin{split}\label{lemproof:expheights:bigcoeffexp3}
				&a \mathtt{y}_{1,0}^{0,0}(-\mathrm{i}u_1,-\mathrm{i}u_2) c^2(u_1+u_1^{-1})(u_2+u_2^{-1})-au_1 u_2\mathtt{y}_{0,0}^{1,1}(-\mathrm{i}u_1,-\mathrm{i}u_2)\\
				&-a u_1 u_2\mathtt{y}_{1,0}^{1,1}(-\mathrm{i}u_1,-\mathrm{i}u_2)+a(1-c^2(u_1+u_1^{-1})^2)u_1u_2\mathtt{y}_{0,1}^{1,1}(-\mathrm{i}u_1,-\mathrm{i}u_2)\\
				&-a(1-c^2(u_1+u_1^{-1})^2)u_1u_2\mathtt{y}_{1,1}^{1,1}(-\mathrm{i}u_1,-\mathrm{i}u_2),
			\end{split}
		\end{equation}
		and the coefficient of $s(-\mathrm{i}u_1)^0s(-\mathrm{i}u_2)^1$ in~\eqref{lemproof:expheights:bigcoeffexp} is given by 
		\begin{equation}
			\begin{split}\label{lemproof:expheights:bigcoeffexp4}
				&a \mathtt{y}_{0,1}^{0,0}(-\mathrm{i}u_1,-\mathrm{i}u_2) c^2(u_1+u_1^{-1})(u_2+u_2^{-1})-au_1 u_2\mathtt{y}_{0,0}^{1,1}(-\mathrm{i}u_1,-\mathrm{i}u_2)\\
				&-a u_1 u_2\mathtt{y}_{0,1}^{1,1}(-\mathrm{i}u_1,-\mathrm{i}u_2)+a(1-c^2(u_1+u_1^{-1})^2)u_1u_2\mathtt{y}_{1,0}^{1,1}(-\mathrm{i}u_1,-\mathrm{i}u_2)\\
				&-a(1-c^2(u_1+u_1^{-1})^2)u_1u_2\mathtt{y}_{1,1}^{1,1}(-\mathrm{i}u_1,-\mathrm{i}u_2).
			\end{split}
		\end{equation}
		Using computer algebra (or an explicit computation by hand), each of~\eqref{lemproof:expheights:bigcoeffexp1},~\eqref{lemproof:expheights:bigcoeffexp2},~\eqref{lemproof:expheights:bigcoeffexp3}, and~\eqref{lemproof:expheights:bigcoeffexp4} are invariant under the map $a \mapsto a^{-1}$ and we can conclude that \eqref{E:expectationexpansion1} and \eqref{E:expectationexpansion} are invariant under the map $a \mapsto a^{-1}$ as required. \end{proof}
	
	We are now in the position to prove \cref{L:expected-middle-height}.

	\begin{proof}[Proof of \cref{L:expected-middle-height}]

		Let $k,\ell$ be as in \cref{L:expected-height-invariant}.  In the proof of the lemma, we let $\mathcal{H}_n^a=\mathcal{H}_n$ to make explicit the height function dependence on $a$. 

        First, using that the boundary faces have deterministic heights, see \cref{S:face-heights}, we have 
        \begin{equation*}
			\begin{split}
				\mathbb{E}\mathcal{H}_n^a(-n+2k+2\ell,-n+2\ell)&=-n+2k+\mathbb{E}\mathcal{H}_n^a(-n+2k+2\ell,-n+2\ell)\\&-\mathbb{E}\mathcal{H}_n^a(-n+2k,-n).
			\end{split}
		\end{equation*}
		Similarly, 
		\begin{equation*}
			\begin{split}
				\mathbb{E}\mathcal{H}_n^a(n-2k-2\ell,-n+2\ell)&=n-2k+\mathbb{E}\mathcal{H}_n^a(n-2k-2\ell,-n+2\ell)\\&-\mathbb{E}\mathcal{H}_n^a(n-2k,-n).
			\end{split}
		\end{equation*}
		Next, by symmetry and a gauge transformation (multiplying all the white vertices by $a$), $\mathcal{H}_n^a(n-2k-2\ell,-n+2\ell)-\mathcal{H}_n^a(n-2k,-n)$ has the same distribution as $-\mathcal{H}_n^{1/a}(-n+2k+2\ell,-n+2\ell)+\mathcal{H}_n^{1/a}(-n+2k,-n)$.  Since 
		\begin{equation*}
			\begin{split}
				&\mathbb{E}\mathcal{H}_n^{1/a}(-n+2k+2\ell,-n+2\ell)-\mathbb{E}\mathcal{H}_n^{1/a}(-n+2k,-n)\\&=\mathbb{E}\mathcal{H}_n^{a}(-n+2k+2\ell,-n+2\ell)-\mathbb{E}\mathcal{H}_n^{a}(-n+2k,-n) 
			\end{split}
		\end{equation*}
		by \cref{L:expected-height-invariant}, it follows that 
		\begin{equation*}
			\begin{split}
				&\mathbb{E}\mathcal{H}^a_n(n-2k-2\ell,-n+2\ell)=n-2k-\mathbb{E}\mathcal{H}^a_n(-n+2k+2\ell,-n+2\ell)\\&-\mathbb{E}\mathcal{H}^a_n(-n+2k,-n)=-\mathbb{E}\mathcal{H}^a_n(-n+2k+2\ell,-n+2\ell).
			\end{split}
		\end{equation*}
		The other three equations follow similarly (or from applying reflection symmetries and gauge transformations). The final statement follows from setting $i = 0$ or $j = 0$ in the chain of equalities.
		        \end{proof}

	\subsection{Proofs of \cref{L:expectation-computation} and \cref{L:expectation-computation-2}}  \label{S:expectation-computations}
	Here, we give the proof of \cref{L:expectation-computation} and \cref{L:expectation-computation-2}. Both of these results involve the refined asymptotic analysis of $K^{-1}_{a,1}$.

	Suppose that  $(\xi_1,\xi_2)$ is on the rough-smooth limit curve \eqref{eq:asymp:limitshape} and that $X$ is defined by \eqref{eq:def:sm}, \eqref{eq:def:smclosetoRS}, \eqref{eq:def:rs} or \eqref{eq:def:rgclosetoRS}. We denote
	\begin{equation*}\label{eq:expectheights:J}
		J_{\xi_1,\xi_2,X}=(\lfloor 4m\xi_1+4mX \rfloor,\lfloor 4m\xi_2+4mX \rfloor)
	\end{equation*}
	which is a face in the Aztec diamond graph.  We prove the following result in which the proof of \cref{L:expectation-computation} will follow.

	\begin{lem} \label{lem:expectheightroughsmooth}
		Let $(\xi_1,\xi_2)$ be on the limit shape curve~\eqref{eq:asymp:limitshape} with $|\xi_1-\xi_2|<m^{-\frac{1}{6}}$ and let $ m^{-\frac{1}{3}} \log^2 m\leq R < m^{-\frac{1}{2}+\delta}$ for $0< \delta <\frac{1}{12}$ fixed. Then, for a constant $\mathtt{C}_1>0$ we have as $m\to \infty$
		\begin{equation*}
			\mathbb{E}\mathcal{H}_n(J_{\xi_1,\xi_2,-R})=-\mathtt{C}_1 R^{\frac{3}{2}}m +o(m^{2\delta}).
		\end{equation*}
	\end{lem}

	\begin{proof}[Proof of \cref{lem:expectheightroughsmooth}]
		First, without loss of generality we may assume that $J_{\xi_1,\xi_2,-R}$ is a $c$-face, since the estimate for other faces can be built from the $c$-face estimate and the fact that height differences on adjacent faces are at most $4$. Next, let $u \in [m^{5/6}, m^{5/6} + 2]$ be such that $\mathtt{J} = J_{\xi_1,\xi_2,u}$ is also a $c$-face. Then, we have that 
		\begin{equation*}
			\mathbb{E}\mathcal{H}_n(J_{\xi_1,\xi_2,-R})=\mathbb{E}(\mathcal{H}_n(J_{\xi_1,\xi_2,-R})-\mathcal{H}_n(\mathtt{J})) -\mathbb{E}\mathcal{H}_n(\mathtt{J}),
		\end{equation*}
		where we have an exponential bound on the last term by \cref{C:expected-sm-height}. Note here that \cref{C:expected-sm-height} does not rely on either \cref{L:expectation-computation} and \cref{L:expectation-computation-2} as an input, so this is not a circular argument. 
		
		Using the relationship between dimer coverings and height functions as well as \cref{localstatisticsthm}, similarly to \eqref{E:Kasteleyn-sum} we have that
		\begin{equation} \label{E:heightdiff-smooth-RS}
			\begin{split}
				&\mathbb{E}(\mathcal{H}_n(J_{\xi_1,\xi_2,-R})-\mathcal{H}_n(\mathtt{J}))\\
                &=4 \sum_{k=0}^\mathtt{m} \sum_{\eps=0}^1 (-1)^{\eps}\mathbb{P}\big[\mbox{Dimer covers the edge } (x_{k, \eps}, y_{k, \eps}) 
            \big] \\ 
				&=-4a\mathrm{i} \sum_{k=0}^{\mathtt{m}} \sum_{\eps=0}^1 (-1)^{\eps}  K^{-1}_{a,1}(x_{k, \eps}, y_{k, \eps}).
			\end{split}
		\end{equation}
		where
        \begin{align*}
           x_{k, \eps} &=(\lfloor 4m \xi_1- 4mR \rfloor +2k+1, \lfloor 4m \xi_1- 4mR \rfloor +2k+2\eps), \\
		y_{k, \eps}&=(\lfloor 4m \xi_1- 4mR \rfloor +2k+2\eps, \lfloor 4m \xi_1- 4mR \rfloor +2k+1). 
        \end{align*}
		and $\mathtt{m} = 2 m^{5/6} + 2R m + O(1)$.
		We use the formula $K_{a,1}^{-1}$ given in \cref{thm:prev:Kinverse}, noting that the smooth phase is flat in expectation, and the bounds given in the first equation of \cref{lem:switchcontoursBtilde}, which gives that \eqref{E:heightdiff-smooth-RS} equals
		\begin{equation*}
			\begin{split}
            -4a\mathrm{i} \sum_{k=0}^{\mathtt{m}} \sum_{\eps=0}^1 (-1)^{\eps} \mathcal{B}_{\eps,\eps}\big(a,
				&(n,n)+x_{k, \eps},(n,n)+y_{k, \eps} \big) +O(e^{-\mathtt{C}_1n})
			\end{split}
		\end{equation*}
		for some positive constant $\mathtt{C}_1$.  
		We split the above sum into three sums, namely
		\begin{equation}
			\begin{split}\label{L:Proof:expectedheight-splitequation}
			-4a\mathrm{i} \sum_{j=0}^2\sum_{k=\mathtt{m}_{j}}^{\mathtt{m_{j+1}}} \sum_{\eps=0}^1 (-1)^{\eps} \mathcal{B}_{\eps,\eps}\big(a,
				&(n,n)+x_{k, \eps},(n,n)+y_{k, \eps} \big) +O(e^{-\mathtt{C}_1n})
			\end{split}
		\end{equation}
		where $\mathtt{m}_3=\mathtt{m}$, $\mathtt{m}_2=\lfloor2( m^{1/3}\log^2 m+Rm)\rfloor$, $\mathtt{m}_1=\lfloor2( -m^{1/3}\log^2 m+Rm)\rfloor$ and $\mathtt{m}_0=0$.  We consider the sum with respect to $j$ separately.  
		
		By statements 1 and 2 in \cref{thm:prev:KinversearoundRS}, we have that the summands when $j=2$  in the above equation are at most $O(\exp(-c\log^3 m))$ for some constant $c > 0$, while the sum is over a polynomial in $m$, and so the $j=2$ sum above is $O(e^{-\mathtt{C}_2 \log^3 n})$ as $m\to \infty$. 
		By statement 3 in \cref{thm:prev:KinversearoundRS}, we have that the summands when $j=1$ are of order $m^{-\frac{1}{3}}$ and so the second sum is $o(m^{\delta})$ as $m\to \infty$. 
		It remains to consider the sum $j=0$ in the above equation. Notice that we can apply Statement 4 in \cref{thm:prev:KinversearoundRS} and we get that the error term, when summed is $o(m^{2\delta})$ as $m\to \infty$. Putting everything together, we get that
		\begin{equation*}
			\begin{split}
				&\mathbb{E}(\mathcal{H}_n(J_{\xi_1,\xi_2,-R})-\mathcal{H}_n(\mathtt{J}))\\
                &=-4a \mathrm{i}\sum_{k=0}^{\mathtt{m_1}} \sum_{\eps=0}^1 (-1)^{\eps}  \mathcal{B}_{\eps,\eps}\big(a, (n,n)+x_{k, \eps},(n,n)+y_{k, \eps} \big)+o(m^\delta)\\
				&=-4a \mathrm{i}\sum_{k=0}^{\mathtt{m_1}}C_{\omega_{c,X(k)}}((1,0),(0,1)) -C_{\omega_{c,X(k)}}((1,2),(2,1))  +o(m^{2\delta})
			\end{split}
		\end{equation*}
		where the error term is out of the sum, and $\omega_{c,X(k)}$ is defined through~\eqref{eq:thm:KinversearoundRS:rough} since $X(k)=2k-4R$. 
		We have that 
		\begin{equation*}
			\begin{split}
				&C_{\omega_{c,X(\ell)}}((1,2),(2,1))-C_{\omega_{c,X(\ell)}}((1,0),(0,1))\\
				&=\frac{1}{2 \pi \mathrm{i}} \int_{\Gamma_{\omega_{c,X(k)}}} \frac{\mathrm{d}\omega}{\omega} (\mathrm{i} V_{1,1}(\omega,\omega)-\mathrm{i}^{-1} V_{0,0}(\omega,\omega) (G(\omega) G(\omega^{-1}))^{-1}) \\
				&=\frac{\mathrm{i}}{2 \pi \mathrm{i}} \int_{\Gamma_{\omega_{c,X(k)}}} \frac{\mathrm{d}\omega}{\omega } \frac{1}{2(1+a^2)\sqrt{\omega^2+2c}\sqrt{\omega^{-2}+2c}} (G(\omega)G(\omega^{-1}) -(G(\omega)G(\omega^{-1}))^{-1})\\
				&=-\frac{\mathrm{i}a}{2 \pi \mathrm{i}} \int_{\Gamma_{\omega_{c,X(k)}}} {\mathrm{d}\omega}\frac{ 1}{\omega^2 \sqrt{\omega^{-2}+2c}} +\frac{1}{\sqrt{\omega^2+2c}}
			\end{split}
		\end{equation*}
		where we have used the formula for $C_{\omega_{c,X(k)}}$ given in \eqref{eq:prev:Comega} and the formula for $V_{\eps_1,\eps_2}(\omega,\omega)$ given in~\eqref{eq:prev:Vww}. The above integral can be computed explicitly, is negative and is bounded below by  $-\mathtt{C}_1|\Gamma_{\omega_{c,X(k)}}|$, the length of the contour which is approximately equal to $4\mathrm{Re}\,\omega_{c,X(k)}$.  The exact value is not important here. We have that $|\Gamma_{\omega_{c,X(k)}}|$ is order $\sqrt{X(k)}$ and since the sum contains approximately $Rm$ terms, we get the desired expansion. 
	\end{proof}

	\begin{proof}[Proof of \cref{L:expectation-computation}]
		Set $\delta=1/30$ and $x=Rm$ in \cref{lem:expectheightroughsmooth} and the result follows for $x\in [-n^{8/15},- \mathtt{C}n^{1/3} (\log n)^2]$ for some constant $\mathtt{C}>0$. For $x \in [-\mathtt{C}n^{1/3}(\log n)^2,2n^{1/3}(\log n)^2]$, we can follow the proof of \cref{lem:expectheightroughsmooth} and get an expected height bound of $O(\log^2 n)$; see the proof of \cref{lem:expectheightroughsmooth} up to \eqref{L:Proof:expectedheight-splitequation} along with the discussion in the following paragraph for the summands $j=1$ and $j=2$.  
	\end{proof}
	
	The next lemma is used in the proof of \cref{L:expectation-computation-2}.  
	\begin{lem}\label{lem:expectheightrough}
		Recall that $\xi_c=-\frac{1}{2}\sqrt{1-2c}$ and let $-\frac{1}{2}\sqrt{1+2c}<X<X'$ where $X'=- m^{-\frac{1}{2}+\delta}$ for $0<\delta<\frac{1}{12}$ fixed. 
		Then, we have
		\begin{equation*}
			\mathbb{E}[\mathcal{H}_n(J_{\xi_c,\xi_c,X})]= \big(p_1(\theta_c)-p_1(\theta_c')\big)2m+o(m^{\frac{1}{2}})
		\end{equation*}
		where 
		\begin{equation*}
			p_1(\theta_c)=\frac{2}{2\pi \mathrm{i}} \big(g_{X,X}(e^{\mathrm{i}\theta_c})-g_{X,X}(e^{-\mathrm{i}\theta_c})+g_{X,X}(-e^{\mathrm{i}\theta_c})-g_{X,X}(-e^{-\mathrm{i}\theta_c}) \big).
		\end{equation*}
		where $ \theta_c$ is defined through $g'_{X,X}(e^{\mathrm{i}\theta_c})=0$ and $ \theta_c'$ is defined through $g'_{X',X'}(e^{\mathrm{i}\theta_c'})=0$; see \eqref{eq:def:gprime}. 
		
	\end{lem}
	
	{

		\begin{proof}
			We compute the expected height difference between $J_{\xi_c,\xi_c,X}$ and $J_{\xi_c,\xi_c,X'}$. This will give the expected height at $J_{\xi_c,\xi_c,X}$ since we have from \cref{lem:expectheightroughsmooth}, the expected height at $J_{\xi_c,\xi_c,X'}$. By following a similar reasoning as in \eqref{E:heightdiff-smooth-RS} and its aftermath, we have
			\begin{equation*}
				\begin{split}
					&\mathbb{E}(\mathcal{H}_n(J_{\xi_c,\xi_c,X})-\mathcal{H}_n(J_{\xi_c,\xi_c,X'}))=-4a\mathrm{i} \sum_{\ell=\lceil mX \rceil}^{\lfloor mX' \rfloor} \sum_{\eps \in \{0,1\}}(-1)^{\eps}\\
					& K_{a,1}^{-1}((\lfloor 4m\rfloor \xi_c+2\ell+1 ,\lfloor 4m \xi_c \rfloor +2 \ell+2 \eps),(\lfloor 4m\rfloor \xi_c+2\ell+2 \eps,\lfloor 4m \xi_c \rfloor +2 \ell+1 ))\\
					&=-4a\mathrm{i} \sum_{\ell=\lceil mX \rceil}^{\lfloor mX' \rfloor} \sum_{\eps \in \{0,1\}}(-1)^{\eps}\mathcal{B}_{\eps,\eps}\bigg(a, (n,n)+(\lfloor 4m\rfloor \xi_c+2\ell+1 ,\lfloor 4m \xi_c \rfloor +2 \ell+2 \eps),\\
					& (n,n) +(\lfloor 4m\rfloor \xi_c+2\ell+2 \eps,\lfloor 4m \xi_c \rfloor +2 \ell+1 )\bigg)
					+O(e^{-\mathtt{C}_1n})
				\end{split}
			\end{equation*}
			where $\mathtt{C}_1$ is a positive constant.  We next use  \eqref{eq:prev:B} to write
			\begin{equation}
\begin{split}\label{eq:lemproof:expectheightintorough}
					&\mathbb{E}(\mathcal{H}_n(J_{\xi_c,\xi_c,X})-\mathcal{H}_n(J_{\xi_c,\xi_c,X'}))=\sum_{\ell=\lfloor m\xi_c+mX \rfloor}^{\lfloor m\xi_c+ mX'\rfloor}\frac{4a}{(2\pi \mathrm{i})^2}  \int_{\Gamma_p}\frac{\mathrm{d} \omega_1}{\omega_1} \int_{\Gamma_{1/p}} \mathrm{d}\omega_2 \frac{1}{\omega_2-\omega_1} \\
					&\times \bigg(\frac{H_{2\ell+2,2\ell}(\omega_1)}{H_{2\ell,2\ell+2}(\omega_2)} \mathrm{i}^{-1}V_{0,0}(\omega_1,\omega_2)-\frac{H_{2\ell+2,2\ell+2}(\omega_1)}{H_{2\ell+2,2\ell+2}(\omega_2)} \mathrm{i}V_{1,1}(\omega_1,\omega_2)\bigg) +O(e^{-\mathtt{C}_1n})\\
					&=\frac{4a}{(2\pi \mathrm{i})^2}  \int_{\Gamma_p}\frac{\mathrm{d} \omega_1}{\omega_1} \int_{\Gamma_{1/p}} \mathrm{d}\omega_2 \frac{1}{\omega_2-\omega_1} \sum_{\ell=-\lfloor m\xi_c+mX' \rfloor}^{-\lfloor m+\xi_c mX \rfloor} \frac{\omega_1^{2m}}{\omega_2^{2m}}T_G(\omega_1,\omega_2)^{\ell} \\
					\times &\frac{1}{G(\omega_1)G(\omega_2^{-1})} \bigg( V_{0,0}(\omega_1,\omega_2) +G(\omega_1^{-1})G(\omega_2) V_{1,1}(\omega_1,\omega_2) \bigg)+O(e^{-\mathtt{C}_1n})
				\end{split}
			\end{equation}
			where we have used ~\eqref{eq:prev:H} and $T_G$ defined in \eqref{eq:T_G}. 
			Let 
			\begin{equation*}
				V(\omega_1,\omega_2)=\frac{\omega_2^{-1} (\omega_2^2-\omega_1^2)}{1-T_G(\omega_1,\omega_2)} \frac{ V_{0,0}(\omega_1,\omega_2) +G(\omega_1^{-1})G(\omega_2) V_{1,1}(\omega_1,\omega_2)}{G(\omega_1)G(\omega_2^{-1})}.
			\end{equation*}
			Then, evaluating the geometric sum in \eqref{eq:lemproof:expectheightintorough} and  using the above definitions gives 
			\begin{equation*}
				\begin{split}
					&\mathbb{E}(\mathcal{H}_n(J_{\xi_c,\xi_c,X})-\mathcal{H}_n(J_{\xi_c,\xi_c,X'}))=\frac{4a}{(2\pi \mathrm{i})^2}  \int_{\Gamma_p}\frac{\mathrm{d} \omega_1}{\omega_1} \int_{\Gamma_{p^{-1}}} \mathrm{d}\omega_2 \frac{\omega_2V_F(\omega_1,\omega_2)}{(\omega_2-\omega_1)^2(\omega_1+\omega_2)} \\
					&\times  \frac{\omega_1^{2m}}{\omega_2^{2m}}\big(1-T_G(\omega_1,\omega_2)^{-\lfloor mX \rfloor +\lfloor mX' \rfloor}\big)T_G(\omega_1,\omega_2)^{- \lfloor mX' \rfloor}+O(e^{-\mathtt{C}_1n})
				\end{split}
			\end{equation*}
			In the above expression, there is a double pole at $\omega_1=\omega_2$  since $V_F(\omega_1,\omega_2)$ is analytic for $\omega_1=\omega_2$. 
			We split the above integral into two and notice that we need to evaluate
			\begin{equation}\label{eq:lemproof:expectheightintorough2}
				\frac{4a}{(2\pi \mathrm{i})^2}  \int_{\Gamma_p}\frac{\mathrm{d} \omega_1}{\omega_1} \int_{\Gamma_{p^{-1}}} \mathrm{d}\omega_2 \frac{\omega_2V_F(\omega_1,\omega_2)T_G(\omega_1,\omega_2)^{-\lfloor mY \rfloor}}{(\omega_2-\omega_1)^2(\omega_1+\omega_2)} \frac{\omega_1^{2m}}{\omega_2^{2m}}
			\end{equation}
			with $Y=X'$ and $Y=X$.  We proceed to evaluate \eqref{eq:lemproof:expectheightintorough2}. 
			By using 
			\begin{equation*}
				\frac{\omega_2}{\omega_2^2-\omega_1^2} \frac{1}{\omega_2-\omega_1}=\frac{1}{2}\bigg( \frac{1}{(\omega_2-\omega_1)^2}+\frac{1}{2\omega_1}\bigg( \frac{1}{\omega_2-\omega_1}-\frac{1}{\omega_1+\omega_2} \bigg)\bigg)
			\end{equation*}
			we split up the integral in \eqref{eq:lemproof:expectheightintorough2} which gives
			\begin{equation*}
				\begin{split}
					&\frac{2a}{(2\pi \mathrm{i})^2}  \int_{\Gamma_p}\frac{\mathrm{d} \omega_1}{\omega_1} \int_{\Gamma_{p^{-1}}} \mathrm{d}\omega_2 \frac{V_F(\omega_1,\omega_2)T_G(\omega_1,\omega_2)^{-\lfloor mY \rfloor}}{(\omega_2-\omega_1)^2} \frac{\omega_1^{2m}}{\omega_2^{2m}} \\
					&+\frac{a}{(2\pi \mathrm{i})^2}  \int_{\Gamma_p}\frac{\mathrm{d} \omega_1}{\omega_1} \int_{\Gamma_{p^{-1}}} \mathrm{d}\omega_2 \bigg(\frac{1}{\omega_2-\omega_1}-\frac{1}{\omega_2+\omega_1} \bigg)V_F(\omega_1,\omega_2) \\
					&\times T_G(\omega_1,\omega_2)^{-\lfloor mY \rfloor} \frac{\omega_1^{2m}}{\omega_2^{2m}} .
					\\
				\end{split}
			\end{equation*}
			For the second integral above, we use a change of variables in the second term of $\omega_2\mapsto -\omega_2$ and so we get that \eqref{eq:lemproof:expectheightintorough2} equals
			\begin{equation*}
				\begin{split}
					&\frac{2a}{(2\pi \mathrm{i})^2}  \int_{\Gamma_p}\frac{\mathrm{d} \omega_1}{\omega_1} \int_{\Gamma_{p^{-1}}} \mathrm{d}\omega_2 \frac{V_F(\omega_1,\omega_2)T_G(\omega_1,\omega_2)^{-\lfloor mY \rfloor}}{(\omega_2-\omega_1)^2} \frac{\omega_1^{2m}}{\omega_2^{2m}} \\
					&+\frac{a}{(2\pi \mathrm{i})^2}  \int_{\Gamma_p}\frac{\mathrm{d} \omega_1}{\omega_1} \int_{\Gamma_{p^{-1}}} \mathrm{d}\omega_2 \frac{V_F(\omega_1,\omega_2)-V_F(\omega_1,-\omega_2)}{\omega_2-\omega_1}  T_G(\omega_1,\omega_2)^{-\lfloor mY \rfloor} \frac{\omega_1^{2m}}{\omega_2^{2m}} .
					\\
				\end{split}
			\end{equation*}
            Here to simplify the change of variables we have used that $G(\omega) = -G(-\omega)$, which implies that $T(\omega_1, \omega_2) = T(\omega_1, -\omega_2)$.
			To compute the asymptotics of the above integrals, we use a saddle point analysis, which we omit the full details as these are given fully in  \cite[Section 3.6]{CJ16}.  We let $\omega_c$ be such that $g'_{Y,Y}(\omega_c)=0$ with $Y=X$ or $Y=X'$. We have that $|\omega_c|=1$ and $0<\arg \omega_c <\pi/2$; see \eqref{eq:def:gprime} and \cite[Lemma 3.18]{CJ16}.  In each of the integrals, we deform the $\Gamma_p$ (the $\omega_1$-contour) to the contour of steepest descent and $\Gamma_{1/p}$ (the $\omega_2$-contour) to the contour of steepest ascent, which results in the contours crossing; see \cite[Lemma 3.20]{CJ16} for a precise description of these contours. For each integral, the contour crossing picks up a single integral term from the residue at $\omega_1=\omega_2$ which has a double pole in the first integral and a simple pole in the second integral.  The contribution from both double integrals (with the contours deformed to that of the steepest ascent and descent) is $O(1)$ which follows from standard saddle point analysis.  
			We obtain that \eqref{eq:lemproof:expectheightintorough2} equals a sum of single integrals, plus an $O(1)$ term:
			\begin{equation}
            \label{E:single-integral-145}
				\begin{split}
					&\frac{2a}{2\pi \mathrm{i}} \int_{\Gamma_1 \backslash \Gamma_{\omega_c}} \mathrm{d} \omega_1\bigg[ \frac{ \mathrm{d}  }{ \mathrm{d} \omega_2}\frac{1}{\omega_1} \frac{\omega_1^{2m}}{\omega_2^{2m}}T_G(\omega_1,\omega_2)^{-\lfloor mY \rfloor} V_F(\omega_1,\omega_2) \bigg]_{\omega_2=\omega_1} \\
					&+\frac{a}{2\pi \mathrm{i}} \int_{\Gamma_1 \backslash \Gamma_{\omega_c}}\frac{ \mathrm{d} \omega}{\omega} V_F(\omega,-\omega)-V_F(\omega,\omega)+O(1).
				\end{split}
			\end{equation}
			We have that the second term above is also $O(1)$.  It remains to compute the first term.  We have that \eqref{E:single-integral-145} equals
			\begin{equation*}
				\begin{split}
					&\frac{2a}{2\pi \mathrm{i}} \int_{\Gamma_1 \backslash \Gamma_{\omega_c}} \mathrm{d} \omega_1\bigg[ \frac{ \mathrm{d}  }{ \mathrm{d} \omega_2}\frac{1}{\omega_1} \frac{\omega_1^{2m}}{\omega_2^{2m}}T_G(\omega_1,\omega_2)^{-\lfloor mY \rfloor} V_F(\omega_1,\omega_2) \bigg]_{\omega_2=\omega_1} +O(1)\\
					&=\frac{2a}{2\pi \mathrm{i}} \int_{\Gamma_1 \backslash \Gamma_{\omega_c}}\mathrm{d} \omega \bigg(-\frac{2m}{\omega^2}  + \lfloor mY \rfloor \frac{G'(\omega)}{G(\omega) \omega}+[mY] \frac{G'(\omega^{-1})}{G(\omega^{-1}) \omega^2}\bigg)V_F(\omega,\omega)+O(1)\\
					&=-\frac{2a}{2\pi \mathrm{i}} \int_{\Gamma_1 \backslash \Gamma_{\omega_c}}\frac{\mathrm{d}  \omega}{\omega} 2m g_{ \lfloor Ym \rfloor, \lfloor Ym \rfloor}'(\omega) V_F(\omega,\omega)+O(1).
				\end{split}
			\end{equation*}
			We can compute that $V_F(\omega,\omega)=\omega/a$, which allow us to compute the above integral  explicitly. Set $\omega_c=e^{\mathrm{i}\theta_c}$ which means that \eqref{eq:lemproof:expectheightintorough2} when $Y=X$ is equal to $p_1(\theta_c)$. We can also set $\omega_c=e^{\mathrm{i}\theta_c'}$ which means that \eqref{eq:lemproof:expectheightintorough2} when $Y=X'$ is equal to $p_1(\theta_c)$. The result follows from taking the difference.	\end{proof}

	}

	We now prove \cref{L:expectation-computation-2}.  
	
	\begin{proof}[Proof of \cref{L:expectation-computation-2}]
		For $x \in (-n^{8/15},2n^{1/3}\log^2n]$, the bound follows from \cref{L:expectation-computation}. For $-\delta n < x \leq -n^{8/15}$ and for fixed $\delta>0$ small, the bound follows from \cref{lem:expectheightrough} along with a Taylor expansion (in $x/n$), while for $x \leq -\delta n$, the bound follows immediately from \cref{lem:expectheightrough}. 
	\end{proof}

	\section{Proof of strengthened result of \cite{BCJ18}}
\label{S:BCJ18}
Here, we prove \cref{T:BCJ18-strong}.   We adopt a similar notation and conventions used in \cite{BCJ18} to make later comparisons to that paper simpler. This helps to keep our exposition as brief as possible. We translate back to the coordinate system in the remainder of the paper only as necessary. 

\subsection{Particle System from \cite{BCJ18}}

We give a similar albeit slightly altered particle system to the one used in \cite{BCJ18}. Let $\hat \beta_m \to \hat{\beta} \in \R$, and introduce 
\begin{equation*}
	\hat{\beta}_m(k)=2 \lfloor \hat{\beta}_m \lambda_2 (2m)^{2/3} \rfloor +2 \lfloor k \lambda_2 \log^2m \rfloor.
\end{equation*}
where $1 \leq k \leq M$ with $M \to \infty$ with $n$ and $M = o(m^{1/3} \log^{-2} m)$.  For $\ell \in 2\mathbb{Z}$, let
\begin{equation*}
	\mathcal{L}_{m}^{\varepsilon}(\ell)=\{ v = (2i-\varepsilon+\frac{1}{2}) e_1 -\ell e_2: i \in \Z, v \in [-n, n]^2 \}, 
\end{equation*}
let $\mathcal{L}_{m}(\hat{\beta}_m(k)) = \mathcal{L}_{m}^0 (\hat{\beta}_m(k))\cup \mathcal{L}_{m}^1(\hat{\beta}_m(k))$ and let
$$
\mathcal{L}_m= \bigcup_{k=1}^{M} \mathcal{L}_{m}(\hat{\beta}_m(k)).
$$
For $z \in \mathcal{L}_{m}^\eps(\hat{\beta}_m(k))$, write $\eps(z)=\eps$.  If $\eps(z) = 0$, then $z$ has \emph{even} parity and if $\eps(z)=1$, then $z$ has \emph{odd} parity.  This relates to the sign of the height function when crossing the edges in the north-east direction along the diagonal of the Aztec diamond. For a particle $z \in \mathcal{L}_{m}(\hat{\beta}_m(k))$, associate the vertex $x(z) \in \mathtt{W}_0 \cup \mathtt{W}_1$ and the vertex $y(z) \in \mathtt{B}_0 \cup \mathtt{B}_1$ and the edge $(y(z),x(z))$ between them, that is the dimer covering the two vertices $y(z)$ and $x(z)$.  Our convention means that 
\begin{equation}\label{eq:dimertoparticle}
	\begin{array}{ll} 
		& x(z)=z-\frac{1}{2} (-1)^{\eps(z)} e_2, \\ 
		& y(z)=z+\frac{1}{2} (-1)^{\eps(z)} e_2,
	\end{array}
\end{equation}
see \cref{fig:aface}.
\begin{figure}
	\begin{center}
		\includegraphics[height=4cm]{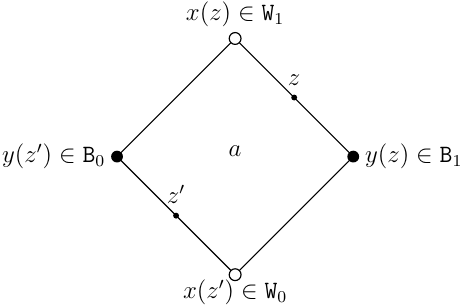}
		\caption{The particles to edge mapping. }
		\label{fig:aface}
	\end{center}
\end{figure}
This gives the particle to edge mapping: 
\begin{align*}
	&\mbox{There is a particle at }\quad &&z \in \mathcal{L}_{m} \qquad
	\\
	\longleftrightarrow \qquad &\mbox{There is a dimer at } \quad &&(x(z),y(z)) \in \mathtt{White} \times \mathtt{Black}.
\end{align*}
We have the following proposition. 
\begin{prop}\label{prop:determinantalpp}
	The particle process on $\mathcal{L}_{m}$ defined above is a determinantal point process with correlation kernel $\mathcal{K}_m$ given by 
	\begin{equation}
		\tilde{\mathcal{K}}_m (z,z') = a\mathrm{i} K_{a,1}^{-1} ( x(z'),y(z))
	\end{equation}
	for $z,z' \in \mathcal{L}_m$. 
\end{prop}
This proposition follows immediately from the local statistics formula, see \cref{localstatisticsthm},  and the fact that $K_{a,1}(y(z),x(z))=a\mathrm{i}$ for $z\in \mathcal{L}_{m}$. We let $\mathbb{I} \subset \mathcal L_m$ denote the point process in \cref{prop:determinantalpp}.

The particle processes on $\mathcal L_m$ can be used to reconstruct height differences between faces along each of the lines $\mathcal L_m(\hat{\beta}_m(k))$. Indeed, for two distinct faces of the form $f = a e_1 - \hat{\beta}_m(k) e_2$ and $f' = b e_1 - \hat{\beta}_m(k) e_2$ with $a < b$, we can write
\begin{equation}
\label{E:height-diff}
\cH_n(f') - \cH_n(f) = 4\sum_{z \in \mathbb{I}} (-1)^{z(\eps)} \mathbf{1}(z = x e_1 - \hat{\beta}_m(k) e_2 \text{ for some } x \in (a, b)).
\end{equation}
Through \eqref{E:height-diff} and averaging, we can then construct differences of the mollified height function $\cH_n^{\phi_n}$ with $\phi_n = \phi_n = \{2\lfloor k \log^2 n \rfloor e_2 : k \in \{0, \dots, M\}\}$ from the particle process $\mathbb{I}$.

Following \cite{BCJ18}, introduce $\tau_m= \lfloor \hat{\beta}^2_m \lambda_1(2m)^{1/3} \rfloor$.  It will be convenient to use the following parameterization. Given $z \in \mathcal{L}_m(\hat{\beta}_m(k))$, there is a $t(z) \in \II{-2\lfloor m\xi_c \rfloor +\tau_m, 4m-2 \lfloor m\xi_c \rfloor +\tau_m}$ such that 
$$
z=(\lfloor 4m\xi_c\rfloor +2(t(z)-\tau_m)-\eps(z)+1/2)e_1-\hat{\beta}_m(k)e_2.
$$
Write for $s \in \mathbb{Z}$,
$$
z_{k}(s)=(\lfloor 4m\xi_c \rfloor +s-2\tau_m+1/2)e_1-\hat{\beta}_m(k)e_2
$$
with $z_{k}(s)\in  \mathcal{L}_m(\hat{\beta}_m(k))$. The function $s \mapsto z_k(s)$ and $z \mapsto t(z)$ are \textit{almost} inverses, with minor differences coming in so as to account for parity. We note that in the $\beta_n$-coordinates introduced prior to \cref{L:expectation-computation}, that
\begin{equation}
\label{E:scaling-translation}
z_k(s) = \beta_n(-\hat{\beta}_m(k), s + o(n^{1/3})).
\end{equation}
where the $o(n^{1/3})$ appears only because in the $\beta_n$-definition we centered around the limit curve $\xi$, whereas in the definition of $t(z)$ above, our centering is around the second-order Taylor approximation to $\xi$ at $(\xi_c, \xi_c)$ (this is where the parabolic correction $\tau_m$ is coming from). 

Next, for $L>0$ and $m \in \N$, let 
$$
\alpha_{1, m}<\alpha_{2, m}' \le \alpha_{2, m}< \alpha_{3, m}' \le \alpha_{4, m} < \dots \le \alpha_{L, m} < \alpha_{L+1, m}' = \log^3 m
$$
be real numbers such that $\alpha_{i, m}, \alpha_{i, m}' \to \alpha_i$, where the limiting sequence satisfies 
$$
\alpha_1 < \alpha_2 < \dots < \alpha_L \le \alpha_{L+1} := \infty.
$$
We allow $\alpha_L \in \R \cup \{\infty\}$, whereas all $\alpha_i, i \le L-1$ must be real.
Now let $A_p=[\alpha_p,\alpha_{p+1})$ for $1 \leq p \leq L$ and define disjoint discrete intervals
\begin{equation*}
	\tilde{A}_p=\{s \in \mathbb{Z}: 2 \lfloor\alpha_{p, m} \lambda_1 (2m)^{1/3}\rfloor-1 \leq s \leq  2\lfloor\alpha_{p+1, m}' \lambda_1 (2m)^{1/3}- 4\rfloor+1\}.
\end{equation*}
Here the separation by $2$ is simply to guarantee that the intervals $\tilde{A}_p$ are disjoint.
The embedding of the interval $\tilde{A}_p$ as a discrete interval in $\mathcal L_m(k)$ is then given by 
$$
I_{p, k} = \{z_k(s) : s \in \tilde{A}_p\}.
$$
We write $\mathbb{I}_{p, k}(z) = \mathbf{1}(z \in I_{p, k})$ for the indicator function for $I_{p, k}$. Now, define the random measure $\mu_m$ by the formula
\begin{equation}
	\label{E:mu-m}
\mu_m(\tilde{A}_p) =	\frac{1}{M} \sum_{k=1}^M \sum_{z \in \mathbb{I}} (-1)^{\eps(z)} \mathbb{I}_{p, k}(z).
\end{equation}
With the mollifier $\phi_n$ as above, and notation as in \eqref{E:mollified-height-fun}, \eqref{E:tilde-Hn}, by \eqref{E:height-diff} we have that
$$
\tilde \cH_n^{\phi_n}(\beta_n, \alpha_{i + 1, m} + o(1)) - \tilde \cH_n^{\phi_n}(\beta_n, \alpha_{i, m} + o(1)) = 4\mu_m(\tilde{A}_p)
$$
where the $o(1)$ terms are coming from scaling translation in \eqref{E:scaling-translation}. Through this translation, and since the approximating sequences $\alpha_{i, m}$ were arbitrary, we can rewrite \cref{T:BCJ18-strong} as the $L = 1$ case of the following theorem.

\begin{thm}\label{T:BCJ18-main-thm-strong-extended}
In the above setup, there exists $R > 0$ such that for all $\omega \in \mathbb{C}^L$ with $|\omega_p|<R$ we have
	\begin{equation}
		\begin{split}
			\lim_{m\to \infty} \mathbb{E}e^{\sum_{p=1}^{L} \omega_p \mu_m(\tilde{A}_p)}=\mathbb{E}e^{\sum_{p=1}^{L} \omega_p \mu_{\mathrm{Ai}}(A_p)}
		\end{split}
	\end{equation}
Here $\mu_{\mathrm{Ai}}$ is the counting measure for the Airy point process $\{\mathcal A_i(0) : i \in \N\}$.	
\end{thm}
The term on the left is a determinant of a finite matrix, while the term on the right is a Fredholm determinant, noting that the interval $A_{L_2}$ is semi-infinite. The proof is developed over the next few subsections, with heavy reference to \cite{BCJ18}.  Note that we are only considering  one fixed time in \cref{T:BCJ18-main-thm-strong-extended} as opposed to  multiple time intervals in \cite[Theorem 1.1]{BCJ18}. 

\subsection{Using  the determinantal point process}
Now, it follows from \eqref{E:mu-m} that
\begin{equation}\label{E:heighttoparticles}
	\sum_{p=1}^L \omega_p \mu_m(\tilde{A}_p)=\frac{1}{M} \sum_{z \in \mathbb{I}} \psi(z)
\end{equation}
where
\begin{equation*}
	\psi(z)=\sum_{p=1}^{L} \sum_{k=1}^M \omega_p (-1)^{\eps(z)} \mathbb{I}_{p,k}(z).
\end{equation*}
As for the correlation kernel of the determinantal point process, with \cref{thm:previousRSasymptotics} in mind, define
\begin{equation*}
	\tilde{ \mathcal{K}}_{m,0}(z,z')=a\mathrm{i} \mathbb{K}_{1,1}^{-1}(x(z'),y(z))
\end{equation*}
and 
\begin{equation*}
	\tilde{ \mathcal{K}}_{m,1}(z,z')=a\mathrm{i} \mathbb{K}_A(x(z'),y(z)).
\end{equation*}
Define for $z \in \mathcal{L}_m(k)$,
\begin{equation*} 
	\begin{split}
		&\gamma_1(z)=\frac{t(z)}{\lambda_1(2m)^{1/3}} \hat{\beta}_m -\frac{1}{3} \hat{\beta}_m^3\\
		&\gamma_2(z)=\eps(z)+\hat{\beta}_m(k)\\
		&\gamma_3(z)=2(t(z)-\tau_m)+\hat{\beta}_m(k).
	\end{split}
\end{equation*}
and let $\mathcal{K}_{m,\delta}$ be such that 
\begin{equation}
	\begin{split}\label{Kmdelta}
		&\tilde{\mathcal{K}}_{m,\delta}(z,z')=\frac{a \mathrm{i}}{2} \sqrt{1-2c} \mathtt{g}_{\eps(z'),\eps(z)}|G(\mathrm{i})|^{2 \eps(z')-2}(-1)^{\eps(z')} e^{\gamma_1(z')-\gamma_1(z)} \\
		&\times |G(\mathrm{i})|^{\gamma_2(z)-\gamma_2(z')} \mathrm{i}^{\gamma_3(z')-\gamma_3(z)} {\mathcal{K}}_{m,\delta}(z,z')
	\end{split}
\end{equation}
for $\delta \in \{0,1\}$.  We set 
\begin{equation*}
	\mathtt{K}_{m, \delta}(z,z')=\frac{a \mathrm{i}}{2} \sqrt{1-2c} \mathtt{g}_{\eps(z'),\eps(z)}|G(\mathrm{i})|^{2 \eps(z')-2}(-1)^{\eps(z')}  \mathcal{K}_{m, \delta}(z,z')
\end{equation*}
and let $\mathtt{K}_m = \mathtt{K}_{m, 0} - \mathtt{K}_{m,1}.$  The above two equations indicate that $\mathtt{K}_m$ is the correlation kernel of the determinantal point process after conjugation and along with \eqref{E:heighttoparticles}, we have
\begin{equation*}
	\begin{split}
		\mathbb{E}e^{\sum_{p=1}^{L} \omega_p \mu_m(\tilde A_p)}&=\det[\mathbbm{I}+(e^{\psi/M}-1)\tilde{\mathcal{K}}_m]_{\mathcal{L}_m}\\
		&=\det[\mathbbm{I}+(e^{\psi/M}-1)\mathtt{K}_m]_{\mathcal{L}_m}.
	\end{split}
\end{equation*}
We have the following proposition, which is an extended version of \cite[Proposition 3.1]{BCJ18}.
\begin{prop} \label{P:airylimit_asymptotics}
	Let $z \in \mathcal{L}_m(k)$, $z' \in \mathcal{L}_m(k')$ and write $t=t(z)$, $t'=t(z')$.  Consider $\mathcal{K}_{m,\delta}(z,z')$ defined by~\eqref{Kmdelta}. 
	The asymptotic formulas and estimates below are uniform as $m\to\infty$ for $-C_1 (2m)^{1/3}< t,t' \leq C_1(2m)^{1/3}(\log m)^3$, for any fixed $C_1>0$ and $1 \leq k \leq M$. 
	\begin{enumerate}[label=\arabic*.]
		\item We have
		\begin{equation*}
			\mathcal{K}_{m,1}(z,z') = \frac{1}{\lambda_1(2m)^{1/3}} \tilde{\mathcal{A}}\left(\frac{t'}{\lambda_1(2m)^{1/3}};\frac{t}{\lambda_1(2m)^{1/3}} \right)(1+o(1)).
		\end{equation*}
		where $\tilde{\mathcal{A}}(x ;y) = \tilde{\mathcal{A}}(0, x; 0, y)$ is given in~\eqref{eq:Airymod}.
		\item Assume that $k>k'$. Then there are constants $c_1,c_2,C_1>0$, so that:
		\begin{enumerate}
			\item If $|t'-t| \leq c_2( (k-k')( \log m )^2)^{7/12}$, then
			\begin{equation*}
				\begin{split}
				\mathcal{K}_{m,0}(z,z')&= \frac{1}{\lambda_1(\log m)} \frac{1}{\sqrt{4 \pi (k-k')}} \\
				&\times \exp \left(-\frac{1}{4(k-k')} \left( \frac{t'-t}{\lambda_1 \log m }\right)^2 \right) (1+o(1)).
				\end{split}
			\end{equation*}
			\item If $ c_2( (k-k')( \log m )^2)^{7/12} \leq  |t'-t| \leq \lambda_2 (k-k') (\log m)^2$, then
			\begin{equation*}
				|\mathcal{K}_{m,0}(z,z')| \leq\frac{C}{(\log m)\sqrt{k-k'}} \exp\left(-\frac{c_1}{(k-k')} \left( \frac{t'-t}{\lambda_1 \log m }\right)^2 \right).
			\end{equation*}
			\item 
		If $|t'-t| \geq \lambda_2 (k-k') (\log m)^2$, then
			\begin{equation*}
				|\mathcal{K}_{m,0}(z,z')| \leq C e^{-c_1(k-k')(\log m)^2}. 
			\end{equation*}
			
		\end{enumerate}

		\item Assume that $k<k'$. Then there are constants $c_1,C>0$ so that
		\begin{equation*}
			|\mathcal{K}_{m,0}(z,z')| \leq C e^{-c_1(k'-k)(\log m)^2}.
		\end{equation*}
		\item  Assume that $k=k'$.  Then there are constants $c_1,C>0$ so that
		\begin{equation*}
			|\mathcal{K}_{m,0}(z,z')| \leq C e^{-c_1|t'-t|}.
		\end{equation*}
	\end{enumerate}
\end{prop}
\begin{proof}[Proof of \cref{P:airylimit_asymptotics}] 
	\cref{P:airylimit_asymptotics} differs from \cite[Proposition 3.1]{BCJ18} in that it allows the parameters to enter the smooth region (by at most $\log^3 m$ amount).  The asymptotics for statements 2, 3, 4 hold from statements 3, 4, 5 from \cite[Proposition 3.1]{BCJ18}.  Statement 1 holds by \cref{thm:previousRSasymptotics}.  	
\end{proof}

\subsection{Proof of \cref{T:BCJ18-main-thm-strong-extended}}

We give a condensed proof of \cref{T:BCJ18-main-thm-strong-extended} by signposting, where appropriate, to the proof of \cite[Theorem 1.1]{BCJ18}. The main difference is that here, the support interval of $\mu_m$ is no longer compact.

For a set $S$, we write $S^r = \{\overline{k} = (k_1, \dots, k_r) : k_i \in S \text{ for all } i = 1, \dots, r\}$. We adopt cyclic notation, that is, $k_{r+1}=k_1$.
Let 
\begin{equation*}
	D_r=\{0,1\}^r \times [M]^r \times [L]^r
\end{equation*}
where $[N]=\{1,2,\dots, N\}$.  
Similar to \cite{BCJ18},
define
\begin{align*}
	D_{r,0} &= \{(\overline{\delta},\overline{k}, \overline{p}) \in D_r; \delta_i =0,k_i=k_{i+1},p_i=p_{i+1} \mbox{ for } 1\leq i \leq r\}, \\
	D_{r,1} &= \{(\overline{\delta},\overline{k}, \overline{p}) \in D_r; \delta_i =0 \mbox{ for } 1\leq i \leq r\mbox{ and } p_i \ne p_{i+1} \mbox{ for some $i$} \}, \\
		D_{r,2} &= \{ (\overline{\delta},\overline{k}, \overline{p}) \in D_r; \delta_i =0, p_i = p_{i+1}\mbox{~for~}1\leq i\leq r  \mbox{~and~} k_i\not=k_{i+1} \mbox{~for some~}i\}, \\
	D_{r,3} &= \{ (\overline{\delta},\overline{k}, \overline{p}) \in D_r; \delta_i =1  \mbox{~for some~} i \}.
\end{align*}
Then $D_r=D_{r,0}\cup D_{r,1}  \cup D_{r,2} \cup D_{r,3}$. Introduce
\begin{equation}  \label{Tjmrell}
	\begin{split}
		T_j(m,r,\overline{\ell}) = &\sum_{\overline{\eps} \in \{0,1\}^r}  \prod_{i=1}^r (-1)^{\ell_i \eps_i} \sum_{(\overline{\delta},\overline{k},\overline{p}) \in D_{r,j}} \prod_{i=1}^r \omega_{p_i}^{\ell_i} (-1)^{\delta_i}  \\
		&\times \sum_{\overline{z} \in( \mathcal{L}_m)^r } \prod_{i=1}^r \mathbbm{I}_{p_i, k_i}^{\eps_i} (z_i)
		\mathtt{K}_{m,\delta_i}(z_i,z_{i+1}),
	\end{split}
\end{equation}
for $j \in \{0,1,2,3\}$ where $\mathbbm{I}_{p,k}^{\eps}$ is the indicator on $\mathcal{L}_m$ for the set
\begin{equation}
	\label{E:Ipke}
	I_{p,k}^{\eps}=\{z \in I_{p,k}: \eps(z)=\eps\}.
\end{equation}
By performing a cumulant expansion, see \cite[Equations (4.1)-(4.6)]{BCJ18} for details, we can write
\begin{equation*}
	\log \det( \mathbbm{I} +(e^{\frac{1}{M} \psi}-1) \mathtt{K}_m) = \sum_{j=0 }^3 U_j(m)
\end{equation*}
where we define
\begin{equation} \label{Ujm}
	U_j(m)=  \sum_{s=1}^\infty \frac{1}{M^s} \sum_{r=1}^s \frac{(-1)^{r+1}}{r} \sum_{\substack{\ell_1+\dots+\ell_r=s \\ \ell_1,\dots,\ell_r \geq 1}} \frac{T_j(m,r,\overline{\ell})}{\ell_1! \dots \ell_r!}.
\end{equation}
We have that $U_0(m)$, $U_1(m)$, and $U_2(m)$ all tend to zero as $m\to \infty$ by \cite[Lemma 4.1 and Lemma 4.2]{BCJ18}, since the same proofs hold as the same estimates are present in \cref{P:airylimit_asymptotics} as \cite[Proposition 3.1]{BCJ18} for $\mathcal{K}_{m,0}$.  It remains to consider $U_3(m)$.   
We introduce some additional notation similar to that of \cite{BCJ18}. Define 
\begin{equation*}
	z_{k}^{\eps}(t) =(\lfloor 4 m \xi_c \rfloor +2(t-\tau_m)-\eps +\frac{1}{2}) e_1 -\hat{\beta}_m(k)e_2,
\end{equation*}
and set
\begin{equation*}
	A_{p,m}=\{t \in \mathbb{Z}: \lfloor \alpha_p \lambda_1 (2m)^{1/3} \rfloor \leq t \leq  \lfloor \alpha_{p+1} \lambda_1 (2m)^{1/3}-\log^3m \rfloor \}.
\end{equation*}
We these definitions, we have that
\begin{equation*}
	I_{p,k}^{\eps} =\{z_{k}^{\eps}(t): t \in A_{p,m} \}. 
\end{equation*}
Next, define
\begin{equation*}
	S_r(\overline{\eps},\overline{\delta},\overline{k},\overline{p}) =\int_{\mathbb{R}^r} \mathrm{d}^r \overline{t} \prod_{i=1}^r \mathbbm{I}_{A_{p,m}}(\lfloor t_i \rfloor) \mathcal{K}^{(i)}_{m,\overline{\eps},\overline{\delta},\overline{k}} (t_i,t_{i+1}) 
\end{equation*}
where
\begin{equation*}
\mathcal{K}^{(i)}_{m,\overline{\eps},\overline{\delta},\overline{k}} (t_i,t_{i+1}) =\mathcal{K}_{m,\delta_i}(z_{k_i}^{\eps_i}(\lfloor t_i \rfloor),z_{k_{i+1}}^{\eps_{i+1}}(\lfloor t_{i+1}\rfloor)).
\end{equation*}
With this notation, following \cite[(4.7)-(4.10)]{BCJ18} we have 
\begin{equation*}
	T_j(m,r,\overline{\ell})=\sum_{\overline{\eps} \in \{0,1\}^r} P(\overline{\eps},\overline{\ell}) \sum_{(\overline{\delta},\overline{k}, \overline{p})\in D_{r,j}} \prod_{i=1}^r \omega_{p_i}^{\ell_i}(-1)^{\delta_i} S_r(\overline{\eps},\overline{\delta},\overline{\ell}, \overline{p}) 
\end{equation*}
with 
\begin{equation*}
	P(\overline{\eps},\overline{\ell})=\prod_{i=1}^r \frac{a \mathrm{i}}{2} \sqrt{1-2c} (-1)^{(1+\ell_i)\eps_i} \mathtt{g}_{\eps_i,\eps_{i+1}} |G(\mathrm{i})|^{\eps_i+\eps_{i+1}-2}.
\end{equation*}
Introduce 
\begin{equation}\label{E:cumulant:d_i}
	d_i=\left\{ \begin{array}{ll}
		\lambda_1 (2m)^{1/3} & \mbox{if }\delta_i=1 \\
		(\log m)\lambda_1 \sqrt{|k_{i+1}-k_i|} &\mbox{if }\delta_i=0,k_i\not = k_{i+1} \\
		1 &\mbox{if }\delta_i=0,k_i = k_{i+1} \end{array} \right. 
\end{equation}
and the change of variables $\tau_i=t_i/d_i$ for $1 \leq i \leq r$.  
After this change of variables, we obtain
\begin{equation*}
	S_r (\overline{\eps}, \overline{\delta} ,\overline{k},\overline{p}  ) = \int_{\mathbb{R}^r} \mathrm{d}^r \overline{\tau} \prod_{i=1}^r \mathbbm{I}_{A_{m}} (\lfloor \tau_i d_i \rfloor ) d_i  \mathcal{K}^{(i)}_{m,\overline{\eps},\overline{\delta},\overline{k} } (\tau_i d_i , \tau_{i+1}d_{i+1} ).
\end{equation*}
We have the following bound.
\begin{lem} \label{L:Srbound}
	There is a constant $C>0$ such that  
	\begin{equation*}
		|S_r (\overline{\eps}, \overline{\delta} ,\overline{k},\overline{p} ) | \leq C^r
	\end{equation*}
	for all $(\overline{\delta} ,\overline{k} ) \in D_{r,3}$ and $\overline{\eps} \in \{0,1\}^r$.  
\end{lem}
\begin{proof}
	The proof of this lemma is the same as the one given in \cite[Lemma 4.3]{BCJ18}, noting the extended ranges given in \cref{P:airylimit_asymptotics}. 
\end{proof}
Let 
\begin{equation*}
	D_{s,3}^*=\{(\overline{\delta},\overline{k}, \overline{p}) \in D_{s,3}: k_i\not= k_j \text{ for all }i \not =j\}
\end{equation*}
and let $Q(\overline{\eps})=P(\overline{\eps},(1,\dots,1))$.  Then define 
\begin{equation*}
	U_3^*(m)=\sum_{s=1}^{\infty}\frac{(-1)^{s+1}}{sM^s}\sum_{\overline{\eps} \in \{0,1\}^r} Q(\overline{\eps}) \sum_{(\overline{\delta},\overline{k}, \overline{p}) \in D_{s,3}^*} \prod_{i=1}^s (-1)^{\delta_i} \omega_{p_i} S_s (\overline{\eps},\overline{\delta},\overline{k},\overline{p})
\end{equation*}
\begin{lem}
There is a constant $C > 0$ such that for $R$ sufficiently small and all $|\omega_p| \le R$, $U_3(m)$ is uniformly convergent and $|U_3^*(m)-U_3(m)|<C/M$.  
\end{lem}
\begin{proof}
	Both results are a consequence of using the bound in~\cref{L:Srbound}; see \cite[Lemma 4.4 and Lemma 4.5]{BCJ18} for further details. 
\end{proof}

The above lemma means that we focus on $U_3^*(m)$. For $(\overline{\delta},\overline{k}, \overline{p}) \in D^*_{s, 3}$, let $1\le j_1<\dots <j_r\leq s$ be the indices $i$ where $\delta_i=1$. Let 
$$
\ell_1=j_1-j_r+s, \quad \ell_2=j_2-j_1, \dots, \quad \ell_r=j_r-j_{r-1}.
$$
Observe that $\ell_i \geq 1$ for all $i$ and $\ell_1+\dots +\ell_r=s$ and that we can uniquely reconstruct $j_1,\dots, j_r$ from $\ell_1,\dots, \ell_r$ provided $1 \leq j_1 \leq \ell_1$ is given.  Write $J=\{j_1,\dots, j_r\}$ and $J'=[s]\backslash J$.  Then, we have
\begin{equation*}
	\begin{split}
		&S_r (\overline{\eps}, \overline{\delta} ,\overline{k},\overline{p}  ) = \int_{\mathbb{R}^r} \mathrm{d}^r \overline{\tau}  \prod_{i \in J}  \mathbbm{I}_{A_{p_i, m}} (\lfloor \tau_i d_i \rfloor)d_i  \mathcal{K}^{(i)}_{m,\overline{\eps},\overline{\delta},\overline{k} } (\tau_i d_i, \tau_{i+1} d_{i+1}) \\
		&\times \prod_{i \in J'}  \mathbbm{I}_{A_{p_i,m}} (\lfloor \tau_id_i \rfloor)d_i  \mathcal{K}^{(i)}_{m,\overline{\eps},\overline{\delta},\overline{k} } (\tau_i d_i , \tau_{i+1} d_{i+1}) 
	\end{split}
\end{equation*}
We can use \cref{P:airylimit_asymptotics} (statements (2) and (3))  to write that for $i \in J'$ 
\begin{equation}\label{E:cumulant:gaussian}
	\begin{split}
		&d_i  \mathcal{K}^{(i)}_{m,\overline{\eps},\overline{\delta},\overline{k} } (\tau_i d_i , \tau_{i+1} d_{i+1}) \\
		&= \frac{1}{\sqrt{4 \pi}} e^{-\frac{(\tau_i-\frac{\tau_{i+1} d_{i+1}}{d_i})^2}{4}} (1+o(1)) \mathbbm{I}_{|\tau_i-\frac{d_{i+1}}{d_i} \tau_{i+1}| \leq c_2 (\log m)^{1/6}} \mathbbm{I}_{k_i>k_{i+1}}
	\end{split}
\end{equation}
Next we have that on $\lfloor d_i \tau_i \rfloor \in A_{p_i,m}$ with $i \in J$, 
\begin{equation*}
	d_i  \mathcal{K}^{(i)}_{m,\overline{\eps},\overline{\delta},\overline{k} } (\tau_i d_i , \tau_{i+1} d_{i+1})=\tilde{\mathcal{A}}(\tau_i;\tau_{i+1})(1+o(1))
\end{equation*}
by Statement (1) in \cref{P:airylimit_asymptotics}. The fact that the Airy functions decay when $\tau_i \to \infty$ means that we can write 
\begin{equation*}
	\begin{split}
		&\lim_{m \to \infty}S_r(\overline{\eps},\overline{\delta},\overline{k},\overline{p})=\prod_{i \in J'}\mathbbm{I}_{k_i>k_{i+1}} 
		\int_{\mathbb{R}^r} \prod_{j \in J}\mathrm{d}\tau_j \times \prod_{i=1}^{r} \mathbbm{I}_{A_{p_i}}(\tau_{j_i})   \tilde{\mathcal{A}}\big(\tau_{j_i},\tau_{j_{i+1}}\big)
	\end{split}
\end{equation*}
where we have computed the Gaussian integral from \eqref{E:cumulant:gaussian} and used the indicator functions for the integral with respect to $\tau_i$ for $i\in J$.  
The proof now proceeds exactly as in \cite{BCJ18} in that one has to use 
$$\lim_{m \to \infty} \frac{1}{M^s}\sum_{\overline{k}\in [M]^s}\prod_{i \in J'} \mathbbm{I}_{k_i>k_{i+1}}= \frac{1}{\ell_1! \dots \ell_r!}.$$
along with \cite[Lemma 4.6]{BCJ18} that states that $\sum_{\overline{\eps} \in \{0,1\}^s}Q(\overline{\eps})=(-1)^s$, and a resummation to get the right side of the statement of the theorem; we omit these details since they can be found in \cite{BCJ18}.

	\section{Proof of results from Section \ref{subsec:KasteleynTheory}} \label{subsec:Kinvasympproofs}
	
	We now give the proofs of the results given in \cref{subsec:KasteleynTheory}, including those relating to the refined asymptotic analysis of $K^{-1}_{a,1}$.  
	
	\subsection{Proof of \cref{lem:switchcontoursBtilde}}
	The next lemma is a restatement of~\cite[Lemma 3.6]{CJ16} and so will be stated without proof. 
	\begin{lem}\label{lem:prev:termsdecay}
		Rescale and center the Aztec diamond so that $x=([4m\eta_1]+\overline{x}_1,[4m\eta_1]+\overline{x}_2) \in \mathtt{W}_{\eps_1}$ and $y=([4m\eta_2]+\overline{y}_1,[4m\eta_2]+\overline{y}_2) \in \mathtt{B}_{\eps_2}$ for $\eps_1,\eps_2\in \{0,1\}$ and $-1<\eta_1,\eta_2<0$ where $\overline{x}_1,\overline{x}_2,\overline{y}_1,\overline{y}_2$ are at most order $m^{5/6}$, there are  positive constants $\mathtt{C}_1$ and $\mathtt{C}_2$ so that
        \begin{align*}
			\left|\frac{1}{a}\mathcal{B}_{1-\eps_1,\eps_2}(a^{-1},n-x_1,n+x_2,n-y_1,n+y_2 \right|&\leq  \mathtt{C}_1 e^{-\mathtt{C}_2 n}, \\
			\left|\frac{1}{a}\mathcal{B}_{\eps_1,1-\eps_2}(a^{-1},n+x_1,n-x_2,n+y_1,n-y_2) \right| &\leq  \mathtt{C}_1 e^{-\mathtt{C}_2 n},\\
			\left|\mathcal{B}_{1-\eps_1,1-\eps_2}(a,n-x_1,n-x_2,n-y_1,n-y_2) \right| &\leq  \mathtt{C}_1 e^{-\mathtt{C}_2 n}.
        \end{align*}
		
	\end{lem}
	The statement in \cite[Lemma 3.6]{CJ16} differs slightly in that the vertices $x$ and $y$ are at the same asymptotic  coordinate (that is, $\eta_1=\eta_2$) and that 
	$\overline{x}_1,\overline{x}_2, \overline{y}_1,\overline{y}_2$ is at most order $m^{2/3}$. However, it follows from the proof of the lemma given there that we can immediately extend to the conditions given in the above lemma. 
	
	Next, set
	\begin{equation}
		\begin{split} \label{eq:prev:Q}
			&Q_{\gamma_1,\gamma_2}^{\eps_1,\eps_2} \left( \omega_1,\omega_2 \right) = (-1)^{\eps_1 +\eps_2 +\eps_1 \eps_2 +\gamma_1 (1+\eps_2) +\gamma_2(1+\eps_1) } \\&
			\times s\left( G\left( \omega_1 \right)\right)^{\gamma_1}   s \left( G\left( \omega_2^{-1} \right)\right)^{\gamma_2} G\left( \omega_1 \right) ^{\eps_1} G\left( \omega_2^{-1}  \right)^{\eps_2} \mathtt{x}^{\eps_1,\eps_2}_{\gamma_1,\gamma_2} \left( a, \omega_1 , \omega_2^{-1}\right),
		\end{split}
	\end{equation}
    where notation is as in \eqref{eq:prev:s} and \eqref{eq:prev:x}. The following corollary from \cite{CJ16} expresses the integrals in \cref{thm:prev:Kinverse} in terms of the $\omega$-coordinates. 
    
	\begin{cor}[Part of Corollary 2.4, \cite{CJ16}] \label{cor:prev:KtoB}
		For $n=4m$, let $x=(x_1,x_2) \in \mathtt{W}_{\eps_1}$ and $y=(y_1,y_2) \in \mathtt{B}_{\eps_2}$ with $\eps_1,\eps_2 \in \{0,1\}$.  Suppose that  $-n < x_1,x_2,y_1,y_2 < 0$, $\sqrt{2c}<p<1$ and  $Q^{\eps_1,\eps_2}_{\gamma_1,\gamma_2}$ is as given in~\eqref{eq:prev:Q}. We have
		\begin{equation}
			\begin{split}
				~\label{cor:eq:prev:KtoB1}
				&\mathcal{B}_{\eps_1,\eps_2} (a,n+x_1,n+x_2,n+y_1,n+y_2) =  \frac{ \mathrm{i}^{\frac{x_2-x_1+y_1-y_2}{2}}}{(2\pi \mathrm{i})^2}  \int_{\Gamma_{p}}\frac{\mathrm{d} \omega_1}{\omega_1} \int_{\Gamma_{1/p}}  \mathrm{d} \omega_2 \\ & \times\frac{\omega_2  }{\omega_2^2-\omega_1^2}  \frac{ H_{x_1+1,x_2}(\omega_1)}{H_{y_1,y_2+1}(\omega_2)} \sum_{\gamma_1,\gamma_2=0}^1 Q_{\gamma_1,\gamma_2}^{\eps_1,\eps_2}(\omega_1 ,\omega_2),
			\end{split}
		\end{equation}
		
			\end{cor}
	
	\begin{proof}[Proof of \cref{lem:switchcontoursBtilde}]
    We have that \eqref{eq:prev:B} is \cite[Equation (3.11)]{CJ16} and we provide a short derivation. 
		The starting point for the formula for $K^{-1}_{a,1}$ is given in \cref{thm:prev:Kinverse}. Using \cref{lem:prev:termsdecay}, we get \eqref{E:first-eq-10.3} in \cref{lem:switchcontoursBtilde}.  We can use partial fractions on \eqref{cor:eq:prev:KtoB1} and a change of variables to get \eqref{eq:prev:B}, which was detailed in \cite[Corollary 2.4 and Eq. (3.11)]{CJ16}.  
		To get \eqref{eq:lem:switchcontours}, we switch contours in \eqref{eq:prev:B} which picks up a single residue at $\omega_2=\omega_1$.  \cref{E:KCS} is given in \cite[Lemma 3.3]{CJ16} while the last statement is \cite[Lemma 2.5]{CJ16}.
	\end{proof}

	\begin{proof}[Proof of \cref{L:Kinv:rotationestimate}]
		This lemma follows from the formulas given in \cite[Corollary 2.4]{CJ16} and using the same steps used in the proof of \cref{lem:prev:termsdecay}; see the proof of \cite[Lemma 3.6]{CJ16}.
	\end{proof}

	\subsection{Proof of \cref{thm:prev:KinversearoundRS}}
    \label{subsec:thm:KinversearoundRS}
	
	Here, we provide the proof of \cref{thm:prev:KinversearoundRS}.  We split into the four cases given in the statement of the theorem.  In each case, we ignore the integer parts as they are not important.  Before the proof of each case, we provide the key lemmas and their proofs. We also need the following lemma, which is used for all the cases given in \cref{thm:prev:KinversearoundRS}. 
	
	\begin{lem} \label{asym:lem:imaginary}
		The behavior of  $\mathrm{Im}\, g_{\xi_1,\xi_2} (\omega)$ for $\omega \in \partial \mathbb{H}_+$ is given by
		\begin{enumerate}
			\item $\mathrm{Im}\, g_{\xi_1,\xi_2} ( \mathrm{i}t)$ decreases from $(1-2\xi_1+3 \xi_2) \frac{\pi}{2}$ to $(1-\xi_1+3\xi_2 ) \frac{\pi}{2}$ for $t$ increasing in $(0,\sqrt{2c})$,
			\item $\mathrm{Im}\, g_{\xi_1,\xi_2} ( \mathrm{i}t)=(1-\xi_1+3\xi_2) \frac{\pi}{2}$ for $t \in(\sqrt{2c},1/\sqrt{2c})$,
			\item $\mathrm{Im}\, g_{\xi_1,\xi_2} ( \mathrm{i}t)$ increases from $(1-\xi_1+3\xi_2 ) \frac{\pi}{2}$ to $(1 -\xi_1+2\xi_2) \frac{\pi}{2}$ for $t$ increasing in $(1/\sqrt{2c},\infty)$,
			\item $\mathrm{Im}\, g_{\xi_1,\xi_2} ( t)= 0$ for $t \in (0,\infty)$,
			\item  $\mathrm{Im}\, g_{\xi_1,\xi_2} (\omega)$ increases from $0$ to $(1-2\xi_1+3\xi_2 ) \frac{\pi}{2}$  depending on the angle of approach as $\omega$ tends to 0,
			\item  $\mathrm{Im}\, g_{\xi_1,\xi_2} (\omega)$ decreases from $(1-\xi_1+2\xi_2 ) \frac{\pi}{2}$  to $0$ depending on the angle of approach as $\omega$ tends to infinity. 
			
		\end{enumerate}
		
	\end{lem}
	The proof of this lemma follows a very similar computation to \cite[Lemma 3.8]{CJ16}  and so we omit the proof and refer the reader there.  
	
	\begin{proof}[Proof of \eqref{thm:eq:KinversearoundRS:sm} in
		\cref{thm:prev:KinversearoundRS}] 
The case when $-\xi_c - \gamma<X,Y \leq  -\xi_c+\gamma$ (for $\gamma$ sufficiently small) follows immediately due to the choices in contours for each of the four integrals involving $\mathcal{B}_{\eps_1,\eps_2}$ having the $\omega_1$-contour contained inside the $\omega_2$-contour and taking a simple bound.  For the case when $m^{-1/6+\delta}\leq X,Y \leq -\xi_c - \gamma$ we need the following lemma, which is proven after the proof of \eqref{thm:eq:KinversearoundRS:sm}.

		\begin{lem}\label{lem:asym:contourssm}
			For the scaling of $x$ and $y$ given in \eqref{eq:def:sm},  we have that $\omega_{c,X}$ and $\tilde{\omega}_{c,Y}$ defined through \eqref{eq:def:omegacX} are such that 
			$\omega_{c,X}, \tilde{\omega}_{c,Y} \in \mathrm{i}(\sqrt{2c},1/\sqrt{2c})$.  Then, we have that 
			\begin{itemize}
				\item $\mathrm{Re}\, g_{\xi_1+X,\xi_2+X}(\omega_{c,X})<0$,
				\item $\mathrm{Re}\, g_{\tilde{\xi}_1+Y,\tilde{\xi}_2+Y}(\tilde{\omega}_{c,Y})>0$,
				\item $g''_{\xi_1+X,\xi_2+X}(\omega_{c,X})<0$,
				\item $g''_{\tilde{\xi}_1+Y,\tilde{\xi}_2+Y}(\tilde{\omega}_{c,Y})>0$,
				\item if $X=Y$ and $(\tilde{\xi}_1,\tilde{\xi}_2)=(\xi_1,\xi_2)$, then $|\omega_{c,X}| <|\tilde{\omega}_{c,Y}|$.
			\end{itemize}
			The curve of steepest descent of $\mathrm{Re}\,g_{\xi_1+X,\xi_2+X}$ goes from $\omega_{c,X}$, perpendicular to the imaginary axis, to 0 and a curve of steepest ascent of $\mathrm{Re}\,g_{\tilde{\xi}_1+Y,\tilde{\xi}_2+Y}$ from $\tilde{\omega}_{c,Y}$, along the imaginary axis, to $\infty$.  
		\end{lem}

		We assume the above lemma and proceed with the proof of \eqref{thm:eq:KinversearoundRS:sm}. We follow the saddle point analysis in \cite[Proposition 3.12]{CJ16} and refer the reader there for more details.  Our starting point is the formula given in \eqref{thmproof:aroundRSsm:B1}.   We deform the contours to the curves given in \cref{lem:asym:contourssm}, splitting the integral into a local contribution around the saddle points $\omega_{c,X}$ and $\tilde{\omega}_{c,Y}$, and a remaining contribution which is exponentially small which can be ignored; see \cite[Proposition 3.12]{CJ16}.  For the local contribution, we apply a change of variables $\omega_1=\omega_{c,X}+(2m)^{-1/2}w_1$ and $\omega_2=\tilde{\omega}_{c,Y}+(2m)^{-1/2}w_2$. Applying a Taylor expansion, the term in the first exponential in \eqref{thmproof:aroundRSsm:B1} becomes
		\begin{equation*}
			\begin{split}
				&2m g_{\xi_1+X,\xi_2+X}(\omega_1) -\log G(\omega_1) =
				2m g_{\xi_1+X,\xi_2+X}(\omega_{c,X}) -\log G(\omega_{c,X})\\
				& +\frac{1}{2}g''_{\xi_1+X,\xi_2+X}(\omega_{c,X})w_1^2+O(w_1^3m^{-1/2})\\
				&=\log H_{x_1+1,x_2}(\omega_{c,X})+\frac{1}{2}g''_{\xi_1+X,\xi_2+X}(\omega_{c,X})w_1^2+O(w_1^3m^{-1/2}).
			\end{split}
		\end{equation*}
		We have a similar formula in the $\omega_2$ variable for the term in the second exponential in \eqref{thmproof:aroundRSsm:B1}. Collecting the error terms in the above expansion means that \eqref{thmproof:aroundRSsm:B1} becomes
		\begin{equation}
			\begin{split}\label{thmproof:aroundRSsm:B2}
				&\mathcal{B}_{\eps_1,\eps_2}(a,n+x_1,n+x_2,n+y_1,n+y_2)\\&=\frac{(1+O(m^{-1/2}))}{2m}\frac{\mathrm{i}^{\frac{x_2-x_1+y_1-y_2}{2}}}{(2\pi \mathrm{i})^2} \int \frac{\mathrm{d}w_1}{\omega_{c,X}} \int \mathrm{d}w_2 \frac{H_{x_1+1,x_2}(\omega_{c,X})}{H_{y_1,y_2+1}(\tilde{\omega}_{c,Y})}\\
				&\times \frac{V_{\eps_1,\eps_2}(\omega_{c,X},\tilde{\omega}_{c,Y})e^{\frac{1}{2}g''_{\xi_1+X,\xi_2+X}(\omega_{c,X})w_1^2-\frac{1}{2}g''_{\tilde{\xi}_1+Y,\tilde{\xi}_2+Y}(\tilde{\omega}_{c,Y})w_2^2}}{\omega_{c,X} - \tilde{\omega}_{c,Y}+(2m)^{-\frac{1}{2}}(w_1-w_2)}  
			\end{split}
		\end{equation}
		where the contours are local contours around the origin (of length $m^{-\epsilon}$ for some $\epsilon>0$) with the $w_1$ and $w_2$ contours are positively oriented along the real and imaginary axis respectively, which follows from \cref{lem:asym:contourssm}. If $\omega_{c,X}=\tilde{\omega}_{c,Y}$, the above singularity is integrable. In any case and noting the signs in the exponent from \cref{lem:asym:contourssm}, we extend the contours to infinity which gives a Gaussian integral and so we obtain a bound
		\begin{equation*}
			|\mathcal{B}_{\eps_1,\eps_2}(a,n+x_1,n+x_2,n+y_1,n+y_2)| \leq \frac{\mathtt{C}_1}{\sqrt{m}} \bigg| \frac{H_{x_1+1,x_2}(\omega_{c,X})}{H_{y_1,y_2+1}(\tilde{\omega}_{c,Y})} \bigg|
		\end{equation*}
		The result then follows by using the exponential form of $H$ given in~\eqref{eq:def:logH} and third and fourth statements in \cref{lem:asym:contourssm}, and noting that $2m\mathrm{Re}\,g_{\xi_1+X,\xi_2+X}(\omega_{c,X})$ is at most $-m^{5/6+\delta}$ from the choice of scaling in \eqref{eq:def:sm}, and that $V$ is uniformly bounded on compact subsets.

		It remains to prove the assumed lemma stated above. 
		\begin{proof}[Proof of \cref{lem:asym:contourssm}]
			The proof of this lemma uses the following two lemmas from \cite{CJ16}.
			\begin{lem}[Lemma 3.9 in \cite{CJ16}]
            \label{lem:zeroes}
				The function $g_{\xi,\xi}(\omega)$ for $\omega \in \mathbb{H}_+ \cup \partial \mathbb{H}_+$ has two distinct saddle points in the interval $(\sqrt{2c},1/\sqrt{2c})\mathrm{i}$ (but not equal to $\mathrm{i}$) if and only if $-1/2 \sqrt{1-2c}<\xi<0$. These saddle points are given by $\omega_c \in (\sqrt{2c},1)\mathrm{i}$ and $-\omega_c^{-1}$. We also have $\mathrm{Re}\, g_{\xi,\xi}(\omega_c)<0$ and $\mathrm{Re}\, g_{\xi,\xi}(-\omega_c^{-1})>0$. 
			\end{lem}
			\begin{lem}[Lemma 3.10 in \cite{CJ16}]\label{lem:lem310inCJ16}
				For $-\frac{1}{2} \sqrt{1-2c} < \xi <0$, choose $\omega_c$ such that $g_{\xi,\xi}'(\omega_c)=0$ with $\omega_c \in (\sqrt{2c},1)\mathrm{i}$. Then $g_{\xi,\xi}'(-\omega_c^{-1})=0$,
				\begin{equation*}
					g_{\xi,\xi}''(\omega_c)<0 \hspace{5mm}\mbox{and} \hspace{5mm} g_{\xi,\xi}''(-\omega_c^{-1})>0.
				\end{equation*}
			\end{lem}

Turning to the proof of \cref{lem:asym:contourssm}, letting $\omega = \mathrm{i} t$ and recalling $\theta$ defined in \eqref{eq:prev:theta}, we have that 
			\begin{equation*}
				0=\mathrm{i} t g'_{\xi_1+X,\xi_2+X}(\mathrm{i}t)= 1 +(\xi_1+X)\theta(t)+(\xi_2+X)\theta(t^{-1}).
			\end{equation*}
            When $\xi_1 = \xi_2$, the right-hand side above has zeroes at points $t \in (\sqrt{2c}, 1)$ and $t^{-1}$; this is part of \cref{lem:zeroes}. Since the function is continuous in $\xi_1, \xi_2$, these two zeroes will move continuously as we vary these parameters in an $O(m^{-1/6})$ window. This produces the two critical points defined in \eqref{eq:def:omegacX}, and verifies the fifth bullet point.

Next, using the above two lemmas and the fact that $|\xi_1-\xi_2|<m^{-\frac{1}{6}}$,  $g_{\xi_1,\xi_2}$ is analytic in $\omega$ and locally smooth as a function of $\omega, \xi_1,\xi_2$, we get the first four bullet points.

			Finally, since $g_{\xi_1+X,\xi_2+X}$ is analytic on $\mathbb{C}\backslash(\mathrm{i}(-\infty,-1/\sqrt{2c}]\cup \mathrm{i} [-\sqrt{2c},\sqrt{2c}] \cup \mathrm{i} [1/\sqrt{2c},\infty))$ the contours of steepest descent of $\mathrm{Re}\,g_{\xi_1+X,\xi_2+X}$ are the level lines of $\mathrm{Im}\, g_{\xi_1+X,\xi_2+X}$. \cref{asym:lem:imaginary} gives that $\mathrm{Im}\, g_{\xi_1+X,\xi_2+X} ( \omega_{c,X})=\frac{\pi}{2}(1-\xi_1+3\xi_2+2X)$ and statement 3 of this lemma means that the descent contour goes perpendicular to the imaginary axis and either ends at 0 or $\infty$ by the fifth and sixth statements in \cref{asym:lem:imaginary}. We have that the steepest ascent contour of  $\mathrm{Re}\,g_{\xi_1+X,\xi_2+X}$ must go along the imaginary axis, by the fourth statement of this lemma, to either 0 or $\infty$ by the fifth and sixth statements of \cref{asym:lem:imaginary}. Since the ascent and descent contours cannot cross and by the ordering given in the fifth statement of this lemma, we have that the descent contour must end at 0, while the ascent contour ends at $\infty$. 
		\end{proof}

	\end{proof}
	
	\begin{proof}[Proof of \eqref{thm:eq:KinversearoundRS:smclosetoRS} in \cref{thm:prev:KinversearoundRS}]
		We follow the same proof as given for \eqref{thm:eq:KinversearoundRS:sm} but the analysis needs to be slightly refined and we need  the following lemma which is proved below.
		
		\begin{lem}\label{lem:asym:secondderclosesm}
			We have that $\omega_{c,X}$ and $\tilde{\omega}_{c,Y}$ satisfy the statements of \cref{lem:asym:contourssm} along with $g''_{\xi_1+X,\xi_2+X}(\omega_{c,X})=O(X)$ and $g''_{\xi_1+Y,\xi_2+Y}(\tilde{\omega}_{c,Y})=O(Y)$.  We have that $|\omega_{c,0}|>|\omega_{c,X}|$ and $\omega_{c,0}-\omega_{c,X}=O(\sqrt{X})$ and 
			$|\tilde{\omega}_{c,0}|<|\tilde{\omega}_{c,Y}|$ and $\tilde{\omega}_{c,0}-\tilde{\omega}_{c,Y}=O(\sqrt{Y})$.
		\end{lem}

		As in the proof of \eqref{thm:eq:KinversearoundRS:sm}, we start with \eqref{thmproof:aroundRSsm:B1} and obtain \eqref{thmproof:aroundRSsm:B2} with the contours for the $w_1$ and $w_2$ as stated there. As given in the proof \eqref{thm:eq:KinversearoundRS:sm} we compute the Gaussian integrals and we obtain the bound
		\begin{equation} \label{thmproof:smclosetoRS:eq:B}
			\begin{split}
				&|\mathcal{B}_{\eps_1,\eps_2}(a,n+x_1,n+x_2,n+y_1,n+y_2) \leq \bigg| \frac{H_{x_1+1,x_2}(\omega_{c,X})}{H_{y_1,y_2+1}(\tilde{\omega}_{c,Y})}\bigg| \\&\times \frac{\mathtt{C}_1}{\sqrt{-mg''_{\xi_1+X,\xi_2+Y}(\omega_{c,X}) g''_{\tilde{\xi}_1+X,\tilde{\xi}_2+X}(\tilde{\omega}_{c,Y}) }}
			\end{split}
		\end{equation}
		where $\mathtt{C}_1$ is a positive constant.
		In the above equation, it suffices to estimate $H_{x_1+1,x_2}(\omega_{c,X})$ as the estimate for $H_{y_1,y_2+1}(\tilde{\omega}_{c,Y})$ follows by a similar computation.  Taking the logarithms and using \eqref{eq:def:logH}, we need to consider
		\begin{equation*}
			\begin{split}
				&g_{\xi_1+X,\xi_2+X}(\omega_{c,X})=g_{\xi_1,\xi_2}(\omega_{c,X}) -X (\log G(\omega_{c,X})-\log G(\omega_{c,X}^{-1}))\\
				&=g_{\xi_1,\xi_2}(\omega_{c,0}) -X(\log G(\omega_{c,0})-\log G(\omega_{c,0}^{-1}))\\ 
				&+(\omega_{c,X}-\omega_{c,0})\frac{X}{\mathrm{i}t_{c,0}}(\theta(t_{c,0})+\theta(t_{c,0}^{-1}))+O(X (\omega_{c,X}-\omega_{c,0})^2)\\
				&=g_{\xi_1+X,\xi_2+X}(\omega_{c,0}) +(t_{c,X}-t_{c,0})\frac{X}{t_{c,0}}(\theta(t_{c,0})+\theta(t_{c,0}^{-1}))+O(X  (\omega_{c,X}-\omega_{c,0})^2)
			\end{split}
		\end{equation*}
		where the second line follows by a Taylor series, setting $\omega_{c,0}=\mathrm{i}t_{c,0}$ and $\omega_{c,X}=\mathrm{i}t_{c,X}$, using that $\omega_{c,0}$ satisfies $g'_{\xi_1,\xi_2}(\omega_{c,0})=g''_{\xi_1,\xi_2}(\omega_{c,0})=0$, and using the definition of $\theta$ in \eqref{eq:prev:theta}. We then have
		\begin{equation*}
			|H_{x_1+1,x_2}(\omega_{c,X})|=|H_{x_1+1,x_2}(\omega_{c,0})| e^{(t_{c,X}-t_{c,0})\frac{X}{t_{c,0}} (\theta(t_{c,0})+\theta(t_{c,0}^{-1}))+O(X^2)}.
		\end{equation*}
		The result follows from using the bound on $(\omega_{c,X}-\omega_{c,0})$ in \cref{lem:asym:secondderclosesm}, applying the same type of computation to ${H_{y_1,y_2+1}(\tilde{\omega}_{c,Y})}$ and subsituting the estimates back into \eqref{thmproof:smclosetoRS:eq:B}.

		\begin{proof}[Proof of \cref{lem:asym:secondderclosesm}]
			We prove the lemma for $X$; the assertions hold for $Y$ by similar computations. 
			Since $(\xi_1+X,\xi_2+X)$ is not on the curve \eqref{eq:asymp:limitshape}, we have that that $g''_{\xi_1+X,\xi_2+X}(\omega_{c,X})\not = 0$ and so $\omega_{c,X}$ satisfy the statements of \cref{lem:asym:contourssm} by analyticity.  
			
			Let $\omega_{c,0}=\mathrm{i} t_{c,0}$, $\omega_{c,X}=\mathrm{i} t_{c,X}$ and recall \eqref{eq:prev:theta}.
			Suppose that $t_{c,X}-t_{c,0}=\epsilon$. From \eqref{eq:def:gprime} with $\xi_i \mapsto \xi_i+X$ for $i \in \{1,2\}$, we get from a Taylor expansion
			\begin{equation}
				\begin{split} \label{lemproof:eq:secondderclosesm}
					0&=1 + (\xi_1+X) \theta( t_{c,X}) + (\xi_2+X)\theta((t_{c,X})^{-1}) \\
					&=1+\xi_1\theta(t_{c,0})+ \xi_2 \theta(t_{c,0}^{-1}) +\epsilon\big(\xi_1\theta'(t_{c,0})+ \xi_2 \frac{\mathrm{d}}{\mathrm{d}t} \theta(t^{-1})\big|_{t=t_{c,0}}\big)\\&  +\frac{\epsilon^2}{2}\big(\xi_1\theta''(t_{c,0})+ \xi_2 \frac{\mathrm{d}^2}{\mathrm{d}t^2} \theta(t^{-1})\big|_{t=t_{c,0}}\big) +X(\theta( t_{c,X})+\theta(t_{c,X}^{-1}))+O(\epsilon^3)\\
					&=\frac{\epsilon^2}{2}\big(\xi_1\theta''(t_{c,0})+ \xi_2 \frac{\mathrm{d}^2}{\mathrm{d}t^2} \theta(t^{-1})\big|_{t=t_{c,0}}\big) +X(\theta( t_{c,X})+\theta(t_{c,X}^{-1}))+O(\epsilon^3)
				\end{split}
			\end{equation}  
			where the last line follows since $(\xi_1,\xi_2)$ is a double zero of \eqref{eq:def:gprime}. The remaining terms in the above equation are non-zero and noting that $\theta''(t_{c,0}),\theta(t_{c,0})>0$ and $\xi_1,\xi_2<0$ it follows $\epsilon$ is real and is $O(\sqrt{X})$.  Finally, we get $|\omega_{c,0}|>|\omega_{c,X}|$ which holds by the fifth statement of \cref{lem:asym:contourssm} and using analyticity. 
	k	
		\end{proof}
	\end{proof}
    
	\begin{proof}[Proof of \eqref{thm:eq:KinversearoundRS:RS} in \cref{thm:prev:KinversearoundRS}] 
	We need the following lemma which we prove first before commencing with the proof of \eqref{thm:eq:KinversearoundRS:RS}. 
		\begin{lem}\label{lem:asym:contoursRS}
			For $\omega_{c,0}=\mathrm{i} t_{c,0}$, we have that  $t_{c,0} \in (\sqrt{2c},1/\sqrt{2c})$  and that
			\begin{equation*}
				g'''_{\xi_1,\xi_2}(\mathrm{i} t_{c,0})=-\mathrm{i} \mathcal{G}_{c,t}
			\end{equation*}
			with $\omega_{c,0}=\mathrm{i} t_{c,0}$ and $t\in (\sqrt{2c},1/\sqrt{2c})$ where 
			\begin{equation*}
				\mathcal{G}_{c,t}=\frac{2c(2c-t^2+8c^2t^2-4ct^4-3t^6+6ct^8)}{t^5(2c-t^2)(-1+2c t^2)(-1+4ct^2-t^4)}.
			\end{equation*}	
			In $\mathbb{H}_+ \cup \partial \mathbb{H}_+$, there is a curve of steepest descent leaving $\omega_{c,0}$ at angle $-\pi/6$ going to the origin and a curve of steepest ascent leaving $\omega_{c,0}$ at angle $\pi/6$ going to $\infty$. 		
		\end{lem}

\begin{proof}[Proof of \cref{lem:asym:contoursRS}]
			We can differentiate ~\eqref{eq:def:gprime} which gives
			\begin{equation}\label{eq:def:gdoubleprime}
				\omega g''_{\xi_1,\xi_2}(\omega)+ g'_{\xi_1,\xi_2}(\omega)= \xi_1 \frac{2c}{(\omega^2+2c)^\frac{3}{2}} -\xi_2 \frac{2c}{(\omega^{-2}+2c)^{\frac{3}{2}}}.
			\end{equation}
			Asserting that $g'_{\xi_1,\xi_2}(\omega_{c,0})=g''_{\xi_1,\xi_2}(\omega_{c,0})=0$ in \eqref{eq:def:gprime} and the above equation gives a matrix equation for $(\xi_1,\xi_2)$ which can be solved giving
			\begin{equation}
				\begin{split}\label{eq:def:xi12rs}
					(\xi_1,\xi_2)&=\bigg( \frac{-\omega_{c,0}(\omega_{c,0}^2+2c)^{\frac{3}{2}}}{1+4c \omega_{c,0}^2+\omega_{c,0}^4} , \frac{-(\omega_{c,0}^{-2}+2c)^{\frac{3}{2}}}{\omega_{c,0}(1+4c \omega_{c,0}^2+\omega_{c,0}^4)} \bigg)\\
					&=\bigg( \frac{-t_{c,0}(t_{c,0}^2-2c)^{\frac{3}{2}}}{1-4c t_{c,0}^2+t_{c,0}^4} , \frac{-(t_{c,0}^{-2}-2c)^{\frac{3}{2}}}{t_{c,0}(1-4c t_{c,0}^2+t_{c,0}^4)} \bigg)
				\end{split}
			\end{equation}
			where $\omega_{c,0}=\mathrm{i} t_{c,0}$. The above equation solves \eqref{eq:asymp:limitshape} and we need to verify that it is on the rough-smooth boundary. To see this, notice that for $t_{c,0}\in (\sqrt{2c},1/\sqrt{2c})$, we have that $\xi_1,\xi_2<0$ and for $t_{c,0}=1$, $\xi_1=\xi_2=\xi_c$ as needed. 
			
			We differentiate \eqref{eq:def:gdoubleprime} which gives
			\begin{equation*}
				\omega g'''_{\xi_1,\xi_2}(\omega)+ 2g''_{\xi_1,\xi_2}(\omega)=-\frac{6c \omega}{(\omega^2+2c)^{\frac{5}{2}}} \xi_1+ \frac{2c(-1+4c\omega^2)}{\sqrt{2c+\omega^{-2}}(\omega+2c \omega^3)^2} \xi_2.
			\end{equation*}
			This gives 
			\begin{equation*}
				g'''_{\xi_1,\xi_2}(\mathrm{i} t_{c,0})=-\frac{6c }{\mathrm{i}(t_{c,0}^2-2c)^{\frac{5}{2}}}+\frac{2c(-1-4ct_{c,0}^2)}{\mathrm{i}t_{c,0}\sqrt{2c+t_{c,0}^{-2}}(t_{c,0}-2c t_{c,0}^3)^2} \xi_2.
			\end{equation*}
			Substituting the values of $\xi_1$ and $\xi_2$ found in \eqref{eq:def:xi12rs} gives the result for $ g'''_{\xi_1,\xi_2}$ after a computation. 
			
			The local angles of the curves of steepest descent and ascent around the double critical point follows from the fact that $\mathcal{G}_{c,t}>0$ for $t \in (\sqrt{2c},1/\sqrt{2c})$.  The rest of the argument for the location of the curves of steepest descent and ascent is the same as the one in \cref{lem:asym:contourssm}. 	
		\end{proof}

    We now proceed with the proof of \eqref{thm:eq:KinversearoundRS:RS} in \cref{thm:prev:KinversearoundRS}.
		We follow the saddle point analysis in \cite[Proposition 3.17]{CJ16} and refer the reader there for more details. 
		Our starting point is the formula given in \eqref{thmproof:aroundRSsm:B1}.   We deform the contours to those given in \cref{lem:asym:contoursRS}, splitting the integral to a local contribution around the saddle points $\omega_{c,0}$ and $\tilde{\omega}_{c,0}$ defined through \eqref{eq:def:omegacX}, and a remaining contribution which is exponentially small which can be ignored; see \cite[Proposition 3.17]{CJ16}.  For the local contribution, we apply a change of variables $\omega_1=\omega_{c,0}+(2m)^{-1/3}w_1$ and $\omega_2=\tilde{\omega}_{c,0}+(2m)^{-1/3}w_2$.  Applying a  Taylor series gives the terms in the exponent with respect to $\omega_1$ becoming
		\begin{equation*}
			\begin{split}
				&2m g_{\xi_1+X,\xi_2+X}(\omega_1) -\log G(\omega_1) \\&=
2m g_{\xi_1,\xi_2}(\omega_1) -(2mX+1)\log G(\omega_1) +2mX \log G(\omega_1^{-1})
\\&=2m g_{\xi_1+X,\xi_2+X}(\omega_{c,0}) -\log G(\omega_{c,0})\\
				&+\frac{1}{6}g'''_{\xi_1,\xi_2}(\omega_{c,0})w_1^3 +(2m)^{\frac{2}{3}}Xw_1\bigg(\frac{1}{\sqrt{\omega_{c,0}^2+2c}}+\frac{1}{\omega_{c,0}^2\sqrt{\omega_{c,0}^{-2}+2c}}\bigg)\\
				& +O(w_1^3m^{-1/3})=\log H_{x_1+1,x_2}(\omega_{c,0})
				+\frac{1}{6}g'''_{\xi_1,\xi_2}(\omega_{c,0})w_1^2\\&+\frac{(2m)^{\frac{2}{3}}Xw_1}{\omega_{c,0}} (\theta (-\mathrm{i}\omega_{c,0})-\theta(\mathrm{i}\omega_{c,0}^{-1}))+O(w_1^3m^{-1/3}).
			\end{split}
		\end{equation*}
		 In the above display, we have rewrite $g_{\xi_1+X,\xi_2+X}(\omega_1) $ in terms of $g_{\xi_1,\xi_2}(\omega_1)$ and an additional term from \eqref{eq:prev:gxi12} in order to facilitate the Taylor series. We have also used $-\mathrm{i}\omega_{c,0}\in (\sqrt{2c},1/\sqrt{2c})$  to keep the notation compact. 
	A similar expression holds for the $\omega_2$ variable. 	
    Collecting the error terms and the above expansion means that \eqref{thmproof:aroundRSsm:B1} becomes
		\begin{equation*}
			\begin{split}
				&\mathcal{B}_{\eps_1,\eps_2}(a,n+x_1,n+x_2,n+y_1,n+y_2)=\frac{(1+O(m^{-1/3}))}{(2m)^{2/3}}\frac{\mathrm{i}^{\frac{x_2-x_1+y_1-y_2}{2}}}{(2\pi \mathrm{i})^2}  \\
				&\times \int \frac{\mathrm{d}w_1}{\omega_{c,X}} \int \mathrm{d}w_2 \frac{V_{\eps_1,\eps_2}(\omega_{c,0},\tilde{\omega}_{c,0})e^{\frac{1}{6}g'''_{\xi_1,\xi_2}(\omega_{c,0})w_1^3-\frac{1}{6}g'''_{\tilde{\xi}_1,\tilde{\xi}_2}(\tilde{\omega}_{c,0})w_2^3}}{\omega_{c,0} - \tilde{\omega}_{c,0}+(2m)^{-\frac{1}{3}}(w_1-w_2)}  \frac{H_{x_1+1,x_2}(\omega_{c,X})}{H_{y_1,y_2+1}(\tilde{\omega}_{c,Y})} \\
				&\times e^{\frac{(2m)^{\frac{2}{3}}Xw_1}{\omega_{c,0}} (\theta (-\mathrm{i}\omega_{c,0})-\theta(\mathrm{i}\omega_{c,0}^{-1}))-\frac{(2m)^{\frac{2}{3}}Yw_2}{\tilde{\omega}_{c,0}} (\theta (-\mathrm{i}\tilde{\omega}_{c,0})-\theta(\mathrm{i}\tilde{\omega}_{c,0}^{-1}))}
			\end{split}
		\end{equation*}
		where the contours are local contours around the origin (of length $m^{-\epsilon}$ for some $\epsilon>0$) with the $w_1$ contour leaving the origin at $-\pi/6$ and $w_2$ contour leaving the origin at $\pi/6$, which follows from \cref{lem:asym:contoursRS}.  From the choice $|\xi_1-\tilde{\xi}_1|<m^{-1/3}$, we have that  $\omega_{c,0}=\tilde{\omega}_{c,0}$ up to an error of $O(m^{-1/3})$ by using \eqref{eq:def:xi12rs} to estimate $|\omega_{c,0}-\tilde{\omega}_{c,0}|$ from $|\xi_1-\tilde{\xi}_1|$.  Note that the above singularity is integrable.  We take the change of variables $w_i\mapsto \mathrm{i}w_i$, extend the contours to infinity and so the above formula is equal to the Airy kernel evaluated at 
		$$
		\bigg(\frac{(2m)^{\frac{2}{3}}X}{\omega_{c,0}} (\theta (-\mathrm{i}\omega_{c,0})-\theta(\mathrm{i}\omega_{c,0}^{-1})),\frac{(2m)^{\frac{2}{3}}Y}{{\omega}_{c,0}} (\theta (-\mathrm{i}{\omega}_{c,0})-\theta(\mathrm{i}{\omega}_{c,0}^{-1}))\bigg)
		$$
		which is bounded due to the choice of $X$ and $Y$ in \eqref{eq:def:rs}.  The result follows.

	\end{proof}
	
	\begin{proof}[Proof of \eqref{eq:thm:KinversearoundRS:rough} in \cref{thm:prev:KinversearoundRS}]
We give a shortened proof of this result since we mirror the proof of \cite[Section 3.6]{CJ16}. 
We need to find the asymptotic behaviour of \eqref{thmproof:aroundRSsm:B1} and so we first analyze~\eqref{eq:def:gprime} to understand the roots of this equation. These roots will enable us to determine the choice of contours in \eqref{thmproof:aroundRSsm:B1}.

It is sufficient to analyze~\eqref{eq:def:gprime} by a Taylor series. 
For $\omega_{c,X}$ defined by \eqref{eq:def:omegacX},  we have that $\overline{\omega_{c,X}}$, $-\omega_{c,X}$ and $-\overline{\omega_{c,X}}$ are also roots of $g'_{\xi_1+X,\xi_2+X}(\omega)=0$. We next determine that $\omega_{c,X} \in \mathbb{H}_+$ for the choice of $X$ given in \eqref{eq:def:rgclosetoRS} since if $\omega_{c,X} \in \mathbb{H}_+$, then these four points are all distinct and define four roots of \eqref{eq:def:gprime} --- \cite[Lemma 3.7]{CJ16} states there are at most four roots of \eqref{eq:def:gprime}.   By a Taylor series (cf. with the computation for~\eqref{lemproof:eq:secondderclosesm}), we get that 
		\begin{equation*}
			0=\frac{\epsilon^2}{2}(\xi_1\theta''(-\mathrm{i} \omega_{c,0})+ \xi_2 \theta''(\mathrm{i}\omega_{c,0}^{-1})) +X(\theta(-\mathrm{i} \omega_{c,X})+\theta(\mathrm{i}\omega_{c,X}^{-1}))+O(\epsilon^3)
		\end{equation*}
		where $\omega_{c,X}=\omega_{c,0}+\epsilon$. Since $X<0,\theta(t)>0,\theta''(t)<0$ for $t \in (\sqrt{2c},1/\sqrt{2c})$ $\xi_1,\xi_2<0$ we have that $\epsilon$ is $O(\sqrt{|X|})$ and real and so $\omega_{c,X}\in \mathbb{H}_+$.   Moreover, we can write $\omega_{c,X}=r_{c,X}e^{\mathrm{i} \theta_{c,X}}$ with $r_{c,X}>0$ and $0<\theta_{c,X}<\pi/2$. Since both $|X-Y| <Cm^{-1}$ and $|\xi_1-\tilde{\xi}_1|<Cm^{-1}$ for some $C>1$ , we have that $\omega_{c,X}=\tilde{\omega}_{c,Y}$ because $g_{\xi_1+X,\xi_2+X}$ and $g_{\tilde{\xi}_1+Y,\tilde{\xi}_2+Y}$ differ by $O(m^{-1})$.

We now choose the contours for the integral in \eqref{thmproof:aroundRSsm:B1} so that they pass through the four roots $\omega_{c,X},-\omega_{c,X}, \overline{\omega_{c,X}} $ and $-\overline{\omega_{c,X}}$ and are given by the curves of steepest ascent and descent for $\mathrm{Re}g_{\xi_1+X,\xi_2+X}$.   By analyticity and since $|\xi_1-\xi_2|<m^{-\frac{1}{6}}$, we get similar steepest descent and ascent curves for $\mathrm{Re}g_{\xi_1+X,\xi_2+X}$ as given in \cite[Lemma 3.20]{CJ16}, namely
\begin{itemize}
\item  a contour of steepest ascent leaving $\omega_{c,X}$ at an angle $\theta_{c,X}-\pi/4$ ending at infinity,
\item a contour of steepest descent leaving $\omega_{c,X}$  at an angle $\theta_{c,X}+\pi/4$ ending at a cut and a descent
contour ending at  $\mathrm{i}/\sqrt{2c}$  traveling via the cut $\mathrm{i}[1/\sqrt{2c},\infty)$,
\item a contour of steepest ascent leaving $\omega_{c,X}$  at an angle $\theta_{c,X}+3\pi/4$ ending at a cut and an ascent
contour ending at $\mathrm{i}/\sqrt{2c}$ traveling via the cut $\mathrm{i}[0,\sqrt{2c}]$, and
\item a contour of steepest descent leaving $\omega_{c,X}$  at an angle $\theta_{c,X}-3\pi/4$ ending at zero.
\end{itemize}
see \cite[Lemma 3.20]{CJ16} for full details on obtaining these contours.  Deforming the contours of \eqref{thmproof:aroundRSsm:B1} to these curves of steepest descent and ascent means that the contours cross, which picks up a single integral term, $C_{\omega_{c,X}}(x,y)$ while the double integral term is $O(m^{-1/2})$ which follows readily by standard saddle point arguments in a rough region.  The length of the contour for $C_{\omega_{c,X}}(x,y)$, $|\Gamma_{\omega_{c,X}}|$, is $2|-\overline{\omega_{c,X}}-\omega_{c,X}|$ which is order $\sqrt{|X|}$ as shown above. Since the integrand for $C_{\omega_{c,X}}(x,y)$ is bounded away from zero, we get that $|C_{\omega_{c,X}}(x,y)|$ is order $\sqrt{|X|}$ as required.  
	\end{proof}

	\subsection{Proof of \cref{thm:previousRSasymptotics}}
	\label{S:thmpreviousRSasymptotics}
	Here, we give the proof of \cref{thm:previousRSasymptotics}.  
	\begin{proof}
		The first equation is simply a restatement of the first equation in \cref{lem:switchcontoursBtilde}. The second equation, holds for $-C<\alpha_x,\alpha_y<C$ by \cite[Theorem 2.7]{CJ16}. To extend the range, the same proof as the one given in \cite[Proposition 3.4]{CJ16} holds; see \cite[Proposition 3.4]{CJ16} for further details as we omit this computation. 
		
		The final equation holds for $-C<\alpha_x,\alpha_y<C$ by \cite[Theorem 2.7]{CJ16}.  We need to check the equation when either $C<\alpha_x<\log^3 m$ or $C<\alpha_y<\log^3 m$ (or both), that is, we need to check that the asymptotics has the same decay as the one for the extended Airy kernel.  We use a refined asymptotic analysis of the one used to prove  \eqref{thm:eq:KinversearoundRS:smclosetoRS}, checking only the case $C<\alpha_x<\log^3 m$ as the others are similar.
		
		We start with \eqref{eq:prev:B}. We note that with the scaling given by \eqref{eq:def:rsrefined} with $f_x=(1,2\eps)$, we have after expanding out that 
		\begin{equation*}
			\begin{split}
				&\log H_{x_1+1,x_2}(\omega_1)=2m g_{\xi_c+X,\xi_c+X}(\omega_1)-\log G(\omega_1)\\
				&-\big(2\lfloor \hat{\beta}_x \lambda_2 (2m)^{\frac{2}{3}} +k_x \lambda_2 \log^2 m\rfloor \big) \big(\log G(\omega_1)+\log (G(\omega_1^{-1})) \big)
			\end{split}
		\end{equation*}
		where $X=\frac{1}{2(2m)}\lfloor  \alpha_x \lambda_1 (2m)^{1/3} \rfloor$ and 
		\begin{equation*}
			g_{\xi_c+X,\xi_c+X}(\omega_1)=g_{\xi_c,\xi_c}(\omega_1)-X\log G(\omega_1)+X\log G(\omega_1^{-1}).
		\end{equation*}
		The $\omega_2$ integral has the exact same computation as the one given in \cite[Proposition 3.17]{CJ16} so we primarily focus on $\omega_1$ terms. 
		
		We have that the critical point of $g_{\xi_c+X,\xi_c+X}(\omega_1)$ is given by $\omega_{c,X}$ where $\omega_{c,X}$ satifies the statements given in \cref{lem:asym:contourssm} and \cref{lem:asym:secondderclosesm}.   We deform the contour with respect to $\omega_1$ to the steepest desecent contour given in \cref{lem:asym:contourssm} while we deform the contour with respect to $\omega_2$ to the steepest ascent contour given in \cite[Lemma 3.15]{CJ16}. 
		
		We next split the integral to a local contribution around the saddle points $\omega_{c,X}$ and $\tilde{\omega}_{c,Y}$, determined in \eqref{lem:asym:contourssm}, and a remaining contribution which is exponentially small which can be ignored; see \cite[Proposition 3.12]{CJ16}. For the local contribution, we apply a change of variables $\omega_1=\omega_{c,X}+(2m)^{-1/3}c_0 w$ and $\omega_2=\mathrm{i}+(2m)^{-1/3}c_0 z$ where $|w|,|z|<m^{\delta}$ for $\delta>0$.  We get from a Taylor series 
		\begin{equation*}
			\begin{split}
				&\frac{H_{x_1+1,x_2}(\omega_1)}{H_{y_1,y_2+1}(\omega_2)}=\frac{H_{x_1+1,x_2}(\omega_{c,X})}{G(\mathrm{i}^{-1})^{\frac{y_2-y_1+1}{2}}}  \exp \Bigg(g_{\xi_c+X,\xi_c+X}''(\omega_{c,X}) c_0^2w^2\\
				&-2\lfloor \hat{\beta}_x \lambda_2 (2m)^{\frac{1}{6}} \rfloor  \bigg(-\frac{1}{\sqrt{\omega_{c,X}^2+2c}}+\frac{1}{\omega_{c,X}^2\sqrt{\omega_{c,X}^{-2}+2c}} \bigg)c_0w\\
				&+\frac{\mathrm{i}}{3}{z^3} +\hat{\beta}_y z^2 -\mathrm{i} \alpha_y z +\mathrm{Err}\Bigg)
			\end{split}
		\end{equation*}
		where $\mathrm{Err}$ is equal to $m^{-1/2}w^3 R_m(w)+m^{-1/3}z^4 S_m(z)$ and both $R_m(w)$ and $S_m(z)$ tend to constants as $m$ tends to infinity for $|z|,|w| < m^\delta$.   We also have that 
		\begin{equation*}
			\frac{V_{\eps_1,\eps_2}(\omega_1,\omega_2)}{\omega_2-\omega_1}=\frac{V_{\eps_1,\eps_2}(\omega_{c,X},\mathrm{i})}{\mathrm{i}-\omega_{c,X}}+O(m^{-\frac{1}{3}}).
		\end{equation*}
		Substituting both of the two expansions back into \eqref{eq:prev:B}, along with the local change of variables, after extending the contours to infinity, we compute the Airy and Gaussian integrals with respect to $z$ and $w$ respectively.     We have that $\omega_{c,X}=\mathrm{i}+c_2\mathrm{i}\sqrt{X}$ from \cref{lem:asym:secondderclosesm}. In fact, from \eqref{lemproof:eq:secondderclosesm} with $\xi_1=\xi_2=\xi_c$, $t_{c,0}=1$, $t_{c,X}=c_2\sqrt{X}$ in that equation,  we get that $c_2=-\frac{(1-2c)^{3/4}}{\sqrt{c(1+c)}}$. Applying another Taylor series and computing the relevant derivatives gives
		\begin{equation*}
			\begin{split}
				& H_{x_1+1,x_2}(\omega_{c,X})= H_{x_1+1,x_2}(\mathrm{i}) \exp \Bigg(-2m \frac{2\mathrm{i}c(1+c)}{3(1-2c)^2} (\mathrm{i} c_2 \sqrt{X})^3 \\
				&-\frac{2 \mathrm{i}  \alpha_xm^{1/3} \lambda_1}{\sqrt{1-2c}}(\mathrm{i} c_2 \sqrt{X}) -\frac{2 c \hat{\beta}_x\lambda_2m^{2/3}}{(1-2c)^{3/2}}(\mathrm{i} c_2 \sqrt{X})^2 +o(1)\Bigg)
			\end{split}
		\end{equation*}
		as $m\to \infty$.  Plugging in the constants for $\lambda_1$, $\lambda_2$ and $c_2$ gives
		\begin{equation*}
			H_{x_1+1,x_2}(\omega_{c,X})= H_{x_1+1,x_2}(\mathrm{i}) \exp \big(-\frac{2}{3} \alpha_x^{3/2} + \hat{\beta}_x \alpha_x \big).
		\end{equation*}
		We also have the coefficient term of 
		\begin{equation*}
			\begin{split}
				&\frac{V_{\eps_1,\eps_2}(\omega_{c,X},\mathrm{i})}{\mathrm{i}-\omega_{c,X}} \frac{c_0 m^{-5/6}}{(2\pi)^{3/2}\sqrt{-g_{\xi_c+X,\xi_c+X}''(\omega_{c,X})}}\\
				&=\frac{V_{\eps_1,\eps_2}(\mathrm{i},\mathrm{i})}{-c_2 \mathrm{i}\sqrt{X}} \frac{c_0m^{-5/6}}{(2 \pi)^{3/2}\sqrt{\frac{4c(c+1)}{(1-2c)^2}c_2}X^{1/4}}+o(1)\\
				&=\frac{V_{\eps_1,\eps_2}(\mathrm{i},\mathrm{i}) c_0m^{-1/3}2^{3/4}}{-c_2\mathrm{i}(2 \pi)^{3/2}\sqrt{\frac{4c(c+1)}{(1-2c)^2}c_2}\alpha^{3/4}\lambda_1^{3/4}}+o(1)\\
			\end{split}
		\end{equation*}
		Multiplying the above two equations, expanding out the constants and multiplying the omitted coefficients gives the same decay and coefficients as the third equation of \cref{thm:previousRSasymptotics} when $C<\alpha_x <\log^3m$ as required.  
	\end{proof}

	\begin{appendix}
		\label{A:limit-shape-curve}

		\section{Limit Shape Curve}
		
		The limit shape curve $\xi$ of the two-periodic Aztec diamond with corners $(-1,-1)$, $(1,-1)$, $(1,1)$ and $(-1,1)$ with $c=a/(1+a^2)$ is given by the equation
		\begin{equation}\label{eq:asymp:limitshape}
			\begin{split} &64 c^6 \left(-1+\xi _1^2\right) \left(-1+\xi _2^2\right)-\left(\xi _1^4+\left(-1+\xi _2^2\right){}^2-2 \xi _1^2 \left(1+\xi _2^2\right)\right){}^2\\ &
				-16 c^4 \left(3 \left(-1+\xi _2^2\right){}^2+\xi _1^2 \left(-6+27 \xi _2^2-20 \xi _2^4\right)+\xi _1^4 \left(3-20 \xi _2^2+16 \xi _2^4\right)\right) \\
				&
				-4 c^2 \left(3 \left(-1+\xi _2^2\right){}^3+\xi _1^6 \left(3+8 \xi _2^2\right)+\xi _1^4 \left(-9+13 \xi _2^2-16 \xi _2^4\right)\right. \\
				&\left.+\xi _1^2 \left(9-30 \xi _2^2+13 \xi _2^4+8 \xi _2^6\right)\right) =0.
			\end{split}
		\end{equation}
		The set of solutions of the above equation has two connected components, which are given by two simple closed curves surrounding the origin. The inner of these two curves is the \textit{rough-smooth limit curve}, and the out curve is the \textit{frozen-rough limit curve}. Only the rough-smooth limit curve is relevant in our setting.

		\section{The formula for $Y_{\gamma_1,\gamma_2}^{\varepsilon_1,\varepsilon_2}$ }\label{sec:appendix:Y}
		Let
		\begin{equation*}
			\begin{split}
				f_{a,b}(u,v)=&
				\left(2 a^2 u v+2 b^2 u v-a b \left(-1+u^2\right) \left(-1+v^2\right)\right)\\&\times \left(2 a^2 u v+2 b^2 u v+a b \left(-1+u^2\right) \left(-1+v^2\right)\right).
			\end{split}
		\end{equation*}
		Define the following rational functions:
        \begin{align*}
            \mathtt{y}_{0,0}^{0,0}(a,b,u,v) &=  \frac{1}{4 \left(a^2+b^2\right)^2f_{a,b}(u,v)}\\
            &\times \Big(
				2 a^7 u^2 v^2-a^5 b^2 \left(1+u^4+u^2 v^2-u^4 v^2+v^4-u^2 v^4\right) \\
				&-a^3 b^4 \left(1+3 u^2+3 v^2+2 u^2 v^2+u^4 v^2+u^2 v^4-u^4 v^4\right) \\
                &-a b^6 \left(1+v^2+u^2 +3u^2 v^2\right)
				\Big),\\
                \mathtt{y}_{0,1}^{0,0}(a,b,u,v)&= \frac{a \left(b^2+a^2 u^2\right) \left(2 a^2 v^2+b^2 \left(1+v^2-u^2+u^2 v^2\right)\right)}{4 \left(a^2+b^2\right)f_{a,b}(u,v)}, \\
                \mathtt{y}_{1,0}^{0,0}(a,b,u,v)&= \frac{a \left(b^2+a^2 v^2\right) \left(2 a^2 u^2+b^2 \left(1-v^2+u^2+u^2 v^2\right)\right)}{4 \left(a^2+b^2\right)f_{a,b}(u,v)}, \\
                \mathtt{y}_{1,1}^{0,0}(a,b,u,v)&= \frac{ a \left(2 a^2 u^2 v^2+b^2 \left(-1+v^2+u^2+u^2 v^2\right)\right) }{4 f_{a,b}(u,v)}.\\
        \end{align*}
		For $i,j \in \{0,1\}$,  define $\mathtt{y}^{0,1}_{i,j}(a,b,u,v)$, $\mathtt{y}^{1,0}_{i,j}(a,b,u,v)$ and $\mathtt{y}^{1,1}_{i,j}(a,b,u,v)$ by
		\begin{equation*}
        \begin{split}
			\mathtt{y}^{0,1}_{i,j}(a,b,u,v)&=\frac{\mathtt{y}^{0,0}_{i,j}(b,a,u,v^{-1})}{v^2}, \\
			\mathtt{y}^{1,0}_{i,j}(a,b,u,v)&=\frac{\mathtt{y}^{0,0}_{i,j}(b,a,u^{-1},v)}{u^2}, \\
			\mathtt{y}^{1,1}_{i,j}(a,b,u,v)&=\frac{\mathtt{y}^{0,0}_{i,j}(a,b,u^{-1},v^{-1})}{u^2v^2}.
            \end{split}
		\end{equation*}
		When $b=1$, we write $y^{\eps_1,\eps_2}_{i,j}(a,1,u,v)=y^{\eps_1,\eps_2}_{i,j}(u,v)$.
		We then have
		\begin{equation}
			\label{asfo:eq:Y}
			\begin{split}
				&Y^{\eps_1,\eps_2}_{\gamma_1,\gamma_2}(u_1,u_2)= Y^{\eps_1,\eps_2}_{\gamma_1,\gamma_2}(a,u_1,u_2)=(1+a^2)^2  \\
				&\times (-1)^{\eps_1 \eps_2+\gamma_1(1 +\eps_2)+\gamma_2(1+\eps_1)}
				s(-\mathrm{i}u_1)^{\gamma_1} s(-\mathrm{i} u_2)^{\gamma_2} \mathtt{y}_{\gamma_1,\gamma_2}^{\eps_1,\eps_2} (-\mathrm{i}u_1 ,- \mathrm{i} u_2)u_1^{\eps_1} u_2^{\eps_2},
			\end{split}
		\end{equation}
		where $s$ is defined in~\eqref{eq:prev:mu_and_s}.

	\end{appendix}

	\bibliographystyle{alpha}
	\bibliography{Biblio}
\end{document}